\documentclass[10pt,reqno]{amsart}

\usepackage{amsfonts}
\usepackage{amsmath}
\usepackage{amssymb}
\usepackage{amsthm}
\usepackage{color}
\usepackage{enumitem}
\usepackage{float}
\usepackage{chngcntr}
\usepackage{setspace}
\usepackage{tikz}
\usepackage{relsize}
\usepackage{stmaryrd}

\usetikzlibrary{decorations.pathreplacing}

\usepackage[pdfencoding=auto]{hyperref}
\hypersetup{
  colorlinks,
  linkcolor={blue!50!black},
  citecolor={green!40!black},
  urlcolor={blue!80!black}
}

\theoremstyle{plain}
\newtheorem{lemma}{Lemma}[subsection]
  \newtheorem{proposition}[lemma]{Proposition}
  \newtheorem{corollary}[lemma]{Corollary}
  \newtheorem{theorem}[lemma]{Theorem}
  \newtheorem*{theorem*}{Theorem}

  \newtheorem*{fact*}{Fact}
  \newtheorem*{claim*}{Claim}
  \newtheorem{claim}[lemma]{Claim}

\theoremstyle{definition}
  \newtheorem{definition}[lemma]{Definition}

\theoremstyle{remark}
  \newtheorem{remark}[lemma]{Remark}
  \newtheorem{example}[lemma]{Example}

\numberwithin{equation}{section}

\DeclareMathOperator*{\EEE}{\mathlarger{\mathbb{E}}}

\setenumerate{label=\textup{(}\roman*\textup{)}} %chktex 9

\makeatletter
\newsavebox{\@brx}
\newcommand{\llangle}[1][]{\savebox{\@brx}{\(\m@th{#1\langle}\)}%
  \mathopen{\copy\@brx\kern-0.5\wd\@brx\usebox{\@brx}}}
\newcommand{\rrangle}[1][]{\savebox{\@brx}{\(\m@th{#1\rangle}\)}%
  \mathclose{\copy\@brx\kern-0.5\wd\@brx\usebox{\@brx}}}
\makeatother
\title[Quantitative bounds in the inverse theorem over cyclic groups]{Quantitative bounds in the inverse theorem for the Gowers $U^{s+1}$-norms over cyclic groups}
\author{Freddie Manners}
\address{Freddie Manners, UC San Diego, 9500 Gilman Drive \#0112, La Jolla, CA 92093, USA}
\email{fmanners@ucsd.edu}

\begin{document}

\newcommand{\eps}[0]{\varepsilon}

\newcommand{\CC}[0]{\mathbb{C}}
\newcommand{\EE}[0]{\mathbb{E}}
\newcommand{\FF}[0]{\mathbb{F}}
\newcommand{\NN}[0]{\mathbb{N}}
\newcommand{\PP}[0]{\mathbb{P}}
\newcommand{\QQ}[0]{\mathbb{Q}}
\newcommand{\RR}[0]{\mathbb{R}}
\newcommand{\TT}[0]{\mathbb{T}}
\newcommand{\ZZ}[0]{\mathbb{Z}}
\newcommand{\cA}[0]{\mathcal{A}}
\newcommand{\cB}[0]{\mathcal{B}}
\newcommand{\cC}[0]{\mathcal{C}}
\newcommand{\cE}[0]{\mathcal{E}}
\newcommand{\cF}[0]{\mathcal{F}}
\newcommand{\cG}[0]{\mathcal{G}}
\newcommand{\cH}[0]{\mathcal{H}}
\newcommand{\cP}[0]{\mathcal{P}}
\newcommand{\cS}[0]{\mathcal{S}}
\newcommand{\fb}[0]{\mathfrak{b}}
\newcommand{\fd}[0]{\mathfrak{d}}
\newcommand{\ff}[0]{\mathfrak{f}}
\newcommand{\fs}[0]{\mathfrak{s}}

\newcommand{\HK}[0]{\operatorname{HK}}
\newcommand{\LHS}[0]{\operatorname{LHS}}
\newcommand{\RHS}[0]{\operatorname{RHS}}
\newcommand{\id}[0]{\operatorname{id}}
\newcommand{\im}[0]{\operatorname{im}}
\newcommand{\poly}[0]{\operatorname{poly}}
\newcommand{\ev}[0]{\operatorname{ev}}
\newcommand{\spn}[0]{\operatorname{span}}
\newcommand{\codim}[0]{\operatorname{codim}}

\newcommand{\tdots}[0]{\mathinner{{\ldotp}{\ldotp}{\ldotp}}}
\newcommand{\wt}[0]{\widetilde}

\newcommand{\heis}[3]{ \left(\begin{smallmatrix} 1 & \hfill #1 & \hfill #3 \\ 0 & \hfill 1 & \hfill #2 \\ 0 & \hfill 0 & \hfill 1 \end{smallmatrix}\right)  }

\newcommand{\uppar}[1]{\textup{(}#1\textup{)}} %chktex 9

\begin{abstract}
  We provide a new proof of the inverse theorem for the Gowers $U^{s+1}$-norm over groups $H=\ZZ/N\ZZ$ for $N$ prime.  This proof gives reasonable quantitative bounds (the worst parameters are double-exponential), and in particular does not make use of regularity or non-standard analysis, both of which are new for $s \ge 3$ in this setting.
\end{abstract}

\maketitle

\setcounter{tocdepth}{2}
\tableofcontents
\parskip1mm
\onehalfspace{}

\section{Introduction}%
\label{sec:intro}

\subsection{Main result}
Our main theorem is the following quantitative version of the inverse theorem for the Gowers norms over the cyclic group $\ZZ/N\ZZ$, for $N$ prime.

\begin{definition}
  Let $s \ge 1$ be an integer and $H$ an abelian group.  By a \emph{nilsequence} on $H$ of degree $s$, dimension $D$, complexity $M$ and parameter $K$, we mean a function $\psi \colon H \to \CC$ given by $\psi= F \circ p$, where:---
  \begin{itemize}
    \item there is a nilmanifold $G/\Gamma$ of degree $s$, dimension $D$ and complexity $M$;
    \item $p \colon H \to G/\Gamma$ is a polynomial map; and
    \item $F \colon G/\Gamma \to \CC$ is a $K$-Lipschitz function.
  \end{itemize}
  We say $\psi$ is a \emph{one-bounded} nilsequence if furthermore $\|F\|_\infty \le 1$.
\end{definition}

\begin{theorem}%
  \label{thm:main}
  Fix an integer $s \ge 1$. Let $N$ be a \uppar{large} prime and write $H = \ZZ/N\ZZ$.  If $f \colon H \to \CC$ is a function such that $\|f\|_\infty \le 1$ and $\|f\|_{U^{s+1}} \ge \delta$, then there exists a one-bounded nilsequence $\psi \colon H \to \CC$ with degree $s$, dimension $D = O_s(\delta^{-O_s(1)})$, parameter $K$ and complexity $M$,
  such that
  \[
    \left| \EE_{x \in H} f(x) \overline{\psi(x)} \right| \ge \eps;
  \]
  where if $s \le 3$ then
  \[
    \eps^{-1}, K, M \le \exp\big(O(\delta)^{O(1)}\big)
  \]
  and if $s \ge 4$ then
  \[
    \eps^{-1}, K, M \le \exp\exp\big(O_{s}(\delta)^{O_s(1)}\big) \ .
  \]
\end{theorem}

The various technical and not entirely standard terms used here (primarily ``complexity'') are discussed in Appendix~\ref{app:nilmanifolds}.  Although they form part of the statement of our main result, the precise definitions are not particularly relevant for the purposes of most of the paper.  We assume the reader is familiar with Gowers uniformity norms $U^{s+1}$ in general; see e.g.\ \cite[\S 11.1]{tao-vu} or~\cite[Appendix B]{gt-primes}.

The case $s=1$ of Theorem~\ref{thm:main} is an exercise in Fourier analysis. The case $s=2$ was proven by Green and Tao~\cite{green-tao-u3}, and it is known~\cite{gt-equiv} that the bounds in that case are the same up to polynomial losses as the best available bounds in a Freiman--Ruzsa theorem, which---following work of Sanders~\cite{sanders}---are significantly better than the ones we give above.  The cases $s=3$ and $s\ge 4$ were proven by Green, Tao and Ziegler (in the equivalent formulation of functions on the interval $\{1,\dots,N\}$) in~\cite{gtz-u4} and~\cite{gtz} respectively.  For $s=3$ the quantitative bounds are extremely poor, owing to repeated use of regularity-type methods; for $s=4$ the proof is phrased in terms of non-standard analysis and so formally gives no bounds at all, and informally could perhaps be rephrased to give a bound far up in the Ackerman hierarchy.  A different but similarly ineffective proof for general $s$ arises from work of Szegedy (in parts together with Anatol{\'\i}n Camarena)~\cite{szegedy,cs}; see~\cite{candela-1,candela-2,gmv-1,gmv-2,gmv-3} for related work.

Working over groups $H=\FF_p^n$ for $p$ fixed and $n$ large, the analogue of Theorem~\ref{thm:main}, with no quantitative bounds, was proven by Tao and Ziegler~\cite{tz}.  Recently, Gowers and Mili\'cevi\'c~\cite{gowers-milicevic} obtained a quantitative result in this setting for $s=3$ (and $p \ge 5$), where the relevant parameters (phrased rather differently to here, but corresponding primarily to the parameter $\eps^{-1}$ above) are given a double-exponential bound.
(The case $p =2,3$, with a triple-exponential bound, was later shown by Tidor~\cite{tidor}.)
Since the first draft of this paper, this was extended to all $s \ge 3$~\cite{gowers-milicevic-2}, with the quality of the bounds a tower of exponentials of fixed height depending on $s$.
There are many points of comparison between this work and~\cite{gowers-milicevic,gowers-milicevic-2}; however, it is not straightforward to make their methods work in the cyclic setting (even for $s=3$), and not straightforward to make the methods of this paper work over finite fields, so these results are related but nonetheless disjoint.

Recent work of Kazhdan and Ziegler~\cite{kz1,kz2,kz3,kz4}, again in the finite field setting $H=\FF_p^n$, also has parallels with the approach here.  These works do not prove an analogue of Theorem~\ref{thm:main} directly, but from part of a programme concerning the structure theory of approximate polynomials, which is the topic of much of this paper (see Section~\ref{subsec:approx-intro} below).  The inverse theorem for the $U^{s+1}$-norms over $\FF_p^n$ is a significant motivation and planned application of their work.  Again, the settings $H=\ZZ/N\ZZ$ and $H=\FF_p^n$ are analogous in some ways but different in others, so none of these results can currently be made to overlap.

\subsection{Applications}

The machinery of Theorem~\ref{thm:main} and its variants, sometimes known as ``higher order Fourier analysis'', has been used to prove some substantial results, a notable example being the work of Green and Tao on asymptotic counts of linear configurations in the primes~\cite{gt-primes}.
The use of an ineffective version of Theorem~\ref{thm:main} was one factor prohibiting a quantitative error term in that theorem.
There is another unrelated obstruction, concerning Siegel zeros; however, if one assumes (say) the generalized Riemann hypothesis, then in principle Theorem~\ref{thm:main} should give a quantitative error term in the main theorem of~\cite{gt-primes}.
Moreover, since the first draft of this paper, Tao and Ter\"av\"ainen~\cite{tao-teravainen} removed the Siegel zero problem using a new argument, and so---using Theorem~\ref{thm:main}---obtained an unconditional quantitative error term on the count of linear configurations in the primes.

We also note that, combining Theorem~\ref{thm:main} with the argument in~\cite{gt-szemeredi}, we can fairly cheaply obtain Szemer\'edi's theorem with a not ridiculous quantitative bound (a few iterated exponentials).  Of the many proofs of Szemer\'edi's theorem, only Gowers' original work~\cite{gowers-kap} gives a quantitative bound of similar shape.  This is maybe not a terribly inspiring application, in that the proof of Theorem~\ref{thm:main} incorporates many ideas from Gowers' argument~\cite{gowers-kap}, and that the latter also gives a strictly better bound; however, we mention this to illustrate the relative strength of Theorem~\ref{thm:main} in the context of other results in this area.

Finally, since the lack of effective bounds has been a frequent reason not to try applying the inverse theorem, the author is optimistic that Theorem~\ref{thm:main} will open up further applications where quantitative results were previously intractable.  (The survey~\cite{hatami-lovett} gives a flavour of some connected areas.)

\subsection{Approximate polynomials}%
\label{subsec:approx-intro}

The key ingredient in proving Theorem~\ref{thm:main} is a structure theorem for objects known as \emph{approximate polynomials}.  We briefly recall some definitions.

\begin{definition}
  Let $s \ge 0$ be an integer, $H$ a finite abelian group and $f \colon H \to A$ a function.  We write
  \begin{align*}
    \partial_h f \colon H &\to A \\
    x &\mapsto f(x) - f(x+h)
  \end{align*}
  for the discrete derivative of $f$.

  We say $f$ is an \emph{approximate} or \emph{1\% polynomial} of degree $s$ and parameter $\eps > 0$, if
  \[
    \left|\left\{ (x,h_1,\dots,h_{s+1}) \in H^{s+2} \colon \partial_{h_1} \dots \partial_{h_{s+1}} f(x) = 0 \right\}\right| \ge \eps |H|^{s+2} ;
  \]
in other words, if an $\eps$ fraction of the order $s+1$ derivatives of $f$ vanish.

  If $f$ is only defined on a subset $S \subseteq H$, we make the same definition, but only counting tuples $(x,h_1,\dots,h_{s+1})$ where the iterated derivative is well-defined.
\end{definition}

When, say, $H = \ZZ/N\ZZ = A$ for $N$ a prime (bigger than $s$) and $\eps=1$, it is a classical fact that this defines the class of polynomials $H \to A$  of degree at most $s$, i.e.\ $f(x) = a_0 + a_1 x + \cdots + a_s x^s$ for $a_0,\dots,a_s \in \ZZ/N\ZZ$.
When $\eps \approx 1$, the class of approximate polynomials consists precisely of functions which agree with some true polynomial at almost all values (see e.g.~\cite{akkl}).

When $\eps$ is small the class is richer.  The archetypal examples, for $A =\RR$, are the so-called \emph{bracket polynomials}. For instance, when $H = \ZZ/N\ZZ$, the function $x \mapsto \{ \beta x / N \}$ for $\beta \in \ZZ$, where $\{\cdot\}$ denotes the fractional part map $\RR \to [0,1)$, is an approximate polynomial for $s=1$ and, say, $\eps = 1/10$.  When $s=2$ more complicated examples arise, such as $x \mapsto \{ a \{ \alpha x / N \} \{ \beta x / N\} \}$ for $a \in \RR$ and $\alpha,\beta \in \ZZ$.  More generally, a bracket polynomial of degree $s$ is any algebraic expression built out of sums, products and fractional part maps; its degree $s$ is one more than the greatest width of product operations in a term. %chktex 9

Other examples include any function that takes only $100$ distinct values; or a function that agrees with a bracket polynomial on a positive density set and is random elsewhere; or a function that switches between a few different bracket polynomials at random (e.g., $f(x) = f_1(x) \text{ or } f_2(x)$, where $f_1,f_2$ are bracket polynomials and the choice is made arbitrarily for each $x$).

Our structure theorem essentially says that these are the only sources of examples, at least for certain groups $H$ and for $A=\RR$.  However, rather than stating the result in terms of bracket polynomials, we use the language of nilmanifolds, which is essentially equivalent; see~\cite[Thoerem A]{bergelson-leibman}.\footnote{In fact the results in that paper are not sharp enough with respect to the degree to show the version of equivalence we allude to here; however, the sharper version is true and in any case not logically required in this paper.}

\begin{definition}%
  \label{def:nil-poly}
  Let $H$ be an abelian group and $f \colon H \to \RR$ a function.  We say $f$ is a \emph{nil-polynomial} of degree $s$, dimension $D$ and complexity $M$ if there exist:
  \begin{itemize}
    \item a nilmanifold $G/\Gamma$ of degree $s$, dimension $D$ and complexity $M$;
    \item a function $r \colon H \to G$ such that $d_G(\id_G,r(x)) \le M$ for all $x \in H$, and $r \bmod \Gamma \colon H \to G/\Gamma$ is a polynomial map; and
    \item a polynomial map $F \colon G \to \RR_{(s)}$;
  \end{itemize}
  such that $f = F \circ r$.
\end{definition}
Again see Appendix~\ref{app:nilmanifolds} for the definitions of terms such as ``complexity'', $\RR_{(s)}$, ``polynomial map'', and the metric $d_G$ on $G$.

\begin{theorem}%
  \label{thm:1pc}
  Suppose $s \ge 1$ is an integer, $H = (\ZZ/N\ZZ)^d$ where $N$ is a large prime and $d$ is a fixed positive integer, $X \subseteq H$ is a subset and $f \colon X \to \RR$ is an approximate polynomial of degree $s$ and parameter $\eps > 0$.

  Then there is a parameter $M$, where if $s \le 2$ we have
  \[
    M \le \exp\big(O_d(\eps)^{O_d(1)}\big)
  \]
  and if $s \ge 3$ then
  \[
    M \le \exp\exp\big(O_{d,s}(\eps)^{O_{d,s}(1)}\big) \ ,
  \]
  such that the following holds: there exists a subset $X' \subseteq X$ of size $|X|' \gg_{s} \eps^{O_{s}(1)} |H|$, and a nil-polynomial $g \colon H \to \RR$ of degree $s$, dimension $D \ll_{d,s} (1/\eps)^{O_{d,s}(1)}$ and complexity $M$, such that $f|_{X'} = g|_{X'}$.
\end{theorem}

Briefly, $f$ agrees on a large set with a nil-polynomial.  Another point of view is that this states that a large part of $f$ can be extended to a genuine (i.e.\ not approximate) polynomial, if we are prepared to pass to some unusual kind of extension of $H$ by the nilpotent group $\Gamma$.

Since the notion of an approximate polynomial is formally very similar to the definition of the Gowers uniformity norms, it is not too surprising that there is a strong connection between the structure theory of approximate polynomials (e.g.\ Theorem~\ref{thm:1pc}) and the inverse theory of Gowers norms (e.g.\ Theorem~\ref{thm:main}).  It is slightly more surprising that the strongest connection is between the inverse theory in degree $s$ and local polynomials of degree $s-1$.  Indeed, in the case $s=2$, this relates the inverse theory of the $U^3$-norm with the theory of approximately linear functions, which in turn is closely related to Freiman's theorem; these connections form the basis of the proofs in~\cite{green-tao-u3,gt-equiv}.  However, the existing proofs of the inverse theorem for $s>2$ do not use such a correspondence.

Our two tasks are therefore (i) to prove Theorem~\ref{thm:1pc}, and (ii) to use it to deduce Theorem~\ref{thm:main}.

Although no reduction of the exact shape of (ii) appears in the literature, this kind of correspondence is closely related to aspects of previous works~\cite{green-tao-u3,gowers-kap,gowers-milicevic}, and is anticipated explicitly in~\cite{kz1,kz3}.  The proof here builds on the previous arguments.  That said, some genuinely new difficulties arise, and we introduce new techniques to handle them, usually of a technical nature.  All this is covered in Section~\ref{sec:deduction}, and we refer to the introduction to that section for further details.

By contrast, the proof of Theorem~\ref{thm:1pc} in Sections~\ref{sec:1pc-global},\ref{sec:hierarchy-improve},\ref{sec:algebraic-poly-system} contains most of the novel content of the paper. No structural results for approximate polynomials of degree $s>1$ were formerly known, and e.g.\ for $H=\FF_p^n$ for $n$ large no such result yet exists.  That said, proving an analogue of Theorem~\ref{thm:1pc} in this finite field setting is a primary motivation for the ongoing programme~\cite{kz1,kz2,kz3}, making this work closely related, although the issues that arise in the finite field setting often have a different flavour.  Similarly,~\cite{gowers-milicevic} treats a related bilinear problem, again over finite fields.

Even in the linear case $s=1$ of Theorem~\ref{thm:1pc} (corresponding to the $U^3$ inverse theorem) the proof here is in principle new, and hence embeds a ``new'' proof of Freiman's theorem (up to standard reductions).  Of course, strong analogies with existing proofs can be detected.

\subsection{Limitations and bounds}

We make some further brief remarks concerning the statements of Theorem~\ref{thm:main} and Theorem~\ref{thm:1pc}.

First, the quality of the bounds on the various quantities is almost certainly not best possible.  It should presumably be possible to take the dimensions $D$ to be logarithmic in $1/\delta$ and the remaining parameters quasi-polynomial.

The primary reason we only get polynomial and single exponential bounds respectively when $s=2$ or $s=3$ is well-understood: it is the same obstruction to getting good bounds in classical proofs of Freiman's theorem, and concerns a failure to split sets up into their weakly-interacting pieces.  Another way of saying this is that when $s=2$, where better bounds are available using Sanders' work on Freiman--Ruzsa~\cite{sanders}, the proof here does not make any attempt to use anything like Sanders' ideas, and defaults to something more like the prior Freiman--Ruzsa proofs.  It is unlikely to be possible to use Sanders' result directly here as the setting is formally different, but it is possible to speculate what it might mean to use some of these ideas indirectly.

The loss of the second exponential when $s \ge 4$ is much less conceptual: it is buried deep within a rather technical part of the argument aimed at recursively eliminating unwanted low-complexity identities between certain functions.  The author has tried to remove this issue without success. It is unclear whether there is some genuine obstruction or whether a modification of the approach could avoid this further loss.

We also comment further on the restriction to cyclic groups of prime order.  In fact the same methods would give an inverse theorem for other abelian groups of bounded rank which have no large subgroups (where ``large'' means index at most some function of $\delta$). However this case can be deduced from the prime cyclic case anyway by elementary manipulations, so for simplicity we do not state this.

 Given~\cite{gowers-milicevic,gowers-milicevic-2,kz1,kz2,kz3}, it would of course be of interest to have a uniform approach over general finite groups.   However, both hypotheses---bounded rank and no large subgroups---are used in ways that appear to be difficult to eliminate.  For general groups, it would almost certainly be necessary to replace the study of approximate polynomials $H \to \RR$ with ones $H \to A$ for other abelian groups $A$: e.g.\ when $H=\FF_p^n$ for $p$ fixed the richest theory concerns approximate polynomials $H \to \FF_p^{m}$ (as discussed in~\cite{kz1,kz2,kz3}).  However the fact that the codomain $\RR$ is $|H|$-divisible is used fairly crucially in the proof here to avoid genuine ``cohomological'' obstructions, and so such a generalization is far from straightforward.

The use of the ``no small subgroups'' hypothesis is more suspect because elementary manipulations of the kind mentioned previously allow us to pass from prime cyclic groups to non-prime ones fairly easily.  The use of this hypothesis is genuine but may be avoidable.

Having said all this, many of the methods in the paper work over arbitrary groups or achieve essentially optimal bounds on the problems they address. The author is optimistic that some of these methods have more to say concerning these and other problems.

\subsection{Structure of the paper}

The proof of Theorem~\ref{thm:1pc} splits into three pieces, covered by Sections~\ref{sec:1pc-global},~\ref{sec:hierarchy-improve} and~\ref{sec:algebraic-poly-system} respectively.

The first part, Section~\ref{sec:1pc-global}, has a combinatorial flavor. In the case $s=1$, it corresponds roughly to taking an approximate polynomial $f \colon X \to \RR$ of degree $1$, and upgrading it to a more regular object, namely a globally defined function $\wt f \colon H \to \RR$ whose derivatives $\partial_{h_1} \partial_{h_2} \wt f(x)$ take (almost surely) a bounded number of distinct values.  The conclusion for general $s$ is that an approximate polynomial can be extended to a global function associated to an object we term a \emph{polynomial hierarchy}.

Next, it is necessary to upgrade this global structure to a stronger version of itself; this is handled in Section~\ref{sec:hierarchy-improve}.  The enemy here is what we call \emph{degeneracies} or \emph{redundancies} in the polynomial hierarchy object, which allow a particular function value to be given more than one low-complexity description.  Removing these ambiguities is a technical and uninspiring task; however, the reward is to begin to identify algebraic structure---in the form of what are given then perhaps non-standard term \emph{cocycles}, following~\cite{cs}---in the strengthened polynomial hierarchy object.

The third task is to show that these more regular global objects correspond, essentially, to nil-polynomials.  This segment, covered in Section~\ref{sec:algebraic-poly-system}, contains the more algebraic content, where we use a few constructions in the category of nilmanifolds to build something that captures the structure emerging from previous arguments.  These constructions are not deep but are rather involved.  In particular some time must be spent supplying the quantitative nilmanifold data we require (i.e., checking bounds on the ``complexity'' of a nilmanifold), which requires the theory expounded in Appendix~\ref{app:nilmanifolds}.

The deduction of the inverse theorem (Theorem~\ref{thm:main}) from the structure theorem for 1\% polynomials (Theorem~\ref{thm:1pc}) appears in Section~\ref{sec:deduction}.  It is in some sense elementary in flavour, consisting of some Fourier analysis and many applications of the Cauchy--Schwarz inequality, and generally gives good bounds.  We also require some nilmanifold constructions at this point.

To keep this introduction manageable, each section and most subsections have their own mini-introduction, giving a heuristic outline of the key tasks and ideas, and further expository passages appear throughout the paper.  The reader could profitably read some or all of these remarks before reading any one section or proof in depth.

A few lemmas and definitions which are either especially technical or outside the main flow of the argument are relegated to the appendices.

\subsection{Notation and definitions}%
\label{subsec:notation}

We write $[k]$ to denote the set $\{1,\dots,k\}$.  As is usual, we write $O_{r_1,\dots,r_k}(1)$ to denote a positive quantity bounded above by a constant depending only on the parameters $r_1,\dots,r_k$.  We write $O_{r_1,\dots,r_k}(X)$ to mean $O_{r_1,\dots,r_k}(1) X$, and $X \ll_{r_1,\dots,r_k} Y$ to mean $X = O_{r_1,\dots,r_k}(Y)$.  Also, for $x \in \RR$ or $\RR/\ZZ$ we write $e(x) = \exp(2 \pi i x)$.

If $X$ is a finite set, we write $\EE_{x \in X} f(x)$ for the average $\frac1{|X|} \sum_{x \in X} f(x)$.  If $S \subseteq X$ is a subset, we write $\mu_X(S)$ for the density $\EE_{x \in X} 1_S(x) = |S|/|X|$ (or just $\mu(S)$ if $X$ is clear from context); i.e., the measure of $S$ under the uniform probability measure on $X$.

Given a logical condition $\cP$ depending on certain variables, we write $[\cP]$ to denote the function taking value $1$ when $\cP$ holds and $0$ when $\cP$ does not hold.

Throughout it will be convenient to use the following terminology regarding ``cubes'' or ``parallelepipeds'', in preference to phrasing things explicitly in terms of iterated derivatives.  This terminology is common in the ``cubespace'' literature~\cite{cs,candela-1,candela-2,gmv-1,gmv-2,gmv-3}; however, we use it here in an elementary fashion and will not require the more abstract axiomatic approach from that area.

We write $\llbracket k \rrbracket$ as shorthand for the discrete cube $\{0,1\}^k$.  If $A$ is a set, we use the notation $A^{\llbracket k \rrbracket}$ to denote the set of functions $\llbracket k \rrbracket \to A$.

For $\omega, \eta \in \llbracket k \rrbracket$ we write $\omega \subseteq \eta$ to denote the usual partial ordering $\omega_i \le \eta_i \ \forall i \in [k]$, and $|\omega|$ to mean $\sum_{i=1}^k \omega_i$.

Given an abelian group $H$, by default we use the notation $C^k(H)$ to denote the subset of $H^{\llbracket k \rrbracket}$ consisting of functions $c \colon \llbracket k \rrbracket \to H$ of the form
\begin{equation}
  \label{eq:cube-definition}
  c(\omega) = x + \omega \cdot h
\end{equation}
where $h=(h_1,\dots,h_k) \in H^k$, $x \in H$ and $\omega \cdot h$ is shorthand for $\sum_{i \colon \omega_i = 1} h_i$.  Elements of $C^k(H)$ are called \emph{$k$-cubes}.\footnote{The term \emph{$k$-dimensional parallelepipeds} is sometimes used and is possibly more accurate, but \emph{$k$-cubes} is now more standard in the literature as it consumes fewer trees.}  It is clear that $C^k(H)$ is a subgroup of $H^{\llbracket k \rrbracket}$ isomorphic to $H^{k+1}$ under the correspondence $c \leftrightarrow (x,h_1,\dots,h_k)$.  Given $(x,h_1,\dots,h_k) \in H^{k+1}$ we sometimes write $\angle(x;h_1,\dots,h_k) \in C^k(H)$ to denote the corresponding cube defined by~\eqref{eq:cube-definition}.

Given an abelian group $A$ and a function $f \colon H \to A$, the $k$-th iterated derivative $\partial^k f \colon C^k(H) \to A$ is defined by
\[
  \partial^k f (c) = \sum_{\omega \in \llbracket k \rrbracket} (-1)^{|\omega|} f(c(\omega)) .
\]
It is straightforward to verify that this does in fact correspond to an iterated discrete derivative, in that if $c = \angle(x; h_1,\dots,h_k)$ then $\partial^k f(c) = \partial_{h_1} \dots \partial_{h_k} f(x)$.

By a \emph{face} of $\llbracket k \rrbracket$ of dimension $d$, we mean a subset of the form $\{\omega \in \llbracket k \rrbracket \colon \eta_0 \subseteq \omega \subseteq \eta_1\}$, where $\eta_0,\eta_1 \in \llbracket k \rrbracket$, $\eta_0 \subseteq \eta_1$ and $|\eta_1| - |\eta_0| = d$.  A face of codimension $d$ is a face of dimension $k-d$.  A face $F$ is called an \emph{upper face} if $\eta_1 = (1,1,\dots,1)$ and a \emph{lower face} if $\eta_0 = (0,0,\dots,0)$. 

For any face $F$ of $\llbracket k \rrbracket$ of dimension $d$, $F= \{\omega \colon \eta_0 \subseteq \omega \subseteq \eta_1\}$, its \emph{active coordinates} are the indices $1 \le i_1 < \cdots < i_d \le k$ such that $(\eta_1)_i = 1$ but $(\eta_0)_i = 0$, i.e.\ the coordinates that are not constant on $F$.  We define a map
\begin{align*}
  \tau_F \colon \llbracket k \rrbracket &\to \llbracket d \rrbracket \\
  \omega &\mapsto (\omega_{i_1},\omega_{i_2},\dots,\omega_{i_d}) .
\end{align*}
Noting that $\tau_F|_F \colon F \to \llbracket d \rrbracket$ is a bijection, we thereby define the inverse map $\imath_F \colon \llbracket d \rrbracket \to \llbracket k \rrbracket$.  If $A$ is some set and $c$ is a configuration in $A^{\llbracket k \rrbracket}$, we abuse notation somewhat to write $c|_F \colon \llbracket d \rrbracket \to A$ for the configuration $c|_F(\omega) = c(\imath_F(\omega))$.

For $\eta \in \llbracket k \rrbracket$, we write $F_\eta$ for the lower face $\{ \omega \in \llbracket k \rrbracket \colon \omega \subseteq \eta\}$ and $F^\eta$ for the upper face $\{ \omega \in \llbracket k \rrbracket \colon \omega \supseteq \eta\}$.  Also, for $i \in [k]$, we write $F_i = \{ \omega \in \llbracket k \rrbracket \colon \omega(i) = 0\}$ and $F^i = \{ \omega \in \llbracket k \rrbracket \colon \omega(i) = 1\}$ as shorthand for the lower and upper codimension one faces in direction $i$.  Although these pieces of notation overlap to some extent, it will always be possible to distinguish them in context based on the type of object $\eta \in \llbracket k \rrbracket$ or $i \in [k]$ respectively (which is further hinted by the use of Greek or Latin alphabets).

If $A$ is some set, $c,c' \in A^{\llbracket k \rrbracket}$ and $i \in [k+1]$, we write $[c,c']_i$ to denote the configuration $\llbracket k+1 \rrbracket \to A$ given by
\[
  [c,c']_i(\omega) = \begin{cases} c(\tau_{F_i}(\omega)) &\colon \omega \in F_i \\ c'(\tau_{F^i}(\omega)) &\colon \omega \in F^i \end{cases}
\]
i.e.\ the configuration obtained by placing $c$ and $c'$ on opposite faces of $\llbracket k+1 \rrbracket$ in direction $i$.

If $F = \{\omega \colon \eta_0 \subseteq \omega \subseteq \eta_1\}$ is a face of $\llbracket k \rrbracket$ of dimension $d$ and $c \in A^{\llbracket d \rrbracket}$, we write $\square^{k-d}_F(c)$ to denote the duplicated configuration in $A^{\llbracket k \rrbracket}$ given by
\[
  \square^{k-d}_F(c)(\omega) = c(\tau_F(\omega))
\]
i.e.\ the configuration which restricts to $c$ on every face parallel to $F$.  We abbreviate $\square^{r}_F$ to $\square^r$ if doing so does not create ambiguity; e.g., whenever $d=0$, in which case $\square^k(x)$ is a constant configuration.

If $\sigma \colon [k] \to [k]$ is a bijection, we abuse notation to write $\sigma \colon \llbracket k \rrbracket \to \llbracket k \rrbracket$ for the map $\sigma(\omega) = \omega \circ \sigma$ (thinking of $\omega$ as a function $[k] \to \{0,1\}$), and call this a \emph{coordinate permutation} of the discrete cube $\{0,1\}^k$.  Given $1 \le i \le k$, the $i$-th \emph{coordinate reflection} is the map $\llbracket k \rrbracket \to \llbracket k \rrbracket$ sending $[c,c']_i$ to $[c',c]_i$.

\subsection{Acknowledgements}

The author is grateful to Thomas Bloom for detailed comments and corrections to an earlier draft of this paper, and to Aled Walker, Ben Green, Timothy Gowers, Tamar Ziegler, Jacob Fox, Luka Mili\'cevi\'c, Yonatan Gutman, P\'eter Varj\'u,  and many others for helpful interactions and discussions concerning this project.

\section{From 1\% polynomials to global structure}%
\label{sec:1pc-global}

Throughout this section, we let $H$ be an arbitrary finite abelian group.  

\subsection{Outline of the argument}%
\label{subsec:improve-outline}

Suppose $f \colon X \to \RR$ is an approximate polynomial as in Theorem~\ref{thm:1pc}.  It may only be defined on a small subset $X \subseteq H$; or even if $X=H$, it may take arbitrary, unstructured values away from a small subset.  By contrast, the statement of Theorem~\ref{thm:1pc} requires us to construct a global function $F \circ p$ from polynomial maps $p \colon H \to G/\Gamma$ and $F \colon G \to \RR$, which implicitly extends $f$ (possibly in a slightly ambiguous or arbitrary fashion) to a function on all of $H$.  A reasonable question is how this extension process will be achieved within the proof.

Our approach is to select some random translates $a_1,\dots,a_L \in H$, and consider functions $f_i \colon x \mapsto f(x+a_i)$ defined on $X - a_i$.  Provided $L$ is large enough, the domains of these translated copies of $f$ will cover almost all of $H$ (even if we throw away unstructured values).  Of course, there is an obstruction to glueing these into a globally defined function, as if $f_i$ and $f_j$ are both defined at some $x$ their values will probably not agree.

However, the failure of $f_i$ and $f_j$ to agree is measured by the function $x \mapsto f_i(x) - f_j(x) = f(x+a_i) - f(x+a_j)$, which is just a translate of the derivative $\partial_{a_j-a_i} f$.  For many values of $h$, $\partial_h f$ is an approximate polynomial of degree $s-1$; so, often $f_i$ and $f_j$ do in fact agree up to a lower-degree correction, whose structure may in turn be understood by induction on $s$.

Suppose then that $\wt{f}$ is a function on almost all of $H$ defined by glueing the functions $f_i$ together fairly arbitrarily.  We then wish to understand (almost all of) the derivatives $\partial^{s+1} \wt f$.  The exact statement is not straightforward, but very roughly we can use similar logic to deduce that that such derivatives may be described in terms of the lower-order functions $\partial_{a_j-a_i} f$ described earlier, in a low-complexity fashion.  The precise statement of this is Lemma~\ref{lem:extend-global}.

In the special case $s=1$, the corrections $f_i - f_j$ are constant functions much of the time, and the final conclusion is that we can extend (a large part of) $f$ to a function $\wt{f}$ with the property that for some not-too-large finite set $T \subseteq \RR$ we have $\partial^{2} \wt{f}(c) \in T$ for almost all $c \in C^2(H)$; i.e., the derivative $\partial^2 \wt f$ almost surely takes a bounded number of distinct values.  This case is useful to keep in mind.  In fact this statement can be proven using well-established tools, such as the Balog--Szemer\'edi--Gowers theorem and the Pl\"unnecke--Ruzsa inequalities~\cite{gowers-kap,green-tao-u3}; however, this method does not generalize well to larger $s$ and so we develop a different one.

Given $\wt{f}$ with this property, there is still a considerable amount of work required to deduce the kind of rigid algebraic structure we want, and this is the topic of the next two sections.

The argument in the remainder of the current section divides further into several stages:---
\begin{itemize}
  \item (discarding bad points) We consider the set of $(s+1)$-cubes $c \in C^{s+1}(H)$ such that $\partial^{s+1} f(c) = 0$, and introduce a notion of a \emph{system of cubes} to capture a robust notion of when cubes $c'$ of dimension between $0$ and $s+1$ are ``popular'' with respect to this set, with a view to discarding unpopular cubes.  This is Section~\ref{subsec:cube-system}.
  \item (Balog--Szemer\'edi--Gowers / connectivity step) To deal with the case where $f$ switches arbitrarily between several unrelated structured pieces, we divide the domain of $f$ into robustly connected pieces in some sense.  This is Section~\ref{subsec:split}.
  \item (extension) With these issues out of the way, the core extension argument described above has a hope of working, and we turn to this in Section~\ref{subsec:improve-core}.
  \item (iteration) In Section~\ref{subsec:hierarchy} we apply these results recursively to the lower-degree functions $\partial_{a_j-a_j} f$ alluded to above, and record precise definitions for the recursive structure (derived from a \emph{polynomial hierarchy}) this imposes on the original $\wt{f}$.
\end{itemize}

\subsection{Systems of cubes}%
\label{subsec:cube-system}

If $X \subseteq H$ and $f \colon X \to \RR$ is a 1\% polynomial of degree $s$ and parameter $\eps$, an important quantity to track is the set of cubes on which the derivative of $f$ vanishes.  Using the notation from Section~\ref{subsec:notation}, we define
\begin{equation}
  \label{eq:good-cubes}
  S_{s+1} = \{ c \in C^{s+1}(H) \cap X^{\llbracket s+1 \rrbracket} \colon \partial^{s+1} f(c) = 0 \}  ;
\end{equation}
so by our hypothesis, $|S_{s+1}| \ge \eps\, |C^{s+1}(H)|$.

It could be that a particular $x \in X$ is not a vertex of many of the cubes $c \in S_{s+1}$.  This will happen if $f$ is a structured function in some places and random elsewhere, and $x$ is one of the random points.  In any case, for such an $x$ we should assume that the value $f(x)$ is not trustworthy and discard it.  Similarly, it will be useful to consider which $1$-dimensional cubes $[x,y]$ are faces of many cubes $c \in S_{s+1}$: if $f$ switches between several unrelated structured functions as described above, then this tells us which pairs $x$ and $y$ are likely to be in the same structured piece.

This could motivate considering sets such as
\[
  S_0 = \left\{ x \in H \colon |\{ c \in S_{s+1} \colon c(0,0,\dots,0) = x \}| \ge (\eps/100) |H|^{s+1} \right\}
\]
or more generally
\[
  S_k = \left\{ c \in C^k(H) \colon |\{ c' \in S_{s+1} \colon c'|_{F} = c \}| \ge (\eps/100) |H|^{s+1-k} \right\}
\]
where (say) $F = \{\omega \in \llbracket s+1 \rrbracket \colon \omega_{k+1} = \cdots = \omega_{s+1} = 0\}$, as giving a notion of ``good'' or ``popular'' cubes of smaller dimension.

The issue with this definition is that a cube in $S_k$ could in principle have a vertex that is not in $S_0$, and so on.  So, when we delete unpopular points not in $S_0$, we also lose some potentially popular $k$-cubes; then this might in turn cause previously popular points in $S_0$ to become unpopular, and so on.

The following definitions and results are designed to address these irritations.

\begin{definition}%
  \label{def:cube-system}
  A \emph{system of cubes} in $H$ of degree $s \ge 0$ and parameter $\delta > 0$ is a collection of subsets $S_k \subseteq C^k(H)$ for $0 \le k \le s+1$, with the following properties (see Section~\ref{subsec:notation} for definitions of the terms used):---
  \begin{enumerate}[label=(\roman*)]
    \item (symmetry) for each $k$, $0 \le k \le s+1$, if $c \in S_k$ then so is any coordinate permutation or coordinate reflection of $c$;

    \item (faces) for each $k$, $1 \le k \le s+1$, if $c \in S_k$ and $F$ is any face of $\llbracket k \rrbracket$ of dimension $d$, then $c|_F \in S_d$;

    \item (popularity) for each $0 \le k \le s$ and $c \in S_k$ there are at least $\delta |H|$ values $h \in H$ such that $[c, c+\square^k(h)]_{k+1} \in S_{k+1}$.
  \end{enumerate}
\end{definition}

\begin{remark}%
  \label{rem:cube-system-large}
  Note that if any $S_k$ is non-empty then $S_0$ is non-empty (by (ii)). For any $x \in S_0$ there are at least $\delta |H|$ cubes $[x,x+h] \in S_1$ (by (iii)), and hence at least $\delta |H|$ elements $x+h \in S_0$ (by (ii)).  Hence, $\mu(S_0) \ge \delta$, and by repeated application of (iii) it follows that $\mu(S_k) \ge \delta^{k+1}$ for each $k$.  Hence, a non-empty system of cubes automatically has positive density in each dimension.
\end{remark}

We first show that we can always obtain such a system from suitable large sets of cubes.

\begin{lemma}%
  \label{lem:cube-system}
  If $s \ge 0$ and $S_0,\dots,S_{s+1}$ are subsets $S_k \subseteq C^{k}(H)$ satisfying parts \uppar{i} and \uppar{ii} of Definition~\ref{def:cube-system}, and such that $\mu(S_{s+1}) \ge \eps$, then there exists a non-empty system of cubes $S'_0,\dots,S'_{s+1}$ with parameter $\delta \gg_s \eps$ such that $S'_{k} \subseteq S_k$ for each $k$.
\end{lemma}
\begin{proof}
  This is very similar to a proof that a dense graph has a subgraph with large minimum degree.  Essentially, we delete cubes whenever we have to and show there is something left at the end.

  Set $\delta = 2^{-2s-4} (s+1)^{-2} \eps$.  We now iteratively perform the following procedure: for some $k$, $0 \le k \le s$, take all cubes $c \in S_k$ which fail to satisfy condition (iii), and remove them.  Next, remove all cubes $c' \in S_{k+1}$ which restrict to any such $c$ on a face of codimension $1$; by hypothesis there are at most $2(k+1) \delta |H|$ of these (since a $(k+1)$-cube has $2(k+1)$ faces of codimension $1$, and for each face and each $c$ there are at most $\delta |H|$ possibilities for $c'$, by symmetry (i) and the failure of (iii) for $c$).  Finally, also remove all cubes in $S_\ell$ for $k+2 \le \ell \le s+1$ which have any of these $(k+1)$-cubes $c'$ as a face.  Since crudely an $\ell$-cube has at most $4^\ell$ faces, at most $4^{\ell}\cdot 2 \ell\, \delta\, |H|^{\ell-k}$ cubes are deleted from $S_{\ell}$ for each original $c \in S_k$.

  Note that after one complete iteration, conditions (i) and (ii) are still satisfied.  This process is repeated until nothing is deleted for any value of $k$.  The final sequence is now a system of cubes $S'_0,\dots,S'_{s+1}$ with parameter $\delta$ (with $S'_k \subseteq S_k$ for each $k$) by construction; the only danger is that it might be empty.

  Since each cube $c \in C^k(H)$ ($1 \le k \le s$) can be removed at most once, the total number of cubes removed from $C^{s+1}(H)$ in the whole process is at most
  \[
    \sum_{k=0}^{s} 4^{s+1}\cdot 2 (s+1)\, \delta\, |H|^{k+1} |H|^{s+1-k} = 2^{2s+3} (s+1)^2\, \delta\, |H|^{s+2} \le (\eps / 2) |H|^{s+2}
  \]
  and hence the final set $S'_{s+1}$ has size at least $|S_{s+1}| - (\eps/2)|H|^{s+2} > 0$ by hypothesis.
\end{proof}

The following is then immediate.
\begin{corollary}%
  \label{cor:cube-system}
  If $X \subseteq H$ and $f \colon X \to \RR$ is an approximate polynomial with degree $s \ge 0$ and parameter $\eps > 0$, then there exists a system of cubes $S_0,\dots,S_{s+1}$ with parameter $\delta \gg_s \eps$ such that $S_0 \subseteq X$ and $\partial^{s+1} f(c) = 0$ for all $c \in S_{s+1}$.
\end{corollary}
\begin{proof}
  We apply Lemma~\ref{lem:cube-system} to the sets $S_k = X^{\llbracket k \rrbracket}$ for $0 \le k \le s$, and
  $S_{s+1}$ as in~\eqref{eq:good-cubes}.
\end{proof}

When the initial sets $S_k \subseteq C^{k}(H)$ in Lemma~\ref{lem:cube-system} all have density close to $1$, we might hope that the densities of $S'_k$ and the parameter $\delta$ in the conclusion could also be taken close to $1$; however, Lemma~\ref{lem:cube-system} does not give this.  Proving such a result is more involved but indispensable throughout the paper for handling error sets.

\begin{lemma}%
  \label{lem:cube-system-ae}
  Suppose $s \ge 0$ and $S_0,S_1,\dots,S_{s+1}$ are sets, $S_k \subseteq C^k(H)$, such that $\mu(S_k) \ge 1-\eps$ for each $k$, $0 \le k \le s+1$. Then there exists a quantity $\delta \ll_s \eps^{O_s(1)}$, and \uppar{if $\delta < 1$} a non-empty system of cubes $S'_0,\dots,S'_{s+1}$ with parameter $1-\delta$, with $S'_k \subseteq S_k$ for all $0 \le k \le s+1$.
\end{lemma}
\begin{proof}
  If we replace $\eps$ with $\eps_0 = 2^{s+1} (s+1)!\, \eps$ and $S_k$ with the set $S_k^{(0)}$ consisting of all $c \in C^k(H)$ all of whose iterated coordinate permutations and reflections lie in $S_k$, we obtain sets $S_0^{(0)},\dots,S_{s+1}^{(0)}$ as in the statement which further obey the symmetry property (i) from Definition~\ref{def:cube-system}.

  Now define $S^{(1)}_k$ for each $0 \le k \le s+1$ to consist of all $c \in C^k(H)$ such that $c|_F \in S^{(0)}_{\dim F}$ for every face $F$ of $\llbracket k \rrbracket$.  (The case $F = \llbracket k \rrbracket$ says $S^{(1)}_k \subseteq S^{(0)}_k$.)
  
  Noting that for a uniform random cube $c \in C^k(H)$ the restriction $c|_F$ is a uniform random cube in $C^{\dim F}(H)$ (e.g., by uniqueness of Haar measure) and again crudely bounding the number of faces of $\llbracket k \rrbracket$ by $4^k$, by a union bound we get $\mu\big(S^{(1)}_k\big) \ge 1-\eps_1$ where $\eps_1 = 4^{s+1} \eps_0$.  Moreover, it is clear that $S^{(1)}_0,\dots,S^{(1)}_{s+1}$ obey properties (i) and (ii) from Definition~\ref{def:cube-system}.

  We now recursively define subsets $T_k, U_k, S^{(2)}_k \subseteq C^k(H)$ for $0 \le k \le s+1$ by $T_{s+1} = U_{s+1} = \emptyset$, $S^{(2)}_{s+1} = S^{(1)}_{s+1}$ and for $0 \le k \le s$,
  \begin{align*}
    T_k &= \big\{ c \in C^k(H) \colon \mu\big(\{ h \in H \colon [c, c+\square^k(h)]_{k+1} \notin S^{(2)}_{k+1}\}\big)  \ge \eta \big\} \\
    U_k &= \big\{ c \in C^k(H) \colon \mu\big(\{ h \in H \colon (c + \square^k(h)) \in T_k \}\big) \ge \eta \big\} \\
    S^{(2)}_k &= C^k(H) \setminus T_k \setminus U_k ,
  \end{align*}
  where $0 < \eta < 1$ is a parameter to be determined.
  
  Informally, for each $c \in C^k(H)$ we are considering a graph $\cG_c$ whose vertices are cubes $c + \square^k(h)$ that lie in $S^{(1)}_k$ and whose edges are pairs such that $[c+\square^k(h), c+\square^k(h')]_{k+1}$ lies in $S^{(2)}_{k+1}$. Then $T_k$ records the event that $c$ has low degree in $\cG_c$ (which is precisely a failure of property (iii)); $U_k$ records the event that many vertices in the graph $\cG_c$ have low degree; and $S^{(2)}_k$ is obtained by removing $c$ whenever either event occurs.

  We observe that $S^{(2)}_k \subseteq S^{(1)}_k$. This is clear for $k=s+1$; for $0 \le k \le s$, observe that if $c \in S^{(2)}_k$ then $[c,c+\square^k(h)]_{k+1} \in S^{(2)}_{k+1}$ for at least one $h \in H$, so $[c,c+\square^k(h)]_{k+1} \in S^{(1)}_{k+1}$ (by induction) and hence $c \in S^{(1)}_k$ (by restricting to a face).  Also, note that if $c \notin U_k$ then $c + \square^k(h) \notin U_k$ for all $h \in H$, and so for any $c \in S^{(2)}_k$ ($0 \le k \le s$) there are at least $(1-\eta) |H|$ values $h \in H$ such that $c + \square^k(h) \in S^{(2)}_k$.

  If $c \in C^k(H)$ and $h \in H$ are chosen uniformly at random then $[c, c+\square^k(h)]$ is a uniform random element of $C^{k+1}(H)$.  Hence we may estimate
  \[
    \PP\big([c,c+\square^k(h)] \notin S^{(2)}_{k+1}\big) \ge \eta\, \PP(c \in T_k)
  \]
  and hence $\mu(T_k) \le \big(1 - \mu\big(S^{(2)}_{k+1}\big)\big) / \eta$. Similarly, if $c \in C^k(H)$ and $h \in H$ are chosen uniformly at random then $c + \square^k(h)$ is a uniform random cube, and so
  \[
    \PP\big(c+\square^k(h) \in T_k \big) \ge \eta\, \PP(c \in U_k) 
  \]
  i.e.\ $\mu(T_k) \ge \eta\, \mu(U_k)$; hence, $\mu(U_k) \le \big(1 - \mu\big(S^{(2)}_{k+1}\big)\big) / \eta^2$.  It follows that
  \[
    1-\mu\big(S^{(2)}_k\big) \le (1/\eta + 1/\eta^2) \big(1 - \mu\big(S^{(2)}_{k+1}\big)\big)
  \]
  for $0 \le k \le s$, and since $1-\mu\big(S^{(2)}_{s+1}\big) \le \eps_1$ we conclude
  \[
    1 - \mu\big(S^{(2)}_k\big) \le \eps_1 (1/\eta + 1/\eta^2)^{s+1-k}
  \]
  for all $0 \le k \le s+1$, and
  \begin{align*}
    \mu(T_k) &\le (\eps_1 / \eta) (1/\eta + 1/\eta^2)^{s-k} \\
    \mu(U_k) &\le (\eps_1 / \eta^2) (1/\eta + 1/\eta^2)^{s-k}
  \end{align*}
  for all $0 \le k \le s$.  In particular, we have $\mu\big(S^{(2)}_0\big) \ge 1 - \eps_1 (1/\eta + 1/\eta^2)^{s+1}$, and so provided we choose, say, $\eta \ge 2 \eps_1^{1/2(s+1)}$, the right hand side is positive and so $S^{(2)}_0$ is non-empty.

  Finally, we define $S^{(3)}_k$ for $0 \le k \le s+1$ to consist of all $c \in C^k(H)$ such that $c|_F \in S^{(2)}_{\dim F}$ for every face $F$ of $\llbracket k \rrbracket$.  Again these are subsets of $S^{(2)}_k$, and it is clear that $S^{(3)}_k$ obey properties (i) and (ii) from Definition~\ref{def:cube-system}.  Moreover, note that $S^{(3)}_0 = S^{(2)}_0$ and hence will be non-empty provided $\eta$ is chosen as above.

  It remains to show that $S^{(3)}_k$ obeys condition (iii) for some parameter $\delta > 0$ to be determined.  Let $0 \le k \le s$. For each $c \in S^{(3)}_k$, we wish to bound
  \[
    \PP_{h \in H} \Big( [c, c+\square^k(h)]_{k+1} \notin S^{(3)}_{k+1} \Big)
  \]
  which by the definition of $S^{(3)}_{k+1}$ is at most
  \begin{equation}
    \label{eq:prob-missing}
    \sum_{F} \PP_{h \in H} \Big( \big([c, c+\square^k(h)]_{k+1}\big)|_F \notin S^{(2)}_{\dim F} \Big) 
  \end{equation}
  where the sum ranges over all faces $F$ of $\llbracket k+1 \rrbracket$.  We consider three kinds of faces $F$ separately:
  \begin{enumerate}[label=(\alph*)]
    \item $F$ is contained in the codimension $1$ lower face $F_{k+1}$;
    \item $F$ is contained in the codimension $1$ upper face $F^{k+1}$;
    \item both $F \cap F_{k+1}$ and $F \cap F^{k+1}$ are faces of strictly smaller dimension than $F$.
  \end{enumerate}
  In case (a), note that $\big([c, c+\square^k(h)]_{k+1}\big)|_F = c|_F$, which always lies in $S^{(3)}_{\dim F}$ by definition.

  In case (b), similarly $\big([c, c+\square^k(h)]_{k+1}\big)|_F = (c+\square^k(h))|_F = c|_F + \square^{\dim F}(h)$.  We remarked above that, for random $h$, this element lies in $S^{(2)}_{\dim F}$ with probability at least $1 - \eta$.

  Finally, in case (c) we have $\big([c, c+\square^k(h)]_{k+1}\big)|_F = [c|_{F \cap F_{k+1}}, c|_{F \cap F_{k+1}} + \square^{\dim F - 1}(h)]$.  Since $c|_{F \cap F_{k+1}} \in S^{(2)}_{\dim F-1}$ by assumption, this configuration lies in $S^{(2)}_{\dim F}$ with probability at least $1 - \eta$.

  So, for each $F$ the probability that $\big([c, c+\square^k(h)]_{k+1}\big)|_F \notin S^{(2)}_{k+1}$ is bounded by $\eta$.  Since again there are at most $4^{k+1}$ faces of $\llbracket k+1 \rrbracket$, the quantity~\eqref{eq:prob-missing} is bounded by $4^{s+1} \eta$.
  
  Therefore, if we set $\eta = 2 \eps_1^{1/2(s+1)}$ and $\delta = 4^{s+1} \eta$, provided $\delta < 1$ we find that $S^{(3)}_k \subseteq S_k$ is a non-empty system of cubes of parameter $1-\delta$, as required.
\end{proof}

We will typically use this in the form of the following corollary, which allows us to recover from losing a few cubes from a system of cubes.

\begin{corollary}%
  \label{cor:cube-system-patch}
  Suppose $s \ge 0$, $S_0,S_1,\dots,S_{s+1}$ is a system of cubes with parameter $\delta$, and $S'_k \subseteq S_k$ are subsets with $\mu(S_k \setminus S'_k) \le \eps$ for each $0 \le k \le s+1$, where $0 < \eps, \delta \le 1$ are parameters.  Then there exists a quantity $\delta' = \delta - O_s\big(\eps^{1/O_s(1)}\big)$, and \uppar{if $\delta' > 0$} a non-empty system of cubes $S''_k \subseteq S'_k$ with parameter $\delta'$, such that $\mu(S_k \setminus S''_k) \ll_s \eps^{1/O_s(1)}$ for each $0 \le k \le s+1$.
\end{corollary}
\begin{proof}
  Let $T_k = C^k(H) \setminus (S_k \setminus S'_k)$ for $0 \le k \le s+1$; so $\mu(T_k) \ge 1 - \eps$.  By Lemma~\ref{lem:cube-system-ae}, we may find a non-empty system of cubes $T'_k \subseteq T_k$ with parameter $1-\eta$, where $\eta \ll_s \eps^{1/O_s(1)}$.  By Remark~\ref{rem:cube-system-large}, $\mu(T'_k) \ge (1-\eta)^{k+1} \ge 1 - (s+2) \eta$ for each $0 \le k \le s+1$.
  
  Now define $S''_k = S_k \cap T'_k$.  It is clear that $\mu(S_k \setminus S''_k) \le 1 - \mu(T'_k) \le (s+2) \eta$ for each $k$; in particular, $S''_0$ is non-empty provided $\delta - \eps - (s+2) \eta > 0$.  Moreover, the intersection of two systems of cubes with parameters $\alpha$ and $\beta$ respectively is a system of cubes with parameter at least $\alpha + \beta - 1$, and hence $S''_k$ form a system of cubes with parameter at least $\delta - \eta$, as required.
\end{proof}

\subsection{Connectivity and dividing into non-interacting pieces}%
\label{subsec:split}

Given an approximate polynomial $f \colon X \to \RR$, we suppose that we have located a system of cubes $S_0,\dots,S_{s+1}$ as in Corollary~\ref{cor:cube-system}.  We now consider the possibility that our approximate polynomial $f \colon S_0 \to \RR$ alternates between several unrelated structured functions.  Our goal is to cut up the domain $S_0$ into pieces, one for each structured component.

As stated above, it should be the case that if $x,y \in S_0$ are in different structured components, then there should be few derivatives $\partial^{s+1} f(c) = 0$ where $c$ has $[x,y]$ as a $1$-dimensional face.  In the language of systems of cubes, the analogous statement is that $[x,y] \notin S_1$.

Therefore, the splitting procedure is roughly to consider a graph whose vertex set is $S_0$ and whose edges are pairs $xy$ such that $[x,y] \in S_1$ (and $x \ne y$), and then to decompose $S_0$ by taking robust connected components of this graph.

This part of the argument plays an analogous role to the Balog--Szemer\'edi--Gowers theorem in related works.  

We fix the following notation.  If $\cG = (V,E)$ is a graph and $X,Y \subseteq V$, we write
\[
  E(X,Y) = \{(x,y) \colon x \in X,\, y \in Y,\, xy \in E \}
\]
and denote by
\[
  H(\cG) = \min_{S \subseteq V,\, 0 < |S| \le |V|/2} \frac{|E(S, V \setminus S)|}{|S|}
\]
the Cheeger (edge expansion) constant.  We also write $h(\cG) = H(\cG) / |V|$ for the normalized Cheeger constant, with $0 \le h(\cG) \le 1$ (except that by convention the one-vertex graph has Cheeger constant $+\infty$).

\begin{lemma}%
  \label{lem:split}
  Suppose $s \ge 0$ and $S_0,\dots,S_{s+1}$ is a system of cubes on $H$ with parameter $\delta$, $0<\delta\le 1$, and let $\eps > 0$ be another parameter.  We assume that $\eps \le \delta^C / C$ is sufficiently small, for some constant $C = C(s)$.

  Then there exists a non-empty system of cubes $S'_0,\dots,S'_{s+1}$ with parameter $\delta' \ge \delta - O_s\big(\eps^{1/O_s(1)} / \delta\big)$, with the following properties.  Write $\cG = (V,E)$ for the graph with $V = S'_0$ and $E = \{ xy \colon [x,y] \in S'_1,\, x \ne y\}$, and write $V_1,\dots,V_K \subseteq V$ for the connected components of $\cG$.  Then:---
  \begin{enumerate}[label=(\alph*)]
    \item $S'_k \subseteq S_k$, and $\mu(S_k \setminus S'_k) \ll_s \eps^{1/O_s(1)} / \delta$, for each $0 \le k \le s+1$;
    \item $K \le 1/\delta'$; and
    \item each subgraph $\cG[V_i]$ has Cheeger constant $H(\cG[V_i])$ at least $\eps |H|$.
  \end{enumerate}
\end{lemma}
\begin{remark}%
  \label{rem:component-cubes}
  It is easy to see that for any $c \in S'_k$ ($1 \le k \le s+1$), all the values $\{c(\omega) \colon \omega \in \llbracket k \rrbracket\}$ are contained in the same component $V_i$ (since $c|_F \in S'_1$ for every $1$-dimensional face $F$, and the graph formed by $1$-dimensional faces of $\llbracket k \rrbracket$ is connected).  So, each set $V_i \subseteq C^0(H)$ carries its own system of cubes $S^{(i)}_0,\dots,S^{(i)}_{s+1}$ with $S^{(i)}_0 = V_i$ and with the same parameter $\delta'$, by simply intersecting $S'_k$ with $V_i^{\llbracket k \rrbracket}$.  In other words, $S'_0,\dots,S'_{s+1}$ decomposes naturally into $K$ non-interacting systems of cubes, one for each component $V_i$.
\end{remark}

\begin{proof}[Proof of Lemma~\ref{lem:split}]
  Throughout, let $\eta=O_s(\eps^{1/O_s(1)})$ be a quantity satisfying the statement of Corollary~\ref{cor:cube-system-patch} with parameter $\eps$.  It is acceptable to assume $\eta \ge \eps$.
  By choosing $C(s)$ in the statement appropriately, we may also assume $\eta \le \delta^2/4$.

  Set $S^{(0)}_k = S_k$ for $0 \le k \le s+1$.
  We will inductively define a sequence of (non-empty) systems of cubes $S^{(i)}_0,\dots,S^{(i)}_{s+1}$ for $0 \le i \le T$, with parameters $\delta^{(i)} := \delta - i \eta$, and satisfying $\mu(S_k \setminus S^{(i)}_k) \le i \eta$.  Here $T \le \lfloor 2/\delta \rfloor$ is some integer.

  First, we observe that $\delta^{(i)} \ge \delta - (2/\delta) \eta \ge \delta/2$ (using the bound $\eta \le \delta^2/4$).
  In particular, $\delta^{(i)}$ remains positive for $0 \le i \le T$, and all $\delta^{(i)}$ obey the bound $\delta^{(i)} \ge \delta - O_s\bigl(\eps^{1/O_s(1)}/\delta\bigr)$.

  Let $\cG^{(i)}$ be the graph with vertices $S^{(i)}_0$ and edges $\bigl\{ xy : [x,y] \in S^{(i)}_1, x \ne y \bigr\}$, as above.  
  Note $\cG^{(i)}$ has minimum degree at least $\delta^{(i)} |H| - 1$, by the definition of a system of cubes (specifically Definition~\ref{def:cube-system} (iii)).
  Hence, every connected component of $\cG^{(i)}$ has size $\ge \delta^{(i)} |H|$, and so $\cG^{(i)}$ has at most $1/\delta^{(i)}$ connected components.
  
  Suppose $S^{(i)}_0,\dots,S^{(i)}_{s+1}$ and $\delta^{(i)}$ have been chosen.
  If $\cG^{(i)}$ has the property that each connected component has Cheeger constant at least $\eps |H|$, then $S^{(i)}_0,\dots,S^{(i)}_{s+1}$ obey all the properties (a), (b), (c) in the statement, with $\delta'=\delta^{(i)}$, so we are done and may terminate the process.

  If not, by definition we can find disjoint sets $A, B \subseteq S^{(i)}_0$ with $0 < |A| \le |B|$, such that $|E(A,B)| < \eps |A|\, |H|$ and $A \cup B$ is a connected component of $\cG^{(i)}$.

  We note that necessarily $|A| > (\delta^{(i)} - \eps) |H|$.
  Indeed, every $x \in A$ has degree at least $\delta^{(i)} |H|-1$ in $\cG^{(i)}$, so has at least $\delta^{(i)} |H| - |A|$ neighbors in $B$; so, $|E(A,B)| \ge |A| \bigl(\delta^{(i)} |H| - |A|\bigr)$.
  Using $|E(A,B)| < \eps |A|\,|H|$ and rearranging gives the stated bound.

  Then, we define $S^{(i+1)}_0,\dots,S^{(i+1)}_{s+1}$ as follows.
  First define $R^{(i)}_1 \subseteq S^{(i)}_1$ by removing all elements $c=[x,y] \in S^{(i)}_1$ with $x \in A, y \in B$ or $x \in B, y \in A$ (of which there are at most $2 \eps |A|\,|H| \le \eps |C^1(H)|$).
  Then apply Corollary~\ref{cor:cube-system-patch} to $S^{(i)}_0,R^{(i)}_1,S^{(i)}_2,\dots,S^{(i)}_{s+1}$ to obtain a system of cubes $S^{(i+1)}_0,\dots,S^{(i+1)}_{s+1}$ with parameter $\delta^{(i+1)} = \delta^{(i)} - \eta$ such that $S^{(i+1)}_k \subseteq S^{(i)}_k$, $\mu(S_k^{(i)} \setminus S_k^{(i+1)}) \le \eta$ for each $0 \le k \le s+1$, and $S^{(i+1)}_1 \subseteq R^{(i)}_1$.
  It is clear that our inductive hypotheses hold for this $S^{(i+1)}_0,\dots,S^{(i+1)}_{s+1}$.
  
  Hence, we are done unless the process continues to time $T=\lfloor 2/\delta \rfloor$ without success, and it suffices to find a contradiction in this case.

  We claim that for $0 \le i \le T$, the graph $\cG^{(i)}$ has at least $i+1$ connected components.
  Indeed, suppose $0 \le i \le T-1$ and $\cG^{(i)}$ has $m$ connected components $X_1,\dots,X_m$ where $X_1=A \cup B$.
  In the graph corresponding to $R^{(i)}_1$, the $m+1$ sets $A,B,X_2,\dots,X_m$ are each a connected component or a union of connected components.
  Moreover, all these sets have size at least $(\delta^{(i)} - \eps) |H|$.
  Passing from $R^{(i)}_1$ to $S^{(i+1)}_1$, the only way the number of connected components could fall below $m+1$ is if all vertices in one of $A,B,X_2,\dots,X_m$ were deleted (i.e., belonged to $S^{(i)}_0 \setminus S^{(i+1)}_0$); however, there are at most $\eta |H|$ deleted vertices, and
  \[
    \delta^{(i)} - \eps - \eta \ge \delta - (i+2) \eta \ge \delta - (T+1) \eta \ge \delta - (2/\delta+1)(\delta^2/4) > 0
  \]
  so this is impossible, and the claim holds.

  However, this means the number of connected components of $\cG^{(T)}$ is at least $T+1$ but at most $1/\delta^{(T)}$, and since $T \le 2/\delta < T+1$, we get
  \[
    \frac2{\delta} < T+1 \le \frac1{\delta^{(T)}} = \frac1{\delta - T \eta} \le \frac1{\delta - \frac{2}{\delta} \frac{\delta^2}{4}} = \frac2{\delta}
  \]
  which gives the desired contradiction.
\end{proof}

\subsection{Extending to a global function}%
\label{subsec:improve-core}

Using the structure obtained in previous parts, we are now in a position to deploy the random sampling argument described in Section~\ref{subsec:improve-outline}.  Recall this allows us (roughly) to extend an approximate polynomial $f$ to a global function $\wt f$, and understand the derivatives $\partial^{s+1} \wt f$ in terms of lower-degree approximate polynomials.

\begin{lemma}%
  \label{lem:extend-global}
  Suppose $s \ge 1$ and $S_0,S_1,\dots,S_{s+1}$ is a non-empty system of cubes with parameter $\delta>0$, a function $f \colon S_0 \to \RR$ satisfies $\partial^{s+1} f(c) = 0$ for all $c \in S_{s+1}$, and the graph $\cG = (V,E)$, where $V=S_0$ and $E=\{ xy \colon [x,y] \in S_1\}$, has normalized Cheeger constant $h = h(\cG) > 0$.

  Let $\eps > 0$ be a further parameter.  Then there exist:---
  \begin{enumerate}[label=(\roman*)]
    \item a set $X \subseteq H$ with $\mu(X) \ge 1 - \eps$ and a function $\wt f \colon X \to \RR$ with $\wt f|_{X \cap S_0} = f|_{X \cap S_0}$;
    \item a set $S \subseteq C^{s+1}(H) \cap X^{\llbracket s+1 \rrbracket}$ with $\mu(S) \ge 1 - \eps$;
    \item positive integers $M,K \ll_s \bigl( \delta^{-1} h^{-4} \log(2/\eps) \log(2/\delta) \log(2/h)^2 \bigr)^{O_s(1)}$;
    \item for each $r \in [K]$, a non-empty system of cubes $S^{(r)}_0,\dots,S^{(r)}_s$ of degree $s-1$ and parameter $\delta$; and
    \item for each $r \in [K]$ a function $g_r \colon S^{(r)}_0 \to \RR$ such that $\partial^{s} g_r (c) = 0$ for all $c \in S^{(r)}_s$;
  \end{enumerate}
  such that for each $c \in S$ there exists a configuration $b \colon \llbracket s+1 \rrbracket \to \ZZ^K$ with $\|b(\omega)\|_1 \le M$ for each $\omega \in \llbracket s+1 \rrbracket$, giving
  \[
    \partial^{s+1} \wt f(c) = \sum_{\omega \in \llbracket s+1 \rrbracket} (-1)^{|\omega|} \sum_{i=1}^K b(\omega)_i g_i(c(\omega))
  \]
  \uppar{where each function $g_i(c(\omega))$ is indeed defined whenever the coefficient $b(\omega)_i$ is non-zero}.
\end{lemma}
In other words, the derivatives $\partial^{s+1} \wt f$ can be understood in terms of a bounded linear combination of approximate polynomials $g_1,\dots,g_K$ of degree $s-1$, evaluated at the vertices of $c$.

\begin{remark}%
  \label{rem:examples}
  We have already discussed the conclusion in the case $s=1$.  When $s=2$, useful examples that showcase the behavior described in the lemma in real life include those of the shape $f \colon \ZZ/N\ZZ \to \RR$ given by $f(x) = \{\alpha x /N\}\,\{ \beta x / N\}$ for arbitrary $\alpha,\beta \in \ZZ/N\ZZ$, where $\{\cdot\} \colon \RR \to [0,1)$ denotes the fractional part map.%chktex 9
  
  It is a somewhat informative exercise to compute the derivative $\partial^3 f(\angle(x;h_1,h_2,h_3))$ explicitly. For current purposes, it suffices to note that it is a bounded integer combination of the functions $x \mapsto \{\alpha x / N\}$, $x \mapsto \{\beta x / N\}$ and $x \mapsto 1$ evaluated at the vertices of $c=\angle(x;h_1,h_2,h_3)$, which is exactly the kind of structure produced by the lemma.
\end{remark}

\begin{remark}%
  \label{rem:case-s-zero}
  It is clear any $f$ obeying the hypotheses of Lemma~\ref{lem:extend-global} with $s=0$ is actually a constant function on its domain (since $f(x) = f(y)$ for all $xy \in E$ and $\cG$ is connected), and so can trivially be extended to a constant function on $H$.  This simpler statement corresponds to the $s=0$ case of the lemma, and may replace it in any inductive arguments.
\end{remark}

\begin{proof}[Proof of Lemma~\ref{lem:extend-global}]
  As stated above, the key step is to pick elements $a_1,\dots,a_L \in H$ uniformly and independently at random, and consider the translates $x \mapsto f(x+a_i)$.  The value of $L$ will be chosen to satisfy the following conditions (some of which may imply others):---
  \begin{enumerate}[label=(\Alph*)]
    \item $\exp\big(-\delta L / 8) \le 2^{-s-1} (\eps/40) (L+1)^{-s-1}$;
    \item $(\delta L/2+1) \exp\big(-(\delta L/2-1) h^4 / 64 \log(2/h^2) \big) \le 2^{-s-1} (\eps/40) (L+1)^{-s-1}$;
    \item $2\exp(-\delta (L-1)/2) \le 2^{-s-1} (\eps / 20) (L+1)^{-s-1}$; and
    \item $(s+1) \exp(-\delta L) \le \eps / 10$.
  \end{enumerate}
  Noting for (A),(B),(C) that for constants $\alpha,\beta \ge 1$ the inequality $\exp(X) \ge \alpha X^\beta$ holds provided (say) $X \ge 2 \beta \log (2 \alpha \beta)$, it follows that there is a choice of $L$ satisfying (A)--(D) while keeping $L \ll_s \delta^{-1} h^{-4} \log(2/\eps) \log(2/\delta) \log(2/h)^2$.
  
  For any pair $a,b \in H$ and $0 \le k \le s$ we define
  \[
    S^{(a,b)}_k = \big\{ c \in C^k(H) \colon [c+\square^{k}(a), c+\square^{k}(b)]_{k+1} \in S_{k+1} \big\} .
  \]
  Each $S^{(a,b)}_0,\dots,S^{(a,b)}_s$ inherits the properties in Definition~\ref{def:cube-system} and so is itself a system of cubes of degree $s$ and parameter $\delta$; however, for a particular $a,b \in H$ it may be empty.  By symmetry (Definition~\ref{def:cube-system}(i)) we also have $S^{(a,b)}_k = S^{(b,a)}_k$.  We can further define
  \[
    g_{a,b}(x) = f(x+a) - f(x+b)
  \]
  as a function $S^{(a,b)}_0 \to \RR$ (it is well-defined there, since $x \in S^{(a,b)}_0$ implies $[x+a,x+b] \in S_1$ which in particular implies $x+a,\, x+b \in S_0$).  For any $c \in S^{(a,b)}_s$, we have
  \[
    \partial^s g_{a,b}(c) = \partial^s f(c + \square^s(a)) - \partial^s f(c + \square^s(b)) = \partial^{s+1} f\big([c+\square^s(a),c+\square^s(b)]_{s+1}\big) = 0
  \]
  by hypothesis.  Note also that $g_{a,b} = -g_{b,a}$.

  If $1 \le k < \ell \le L$ and $x+a_k,\, x+a_\ell \in S_0$ then it makes sense to consider the difference $f(x+a_k) - f(x+a_\ell)$, which as discussed above measures our inability to glue the translates $x \mapsto f(x+a_k)$ and $x \mapsto f(x+a_\ell)$ together.  When furthermore $x \in S_0^{(a_k,a_\ell)}$, i.e.\ $[x+a_k,x+a_\ell] \in S_1$, this difference is exactly measured by the value $g_{a_k,a_\ell}(x)$.  However, this latter condition is too restrictive for what we need.  The next claim allows us to describe $f(x+a_k) - f(x+a_\ell)$ (almost surely) wherever it is defined, using a handful of different functions $g_{a_i,a_j}$, by exploiting our robust connectivity hypothesis on the graph $\cG$.

  \begin{claim}%
    \label{claim:samples-conn}
    Fix $x \in H$.  With probability at least $1 - 2^{-s-1} (\eps/20) (L+1)^{-s-1}$ with respect to $a_1,\dots,a_L$, the following hold: (i) there is at least one $i$ such that $x+a_i \in S_0$; and (ii) for every pair $k,\ell \in [L]$ such that $x+a_k,\, x+a_\ell \in S_0$, there exist integers $b_{i,j} \in \{-1,0,1\}$ for $1 \le i < j \le L$ such that
    \[
      f(x+a_{k}) - f(x+a_{\ell}) = \sum_{1 \le i < j \le L} b_{i,j} g_{a_i,a_j}(x)
    \]
    and with $x \in S^{(a_i,a_j)}_0$ whenever $b_{i,j} \ne 0$ \uppar{i.e., the expression is well-defined}.
  \end{claim}
  \begin{proof}[Proof of claim]
    We recall the graph $\cG$ from the statement.  Also let $V' = \big\{ i \colon x + a_i \in S_0 \big\}$ and $E' = \big\{ ij \colon [x+a_i, x+a_j] \in S_1\big\}$, and consider the graph $\cG'=(V',E')$ (which depends on $x$).  Since the elements $x+a_i$ for $i \in V'$ are independent uniform random samples of $S_0$, this graph $\cG'$ is the induced subgraph of $\cG$ on a random set of $|V'|$ vertices (allowing repetitions), in the precise sense of Lemma~\ref{lem:subgraph-connected}.  Since $|S_0| \ge \delta |H|$ (see Remark~\ref{rem:cube-system-large}), we have
    \[
      \PP\big[ |V'| \ge \delta L / 2 \big] \ge 1 - \exp(-\delta L / 8) \ge 2^{-s-1} (\eps/40) (L+1)^{-s-1}
    \]
    by Chernoff and (A).  Conditional on this event, by Lemma~\ref{lem:subgraph-connected} we have
    \[
      \PP\big[\cG' \text{ is connected}\big] \ge 1 - (\delta L/2+1) \exp\Big(-(\delta L/2-1) h^4 / 64 \log\big(2/h^2\big) \Big) ;
    \]
    and by (B) this probability is at least $1 - 2^{-s-1} (\eps / 40) L^{-s-1}$.

    If $k,\ell \in V'$, $k \ne \ell$, and $k=i_0,i_1,\dots,i_m=\ell$ is a path from $k$ to $\ell$ in $\cG'$, it is immediate that
    \[
      f(x+a_k) - f(x+a_\ell) = \sum_{r=0}^{m-1} \big( f(x+a_{i_r}) - f(x+a_{i_{r+1}}) \big) = \sum_{r=0}^{m-1} g_{i_{r},i_{r+1}}(x) .
    \]
    For each $1 \le i < j \le L$ the functions $g_{i,j}$, $g_{j,i}$ are used at most once in total in this sum, and so this expression has the required form.
  \end{proof}

  Let $X_0$ denote the set of all $x \in H$ for which the conclusion of Claim~\ref{claim:samples-conn} holds.  By the claim, $\EE[\mu(X_0)] \ge 1 - 2^{-s-1} (\eps/20) L^{-s-1}$.  In particular, for every $x \in X_0$ there is some $i \in [L]$ such that $x+a_i \in S_0$; so, we may define a function $f' \colon X_0 \to \RR$ by $f'(x) = f(x+a_i)$ where $i \in [L]$ is the least value for which $x+a_i \in S_0$.
  This choice is slightly arbitrary, but the second part of Claim~\ref{claim:samples-conn} allows us to pass freely between the values $f(x+a_i)$ so this is usually not important.
  
  We now need to analyse the derivatives of $f'$.  As progress towards this, the next claim shows that we can (almost surely) understand any \emph{fixed} single derivative $f'(x) - f'(x+h)$, at the expense of considering a collection of functions $g_{a,b}$ where $a$ and $b$ depend on $h$.

  \begin{claim}%
    \label{claim:hopping}
    Fix any $x, h \in H$.  Then with probability at least $1 - 2^{-s-1} (\eps/20)(L+1)^{-s-1}$ with respect to $a_1,\dots,a_L$, there exists $k,\ell \in [L]$ such that $[x+a_k,\, x+a_\ell+h] \in S_1$.

    If furthermore $x,x+h \in X_0$, then for some integers $b_{i,j}, b'_{i,j} \in \{-1,0,1\}$ for $1 \le i < j \le L$, we have
    \[
      f'(x) - f'(x+h) = g_{a_k,a_\ell+h}(x) + \sum_{1 \le i < j \le L} \left(b_{i,j} g_{a_i,a_j}(x) + b'_{i,j} g_{a_i,a_j}(x+h) \right) 
    \]
    where again we require that whenever some $g_{a,b}(y)$ appears with a non-zero coefficient it must be well-defined \uppar{in that $y \in S_0^{(a,b)}$}.
  \end{claim}
  \begin{proof}[Proof of claim]
    For the first part, divide the random variables $a_i$ into two halves $a_1,\dots,a_{\lfloor L/2 \rfloor}$ and $a_{\lfloor L/2 \rfloor+1},\dots,a_L$.  Since the values $x+a_k$ for $k \in \{1,\dots,\lfloor L/2 \rfloor\}$ are independent random samples of $H$, with probability at least $1 - (1 - \delta)^{(L-1)/2} \ge 1 - \exp(-\delta (L-1)/2)$ there is some $k$ such that $x+a_k \in S_0$.  Fixing this $k$, again $x+a_\ell+h$ for $\lfloor L/2\rfloor + 1 \le \ell \le L$ are uniform random samples from $H$, so by the definition of a system of cubes the probability that $[x+a_k,\, x+a_\ell+h] \in S_1$ for some $\ell$ in this range is at least $1 - (1 - \delta)^{L/2} \ge 1 - \exp(-\delta L/2)$.  By (C), the total failure probability is at most $2^{-s-1} (\eps/20) L^{-s-1}$ as required.

    For the second part, suppose $t,u \in [L]$ are values such that $f'(x) = f(x+a_t)$ and $f'(x+h) = f(x+h+a_u)$.  Then we may expand
    \[
      f'(x) - f'(x+h) = \big(f(x+a_t) - f(x+a_k)\big) - \big(f(x+h+a_u) - f(x+h+a_\ell)\big) + \big(f(x+a_k) - f(x+h+a_\ell) \big) .
    \]
    Since $x, x+h \in X_0$, the first two terms may be expanded in terms of $g_{a_i,a_j}$ as in Claim~\ref{claim:samples-conn}.  The last term is exactly $g_{a_k,a_\ell+h}(x)$ and is well-defined by the first part.
  \end{proof}

  We now choose a further set of uniform independent random translates $h_1,\dots,h_L \in H$, and write $T = \{0,h_1,\dots,h_L\}$ and $(s+1) T$ for the $(s+1)$-fold sumset of $T$.  Accordingly, we define $X_1 \subseteq H$ to consist of all $x \in H$ such that the first part of Claim~\ref{claim:hopping} succeeds for $x$ and for each $h \in (s+1) T$.  The bound from the claim shows that $\EE[\mu(X_1)] \ge 1 - 2^{-s-1} \eps/20$.

  We further define $X_2 \subseteq H$ to consist of all $x \in X_1$ such that $x+t \in X_0$ for each $t \in (s+1) T$ (so in particular, $x \in X_0$).  A union bound shows $\EE[\mu(X_2)] \ge 1 - 2^{-s-1} \eps / 10$.

  Finally, define $K$ and functions $g'_r$ for $r \in [K]$ to be, in any order, all functions of the form $g_{a_i,a_j}$, $g_{a_i+t, a_j+t}$ or $g_{a_i,a_j+t}$ for $1 \le i < j \le L$ and $t \in (s+1) T$, whose corresponding domains $S_0^{(a,b)}$ are non-empty.  It follows that $K \le 3 L^{s+3}$.

  With this is place, we can describe (almost all) general derivatives $\partial^{s+1} f'$.
  \begin{claim}
    Fix $c \in C^{s+1}(H)$, and suppose that $c(\omega) \in X_2$ for each $\omega \in \llbracket s+1 \rrbracket$.
    Write $c = \angle(x; u_1,\dots,u_{s+1})$ \uppar{see Section~\ref{subsec:notation}}, and let $k \in [L]$ be any index such that $x + a_k \in S_0$ \uppar{possible since $x = c(0,\dots,0) \in X_0$}.
    
    Then with probability at least $1-\eps/10$ with respect to $h_1,\dots,h_L$, there exist $t_1,\dots,t_{s+1} \in T$ such that $\angle(x+a_k; u_1+t_1,\dots,u_{s+1}+t_{s+1}) \in S_{s+1}$.

    If this conclusion holds, then there exist coefficients $b(\omega) \in \ZZ^K$ with $\|b(\omega)\|_1 \le 3 L^2 + 1$ for each each $\omega \in \llbracket s+1 \rrbracket$ \uppar{depending on $c$}, such that
    \[
      \partial^{s+1} f'(c) = \sum_{\omega \in \llbracket s+1 \rrbracket} (-1)^{|\omega|} \sum_{r=1}^K b(\omega)_r\, g'_r(c(\omega))
    \]
    and where again the coefficient $b(\omega)_r$ is only non-zero when $g'_r(c(\omega))$ is well-defined.
  \end{claim}

  Given the claim, we can finish the proof as follows. Set $S'$ to the set of all $c \in C^{s+1}(H)$ for which the claim succeeds; so,
  \[
    \EE[\mu(S')] \ge 1 - \eps/10 - 2^{s+1} \EE[1 - \mu(X_2)] \ge 1 - \eps/5.
  \]
  Therefore with probability at least $4/5$ each (say), we get that $\mu(S') \ge 1 - \eps$ and $\mu(X_2) \ge 1 - \eps$.  We can therefore fix some $a_1,\dots,a_L$ and $h_1,\dots,h_L$ so that both of these estimates hold simultaneously.

  Finally we apply a global translation to bring everything into line with the original function $f$.  Define $X \subseteq H$ by $X = X_2 + a_1$, $S \subseteq C^{s+1}(H)$ by $S = S' + \square^{s+1}(a_1)$, $\wt f \colon X \to \RR$ by $\wt f(x) = f'(x-a_1)$, and $g_r(x) = g'_r(x-a_1)$ for $r \in [K]$.
  
  By definition of $f'$, whenever $x \in S_0 \cap X$ we have $\wt f(x) = f'(x-a_1) = f(x)$ (as $1$ is the smallest index $i \in [L]$ such that $(x-a_1)+a_i \in S_0$), so (i) holds.  The bounds in (ii),(iii) are clear.  The systems of cubes attached to $g_r$ is the corresponding translate of that attached to $g'_r$, which has the required properties for (iv),(v).  The grand conclusion then follows immediately from the claim.

  \begin{proof}[Proof of claim]
    For the first part, we choose the values $t_1,\dots,t_{s+1} \in T$ in a recursive procedure.
    
    For the recursive step, we show that for any $0 \le m \le s$, any cube $\angle(y;e_1,\dots,e_m) \in S_m$ and any $e \in H$, with probability at least $1 - \exp(-\delta L)$ with respect to $h_1,\dots,h_L$ there exists some $t \in T$ such that $\angle(y; e_1,\dots,e_m, e + t) \in S_{m+1}$.  Indeed, by the definition of a system of cubes there are at least $\delta |H|$ values $z \in H$ such that $\angle(y;e_1,\dots,e_m,z) \in S_{m+1}$, and each $e + h_i$ for $i \in [L]$ is an independent uniform random sample from $H$, so one lands in the good set with probability at least $1 - (1-\delta)^L$.
    
    Starting with $x+a_k \in S_0$ and applying this for each $m$ from $0$ to $s$ increasing,
    we find the values $t_1,\dots,t_{s+1} \in T$ we need, with the total failure probability bounded by $(s+1) \exp(-\delta L)$, which in turn is at most $\eps/10$ (by (D)).

    For the second statement, we write $c' = \angle(x+a_k; u_1+t_1,\dots,u_{s+1}+t_{s+1}) \in S_{s+1}$ as above, and note that $\partial^{s+1} f(c') = 0$.  It follows that
    \[
      \partial^{s+1} f'(c) = \sum_{\omega \in \llbracket s+1 \rrbracket} (-1)^{|\omega|} \big(f'(c(\omega)) - f(c'(\omega)) \big).
    \]
    For each $\omega \in \llbracket s+1 \rrbracket$, we have $c'(\omega) = c(\omega) + a_k + t$ where $t = \omega \cdot (t_1,\dots,t_{s+1})$.  Since $c' \in S_{s+1}$ we also have $c'(\omega) \in S_0$, and since $c(\omega) \in X_2$ and $t \in (s+1) T$ it follows that $c(\omega) + t \in X_0$.  By Claim~\ref{claim:samples-conn} we can therefore write
    \begin{equation}
      \label{eq:first-fprime}
      f'(c(\omega) + t) - f(c'(\omega)) = \sum_{1 \le i < j \le L} b_{i,j} g_{a_i,a_j}(c(\omega)+t)
    \end{equation}
    where $b_{i,j} \in \{-1,0,1\}$. Since $g_{a_i,a_j}(x+t) = g_{a_i+t,a_j+t}(x)$ the terms on the right hand side can be expressed as a sum of values $g'_r(c(\omega))$ with total weight at most $L^2$.

    Next, since $c(\omega) \in X_1$ and $t \in (s+1) T$ still, by Claim~\ref{claim:hopping} we deduce that
    \begin{equation}
      \label{eq:second-fprime}
      f'(c(\omega)) - f'(c(\omega)+t) = g_{a_\ell,a_m+t}(x) + \sum_{1 \le i < j \le L} \left(b'_{i,j} g_{a_i,a_j}(x) + b''_{i,j} g_{a_i,a_j}(x+t) \right) 
    \end{equation}
    for some $b'_{i,j},\, b''_{i,j} \in \{-1,0,1\}$ and $\ell,m \in [L]$, and again every term on the right is one of the values $g'_r(c(\omega))$, with total weight $2 L^2 + 1$.  Adding~\eqref{eq:first-fprime} and~\eqref{eq:second-fprime} gives an expression for $\partial^{s+1} f'(c)$ of the desired form.
  \end{proof}
  This concludes the proof of Lemma~\ref{lem:extend-global}.
\end{proof}

\subsection{Iterating global structure and polynomial hierarchies}%
\label{subsec:hierarchy}

We now record the conclusions we get by applying Lemma~\ref{lem:extend-global} recursively to the lower-degree functions $g_1,\dots,g_K$ it produces, and so on.\footnote{We could instead apply a much stronger structure theorem, along the lines of Theorem~\ref{thm:1pc}, inductively to $g_1,\dots,g_K$.  It turns out that technical issues related to ``reducibility'', discussed in the next section, are problematic in that setting, but are more tractable if we continue to work with softer statements in the spirit of Lemma~\ref{lem:extend-global} for now.}    The following definitions are designed to capture the data that arises when we do so.

\begin{definition}%
  \label{def:hierarchy}
  Let $H$ be an abelian group, $X \subseteq H$ a subset, and let integers $s \ge 0$ and $d_0,d_1,\dots,d_s \ge 0$ be given.

  Suppose we have a tuple $f$ of functions $f_{i,j} \colon X \to \RR$ for $0 \le i \le s$ and $1 \le j \le d_i$.  In general we write $f_{\le t}$ for the sub-tuple $(f_{i,j})$ with $i \le t$, and similarly $f_{=t}$, etc..
  
  Given some other function $g \colon X \to \RR$, integers $k,t$ (with $0 \le k,t \le s+1$), a set $S \subseteq C^k(H) \cap X^{\llbracket k \rrbracket}$, and further function
  \[
    b \colon S \times \llbracket k \rrbracket \to \bigoplus_{i=0}^{t-1} \ZZ^{d_i}
  \]
  (so, $b(c, \omega)$ is a tuple of integers mirroring the tuple $f_{\le t-1}$),
  we say $g$ and $b$ satisfy the \emph{$k$-derivatives condition} on $S$ with respect to $f$, with degree $t$ and parameter $M \ge 1$, if
  \begin{equation}
    \label{eq:deriv-condition}
    \partial^{k} g (c) = \sum_{\omega \in \llbracket k \rrbracket } (-1)^{|\omega|} \big(b(c, \omega) \cdot f_{\le t-1}(c(\omega)) \big)
  \end{equation}
  holds for all $c \in S$, and if $\|b_{=r}(c, \omega)\|_1 \le M^{t-r}$ for all $(c,\omega) \in S \times \llbracket k \rrbracket$ and each $r$, $0 \le r \le t-1$.

  We often suppress explicit mention of $b$, and say more briefly that $g$ obeys the \emph{$(f,k,t,M)$-derivatives condition} on $S$ if a suitable $b$ exists.

  Now let $S_0,\dots,S_{s+1}$ be a non-empty system of cubes, with $S_0 \subseteq X$.
  If for each $i,j$ ($0 \le i \le s$, $1 \le j \le d_i$) there is a function $b^{i,j}$ such that $f_{i,j}$ and $b^{i,j}$ obey the $(i+1)$-derivatives condition on $S_{i+1}$ with respect to $f_{\le i-1}$, with degree $i$ and parameter $M$, we say $f$ (together with the functions $b^{i,j}$) forms a \emph{polynomial hierarchy} (on $S_0,\dots,S_{s+1}$, with degree $s$, dimensions $(d_0,d_1,\dots,d_s)$ and parameter $M$).

  Again we often suppress all mention of $b^{i,j}$ and say simply that $f$ is an \emph{$(s,d,M)$-polynomial hierarchy} on $S_0,\dots,S_{s+1}$.
\end{definition}

\begin{remark}
  In practice, the system of cubes $S_0,\dots,S_{s+1}$ in the definition of a hierarchy will always have parameter $1-\eps$ for some small $\eps > 0$, so should be thought of as encoding a precise technical notion of ``almost all cubes''.  

  Also, note that the expression $M^{t-r}$ in the norm bound on $b$ is equivalent up to polynomial losses (for fixed $s$) to simply using $M$ as the bound; however, the graded bound behaves better when we combine derivative conditions of different degrees.
\end{remark}

\begin{remark}%
  \label{rem:hierarchy-t0}
  Our convention is that having a function $g$ obey a $k$-derivatives condition with degree $t=0$ on a set $S$ is interpreted as a complicated way of saying that $\partial^k f(c) = 0$ for all $c \in S$.
  
  Most statements---but not all---avoid this degenerate case. For instance, it is implicit in the description of the bottom level functions $f_{0,j}$ in the hierarchy, for $j \in [d_0]$: by this convention they satisfy $f_{0,j}(x) - f_{0,j}(y) = 0$ whenever $[x,y] \in S_1$.  When $S_0,\dots,S_{s+1}$ has parameter greater than $1/2$, this is equivalent to saying that $f_{0,j}$ is a constant function on $S_0$ (since for any $x,y \in S_0$ we can find some $z \in H$ with $[x,z],[y,z] \in S_1$, so $f_{0,j}(x) = f_{0,j}(z) = f_{0,j}(y)$).
\end{remark}

We can conveniently summarize all the work of this section in this terminology, including a recursive application of Lemma~\ref{lem:extend-global}, in the following result.

\begin{corollary}%
  \label{cor:poly-hierarchy-everything}
  Let $s \ge 1$, $X \subseteq H$, and suppose $g \colon X \to \RR$ is an approximate polynomial with degree $s$ and parameter $\delta > 0$.  Also let $\eps > 0$ be a parameter.

  Then there exists a system of cubes $S_0,\dots,S_{s+1}$ with parameter $1-\eps$, an $(s-1,d,M)$-polynomial hierarchy $f$ on $S_0,\dots,S_{s}$, a set $Y \subseteq S_0 \cap X$, and a function $\wt{g} \colon S_0 \to \RR$, such that
  \begin{itemize}
    \item $g|_Y = \wt{g}|_Y$ and $\mu(Y) \gg_s \delta^{O_s(1)}$;
    \item $\wt{g}$ obeys the $(f,s+1,s,M)$-derivatives condition on $S_{s+1}$; and
    \item the quantities $D = \sum_{i=0}^{s-1} d_i$ and $M$ satisfy $D,M \ll_{s} (\eps \delta)^{-O_{s}(1)}$.
  \end{itemize}
\end{corollary}
\begin{proof}
  We apply Corollary~\ref{cor:cube-system} to $g$ to obtain a system of cubes $S_0^{(0)},\dots,S_{s+1}^{(0)}$ of parameter $\delta_0 \gg_s \delta$, with $S_0 \subseteq X$ and $\partial^{s+1} g(c) = 0$ for all $c \in S_{s+1}$.
  
  Next we apply Lemma~\ref{lem:split} to $S_0^{(0)},\dots,S_{s+1}^{(0)}$ to obtain a subsystem $S^{(1)}_0,\dots,S^{(1)}_{s+1}$ with parameter $\delta_1 \gg \delta_0 - O_s\big(h_0^{1/O_s(1)} / \delta_0\big)$ with the property that each connected component of $S^{(1)}_0$ has normalized Cheeger constant at least $h_0$.
  By a suitable choice of $h_0 \gg_s \delta^{O_s(1)}$ we can ensure $\delta_1 \ge \delta_0 / 2$.
  
  By Remark~\ref{rem:component-cubes}, we can pass to just one of these components to obtain a system of cubes $S^{(2)}_0,\dots,S^{(2)}_{s+1}$ with the same parameters.  Note $S^{(2)}_0 \subseteq S$ is large, in that $\mu\big(S^{(2)}_0\big) \ge \delta_1$ (see Remark~\ref{rem:cube-system-large} again).

  We now state for induction the following claim.
  \begin{claim}
    Let $t \ge 1$ and suppose $T_0,\dots,T_{t+1}$ is a system of cubes with parameter $\delta'>0$, such that the graph on $T_0,T_1$ has normalized Cheeger constant $h > 0$ \uppar{in the sense of Lemma~\ref{lem:split}}. Also suppose $g \colon T_0 \to \RR$ satisfies $\partial^{t+1} g(c) = 0$ for all $c \in T_{t+1}$, and let $\eps'>0$ be a parameter.
    
    Then there exists a system of cubes $S_0,\dots,S_{t+1}$ with parameter $1-\eps'$, a $(t-1,d,M)$-polynomial hierarchy $f$ on $S_0,\dots,S_{t}$, and a function $\wt{g} \colon S_0 \to \RR$, such that
    \begin{itemize}
      \item $g|_{S_0 \cap T_0} = \wt{g}|_{S_0 \cap T_0}$;
      \item $\wt{g}$ obeys the $(f,t+1,t,M)$-derivatives condition; and
      \item the quantities $D = \sum_{i=0}^{t-1} d_i$ and $M$ satisfy $D,M \ll_{t} (\eps' \delta' h)^{-O_{t}(1)}$.
    \end{itemize}
  \end{claim}

  To complete the proof, we apply the claim to $g$ and $T_k = S^{(2)}_k$, with $t=s$, $\delta' = \delta_1$, $h=h_0$ and $\eps' = \min(\eps, \delta_1 / 2)$. Using the function $\wt{g}$ and the system $S_0,\dots,S_{s+1}$ returned by the claim, together with $Y = S^{(2)}_0 \cap S_0$ (so that $\mu(Y) \ge \delta_1 - \eps' \ge \delta_1/2$), gives the result.

  \begin{proof}[Proof of claim]
    We apply Lemma~\ref{lem:extend-global} to $g$ with the further parameter $\eps_0$ to be determined. This gives integers $K, M_1 \ll_s (h \delta')^{-O_s(1)} \log(2/\eps_0)^{O_s(1)}$; functions $g_1,\dots,g_K$, where each $g_i$ for $1 \le i \le K$ is equipped with a system of cubes $T_0^{(i)},\dots,T_t^{(i)}$ of degree $t-1$ and parameter $\delta'$; sets $X \subseteq H$ and $S \subseteq C^{t+1}(H)$ with $\mu(X),\mu(S) \ge 1-\eps_0$; and a function $\wt{g} \colon X \to \RR$; all with the properties given in the statement.  
    
    By Lemma~\ref{lem:cube-system-ae}, we can build a system of cubes $S'_0,\dots,S'_{t+1}$ with parameter $1-O_t\big(\eps_0^{1/O_t(1)}\big)$ such that $X \subseteq S'_0$ and $S'_{t+1} \subseteq S$.  Then, $\wt{g}$ obeys the $(t+1)$-derivatives condition on $S'_{t+1}$ with respect to the tuple $(g_1,\dots,g_K)$ (i.e., having $d_{t-1} = K$ and $d_i=0$ for $0 \le i < t-1$), with degree $t$ and parameter $M_1$.

    Now we apply Lemma~\ref{lem:split} again to each function $g_i$ (or more accurately, to its system of cubes $T_0^{(i)},\dots,T_t^{(i)}$), with some parameter $h_1$, to get subsystems $T_0'^{(i)},\dots,T_t'^{(i)}$ with parameter $\delta_1 = \delta' - O_s\big(h_1^{1/O_s(1)} / \delta'\big)$ and $\mu\big(T_k'^{(i)} \setminus T_k^{(i)}\big) \ll_s h_1^{1/O_s(1)} / \delta'$, whose normalized Cheeger constants on each connected component are bounded below by $h_1$.  By Remark~\ref{rem:component-cubes} we may split these systems into connected components to get systems $T_0^{(i,j)},\dots,T_t^{(i,j)}$ for $1 \le i \le K$ and $1 \le j \le k_i$, where $k_i \le 1/\delta_1$ for each $i$.
    
    Now, each restriction $g_i|_{T_0^{(i,j)}}$ obeys the hypotheses of the claim, replacing $t$ by $t-1$, $h$ by $h_1$ and $\delta'$ by $\delta_1$.
    Applying the claim inductively (and referring to Remark~\ref{rem:case-s-zero} and Remark~\ref{rem:hierarchy-t0} for the absent base case $t=0$) with some new parameter $\eps_1$, for each such restriction we get a system of cubes $S_0^{(i,j)},\dots,S_t^{(i,j)}$ of parameter at least $1-\eps_1$, a polynomial hierarchy $f^{(i,j)}$ of degree $t-1$, total dimension $D^{(i,j)}$ and parameter $M^{(i,j)}$, and an extension of $g_{i,j}$ to a function $\wt g_{i,j}$ on $S_0^{(i,j)}$ that obeys the $\big(f^{(i,j)},t,t-1,M^{(i,j)}\big)$-derivatives condition, obeying the various conditions in the statement.

    Applying Lemma~\ref{lem:cube-system-ae} (or less wastefully, deploying some of the arguments in the proof), we can replace each of these systems of cubes with another one $S_0'^{(i,j)},\dots,S_{t+1}'^{(i,j)}$ of degree $t$ rather than $t-1$, with parameters at least $1-O_t\big(\eps_1^{1/O_t(1)}\big)$, and such that $S_k'^{(i,j)} \subseteq S_k^{(i,j)}$ for each $1\le k \le t$.

    If we now intersect all of the systems of cubes in sight, i.e, setting
    \[
      S_k = S'_k \cap \bigcap_{i,j} S_k'^{(i,j)}
    \]
    for each $0 \le k \le t+1$, then this is a system of cubes of parameter $1-\eps_2$ where
    \begin{equation}
      \label{eq:eps2}
      \eps_2 \ll_t \eps_0^{1/O_t(1)} + (K/\delta_1) \eps_1^{1/O_t(1)}
    \end{equation}
    (by a union bound).  Then, we can form a polynomial hierarchy $f$ by combining all the hierarchies $f^{(i,j)}$, as well as the functions $(\wt g_{i,j})_{i \in [K], j \in [k_i]}$; i.e.,
    \[
      f_{=k} = \bigoplus_{i,j} f_{=k}^{(i,j)}
    \]
    for $1 \le k < t-1$, and $f_{=t-1} = (\wt g_{1,1} ,\dots,\wt g_{K,k_K})$.  This is a $(t-1,d,M)$-polynomial hierarchy on $S_0,\dots,S_{t+1}$, where
    \[
      D = \sum_{k=0}^{t-1} d_i \le K/\delta_1 + \sum_{i,j} D^{(i,j)} \ll_t (K / \delta_1) (\eps_1 \delta_1 h_1)^{-O_t(1)}
    \]
    and
    \[
      M = \max \left(M_1, \max_{i,j} M^{(i,j)}\right) \ll_t \max \left((\delta' h)^{-O_t(1)} \log(2/\eps')^{O_t(1)}, (h_1 \delta_1 \eps_1)^{-O_t(1)} \right) \ ,
    \]
    following directly from the facts that each $f^{(i,j)}$ is a polynomial hierarchy and that each $\wt g_{i,j}$ obeys a derivatives condition on $f^{(i,j)}$.  Moreover, by construction $\wt{g}$ obeys the $(f,t+1,t, M)$-derivatives condition on $S_{t+1}$: indeed, it obeys a derivatives condition with respect to the functions $g_{i,j}$ on $S'_{t+1}$; the tuple $f$ contains all functions $\wt g_{i,j}$; and for any $c \in S_{t+1}$ we have $g_{i,j}(c(\omega)) = \wt g_{i,j}(c(\omega))$ for each $\omega \in \llbracket t+1 \rrbracket$ (since $S_0 \supseteq S_0^{(i,j)}$).

    Finally, we must exhibit a choice of the parameters $\eps_0$, $\eps_1$, $h_1$ so that $\eps_2 \le \eps'$ but $M$,$D$ are not too large.  For a suitable $h_1 \gg_t \delta'^{O_t(1)}$ we can guarantee $\delta_1 \ge \delta'/2$. Next we pick a suitable $\eps_0 \gg_t \eps'^{O_t(1)}$ so that the first term in~\eqref{eq:eps2} is at most $\eps'/2$.  This means $K \ll_t (\delta' \eps' h)^{-O_t(1)}$, and so we can pick $\eps_1$ small enough that the second term in~\eqref{eq:eps2} is again at most $\eps'/2$, while keeping $\eps_1 \gg_t (\delta' \eps')^{O_t(1)}$.
  \end{proof}

  This completes the proof of Corollary~\ref{cor:poly-hierarchy-everything}.
\end{proof}

\section{Improving the global structure}%
\label{sec:hierarchy-improve}

Again in this section we let $H$ denote an arbitrary finite abelian group unless otherwise stated. We will need to assume at times that $H$ has no large subgroups.

Our current task is to show that polynomial hierarchies (in the sense of Definition~\ref{def:hierarchy}) can always be replaced by ones with additional structural properties.  This prepares us for proving a structure theorem for polynomial hierarchies in the next section.

\subsection{Outline of the argument}

It is necessary to refine the derivatives condition from Definition~\ref{def:hierarchy} to get more information about the functions $b(c,\omega)$.  If we can assume these functions exhibit some kind of rigid algebraic structure, this can later be exploited to obtain genuinely algebraic objects.

To illustrate this we consider the case $s=1$.  For simplicity we assume the underlying system of cubes is everything.  The set-up is that we have one function $f \colon H \to \RR$ and a tuple of constant functions $f_{=0} = (f_{0,1},\dots,f_{0,d_0})$ such that for some $b \colon C^2(H) \times \llbracket 2 \rrbracket \to \ZZ^{d_0}$ we have
\begin{equation}
  \label{eq:s2-cocycle}
  \partial^2 f(c) = \big(b(c,00) - b(c,01) - b(c,10) + b(c,11)\big) \cdot f_{=0}  ;
\end{equation}
this is obtained from~\eqref{eq:deriv-condition} by noting that since $f_{=0}$ is constant we can drop the evaluation at $c(\omega)$ and group terms.  Our first refinement is to note that allowing four values $b(c, \omega)$ for each $c$ is clearly too much flexibility, and they can be replaced with a single value $B(c) = b(c,00) - b(c,01) - b(c,10) + b(c,11)$, the remaining freedom in the choice of $b(c,-)$ being entirely spurious.  Hence, from now on we work with $B(c)$ instead of $b(c,\omega)$.

The next step is to investigate how $B(c)$ behaves across different cubes $c$. Specifically, we consider the operation of ``glueing'' two cubes together along a common face; see Figure~\ref{fig:glueing} below for a graphical representation.  In symbols, if $c_0 = [[x,x+h],[z,z+h]]$ and $c_1=[[y,y+h],[z,z+h]]$ are two elements of $C^2(H)$ with a common upper face $[z,z+h]$, we can glue along this face to obtain $c = [[x,x+h],[y,y+h]]$.

If we know the values of $B(c_0)$, $B(c_1)$, then a possible value of $B(c)$ is just $B(c_0)-B(c_1)$: this will automatically satisfy~\eqref{eq:s2-cocycle}, since any derivative $\partial^2 g$ obeys the cocycle-type identity
\[
  \partial^2 g([[x,x+h],[y,y+h]]) = \partial^2 g([[x,x+h],[z,z+h]]) - \partial^2 g([[y,y+h],[z,z+h]])
\]
which we apply with $g=f$.  However, it is not necessarily true that $B(c) = B(c_0)-B(c_1)$, since $B(c)$ and $B(c_0)-B(c_1)$ could be two different bounded elements of $\ZZ^{d_0}$ whose dot products with $f_{=0}$ give the same value.

If this happens, though, it means that there is some low-weight relation between the constants $(f_{0,1},\dots,f_{0,d_0})$.  We think of this as a kind of \emph{redundancy} in the hierarchy at degree $0$.  In this case, we can apply some change of coordinates to find a smaller set of constants $f'_{=0}$, such that each original constant $f_{0,j}$ is a bounded integer combination of the new constants, but all the original low-weight relations have been quotiented out.  We then need to modify $B$ to $B'$ to reflect this coordinate change.

After this correction, the cocycle-like condition $B'(c) = B'(c_0) - B'(c_1)$ will indeed hold.  In Lemma~\ref{lem:cocycle} we prove a strong classification of functions satisfying cocycle-like conditions of this type: in this case, the conclusion would be that $B' = \partial^2 g$ for some $g \colon H \to \RR$ which is itself a nil-polynomial of degree $1$, and this is enough to complete the proof (although for $s>1$ there is rather more left to do, as covered in the next section).

The remaining task is to formalize the preceding discussion, introduce tolerance for errors (i.e., when $S_k \ne C^k(H)$), and generalize it to the case of higher $s$.  Unfortunately the last two of these, and the last in particular, seem to require far greater technical complexity than this sketch might suggest.  We briefly summarize some of the necessary components.

\begin{itemize}
  \item We first address the unnecessary freedom implicit in allowing $2^k$ distinct values $b(c,\omega)$ in~\eqref{eq:deriv-condition} for each $c$, for general $k$ and $t$.  Because $f_{\le t-1}$ no longer consists of constant functions this is a bit more involved.  However, it is possible (and useful) to put $b(c,\omega)$ into a simply stated normal form; see Definition~\ref{def:normal-form}. This is all addressed in Section~\ref{subsec:normal-form}.

  \item In Section~\ref{subsec:cocycles-and-strong} we investigate the behavior of $b(c,\omega)$ (assumed to be in normal form) under glueing cubes together.  This leads us to define generalizations of the cocycle-type conditions discussed above, which describe what we expect to see in the absence of ``redundancies''.  This is formalized in the key notion of the \emph{strong derivatives condition} in Definition~\ref{def:strong}, which uses among other things the notion of a what we term a \emph{generalized cocycle} in Definition~\ref{def:gen-cocycle}. 
    
    In the same section we state and prove an important structural result classifying these generalized cocycles (Lemma~\ref{lem:gen-cocycle}).  The core of the proof is an averaging argument that appears in~\cite{cs}, and which is not conceptually difficult; however, the need to handle the $99\%$ case (where $S_k \ne C^k(H)$) efficiently, imposes costs on the proof in terms of complexity.

  \item Lastly, in Section~\ref{subsec:redundancy} we consider what to do when ``redundancies'' are in fact present.  In the sketch above, a redundancy meant a linear combination of functions which is zero; in general the term has to include linear combinations of functions which behave a bit like they had smaller degree than they are supposed to.
    
    The main steps are to (i) formalize a precise notion of ``redundancy'', (ii) show that any failure to obey a strong derivatives condition must come from a redundancy in this precise sense, and (iii) demonstrate how to eliminate redundancies from a hierarchy in a protracted inductive fashion.  Again, this involves a great deal of technical work.
\end{itemize}

\subsection{A normal form for $b$-functions}%
\label{subsec:normal-form}

Recall our goal is to reduce the amount of freedom in functions $\omega \mapsto b(c,\omega)$ appearing in the derivatives condition~\eqref{eq:deriv-condition}.  That is, we would like a normal form for such $b$ which is unique up to elementary manipulations that do not alter the right hand side of~\eqref{eq:deriv-condition}.  In the discussion above for $s=1$, this exactly means those manipulations that do not alter the value of the quantity we termed $B(c)$.

The notation $F_\eta$, $F^\eta$ and $F_i$, $F^i$ to denote certain upper or lower faces of $\llbracket k \rrbracket$ will be used extensively in this section, and is given in Section~\ref{subsec:notation}.

The normal form we will use is not hard to define.

\begin{definition}%
  \label{def:normal-form}
  If $d_0,\dots,d_{t-1}$ are integers, we say a function
  \[
    \fb \colon \llbracket k \rrbracket \to \bigoplus_{r=0}^{t-1} \ZZ^{d_r}
  \]
  is in \emph{normal form}, if $\fb_{=i}(\omega) = 0$ for all $0 \le i \le t-1$ and all $\omega \in \llbracket k \rrbracket$, $|\omega| \ge i+1$.
  Similarly, if $S \subseteq C^k(H)$ we say a function
  \[
    b \colon S \times \llbracket k \rrbracket \to \bigoplus_{r=0}^{t-1} \ZZ^{d_r}
  \]
  is in normal form if $\omega \mapsto b(c,\omega)$ is normal form for each $c \in S$.
\end{definition}

We can show that any function $b$ appearing in a derivatives condition on a polynomial hierarchy can be put into normal form.

\begin{lemma}%
  \label{lem:canonical-b}
  Suppose that $f$ is an $(s,d,M)$-polynomial hierarchy on a system of cubes $S_0,\dots,S_{s+1}$; that $S \subseteq C^k(H)$ is a set of cubes such that $c|_F \in S_{\dim F}$ for each $c \in S$ and each face $F$ of $\llbracket k \rrbracket$ of dimension at most $s+1$; and that $g\colon S_0 \to \RR$ and $b$ obey the $(f,k,t,M)$-derivatives condition on $S \subseteq C^k(H)$.  Here $k,t$ are integers satisfying $k \ge 0$ and $1 \le t \le s+1$.
  
  Then there exists a function $b' \colon S \times \llbracket k \rrbracket \to \bigoplus_{i=0}^{t-1} \ZZ^{d_i}$ in normal form such that $g$ and $b'$ obey the $\big(f,k,t, O_{k,t}(M)\big)$-derivatives condition.  Also, for the part $b'_{=t-1}$ we may give an explicit formula:---
  \[
    b'_{=t-1}(c,\omega') = \sum_{\omega \in \llbracket k \rrbracket} Z_{t-1}(\omega,\omega') b_{=t-1}(c,\omega)
  \]
  where $Z_r(\omega,\omega')$ are the integer coefficients
  \begin{equation}
    \label{eq:zr}
    Z_r(\omega,\omega') = \sum_{\substack{\eta \in \llbracket k \rrbracket \\ \omega \supseteq \eta \supseteq \omega' \\ |\eta| \le r}} (-1)^{|\omega|-|\eta|}  .
  \end{equation}
\end{lemma}

We will not prove anything along those lines now, but it will turn out that this does indeed account for all the surplus freedom in the choice of $b$; i.e., $b'$ satisfying these conditions is unique under elementary manipulations.

\begin{remark}%
  \label{rem:normal-form}
  Given any abelian group $A$ and any function $\tau \colon \llbracket k \rrbracket \to A$, and defining $\tau' \colon \llbracket k \rrbracket \to A$ as above by
  \[
    \tau'(\omega') = \sum_{\omega \in \llbracket k \rrbracket} Z_r(\omega,\omega') \tau(\omega)
  \]
  using the same coefficients $Z_r$, it is automatic that $\tau'(\omega') = 0$ when $|\omega'| \ge r+1$, and we also observe
  \[
    (\tau - \tau')(\omega') = -\sum_{\substack{\omega \supseteq \eta \supseteq \omega' \\ |\eta| \le r \\ \omega \ne \omega'}} (-1)^{|\omega| - |\eta|} \tau(\omega) 
    = \sum_{\substack{\omega \supseteq \eta \supseteq \omega' \\ |\eta| \ge r+1 \\ \omega \ne \omega'}} (-1)^{|\omega| - |\eta|} \tau(\omega) \\
  \]
  (see~\eqref{eq:cube-orthog} below)
  or equivalently
  \[
    \tau - \tau' = \sum_{|\eta| \ge r+1} (-1)^{|\eta|} 1_{F_\eta}(\omega') \left(\sum_{\omega \supseteq \eta} (-1)^{|\omega|} \tau(\omega) \right)  ;
  \]
  that is, $\tau-\tau'$ is a linear combination of indicator functions of (lower) faces of dimension at least $r+1$.  In fact these two properties uniquely determine the normal form transformation described by the matrix $Z_r$.
\end{remark}

\begin{proof}[Proof of Lemma~\ref{lem:canonical-b}]

  First we note that we may as well handle each $c \in S$ independently (and given it is possible to have $|S|=1$ this is no loss).  That is: fix some $c \in C^k(H)$ whose faces $c|_F$ all lie in $S_{\dim F}$ (provided $\dim F \le s+1$), and some configuration $\fb \colon \llbracket k \rrbracket \to \bigoplus_{r=0}^{t-1} \ZZ^{d_r}$.  Suppose that $\|\fb_{=r}(\omega)\|_1 \le M^{t-r}$ for each $r$ ($0 \le r \le t-1$) and $\omega \in \llbracket k \rrbracket$.  Then our goal is to find some $\fb' \colon \llbracket k \rrbracket \to \bigoplus_{r=0}^{t-1} \ZZ^{d_r}$ such that
  \begin{equation}
    \label{eq:pointwise-derivs}
    \sum_{\omega \in \llbracket k \rrbracket} (-1)^{|\omega|} \fb(\omega) \cdot f(c(\omega))
    =
    \sum_{\omega \in \llbracket k \rrbracket} (-1)^{|\omega|} \fb'(\omega) \cdot f(c(\omega)) \ ;
  \end{equation}
  $\fb'$ is in normal form; and $\|\fb'_{=r}(\omega)\|_1 \ll_{k,t} M^{t-r}$ for each $r$ ($0 \le r \le t-1$) and $\omega \in \llbracket k \rrbracket$.

  We record the following general lemma, which can be thought of as a discrete variant of the Leibniz rule for derivatives of products.
  \begin{lemma}[``Discrete Leibnitz rule'']%
    \label{lem:product-rule}
    Suppose $\alpha,\beta \colon \llbracket k \rrbracket \to \RR$ are two functions.  Then
    \[
      \sum_{\omega \in \llbracket k \rrbracket} (-1)^{|\omega|} \alpha(\omega) \beta(\omega)
       = \sum_{\eta \in \llbracket k \rrbracket} (-1)^{|\eta|} \left(\sum_{\omega \subseteq \eta} (-1)^{|\omega|} \alpha(\omega) \right) \left(\sum_{\omega' \supseteq \eta} (-1)^{|\omega'|} \beta(\omega') \right)  .
    \]
  \end{lemma}
  \begin{proof}
    The right hand side expands to
    \[
      \sum_{\substack{\eta, \omega, \omega' \in \llbracket k \rrbracket \\ \omega \subseteq \eta \subseteq \omega'}} (-1)^{|\eta| + |\omega'| - |\omega|} \alpha(\omega) \beta(\omega')
    \]
    and so the result follows provided that for each pair $\omega \subseteq \omega'$ in $\llbracket k \rrbracket$,
    \begin{equation}
      \label{eq:cube-orthog}
      \sum_{\eta \colon \omega \subseteq \eta \subseteq \omega'} (-1)^{|\eta|} = \begin{cases} (-1)^{|\omega|} &\colon \omega = \omega' \\ 0 &\colon \omega \ne \omega' . \end{cases} 
    \end{equation}
    Indeed, the left hand side may be expressed as
    \[
      (-1)^{|\omega|} (1 - 1)^{|\omega'| - |\omega|}
    \]
    which is equal to the right hand side.
  \end{proof}

  One consequence is that for any $r$, $0 \le r \le t-1$,
  \[
    \sum_{\omega \in \llbracket k \rrbracket } (-1)^{|\omega|} \big(f_{=r}(c(\omega)) \cdot \fb_{=r}(\omega) \big) 
    = \sum_{\eta \in \llbracket k \rrbracket } (-1)^{|\eta|} \left(\sum_{\omega \subseteq \eta} (-1)^{|\omega|} f_{=r}(c(\omega)) \right) \cdot \left(\sum_{\omega' \supseteq \eta} (-1)^{|\omega'|} \fb_{=r}(\omega') \right) .
  \]
  We note that the left bracket is exactly $\partial^{|\eta|} f_{= r} \big(c|_{F_\eta}\big)$, which for terms with $|\eta| \ge r+1$ may be re-expressed in terms of lower-degree functions $f_{\le r-1}$ in the hierarchy, using the derivatives condition on each function in $f_{=r}$.  Hence, up to making lower-degree corrections we have more-or-less total freedom to choose the values $\sum_{\omega' \supseteq \eta} (-1)^{|\omega'|} \fb_{=r}(\omega')$ for $|\eta| \ge r+1$, and in particular could choose them to be zero, without affecting the left hand side of~\eqref{eq:pointwise-derivs}.  Iterating this observation for $r$ ranging from $t-1$ down to $0$ in that order will give the result.

  Formally, we set up the following claim for inductive purposes.
  \begin{claim}
    Let $f$, $k$, $t$, $M$ and $c$ be as above, let $r$ be an integer, $0 \le r \le t-1$, and suppose
    $
      \alpha \colon \llbracket k \rrbracket \to \bigoplus_{i=0}^{r} \ZZ^{d_i}
    $
    is any configuration with $\|\alpha_{=i}(\omega)\|_1 \le R^{t-i}$ for each $i$ \uppar{$0\le i \le r$} and $\omega \in \llbracket k \rrbracket$. Then there exist
    $
      \alpha' \colon \llbracket k \rrbracket \to \bigoplus_{i=0}^{r-1} \ZZ^{d_i}
    $
    and
    $
      \beta \colon \llbracket k \rrbracket \to \ZZ^{d_{r}}
    $
    such that $\|\beta(\omega)\|_1 \le 4^k R^{t-r}$ for each $\omega$; $\|\alpha'_{=i}(\omega)\|_1 \le R^{t-i} + 4^k R^{t-r} M^{r-i}$ for each $i$ \uppar{$0 \le i \le r-1$} and each $\omega$; $\beta(\omega) = 0$ for each $|\omega| \ge r+1$; and
    \begin{equation}
      \label{eq:claim-alpha-beta}
      \sum_{\omega \in \llbracket k \rrbracket} (-1)^{|\omega|} \alpha(\omega) \cdot f_{\le r}(c(\omega))
      = 
      \sum_{\omega \in \llbracket k \rrbracket} (-1)^{|\omega|} \alpha'(\omega) \cdot f_{\le r-1}(c(\omega))
      +
      \sum_{\omega \in \llbracket k \rrbracket } (-1)^{|\omega|} \beta(\omega) \cdot f_{=r}(c(\omega)) .
    \end{equation}
    Moreover, explicitly we have
    \begin{equation}
      \label{eq:claim-beta-formula}
      \beta(\omega') = \sum_{\omega \in \llbracket k \rrbracket} Z_{r}(\omega,\omega') \alpha_{=r}(\omega)
    \end{equation}
    where $Z_r$ are the coefficients in~\eqref{eq:zr}.
  \end{claim}
  \begin{proof}[Proof of claim]
    First set
    \[
      B(\eta) = \sum_{\omega \supseteq \eta} (-1)^{|\omega|} \alpha_{=r}(\omega)
    \]
    for each $\eta \in \llbracket k \rrbracket$, so that
    \[
      \sum_{\omega \in \llbracket k \rrbracket} (-1)^{|\omega|} \alpha_{=r}(\omega) \cdot f_{= r}(c(\omega))
      = \sum_{\eta \in \llbracket k \rrbracket} (-1)^{|\eta|} B(\eta) \cdot \partial^{|\eta|} f_{= r}\big(c|_{F_\eta}\big)
    \]
    by Lemma~\ref{lem:product-rule}.  We then split the sum into cases $|\eta| \le r$ and $|\eta| \ge r+1$.

    Since $f$ is a polynomial hierarchy, for each $1 \le i \le d_{r}$ and each $\eta \in \llbracket k \rrbracket$, $|\eta| \ge r+1$, there is a configuration
    $
      \gamma^{i,\eta} \colon F_\eta \to \bigoplus_{j=0}^{r-1} \ZZ^{d_j}
    $
    with $\|\gamma^{i,\eta}_{=j}(\omega)\|_1 \le M^{r-j}$ for each $j$ and each $\omega \in F_\eta$, such that
    \[
      \partial^{|\eta|} f_{i}\big(c|_{F_\eta}\big) = \sum_{\omega \in F_{\eta}} (-1)^{|\omega|} \gamma^{i,\eta}(\omega) \cdot f_{\le r-1} (c(\omega)) \ ;
    \]
    indeed, when $|\eta| = r+1$ this is by definition, and when $|\eta| > r+1$ this follows by decomposing $\partial^{|\eta|} f_i\big(c|_{F_\eta}\big)$ as a sum of $(r+1)$-derivatives of $f_{i}$ on faces of $c|_{F_{\eta}}$ of dimension exactly $r+1$.\footnote{Another way of saying this is that the $(f,k,t,M)$-derivatives condition immediately implies the $(f,k',t,M)$-derivatives condition for $k' > k$.}

    We deduce that
    \[
      \sum_{\omega \in \llbracket k \rrbracket} (-1)^{|\omega|} \alpha_{=r}(\omega) \cdot f_{= r}(c(\omega))
      = \begin{aligned}[t]
        & \sum_{|\eta| \le r} (-1)^{|\eta|} B(\eta) \cdot \left(\sum_{\omega' \subseteq \eta} (-1)^{|\omega'|} f_{= r}(c(\omega')) \right) \\
        +\ &
      \sum_{|\eta| \ge r+1} (-1)^{|\eta|}  \sum_{i=1}^{d_r} B(\eta)_i \left(\sum_{\omega' \subseteq \eta} \gamma^{i,\eta}(\omega') \cdot f_{\le r-1}(c(\omega')) \right) \ .
      \end{aligned}
    \]
    Hence, we may define
    \[
      \alpha'(\omega') = \alpha_{\le r-1}(\omega') + \sum_{|\eta| \ge r+1} (-1)^{|\eta|} \sum_{i=1}^{d_{r}} B(\eta)_i \gamma^{i,\eta}(\omega')
    \]
    and
    \[
      \beta(\omega') = \sum_{\substack{\eta \supseteq \omega' \\ |\eta| \le r}} (-1)^{|\eta|} B(\eta)
    \]
    which by the preceding discussion satisfy~\eqref{eq:claim-alpha-beta}.  Unwrapping the definitions of $\beta$, $B$ and $Z_r$ immediately gives~\eqref{eq:claim-beta-formula}, and the bounds on $\|\beta(\omega)\|_1$ and $\|\alpha'_{=i}(\omega)\|_1$ follow from their definitions.
  \end{proof}

  To finish the proof of the lemma, we define $\fb'_{=r}$ and auxiliary configurations 
  \[
    \alpha^{(r)} \colon \llbracket k \rrbracket \to \bigoplus_{i=0}^r \ZZ^{d_i}
  \]
  for $r$ from $t-1$ down to $0$ recursively as follows: set $\alpha^{(t-1)} = \fb$, and then let $\alpha^{(r-1)}$ and $\fb'_{=r}$ be the configurations $\alpha'$ and $\beta$ respectively obtained by applying the claim to $\alpha^{(r)}$.  
\end{proof}

We note that a direct formula for $b'_{=i}$ in terms of $b_{=i}$ is not possible for $i < t-1$, as there are necessarily correction terms which depend on the polynomial hierarchy $f$.

\subsection{Cocycles and the strong derivatives condition}%
\label{subsec:cocycles-and-strong}

For this discussion suppose that $S_0,\dots,S_{s+1}$ is a system of cubes, $f$ is an $(s, d, M)$-polynomial hierarchy on $S_0,\dots,S_{s+1}$, that $g \colon S_0 \to \RR$ and $b(c,\omega)$ obey an $(f,k,t,M)$-derivatives condition on $S \subseteq C^k(H)$ (where $k \ge 0$ and $1 \le t \le s+1$), and that $b$ is in normal form in the sense of Definition~\ref{def:normal-form}.
We now consider compatibility conditions on $b(c,\omega)$ across different choices of $c$.

The key set-up consists of two cubes that can be glued together to form another cube.  Specifically, let $T_{k,i}$ denote the set of all triples of cubes $(c_0,c_1,c)$ in $C^k(H)$, such that $c_0$ and $c_1$ have the same upper face in direction $i$, i.e.\ $c_0|_{F^i} = c_1|_{F^i}$, and $c=[c_0|_{F_i}, c_1|_{F_i}]_i$ is obtained by glueing $c_0$ and $c_1$ together along that face.   Concretely, $T_{k,i}$ is indexed by tuples $(x,h_1,\dots,h_k,h_i') \in H^{k+2}$, as follows:
\begin{equation}
  \label{eq:tk}
  T_{k,i} = \Big\{ (c_0,c_1,c) \colon
  \begin{aligned}[t]
    &(x,h_1,\dots,h_k,h_i') \in H^{k+2}, \\
    c &= \angle(x; h_1,\dots,h_i,\dots,h_k) \\
    c_0 &= \angle(x; h_1,\dots,h_i+h_i',\dots,h_k)\\
    c_1 &= \angle(x+h_i; h_1,\dots,h_i',\dots,h_k)\ \Big\} \ .
  \end{aligned}
\end{equation}
The derivative $\partial^k g$ always obeys the cocycle-like condition $\partial^k g(c) - \partial^k g(c_0) + \partial^k g (c_1) = 0$ for any triple $(c_0,c_1,c) \in T_{k,i}$ where it is well-defined: indeed, each value $g(x)$ appears an equal number of times with a positive and negative sign.
When $k=2$ and $i=1$ this set-up is illustrated in Figure~\ref{fig:glueing}.

\begin{figure}
  \begin{center}
    \begin{tikzpicture}
      \draw
        (0,0) node[below left]  {$\scriptscriptstyle x$} --
        (3,0) node[below] {$\scriptscriptstyle x+h_1$} --
        (3,2) node[above] {$\scriptscriptstyle x+h_1+h_2$} --
        (0,2) node[above left]  {$\scriptscriptstyle x+h_2$} --
        (0,0)
        (1.5,1) node {$c$}

        (3,0) --
        (6,0) node[below right] {$\scriptscriptstyle x+h_1+h_1'$} --
        (6,2) node[above right] {$\scriptscriptstyle x+h_1+h_1'+h_2$} --
        (3,2)

        (4.5,1) node {$c_1$};

      \draw[decorate,decoration={brace,amplitude=10pt}]
        (6,-0.5) --
        (0,-0.5) node [midway,below,yshift=-10pt] {$c_0$};
    \end{tikzpicture}
    \caption{Glueing cubes together with $k=2$, $i=1$\label{fig:glueing}}
  \end{center}
\end{figure}
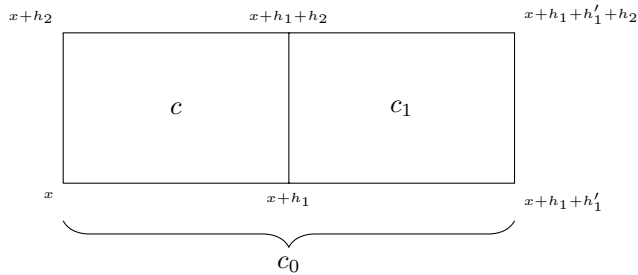

It follows that (provided $c_0,c_1 \in S$)
\begin{equation}
  \label{eq:trying-to-glue}
  \partial^k g(c) = \sum_{\omega \in \llbracket k \rrbracket} (-1)^{|\omega|} f_{\le t-1}(c_0(\omega)) \cdot b(c_0,\omega) - 
\sum_{\omega \in \llbracket k \rrbracket} (-1)^{|\omega|} f_{\le t-1}(c_1(\omega)) \cdot b(c_1,\omega)
\end{equation}
by combining the derivatives conditions at $c_0$ and $c_1$.
Regrouping terms, we see that
\begin{equation}
  \label{eq:really-trying-to-glue}
  \partial^k g(c) = \sum_{\omega \in \llbracket k \rrbracket} (-1)^{|\omega|} f_{\le t-1}(c(\omega)) \cdot \fb(\omega) + 
\sum_{\omega \in F^i} (-1)^{|\omega|} f_{\le t-1}(c_0(\omega)) \cdot (b(c_0,\omega) - b(c_1,\omega))
\end{equation}
where
\[
  \fb(\omega) = \begin{cases} b(c_0, \omega) &\colon \omega(i) = 0 \\ b(c_1, \omega \setminus \{i\}) &\colon \omega(i) = 1  . \end{cases}
\]
We would like this to be an alternative derivatives condition for $g$ at $c$, and then use some sort of uniqueness to compare $\fb$ with $b(c,\omega)$ and thereby get a cocycle-type compatibility statement between $b(c,-)$, $b(c_0,-)$ and $b(c_1,-)$.  There are two problems:---
\begin{enumerate}[label=(\roman*)]
  \item The second term involves values of $f$ at the vertices of the common upper face $c_0|_{F^i} = c_1|_{F^i}$, which are not vertices of $c$.
    Hence, if the associated coefficients are non-zero this is not a legal derivatives condition at $c$.
  \item If it happens to be true that $b(c_0,\omega) = b(c_1,\omega)$ for all $\omega \in F^i$, that term disappears and we do get an alternative derivatives condition at $c$ using the configuration $\fb$.
    However, $\fb$ is not typically in normal form, and would have to be placed in normal form before we could compare it to $b(c,-)$, which means in particular that we would lose explicit control over all but the top piece $\fb_{=t-1}$.
\end{enumerate}

It turns out it is natural to deal with (i) by introducing a requirement that $b(c_0,\omega) = b(c_1,\omega)$ for all $\omega \in F^i$.  This is usually true in practice, e.g.\ for the example functions and suitable hierarchies considered in Remark~\ref{rem:examples} (with everything put in normal form).  A failure of this requirement will be interpreted as evidence of some redundancy.

The objection in (ii) is unavoidable: a na{\"\i}ve cocycle-like compatibility condition simply does not hold for $b_{=i}$ for $i < t-1$ in this setting.  However, in the overall geography of the proof, having structure in $b_{=t-1}$ will be good enough to start an induction.

It is therefore natural to make the following definitions concerning the structure we expect to see in $b$, if there are no redundancies (a term which as yet still has no precise definition).  We refer to the definition of $T_{k,i}$ from~\eqref{eq:tk}.

\begin{definition}%
  \label{def:compatible}
  If $S \subseteq C^k(H)$, $b \colon S \times \llbracket k \rrbracket \to \ZZ^{d}$ is a function and $\delta>0$ a parameter, we say $b$ is \emph{$\delta$-almost upper compatible} on $S$ if the following holds: for all $\omega \in \llbracket k \rrbracket$ and all $i \in [k]$ with $\omega(i) = 1$, we have
  \[
    \big|\big\{ (c_0,c_1,c) \in T_{k,i} \colon c_0,c_1 \in S,\ b(c_0,\omega) \ne b(c_1,\omega) \big\}\big| \le \delta |H|^{k+2} \ .
  \]
  If this holds for $\delta=0$, we just say $b$ is \emph{upper compatible} on $S$. In that case, this equivalently states that for each fixed $\omega \in \llbracket k \rrbracket$ the function $c \mapsto b(c,\omega)$ only depends on the value of the upper face $c|_{F^\omega}$.
\end{definition}

\begin{definition}%
  \label{def:gen-cocycle}
  Let $S \subseteq C^k(H)$, let $r$ be an integer \uppar{$0 \le r \le k-1$} and let $b \colon S \times \llbracket k \rrbracket \to \ZZ^{d}$ a function in normal form.  Finally let $\delta>0$ be a parameter.
  
  We say $b$ is a \emph{generalized $k$-cocycle of type $r$ and loss $\delta>0$} on $S$, if (i) $b$ is $\delta$-almost upper compatible on $S$, and (ii) the following holds for each $i \in [k]$.  Of the triples $(c_0,c_1,c) \in T_{k,i}$ with $c_0,c_1,c \in S$, at most $\delta |H|^{k+2}$ fail to have the following property: for each $\omega' \in \llbracket k \rrbracket$ we have
  \[
    b(c,\omega') = \sum_{\omega \in \llbracket k \rrbracket} Z_r(\omega,\omega') \fb(\omega)
  \]
  where
  \[
    \fb(\omega) = \begin{cases} b(c_0, \omega) &\colon \omega(i) = 0 \\ b(c_1, \omega \setminus \{i\}) &\colon \omega(i) = 1 \end{cases}
  \]
  and $Z_r(\omega,\omega')$ are the coefficients defined in~\eqref{eq:zr}.
\end{definition}

\begin{definition}%
  \label{def:strong}
  If $X \subseteq H$, $(f_{i,j})$ is a tuple of functions $X \to \RR$ of shape $d_0,\dots,d_s$, and $g \colon X \to \RR$ and $b$ obey the $(f,k,t,M)$-derivatives condition on $S \subseteq C^k(H)$ (where $k \ge 0$ and $1 \le t \le s+1$), we say they satisfy the \emph{strong $k$-derivatives condition} on $S$ with respect to $f$, with degree $t$, parameter $M \ge 1$, and loss $\delta>0$, if:---
  \begin{enumerate}[label=(\roman*)]
    \item $b$ is in normal form;
    \item $b_{=i}$ is $\delta$-almost upper compatible on $S$ for each $i$ ($1 \le i \le t-1$); and
    \item $b_{=t-1}$ is a generalized $k$-cocycle of type $t-1$ and loss $\delta$ on $S$.
  \end{enumerate}
  As before we often just say $g$ obeys the \emph{$(f,k,t,M,\delta)$-strong derivatives condition} on $S$.
\end{definition}

\begin{remark}%
  \label{rem:b0-compat}
  Note that for $b$ in normal form, $b_{=0}$ satisfies $b_{=0}(c,\omega) = 0$ whenever $\omega \ne (0,\dots,0)$. It follows that $b_{=0}$ is automatically upper compatible, because all the values $b(c,\omega)$ considered in Definition~\ref{def:compatible} are zero.  Hence, it is not important whether Definition~\ref{def:strong}(ii) is stated for $1 \le i \le t-1$ or $0 \le i \le t-1$.
\end{remark}

The definition of a generalized $k$-cocycle of type $r$ given here, motivated by the discussion above, appears fairly ad-hoc: it just records the conditions we expect the top-level coefficients $b_{=t-1}$ to satisfy.
Nonetheless, such objects turn out to have a rigid algebraic classification.

\begin{lemma}%
  \label{lem:gen-cocycle}
  Let $S_0,\dots,S_{k}$ be a system of cubes of parameter $1-\eps$, where $\eps \le \eps_0(k)$ is sufficiently small, and let $\rho \colon S_k \times \llbracket k \rrbracket \to \ZZ^d$ be a generalized $k$-cocycle of type $r$ and loss at most $\delta$ on $S_k$.

  Then there exists a function $\lambda \colon H \to \RR^d$ with the following properties.  For each $c \in C^k(H)$, we consider the normal form of the configuration $\omega \mapsto \lambda(c(\omega))$, in the sense of Lemma~\ref{lem:canonical-b} and Remark~\ref{rem:normal-form}: i.e., we define
  \begin{equation}
    \label{eq:big-lambda}
    \Lambda(c,\omega') = \sum_{\omega \in \llbracket k \rrbracket} Z_r(\omega,\omega') \lambda(c(\omega))
  \end{equation}
  where $Z_r(-,-)$ are defined in~\eqref{eq:zr}.  Then:---
  \begin{enumerate}[label=(\roman*)]
    \item there are at most $O_{k}\big(\delta^{1/O_k(1)}\big) |C^k(H)|$ cubes $c \in S_k$ which do not satisfy $\Lambda(c,\omega) = \rho(c,\omega)$ for each $\omega \in \llbracket k \rrbracket$;
    \item for every $c \in C^{k-r}(H)$, the derivative $\partial^{k-r} \lambda(c)$ lies in $\ZZ^d$; equivalently, $\lambda \bmod \ZZ^d \colon H \to \RR^d / \ZZ^d$ is a polynomial map of degree $k-r-1$ \uppar{see Definition~\ref{def:poly-map}};
    \item we have $\sup_{x \in H} \|\lambda(x)\|_1 \ll_k \sup_{c \in S_k} \|\rho(c)\|_1 + d$; and
    \item for all $c \in C^{k}(H)$ and $\omega \in \llbracket k \rrbracket$ we have $\Lambda(c,\omega) \in \ZZ^d$.
  \end{enumerate}
\end{lemma}

It is a straightforward fact (which we will not prove) that every function $\Lambda$ obtained in this way---i.e., by applying the normal form transformation to some $\lambda$ as in~\eqref{eq:big-lambda}---is automatically a generalized $k$-cocycle of type $r$.
So, the lemma implies we can ``explain'' every generalized $k$-cocycle $\rho$ of type $r$ in a natural way; specifically by (i) finding a polynomial map $H \to \RR^d \bmod \ZZ^d$, (ii) choosing some lift to a function $\lambda \colon H \to \RR^d$, and (iii) taking the normal form of $(c,\omega) \mapsto \lambda(c(\omega))$, in the sense of Lemma~\ref{lem:canonical-b}, to obtain $\rho$.
Moreover, this statement is robust under small losses $\delta>0$ (and handling this loss accounts for a large part of the complexity of the proof.)

The remainder of this subsection is spent in the proof of Lemma~\ref{lem:gen-cocycle}, which is fairly involved.  We first switch the discussion from generalized cocycles to simpler objects which, following~\cite{cs}, we refer to simply as (non-generalized) \emph{cocycles}.  The definition differs from~\cite{cs} only to introduce epsilon errors.

\begin{definition}%
  \label{def:cocycle}
  Let $H$ and $A$ be any abelian groups, and $S \subseteq C^k(H)$ a subset, where $k \ge 1$.  A function $\rho \colon S \to A$ is called a \emph{$k$-cocycle} on $S$ with loss $\delta$, if the following holds for each $i \in [k]$:
    at most $\delta |H|^{k+2}$ triples $(c_0,c_1,c) \in T_{k,i}$ with $c_0,c_1,c \in S$, fail to satisfy the identity
    \begin{equation}
      \label{eq:proper-cocycle}
      \rho(c) = \rho(c_0) - \rho(c_1) .
    \end{equation}
    In other words, $\rho$ is (almost surely) additive under glueing two cubes together along a common face in direction $i$.
\end{definition}

We remark that $k$-cocycles (on $S$, with loss $\delta$) are precisely generalized $k$-cocycles of type $0$ (on $S$, with loss $\delta$), although in this special case the definition is much more straightforward and the equivalence is not totally obvious.

The key example of cocycles are derivatives $\rho = \partial^k \lambda$ where $\lambda \colon X \to A$ is any function on $X \subseteq A$; these could be considered ``coboundaries'', although we will not formally adopt this term.  For general groups $A$, and even in the error-free case $S=C^k(H)$, $\delta=0$, there may exist cocycles $\rho$ which are not of this type.  However, when $A=\RR^d$ the ``cohomology'' is trivial, and every $k$-cocycle agrees almost surely (as $\delta \rightarrow 0$, and provided $\mu(S) \approx 1$) with a coboundary $\partial^k \lambda$.

The technical statement along these lines that we will need is the following (which is in fact a special case of Lemma~\ref{lem:gen-cocycle}).

\begin{lemma}%
  \label{lem:cocycle}
  Let $H$ be a finite abelian group, $S_0,\dots,S_{k}$ a system of cubes on $H$ with parameter $1-\eps$ where $\eps \le \eps_0(k)$ is sufficiently small, and $\rho \colon S_k \to \ZZ^d$ a $k$-cocycle on $S_k$ with loss $\delta$.

  Then there exists a function $\lambda \colon H \to \RR^d$ such that:---
  \begin{enumerate}[label=(\roman*)]
    \item $\partial^k \lambda(c) = \rho(c)$ for all but at most $O_k\big(\delta^{1/O_k(1)}\big) |C^k(H)|$ cubes $c \in S_k$;
    \item $\partial^k \lambda(c) \in \ZZ^d$ for every $c \in C^k(H)$, or equivalently $\lambda \bmod \ZZ^d \colon H \to \RR^d/\ZZ^d$ is a polynomial map of degree $k-1$ \uppar{see Definition~\ref{def:poly-map}}; and
    \item we have $\sup_{x \in H} \|\lambda(x)\|_1 \ll_k \sup_{c \in S_k} \|\rho(c)\|_1 + d$.
  \end{enumerate}
\end{lemma}

The case $\eps=0$ and $S = C^k(H)$ essentially appears as a special case of~\cite[Lemma 3.19]{cs} (or~\cite[Theorem 4.11]{gmv-2} or~\cite[Lemma 2.5.7]{candela-2}).  The main idea in this case is to define $\lambda$ by a straightforward averaging process:
\[
  \lambda(x) := \EE_{h_1,\dots,h_k \in H}\ \rho(\angle(x;h_1,\dots,h_k))
\]
and then verify that $\partial^k \lambda  = \rho$ for this choice of $\lambda$.  This idea is still at the heart of the proof of Lemma~\ref{lem:cocycle}; however, the presence of errors mean that these straightforward formulae are simply not true and---along with many other things---must be corrected.

Specifically, the proof of Lemma~\ref{lem:cocycle} splits into two parts.
First we find a function $\lambda' \colon H \to \RR^d$ such that parts (i) and (iii) of the statement hold.
By (i), we have that $\partial^k \lambda'(c) \in \ZZ^d$ for almost all $c \in S_k$ (since $\rho(c) \in \ZZ^d$ for these $c$), but for the remaining few $c \in C^k(H)$ part (ii) may fail.
In the second stage, we correct $\lambda'$ (if necessary) to a function $\lambda$ which further satisfies $\partial^k \lambda(c) \in \ZZ^d$ for every $c \in C^k(H)$, while leaving most of its values unchanged.
That is, we prove the following two statements.

\begin{lemma}%
  \label{lem:rough-cocycle}
  If $H$, $S_0,\dots,S_{k}$, $\eps$, $\rho$ and $\delta$ are as in the hypotheses of Lemma~\ref{lem:cocycle}, then there exists a function $\lambda \colon H \to \RR^d$ such that
  \begin{enumerate}[label=(\roman*)]
    \item $\partial^k \lambda(c) = \rho(c)$ for all but at most $O_k\big(\delta^{1/O_k(1)}\big) |C^k(H)|$ cubes $c \in S_k$; and
    \item $\sup_{x \in H} \|\lambda(x)\|_1 \ll_k \sup_{c \in S_k} \|\rho(c)\|_1$.
  \end{enumerate}
\end{lemma}

\begin{lemma}%
  \label{lem:99pc-poly}
  Let $s \ge 0$ be an integer, $H$ a finite abelian group and $A$ any abelian group.  Given $f \colon H \to A$, let
  \[
    S = \big\{ c \in C^{s+1}(H) \colon \partial^{s+1} f(c) = 0 \big\}
  \]
  and suppose $\mu(S) \ge 1 - \eps$ where $\eps < 2^{-2s-2}$.  Also define
  \begin{equation}
    \label{eq:bigX}
    X = \big\{ x \in H \colon |\{ c \in S \colon c(\vec0) = x \}| \ge (1/2) |H|^{s+1} \big\}
  \end{equation}
  \uppar{so necessarily $\mu(X) \ge 1 - 2 \eps$}.

  Then there exists a function $\wt f \colon H \to A$ such that $\partial^{s+1} \wt f$ is identically zero, and with $f(x) = \wt f(x)$ for all $x \in X$.
\end{lemma}

In the application to Lemma~\ref{lem:cocycle}, the function $f$ in Lemma~\ref{lem:99pc-poly} will be $f = \lambda \bmod \ZZ^d$ where $\lambda$ is the function obtained from Lemma~\ref{lem:rough-cocycle} (so $A = \RR^d/\ZZ^d$).

\begin{proof}[Proof of Lemma~\ref{lem:cocycle} assuming Lemma~\ref{lem:rough-cocycle} and Lemma~\ref{lem:99pc-poly}]
  We first choose $\lambda' \colon H \to \RR^d$ according to Lemma~\ref{lem:rough-cocycle}; so, $\rho(c) = \partial^k \lambda'(c)$ for all $c \in S$, where $S \subseteq S_k$ is a set such that $|S_k \setminus S| \ll_k \delta^{1/O_k(1)} |C^k(H)|$, and $\lambda'$ obeys the size bound (iii).

  Let $X \subseteq H$ be defined as in~\eqref{eq:bigX} with respect to this $S$ and with $s=k-1$.  For each $x \in S_0$ there are at least $(1 - \eps)^k |H|^k$ cubes $c \in S_k$ with $c(\vec 0) = x$. Therefore every $x \in S_0 \setminus X$ accounts for at least $\big((1-\eps)^k - 1/2\big) |H|^k$ elements $c \in S_k \setminus S$ with $c(0,\dots,0) = x$, and hence
  \[
    |S_0 \setminus X| \le |S_k \setminus S|\ \big((1 - \eps)^{k} - 1/2\big)^{-1} |H|^{-k} \ll_k \delta^{1/O_k(1)} |H|
  \]
  provided $\eps$ is sufficiently small in terms of $k$.
  
  Letting $f \colon H \to \RR^d/\ZZ^d$ be the function $f(x) = \lambda'(x) \bmod \ZZ^d$, we have in particular that $\partial^k f(c) = 0$ for all $c \in S$.  Applying Lemma~\ref{lem:99pc-poly}, we can find $\wt{f} \colon H \to \RR^d / \ZZ^d$ such that $\partial^k \wt{f}(c) = 0$ for all $c$, and with $f(x) = \wt{f}(x)$ for all $x \in X$.

  Finally, define $\lambda \colon H \to \RR^d$ by setting $\lambda(x) = \lambda'(x)$ if $x \in X$, and otherwise by taking the value $\lambda(x)$ such that $\lambda(x) \bmod \ZZ^d = \wt{f}(x)$ and $\big\|\wt{\lambda}(x)\big\|_1$ is as small as possible (and in particular at most $d$).  It is immediate that $\partial^k \lambda(c) \in \ZZ^d$ for any $c \in C^k(H)$. Moreover, $\partial^k \lambda(c) = \partial^k \lambda'(c) = \rho(c)$ for all $c \in S$ such that $c(\omega) \in X$ for all $\omega \in \llbracket k \rrbracket$, and the number of $c \in S_k$ which do not have these properties is at most
  \[
    |S_k \setminus S| + 2^k |S_0 \setminus X|\,|H|^k \ll_k \delta^{1/O_k(1)} |C^k(H)|
  \]
  as required.
\end{proof}

We next consider Lemma~\ref{lem:99pc-poly}, which is a ``property testing'' result for polynomial maps $H \to A$ of degree $s$.
The case $H=\FF_2^d$, $A = \FF_2$ appears in~\cite{akkl}, and various related results appear in the literature (e.g.\ \cite[Lemma 4.5]{tz}).  However, as far as the author is aware the general case is not quotable in the form we need, although it is essentially well-known.
For completeness, then, we record a proof of Lemma~\ref{lem:99pc-poly} in general, using similar ideas to~\cite{akkl} but being brief with some details.

\begin{proof}[Proof of Lemma~\ref{lem:99pc-poly}]
  The key idea is to show that for $x \in H$, the value
  \[
    g_{\vec h}(x) = -\sum_{\omega \in \llbracket s+1 \rrbracket \setminus \vec0} (-1)^{|\omega|} f(x + \omega \cdot \vec h)
  \]
  where $\vec h \in H^{s+1}$, is almost surely constant as $\vec h$ varies.  Indeed, if $\vec h = (h_1,\dots,h_{s+1})$ and $\vec h' = (h'_1,\dots,h'_{s+1})$, then
  \[
    g_{\vec h}(x) - g_{\vec h'}(x) = \partial^{s+1} f(\angle(x; h)) - \partial^{s+1} f(\angle(x; h'))
  \]
  and we also have the identity
  \begin{equation}
    \label{eq:tricube}
    \partial^{s+1} f(\angle(x; h')) = \sum_{\eta \in \llbracket s+1 \rrbracket} (-1)^{|\eta|} \partial^{s+1} f(\angle(x_\eta; \vec h_{\eta}))
  \end{equation}
  where $x_\eta = x + \eta \cdot h'$, and $\vec h_\eta = (r_1,\dots,r_{s+1})$ where $r_i = h_i$ if $\eta_i = 0$ or $r_i = h_i-h'_i$ if $\eta_i=1$. This corresponds to a ``tricube'' configuration in the language of~\cite{cs}; e.g.\ when $s=1$ this corresponds to the diagram in Figure~\ref{fig:tricube}, which demonstrates that cancellation occurs at all the internal vertices.

  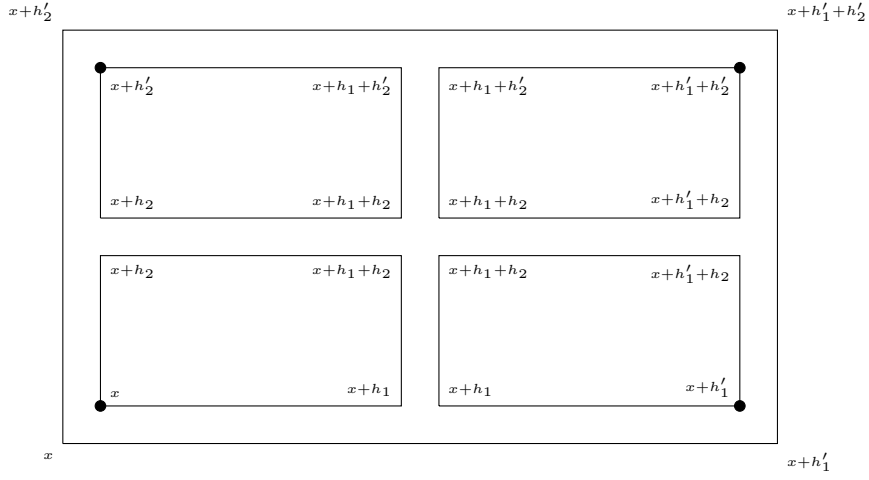
\begin{figure}%
    \label{fig:tricube}
    \begin{center}
      \begin{tikzpicture}[xscale=2.0]
        \tikzstyle{mydot} = [circle, draw, fill = black, inner sep = 0pt, minimum size = 4pt]
        \draw
          (0,0) node[above right] {$\scriptscriptstyle x$} node[mydot] {} --
          (2,0) node[above left]  {$\scriptscriptstyle x+h_1$} --
          (2,2) node[below left]  {$\scriptscriptstyle x+h_1+h_2$} --
          (0,2) node[below right] {$\scriptscriptstyle x+h_2$} --
          (0,0)

          (2.25,0) node[above right] {$\scriptscriptstyle x+h_1$} --
          (4.25,0) node[above left]  {$\scriptscriptstyle x+h_1'$} node[mydot] {} --
          (4.25,2) node[below left]  {$\scriptscriptstyle x+h_1'+h_2$} --
          (2.25,2) node[below right] {$\scriptscriptstyle x+h_1+h_2$} --
          (2.25,0)

          (0,2.5) node[above right] {$\scriptscriptstyle x+h_2$} --
          (2,2.5) node[above left]  {$\scriptscriptstyle x+h_1+h_2$} --
          (2,4.5) node[below left]  {$\scriptscriptstyle x+h_1+h_2'$} --
          (0,4.5) node[below right] {$\scriptscriptstyle x+h_2'$} node[mydot] {} --
          (0,2.5)

          (2.25,2.5) node[above right] {$\scriptscriptstyle x+h_1+h_2$} --
          (4.25,2.5) node[above left]  {$\scriptscriptstyle x+h_1'+h_2$} --
          (4.25,4.5) node[below left]  {$\scriptscriptstyle x+h_1'+h_2'$} node[mydot] {} --
          (2.25,4.5) node[below right] {$\scriptscriptstyle x+h_1+h_2'$} --
          (2.25,2.5)

          (-0.25,-0.5) node[below left]  {$\scriptscriptstyle x$} --
          (4.5,-0.5)    node[below right] {$\scriptscriptstyle x+h_1'$} --
          (4.5,5)       node[above right] {$\scriptscriptstyle x+h_1'+h_2'$} --
          (-0.25,5)    node[above left]  {$\scriptscriptstyle x+h_2'$} --
          (-0.25,-0.5);
      \end{tikzpicture}
      \caption{A tricube configuration}
    \end{center}
   \end{figure}

  Note $(x_{\vec 0}; \vec h_{\vec 0}) = (x; \vec h)$. On the other hand, for each $\eta \ne 0$, when $x$ is fixed and $\vec h, \vec h'$ are chosen uniformly at random, $\angle(x_\eta; \vec h_\eta)$ is a uniform random element of $C^{s+1}(H)$, and hence $\partial^{s+1}(\angle(x_\eta; \vec h_\eta)) = 0$ holds with probability at least $1-\eps$.
  Hence, for each $x \in H$, we have that
  \[
    \partial^{s+1} f(\angle(x; h)) - \partial^{s+1} f(\angle(x; h')) = 0
  \]
  for at least a $1 - (2^{s+1}-1) \eps$ fraction of pairs $(\vec h,\vec h')$, and so there is some $a \in A$ such that $g_{\vec h}(x)=a$ holds for at least $(1-2^{s+1} \eps) |H|^{s+1}$ values $\vec h \in H^{s+1}$.

  Define $g \colon H \to A$ by picking this majority value.  Note that for $x \in X$, for at least a $1/2-2^{s+1}\eps$ proportion of $\vec h \in H^{s+1}$ we have
  \[
    f(x) - g(x) = f(x) - g_{\vec h}(x) = \partial^{s+1} f(x; \vec h) = 0
  \]
  and so $f(x) = g(x)$ for all such $x$.
  
  We now claim $\partial^{s+1} g(c) = 0$ for all $c \in C^{s+1}(H)$.  Indeed, using the same set-up as~\eqref{eq:tricube}, we have
  \[
    \sum_{\eta \in \llbracket s+1 \rrbracket} (-1)^{|\eta|} g_{\vec h_{\eta}}(x_\eta) = 0
  \]
  for all $x \in H$ and $\vec h, \vec h' \in H^{s+1}$.  Fixing $x$ and $\vec h'$ and letting $\vec h$ vary uniformly at random, each $\vec h_\eta$ is also uniform random, and hence with probability at least $1 - 2^{2s+2} \eps$ we have $g_{\vec h_{\eta}}(x_\eta) = g(x_\eta)$ for each $\eta \in \llbracket s+1 \rrbracket$.  So, since $\eps < 2^{-2s-2}$ we have $\partial^{s+1} g(\angle(x; \vec h')) = 0$, for any $x$ and $\vec h'$, as required.
\end{proof}

We now prove Lemma~\ref{lem:rough-cocycle}.  We recall the definition of $T_{k,i}$ in~\eqref{eq:tk}, and will pass freely between thinking of $t \in T_{k,i}$ as a triple $(c_0,c_1,c)$, or by abuse of notation as a tuple $(x,h_1,\dots,h_k,h_i') \in H^{k+2}$, under the bijection given again in~\eqref{eq:tk}.

\begin{proof}[Proof of Lemma~\ref{lem:rough-cocycle}]
  Note that if $\delta$ is large then the conclusion is vacuous (e.g., set $\lambda = 0$), so we are free to assume $\delta \le \delta_0(k)$ is sufficiently small.

  We first do some combinatorics to nominate a set of cubes $S \subseteq S_k$ on which $\rho$ and $\partial^k \lambda$ will eventually agree.
  \begin{claim}%
    \label{claim:siprops}
    There exist sets $S \subseteq S_k$, and $I_i \subseteq T_{k,i}$ for each $i \in [k]$, such that:---
    \begin{enumerate}[label=(\roman*)]
      \item $\mu(S_k \setminus S) \ll_k \delta^{1/O_k(1)}$;
      \item for each $i \in [k]$ and $(c_0,c_1,c) \in I_i$ we have $c_0,c_1,c \in S$ and $\rho(c) = \rho(c_0) - \rho(c_1)$; and
      \item for each $c' \in S$ and each $i \in [k]$ we have
        \[
          \big|\big\{(c_0,c_1,c) \in I_i \colon c = c'\big\}\big| \ge (1 - \eta) |H| 
        \]
        where $\eta \le 2 \eps + O_k\big(\delta^{1/O_k(1)}\big)$.
    \end{enumerate}
  \end{claim}
  \begin{proof}[Proof of claim]
    Finding these sets is reminiscent of the proof of Lemma~\ref{lem:cube-system-ae}.  We first define $I'_i \subseteq T_{k,i}$ to consist of all triples $(c_0,c_1,c)$ such that $c_0,c_1,c \in S_k$ and $\rho(c) = \rho(c_0) - \rho(c_1)$.  By hypothesis, for each $i \in [k]$,
  \[
    \big| \big(T_{k,i} \cap S_k^3\big) \setminus I'_i \big| \le \delta |H|^{k+2} \ .
  \]
  Also note that for each $c' \in S_k$ there are at least $(1-2\eps)|H|$ triples $(c_0,c_1,c) \in T_{k,i}$ with $c_0,c_1,c \in S_k$ and $c = c'$: given $c = c'$ at least $(1-\eps)|H|$ of the valid choices of $c_0$ are in $S_k$ (since $S_k$ is a system of cubes with parameter $1-\eps$) and similarly for $c_1$.

  Set
  \[
    \cA_i = \big\{ c' \in S_k \colon |\{ (c_0,c_1,c) \in I'_i \colon c = c'\}| \ge (1 - 2\eps - \delta^{1/2})|H| \big\}
  \]
  for each $i \in [k]$.  Since each $c' \in S_k \setminus \cA_i$ contributes at least $\delta^{1/2}|H|$ elements $(c_0,c_1,c') \in \big(T_{k,i} \cap S_k^3 \big) \setminus I'_i$, we have $\mu(S_k \setminus \cA_i) \le \delta^{1/2}$.  We then take $S' = \bigcap_{i=1}^k \cA_i$, meaning $\mu(S_k \setminus S') \le k \delta^{1/2}$.

  We now pick $S'_0,\dots,S'_{k}$ to be a system of cubes with parameter $1-\delta'$, where $\delta' \ll_k \delta^{1/O_k(1)}$, such that $S'_k$ is disjoint from $S_k \setminus S'$ (by Lemma~\ref{lem:cube-system-ae}). We then take $S = S'_k \cap S_k$.  Finally, set
  \[
    I_i = \big\{ (c_0,c_1,c) \in I'_i \colon c_0,c_1,c \in S \big\} .
  \]
  Properties (i) and (ii) from the claim are immediate (as ever noting Remark~\ref{rem:cube-system-large}), so we check (iii).  For a fixed $c' \in S$, since $c' \in S'$ there are at least $(1 - 2 \eps - \delta^{1/2}) |H|$ triples $(c_0,c_1,c) \in I'_i$ with $c=c'$, and repeating the argument above on the system of cubes $S'_0,\dots,S'_{k}$ there are at least $(1 - 2 \delta') |H|$ triples $(c_0,c_1,c) \in T_{k,i}$ with $c=c'$ such that $c_0,c_1 \in S'_k$.  It follows that at least $(1 - 2 \eps - \delta^{1/2} - 2 \delta')|H|$ of these triples satisfy both conditions and hence lie in $I_i$, as required.
  \end{proof}

  From now on we fix a choice of sets $S$ and $I_i$ with these properties, and will show that there exists a choice of $\lambda \colon H \to \RR^d$ such that $\partial^k \lambda(c) = \rho(c)$ for all $c \in S$, and such that the size bound (ii) holds.
  For this stronger claim, we lose nothing by working with each coordinate of $\RR^d$ separately, and hence we will now assume without loss of generality that $d=1$.  (For the size bound (ii), the $d=1$ version gives us that $\sup_{x \in H} |\lambda(x)_i| \ll_k \sup_{c \in S_k} |\rho(c)_i|$ for each $i \in [d]$, and summing over $i$ gives the desired bound.)
  Also note that the remaining task is a problem concerning satisfiability of a system of linear equations, and so we will use linear algebraic rather than combinatorial tools.
  
  We first do some analysis that corresponds to the setting where there are no losses, i.e.\ as if $S = C^k(H)$ and $I_i = T_{k,i}$ for each $i \in [k]$; this will be useful for the general case.  We set up some notation.  Let $U$ denote the real vector space of all functions $H \to \RR$ and $V$ the space of all functions $C^k(H) \to \RR$.  So, $\partial^k \colon U \to V$ is a linear map, which we abbreviate to $\partial$.

  Also, for each $i \in [k]$ let $W_i$ be the space of all functions $T_{k,i} \to \RR$ (or by abuse of notation, $H^{s+2} \to \RR$), and define a map
  \begin{align*}
    \delta_i \colon V &\to W_i \\
    f &\mapsto
      \big( (c_0,c_1,c) \mapsto f(c_0) - f(c_1) - f(c)
          \big)
  \end{align*}
  measuring the success or failure of~\eqref{eq:proper-cocycle} (so e.g.\ $\delta_i f = 0$ for all $i$ if and only if $f$ is a $k$-cocycle on $C^k(H)$ with loss $0$).  We note that $\delta_i \circ \partial \equiv 0$ for each $i$: this precisely records the fact that derivatives (or ``coboundaries'') are cocycles.

  We also define linear maps
  \begin{align*}
    \sigma \colon V &\to U \\
    f &\mapsto \big(x \mapsto \EE_{h_1,\dots,h_k \in H} f(\angle(x;h_1,\dots,h_k)) \big)
  \end{align*}
  i.e.\ the function obtained by averaging $f$ over all $c \in C^k(H)$ whose bottom left vertex is $x$; and
  \begin{align*}
    \tau_i \colon W_i &\to V \\
     g &\mapsto \left(\angle(x;h_1,\dots,h_k) \mapsto \EE_{h \in H} g(x,h_1,\dots,h_k,h) \right)
  \end{align*}
  (noting the abuse of notation); equivalently, the function whose value at $c'$ is obtained by averaging $g$ over all $(c_0,c_1,c) \in T_{k,i}$ with $c=c'$.
  
  Finally, define a linear map $Z \colon V \to V$ by
  \[
    Z = \partial \circ \sigma - \prod_{i=1}^k (\id_V + \tau_i \circ \delta_i) .
  \]
  (It is not hard to show that the maps $\tau_i \circ \delta_i$ commute, but alternatively we can just compose the maps $\id_V + \tau_i \circ \delta_i$ in a fixed order.)

  The following observation is key in this setting.
  \begin{claim}
    The map $Z$ is identically zero.
  \end{claim}
  Note that this suffices for the case $\eps = \delta = 0$: if $f$ is a $k$-cocycle on $C^k(H)$ with loss $0$, so $\delta_i(f) = 0$ for each $i$, then $Z(f) = \partial(\sigma(f)) - f = 0$ and hence $f = \partial(\sigma(f))$ is the derivative of $\sigma(f)$.

  \begin{proof}[Proof of claim]
    Expanding the definitions, for $f \in V$ we have
    \[
      \partial(\sigma(f))(\angle(x;h_1,\dots,h_k)) = \sum_{\omega \in \llbracket k \rrbracket} (-1)^{|\omega|} \EE_{h'_1,\dots,h'_k \in H}\ f\big(\angle(x + \omega \cdot \vec h; h'_1,\dots,h'_k)\big) .
    \]
    Similarly,
    \begin{align*}
      &\tau_i(\delta_i(f)) (\angle(x;h_1,\dots,h_k))
      \\ &= \EE_{h'_i \in H} 
      \Big[ f\big(\angle(x;h_1,\tdots,h_i+h'_i,\tdots,h_k)\big)
        -
        f\big(\angle(x;h_1,\tdots,h_i,\tdots,h_k)\big)
        -
        f\big(\angle(x+h_i;h_1,\tdots,h'_i,\tdots,h_k)\big) \Big] \\
      &= -f\big(\angle(x;h_1,\dots,h_k)\big)
        + \EE_{h'_i \in H}\ f\big(\angle(x;h_1,\tdots,h'_i,\tdots,h_k)\big)
        - \EE_{h'_i \in H}\ f\big(\angle(x+h_i;h_1,\tdots,h'_i,\tdots,h_k)\big)
    \end{align*}
    where we used shift-invariance of the average $\EE_{h'_i \in H}$.  It follows that
    \[
      (\id_V + \tau_i \circ \delta_i)(f)(\angle(x;h_1,\dots,h_k)) = \EE_{h'_i \in H} \big[ f\big(\angle(x;h_1,\tdots,h'_i,\tdots,h_k)\big) - f\big(\angle(x+h_i;h_1,\tdots,h'_i,\tdots,h_k)\big) \big]
    \]
    and hence
    \[
      \left(\prod_{i=1}^k (\id_V + \tau_i \circ \delta_i)\right)(f)\big(\angle(x;h_1,\dots,h_k)\big) = \EE_{h'_1,\dots,h'_k \in H} \sum_{\omega \in \llbracket k \rrbracket} (-1)^{|\omega|} f\big(\angle(x + \omega \cdot \vec h; h'_1,\dots,h'_k)\big)
    \]
    which, comparing with the expression above, proves the claim.
  \end{proof}

  We now adapt this set-up to accommodate our sets $S$ and $I_i$.  We define some more linear operators:
  \begin{align*}
    \pi_S \colon V &\to V \\
    \pi_{I_i} \colon W_i &\to W_i 
  \end{align*}
  given by $\pi_S(f)(c) = 1_S(c) f(c)$ and $\pi_{I_i}(g)(z) = 1_{I_i}(z) g(z)$; i.e., setting values to zero outside the sets $S$ or $I_i$.  Also define a modified map $Z' \colon V \to V$ by
  \[
    Z' = \pi_S \circ \partial \circ \sigma - \prod_{i=1}^k \pi_S \circ (\id_V + \tau_i \circ \pi_{I_i} \circ \delta_i )  .
  \]
  Further let $V' \subseteq V$ be the subspace consisting of functions $f \in V$ such that $\delta_i(f)|_{I_i} \equiv 0$ for each $i \in [k]$ and $f(c) = 0$ whenever $c \notin S$; equivalently,
  \[
    V' = \im \pi_S \cap \bigcap_{i=1}^k \ker (\pi_{I_i} \circ \delta_i)  .
  \]
  Note that if $g \in \ker (\pi_{I_i} \circ \delta_i)$ then $\pi_S(g) \in \ker (\pi_{I_i} \circ \delta_i)$ also, since the value of $(\pi_{I_i} \circ \delta_i)(g)$ depends only on values $g(c)$ for $c \in S$, by Claim~\ref{claim:siprops}(ii).  Therefore, $\im(\pi_S \circ \partial) \subseteq V'$ (since $\delta_i \circ \partial = 0$ for all $i \in [k]$).  It follows that if $g \in V'$ then
  \[
    Z'(g) = (\pi_S \circ \partial)(\sigma(g)) - g \in V'
  \]
  i.e., $Z'(V') \subseteq V'$.

  We also consider the supremum norms $\|\cdot\|_\infty$ on $U$, $V$, $W_i$ or $V'$, and write $\|\cdot\|$ for the $\ell^\infty \to \ell^\infty$ operator norm of a linear map between any of these spaces.  It is clear from the definitions that $\|\partial\| \le 2^k$, $\|\delta_i\| \le 3$ for each $i \in [k]$, $\|\sigma\| \le 1$, $\|\tau_i\| \le 1$ for each $i \in [k]$, and $\|\pi_S\| \le 1$ and $\|\pi_{I_i}\| \le 1$ for each $i \in [k]$.

  We make the following claim.
  \begin{claim}%
    \label{cl:triv-kernel}
    Provided $\eps \le \eps_0(k)$ and $\delta \le \delta_0(k)$ are small enough, we have $\|Z'\| \le 1/2$ \uppar{as a map $V \to V$}. It follows that $\|(\id_{V'} + Z')(g)\|_\infty \ge (1/2) \|g\|_\infty$ for any $g \in V$, and hence the map $\id_{V'} + Z' \colon V' \to V'$ is injective and therefore surjective.
  \end{claim}
  This is enough to complete the proof, as follows.  Let $f = \pi_S(\rho)$; by hypothesis $f \in V'$.  By the claim, there is some $g \in V'$ such that
  \[
    f = g + Z'(g) = (\pi_S \circ \partial \circ \sigma)(g)
  \]
  (as $g \in V'$ and so all the terms involving $\pi_{I_i} \circ \delta_i$ vanish), and hence $f|_S = \partial(\sigma(g))|_S$.  Then $\lambda = \sigma(g)$ satisfies $\partial \lambda(c) = f(c)$ for all $c \in S$ as well as $\|\lambda\|_\infty \le \|g\|_\infty \le 2 \|f\|_{\infty} \le 2 \|\rho\|_\infty$.
  Hence $\lambda$ obeys parts (i) and (ii) as required, with the implied constant in (ii) in fact being $2$.

  \begin{proof}[Proof of Claim~\ref{cl:triv-kernel}]
    It is immediate from the definition that $\pi_S \circ Z' = Z'$.  Since $Z=0$, we can write
    \[
      Z' = \pi_S \circ (Z' - Z)
      = \left(\prod_{i=1}^k \pi_S \circ \big(\id_V + \tau_i \circ \pi_{I_i} \circ \delta_i\big)\right) - \pi_S \circ \left(\prod_{i=1}^k \big(\id_V + \tau_i \circ \delta_i\big)\right) \ .
    \]
    For convenience we define
    \begin{align*}
      \phi_i &= \id_V + \tau_i \circ \delta_i \\
      \phi'_i &= \id_V + \tau_i \circ \pi_{I_i} \circ \delta_i
    \end{align*}
    so that
    \[
      Z' = \prod_{i=1}^k \pi_S \circ \phi'_i - \pi_S \circ \prod_{i=1}^k \phi_i .
    \]
    Note that it is immediate from the operator norm bounds given above that $\|\phi_i\|,\, \|\phi'_i\| \le 4$ for each $i \in [k]$.
    We also claim the bound $\|\pi_S \circ \phi_i \circ (1 - \pi_S)\| \le 2 \eta$, for $\eta$ as in Claim~\ref{claim:siprops}(iii).  Indeed, since $\pi_S \circ (1 - \pi_S) = 0$ we have
    \[
      \pi_S \circ \phi_i \circ (1 - \pi_S) = \pi_s \circ \tau_i \circ \delta_i \circ (1 - \pi_S)
    \]
    and by unwrapping the definitions, if $g$ is supported on $C^k(H) \setminus S$ and $c' \in S$ then
    \[
      \big|(\tau_i \circ \delta_i)(g)(c')\big| = \big|\EE_{(c_0,c_1,c) \in T_{k,i} \colon c = c'} \big(g(c_0) - g(c_1) - g(c)\big) \big|
      \le 2 \eta \|g\|_{\infty}
    \]
    by Claim~\ref{claim:siprops}(ii,iii), since at most $\eta |H|$ terms in the average fail to have $c_0,c_1 \in S$.  The claimed norm bound follows. 
    
    For $1 \le j \le k-1$ we note
    \begin{align*}
      &\left(\prod_{i=1}^j (\pi_S \circ \phi_i)\right) \circ \left(\prod_{i=j+1}^k \phi_i\right) - \left(\prod_{i=1}^{j+1} (\pi_S \circ \phi_i) \right) \circ \left( \prod_{i=j+2}^k \phi_i \right) \\
      =\ & \left(\prod_{i=1}^{j-1} (\pi_S \circ \phi_i)\right) \circ \big(\pi_S \circ \phi_j \circ \phi_{j+1} - \pi_S \circ \phi_j \circ \pi_S \circ \phi_{j+1}\big) \circ \left(\prod_{i=j+2}^k \phi_i \right) \\
    =\ & \left(\prod_{i=1}^{j-1} (\pi_S \circ \phi_i)\right) \circ \left(\pi_S \circ \phi_j \circ (1 - \pi_S) \right) \circ \left(\prod_{i=j+1}^k \phi_i \right)
    \end{align*}
    and by telescoping we deduce that
    \[
      \left\| \pi_S \circ \prod_{i=1}^k \phi_i - \prod_{i=1}^k (\pi_S \circ \phi_i) \right\| \le (k - 1) 4^{k-1} (2 \eta) .
    \]

    Next we claim that $\|\pi_S \circ (\phi_i - \phi'_i)\| \le 3 \eta$.  We have
    $\phi_i - \phi_i' = \tau_i \circ (1 - \pi_{I_i}) \circ \delta_i$. So, for $g \in V$ and $c' \in S$, unwrapping the definitions gives
    \[
      \big|(\phi_i - \phi'_i)(g)(c')\big| = \left|\EE_{(c_0,c_1,c) \in T_{k,i} \colon c = c'} \ (1 - 1_{I_i}(c_0,c_1,c)) \big(g(c_0) - g(c_1) - g(c) \big)\right|
      \le 3 \eta \|g\|_\infty
    \]
    as required, again using Claim~\ref{claim:siprops}(iii).
    By telescoping again, we have
    \begin{align*}
      \left\|\prod_{i=1}^k (\pi_S \circ \phi'_i) - \prod_{i=1}^k (\pi_S \circ \phi_i) \right\| &\le \sum_{j=1}^k \left\| \left(\prod_{i=1}^{j-1} (\pi_S \circ \phi_j) \right) \circ \big(\pi_S \circ (\phi'_j - \phi_j)\big) \circ \left(\prod_{i=j+1}^{k} (\pi_S \circ \phi'_i)\right) \right\| \\
      &\le k\, 4^{k-1} (3 \eta)
    \end{align*}
    and combining these estimates yields $\|Z'\| \le 5 k\, 4^{k-1} \eta$, which for $\eps$ and $\delta$ sufficiently small in terms of $k$ is at most $1/2$, as required.
  \end{proof}

  This completes the proof of Lemma~\ref{lem:rough-cocycle} (and therefore of that of Lemma~\ref{lem:cocycle}).
\end{proof}

Finally we deduce the classification of generalized cocycles from that of non-generalized cocycles.

\begin{proof}[Proof of Lemma~\ref{lem:gen-cocycle}]
  As usual, we are free to assume $\delta \le \delta_0(k)$ is sufficiently small.
  We fix $\rho \colon S_k \times \llbracket k \rrbracket \to \ZZ^d$ as in the statement.

  The argument has several stages.  First, as we have done on other occasions, we locate some sets of configurations with good properties where the relevant hypotheses of Definition~\ref{def:gen-cocycle} and Definition~\ref{def:compatible} hold.  Second, we show that, on these sets, $\rho$ induces a collection of un-generalized cocycles $\rho_F$ (one for each face $F \subseteq \llbracket k \rrbracket$ of codimension $r$), to which we can apply Lemma~\ref{lem:cocycle}; i.e., we find functions $\lambda_F \colon H \to \RR$ such that $\rho_F = \partial^{k-r} \lambda_F$ on its domain.  Finally, we use a separate argument to show that (without loss of generality) all these $\lambda_F$ are the same, and this common function gives the $\lambda$ we require.

  First, for each $i \in [k]$, let $I'_i \subseteq T_{k,i}$ denote the set of all triples $(c_0,c_1,c)$ such that $c_0,c_1,c \in S_k$, and which are not in the exceptional sets for $\rho$ described in Definition~\ref{def:compatible} or Definition~\ref{def:gen-cocycle}: that is, for all $\omega \in F^i$ we have $\rho(c_0,\omega) = \rho(c_1,\omega)$, and for all $\omega \in \llbracket k \rrbracket$ we have
  \[
    \rho(c,\omega') = \sum_{\omega \in \llbracket k \rrbracket} Z_r(\omega,\omega') \fb(\omega)
  \]
  where
  \[
    \fb(\omega) = \begin{cases} b(c_0, \omega) &\colon \omega(i) = 0 \\ b(c_1, \omega \setminus \{i\}) &\colon \omega(i) = 1 . \end{cases}
  \]
  The hypotheses of being a generalized $k$-cocycle of type $r$ and loss $\delta$ immediately implies that
  \[
    \big| \big(T_{k,i} \cap S_k^3\big) \setminus I'_i \big| \le (2^k + 1) \delta |H|^{k+2}
  \]
  for each $i \in [k]$.  The next claim locates cubes $c \in S_k$ which are trustworthy in the sense of belonging to many triples of $I'_i$ for each $i$, and so on in a hereditary fashion, in the spirit of Definition~\ref{def:cube-system}.
  \begin{claim}%
    \label{claim:siprops2}
    There exists a system of cubes $S'_0,\dots,S'_k$ with $S'_j \subseteq S_j$ for $0 \le j \le k$, and sets $I_i \subseteq I'_i$ for each $i \in [k]$, such that:---
    \begin{enumerate}[label=(\roman*)]
      \item $S'_0,\dots,S'_k$ has parameter $1-\eps - O_k\big(\delta^{1/O_k(1)}\big)$, and $\mu(S_j \setminus S_j') \ll_k \delta^{1/O_k(1)}$ for each $0 \le j \le k$;
      \item for each $i \in [k]$ and $(c_0,c_1,c) \in I_i$ we have $c_0,c_1,c \in S'_k$;
      \item for each $c' \in S$ and each $i \in [k]$ we have
        \begin{align*}
          \big|\big\{(c_0,c_1,c) \in I_i \colon c_0 = c'\big\}\big| &\ge (1 - \eta) |H| \\
          \big|\big\{(c_0,c_1,c) \in I_i \colon c_1 = c'\big\}\big| &\ge (1 - \eta) |H| \\
          \big|\big\{(c_0,c_1,c) \in I_i \colon c = c'\big\}\big| &\ge (1 - \eta) |H| 
        \end{align*}
        where $\eta \le 2 \eps + O_k\big(\delta^{1/O_k(1)}\big)$; and
      \item we have 
        \[
          \big| \big(T_{k,i} \cap {S'}_k^3\big) \setminus I_i \big| \ll_k \delta^{1/O_k(1)} |H|^{k+2}
        \]
        for each $i \in [k]$.
    \end{enumerate}
  \end{claim}
  \begin{proof}[Proof of claim]
    This is almost identical to the proof of Claim~\ref{claim:siprops}, so we will be brief and focus on what is different.  The definition of $I'_i$ above directly replaces that from Claim~\ref{claim:siprops}.
    As well as 
    \[
      \cA_i = \big\{ c' \in S_k \colon |\{ (c_0,c_1,c) \in I'_i \colon c = c'\}| \ge (1 - 2\eps - \delta^{1/2})|H| \big\}
    \]
    we define
    \begin{align*}
      \cB_i &= \big\{ c' \in S_k \colon |\{ (c_0,c_1,c) \in I'_i \colon c_0 = c'\}| \ge (1 - 2\eps - \delta^{1/2})|H| \big\} \\
      \cC_i &= \big\{ c' \in S_k \colon |\{ (c_0,c_1,c) \in I'_i \colon c_1 = c'\}| \ge (1 - 2\eps - \delta^{1/2})|H| \big\}
    \end{align*}
    so that $\mu(S_k \setminus \cA_i) \ll_k \delta^{1/2}$ as before, and similarly for $\cB_i$ and $\cC_i$.  We then set $S'$ to be the intersection of $\cA_i \cap \cB_i \cap \cC_i$ over all $i \in [k]$, and locate a system of cubes $S''_0,\dots,S''_k$ with parameter $1-O_k\big(\delta^{1/O_k(1)}\big)$ such that $S_k \setminus S'$ is disjoint from $S''_k$. Finally we set $S'_j = S_j \cap S''_j$, and
    \[
      I_i = \big\{ (c_0,c_1,c) \in I'_i \colon c_0,c_1,c \in S'_k \big\} .
    \]
    The argument to show that (iii) holds is unchanged, and it is immediate that
    \[
      \mu(I'_i \setminus I_i) \le 3 \mu(S_k \setminus S'_k) \ll_k \delta^{1/O_k(1)}
    \]
    which implies (iv).
  \end{proof}

  For any face $F \subseteq \llbracket k \rrbracket$ of codimension $r$, we define
  \begin{align*}
    \rho_F \colon S_k &\to \ZZ^k \\
    c &\mapsto \sum_{\omega \in F} (-1)^{|\omega|} \rho(c,\omega) .
  \end{align*}
  Our next task is to use these functions to reformulate the definition of a generalized cocycle in a more natural and symmetrical way.
  \begin{claim}%
    \label{claim:nice-cocycle}
    Take any $i \in [k]$ and $(c_0,c_1,c) \in I'_i$. For each face $F \subseteq \llbracket k \rrbracket$ of codimension $r$:---
    \begin{enumerate}[label=(\roman*)]
      \item if $F \subseteq F^i$, then $\rho_F(c_0) = \rho_F(c_1)$;
      \item if $F \subseteq F_i$, then $\rho_F(c_0) = \rho_F(c)$;
      \item if $F \nsubseteq F^i$ and $F \nsubseteq F_i$, then $\rho_F(c) = \rho_F(c_0) - \rho_F(c_1)$;
      \item if $F \subseteq F^i$ and $F' = \{ \omega \setminus \{i\} \colon \omega \in F\}$ is the opposite face in direction $i$, then $\rho_F(c) = -\rho_{F'}(c_1)$.
    \end{enumerate}
  \end{claim}
  Statement (i) is immediate from the hypothesis of upper compatibility. By contrast, statement (ii) is symmetrically a statement about ``lower compatibility'', which we did not assume, and which we must justify in a more roundabout way.  Taken together, these roughly state (in a way we will make precise shortly) that $\rho_F(c)$ only really depends on the face $c|_F$, not on the rest of $c$.  After that, we will use (iii,iv) to deduce that the implied functions $S_{k-r} \to \ZZ^d$ obtained this way are non-generalized cocycles.
  \begin{proof}[Proof of claim]
    As stated, (i) is immediate from upper compatibility:
    \[
      \rho_F(c_0) = \sum_{\omega \in F} (-1)^{|\omega|} \rho(c_0,\omega) = \sum_{\omega \in F} (-1)^{|\omega|} \rho(c_1,\omega) = \rho_F(c_1) .
    \]
    For the remaining parts, we recall by Remark~\ref{rem:normal-form} that $\omega \mapsto \fb(\omega) - \rho(c,\omega)$ is a linear combination of indicator functions of faces $\cF$ each of dimension at least $r+1$.  For each such $\cF$ we have that $F \cap \cF$ is a face of dimension at least $1$, and hence $\sum_{\omega \in F \cap \cF} (-1)^{|\omega|} = 0$ (see~\eqref{eq:cube-orthog}).  Therefore,
    \[
      \rho_F(c) = \sum_{\omega \in F} (-1)^{|\omega|} \rho(c,\omega) = \sum_{\omega \in F} (-1)^{|\omega|} \fb(\omega)  .
    \]
    If we now expand the definition of $\fb$, we deduce
    \[
      \rho_F(c) = \sum_{\omega \in F \cap F_i} (-1)^{|\omega|} \rho(c_0, \omega) + \sum_{\omega \in F \cap F^i} (-1)^{|\omega|} \rho(c_1,\omega \setminus \{i\}) .
    \]
    If $F \subseteq F_i$, this just says that $\rho_F(c) = \rho_F(c_0)$, as required for (ii).  If $F \subseteq F^i$, this is exactly the statement we need for (iv).
    
    In the last case (iii), we have $\{ \omega \setminus \{i\} \colon \omega \in F \cap F^i \} = F \cap F_i$, and so we can rewrite this as
    \[
      \rho_F(c) = \sum_{\omega \in F \cap F_i} (-1)^{|\omega|} \rho(c_0, \omega) - \sum_{\omega \in F \cap F_i} (-1)^{|\omega|} \rho(c_1,\omega)
    \]
    and recalling that $\rho(c_0,\omega) = \rho(c_1,\omega)$ for all $\omega \in F^i$, this implies
    \[
      \rho_F(c) = \sum_{\omega \in F} (-1)^{|\omega|} \rho(c_0, \omega) - \sum_{\omega \in F} (-1)^{|\omega|} \rho(c_1,\omega)
    \]
    or $\rho_F(c) = \rho_F(c_0) + \rho_F(c_1)$, as required.
  \end{proof}

  We now formalize the statement above, that there exist non-generalized cocycles corresponding to each $\rho_F$, and state some relationships between them.
   \begin{claim}%
     \label{claim:cocycles-well-defined}
    For each face $F \subseteq \llbracket k \rrbracket$ of codimension $r$ there exists a $(k-r)$-cocycle $\rho'_F \colon S'_{k-r} \to \ZZ^d$ on $S'_{k-r}$ with loss $O_k\big(\delta^{1/O_k}\big)$, such that $\rho_F(c) = \rho'_F(c|_F)$ for every $c \in S'_k$.
     
     Moreover, if faces $F = \{\omega \colon \eta_0 \subseteq \omega \subseteq \eta_1\}$ and $F' = \{\omega \colon \eta'_0 \subseteq \omega \subseteq \eta'_1\}$ of codimension $r$ are parallel, in the sense that they have the same set of active coordinates \uppar{see Section~\ref{subsec:notation}}, then $\rho'_F = (-1)^{|\eta_0| - |\eta_0'|} \rho'_{F'}$.
  \end{claim}
  \begin{proof}[Proof of claim]
    The first step is to show that $\rho_F$ descends to a well-defined function on $S'_{k-r}$.  That is, we claim that if $c,c' \in S'_k$ and $c|_F = c'|_F$ then $\rho_F(c) = \rho_F(c')$.  We can then define $\rho'_F(\wt c)$ for $\wt c \in S'_{k-r}$ unambiguously to be $\rho_F(c)$ for any $c \in S'_{k}$ which has $c|_F = \wt c$: such a $c$ always exists, as $S'_0,\dots,S'_k$ is a system of cubes with positive parameter.

    To prove that $\rho_F(c) = \rho_F(c')$ whenever $c|_F=c'|_F$, we show the following: for every $c \in S'_k$ there are more than $(1/2) |H|^r$ cubes $c'' \in S'_k$ with $c|_F = c''|_F$ and $\rho_F(c) = \rho_F(c'')$.  Applying this to $c'$ as well, the two sets of cubes $c''$ must intersect, and hence $\rho_F(c) = \rho_F(c')$ as required.

    We exhibit a large collection of such cubes $c''$ by making changes in one coordinate $i$ at a time (without ever changing the restriction to $F$).  Take any $c_1 \in S'_k$ with $c_1|_F = c|_F$, and write $c_1 = \angle(x;h_1,\dots,h_k)$.
    
    If $i \in [k]$ and $F \subseteq F_i$, we note that are at least $(1-\eta) |H|$ values $h \in H$ such that $c_2 = \angle(x;h_1,\dots,h_i+h,\dots,h_k)$ is in $S'_k$, and has $c_1|_F = c_2|_F$ and $\rho_F(c_1) = \rho_F(c_2)$, by Claim~\ref{claim:siprops2}(ii,iii) and Claim~\ref{claim:nice-cocycle}(i). Similarly, if $F \subseteq F^i$, there are at least $(1-\eta)|H|$ values $h \in H$ such that $c_3 = \angle(x-h;h_1,\dots,h_i+h,\dots,h_k)$ is in $S'_k$, and has $c_1|_F = c_3|_F$ and $\rho_F(c_1) = \rho_F(c_3)$, by Claim~\ref{claim:siprops2}(ii,iii) and Claim~\ref{claim:nice-cocycle}(ii).

    Applying this to each of the $i \in [k]$ which are not active coordinates of $F$ (of which there are $r$) in turn, we obtain $(1-\eta)^r |H|^r$ cubes $c''$ with the desired properties (noting all possible outcomes are distinct).  Provided $(1-\eta)^r > 1/2$, which we may assume, this suffices.

    The second step is to show that these $\rho'_F$ are $(k-r)$-cocycles with suitable parameters.  Fix a triple $(c_0,c_1,c) \in T_{k-r,i}$ for $i \in [k-r]$, such that $c_0,c_1,c \in S'_{k-r}$ but $\rho'_F(c) \ne \rho'_F(c_0) - \rho'_F(c_1)$. The goal is to show the number of such triples is small.
    
    By Claim~\ref{claim:nice-cocycle}(iii), there cannot be a triple $(\wt{c_0},\wt{c_1},\wt{c}) \in I_{j}$ such that $\wt{c_0}|_F = c_0$, $\wt{c_1}|_F = c_1$, $\wt{c}|_F = c$, where $j \in [k]$ is the index corresponding to $i \in [k-r]$ when we restrict to $F$.  However, since $S'_0,\dots,S'_k$ is a system of cubes with parameter $1 - \eps - O_k\big(\delta^{1/O_k}\big)$, by a union bound there are always at least $\big(1 - 3 \eps - O_k\big(\delta^{1/O_k(1)}\big)\big)|H|$ ways to augment a triple in the system to one of dimension one higher, and so
    \[
      \big|\big\{ (\wt{c_0}, \wt{c_1}, \wt{c}) \in T_{k,j} \colon \wt{c_0}|_F = c_0,\, \wt{c_1}|_F = c_1,\, \wt{c}|_F = c,\ \wt{c_0}, \wt{c_1}, \wt{c} \in S'_k \big\}\big| \ge \big(1 - 3 \eps - O_k\big(\delta^{1/O_k(1)}\big)\big)^r |H|^r .
    \]
    Again we are free to assume the right hand side is at least $|H|^r/2$. So, every original triple $(c_0,c_1,c)$ gives rise to at least $|H|^r/2$ triples  $(\wt{c_0},\wt{c_1},\wt{c})$ in $\big(T_{k,j} \cap {S'}_k^3\big) \setminus I_j$, all distinct, and so by Claim~\ref{claim:siprops2} that there are at most $O_k\big(\delta^{1/O_k(1)}\big)$ such bad triples $(c_0,c_1,c)$, as required.

    Lastly we verify the second part of the claim.  In the first instance suppose that $F,F'$ are two parallel faces of $\llbracket k \rrbracket$ of codimension $r$, and moreover that $F \subseteq F^i$ and $F' = \{ \omega \setminus \{i\} \colon \omega \in F\}$ for some $i \in [k]$ (in other words, these are adjacent parallel faces in direction $i$).
    
    For any $c \in S'_{k-r}$, pick any $\wt{c} \in S'_k$ such that $\wt{c}|_F = c$ (which we may), and then pick any $(\wt{c_0},\wt{c_1},\wt{c}) \in I_i$; so, $\wt{c_1}|_{F'} = c$.  By Claim~\ref{claim:nice-cocycle}(iv) we have $\rho_F(\wt{c}) = -\rho_{F'}(\wt{c_1})$, and hence $\rho'_F(c) = -\rho'_{F'}(c)$, which agrees with what we were supposed to prove in this special case.
    
    For general parallel faces $F$ and $F'$, we note that it is possible to move from any face to any parallel one in at most $r$ steps, where each step consists of moving to an adjacent parallel face in the above sense, so the claimed statement follows from applying this special case at most $r$ times.
  \end{proof}

  We can now apply Lemma~\ref{lem:cocycle} to deduce that for each face $F$ of codimension $r$, there is some $\lambda_F \colon H \to \RR^d$ such that $\rho'_F(c) = \partial^{k-r} \lambda_F(c)$ for all but $O_k\big(\delta^{1/O_k(1)}\big) |C^{k-r}(H)|$ cubes $c \in S'_{k-r}$, and moreover that properties (ii) and (iii) from Lemma~\ref{lem:cocycle} hold for $\lambda_F$.  Also, we may assume that if $F = \{\omega \colon \eta_0 \subseteq \omega \subseteq \eta_1\}$ and $F' = \{\omega \colon \eta'_0 \subseteq \omega \subseteq \eta'_1\}$ are parallel faces of codimension $r$ as in Claim~\ref{claim:cocycles-well-defined}, then $\lambda_{F} = (-1)^{|\eta_0| - |\eta_1|} \lambda_{F'}$.

  In fact, we claim that without loss of generality all the functions $\lambda_F$ are equal up to sign changes.

  \begin{claim}%
    \label{claim:lambda-compat}
    There exists a single function $\lambda \colon H \to \RR^d$ such that $\lambda_F(x) = (-1)^{|\eta_0|} \lambda(x) + C_F$ for each face $F = \{\omega \colon \eta_0 \subseteq \omega \subseteq \eta_1\}$ of codimension $r$ and each $x \in S'_0$, where $C_F \in \RR^d$ are constants \uppar{and $C_F = 0$ for at least one $F$}.
  \end{claim}

  We claim this $\lambda$ has all the properties required to complete the proof of Lemma~\ref{lem:gen-cocycle}.  Indeed, parts (ii) and (iii) of the statement follow immediately from the corresponding parts of Lemma~\ref{lem:cocycle}.  For (i), we assert that---in the notation of the statement---we have $\rho(c,\omega') = \Lambda(c,\omega')$ for each $c \in S'_k$ and $\omega' \in \llbracket k \rrbracket$, which is sufficient (noting Claim~\ref{claim:siprops2}(i)).  To see this, we note that for any face $F$ of dimension $k-r$ and $c \in S'_k$ we have
  \[
    \sum_{\omega \in F} (-1)^{|\omega|} \lambda(c(\omega)) = \rho_F(c) = \sum_{\omega \in F} (-1)^{|\omega|} \rho(c,\omega)
  \]
   since $\partial^{k-r}\lambda = \partial^{k-r}\lambda_F$ for any $F$ (as the constants $C_F$ do not affect the derivative, since $k>r$).  It follows that for any $F$ of dimension at least $k-r$ we also have
  \begin{equation}
    \label{eq:lambda-vs-rho}
    \sum_{\omega \in F} (-1)^{|\omega|} \lambda(c(\omega)) = \sum_{\omega \in F} (-1)^{|\omega|} \rho(c,\omega)
  \end{equation}
  by summing over smaller faces.
  We can now compute directly from the definitions of $\Lambda$ and $Z_r$ (i.e.,~\eqref{eq:big-lambda} and~\eqref{eq:zr}) that for all $c \in S'_k$ and $\omega' \in \llbracket k \rrbracket$,
  \begin{align}
    \Lambda(c,\omega') &= \sum_{\omega \in \llbracket k \rrbracket} Z_r(\omega,\omega') \lambda(c(\omega))
    = \sum_{\substack{|\eta| \le r \\ \eta \supseteq \omega'}} (-1)^{|\eta|} \sum_{\omega \supseteq \eta} (-1)^{|\omega|} \lambda(c(\omega)) \label{eq:big-lambda-as-deriv}\\
    &= \sum_{\substack{|\eta| \le r \\ \eta \supseteq \omega'}} (-1)^{|\eta|} \sum_{\omega \supseteq \eta} (-1)^{|\omega|} \rho(c,\omega) \nonumber
  \end{align}
  by~\eqref{eq:lambda-vs-rho}, and rearranging gives
  \[
    \Lambda(c,\omega')
    = \sum_{\omega \supseteq \omega'} (-1)^{|\omega|} \rho(c,\omega) \sum_{\substack{|\eta| \le r \\ \omega \supseteq \eta \supseteq \omega'}} (-1)^{|\eta|} .
  \]
  The outer summand is zero if $|\omega| > r$ (as then $\rho(c,\omega) = 0$) and so the condition $|\eta| \le r$ is redundant.  Therefore the inner sum is given by~\eqref{eq:cube-orthog}, and so the right hand side is precisely $\rho(c,\omega')$, so $\Lambda(c,\omega')=\rho(c,\omega')$ as required for (i).

  Finally we show part (iv).
  For any $c \in C^k(H)$, by (ii) we have $\sum_{\omega \in F} (-1)^{|\omega|} \lambda(c(\omega))  \in \ZZ^d$ for any face $F \subseteq \llbracket k \rrbracket$ of dimension $k-r$, and hence (by summing over smaller faces as above) for any face $F$ of dimension at least $k-r$.
  But then for any $\omega' \in \llbracket k \rrbracket$ we have $\Lambda(c,\omega') \in \ZZ^d$ by~\eqref{eq:big-lambda-as-deriv} (which holds for any $c$).

  \begin{proof}[Proof of Claim~\ref{claim:lambda-compat}]
  Consider some pair of lower faces $F_{\eta}, F_{\eta'}$ where $\eta,\eta' \in \llbracket k \rrbracket$ have $|\eta|, |\eta'| = k-r$.  We observe that for any $c \in S'_k$,
  \[
    \partial^k \lambda_{F_{\eta}} (c) = \sum_{\omega \in \llbracket k \rrbracket} (-1)^{|\omega|} \rho(c,\omega) = \partial^k \lambda_{F_{\eta'}} (c)
  \]
  where we applied the proof of~\eqref{eq:lambda-vs-rho} to the face $\llbracket k \rrbracket$.  Hence, writing $\tau = \lambda_{F_{\eta}} - \lambda_{F_{\eta'}}$, we have that $\partial^k \tau(c) = 0$ for all $c \in S'_k$.  Now by Lemma~\ref{lem:99pc-poly}, there is some function $\tau' \colon H \to \RR^d$ such that $\tau(x) = \tau'(x)$ for all $x \in S'_0$ (as for $\eps,\delta$ small enough the set $X$ in~\eqref{eq:bigX} is all of $S'_0$, by properties of the system of cubes $S'_0,\dots,S'_k$) and such that $\partial^k \tau'(c) = 0$ for all $c \in C^k(H)$.  This last condition forces $\tau'$ to be constant: for instance, we can compute that for each $x \in H$,
  \begin{equation}
    \label{eq:triv-poly}
    0 = \EE_{h_1,\dots,h_k \in H} \partial^k \tau'(\angle(x;h_1,\dots,h_k)) = \tau'(x) - \EE_{y \in H} \tau'(y)
  \end{equation}
    and hence $\tau'$ takes its average value anywhere.  Hence, $\lambda_{F_{\eta}}$ and $\lambda_{F_{\eta'}}$ differ by a global constant.  Because every face is parallel to a lower face, by the second part of Claim~\ref{claim:cocycles-well-defined} we deduce that for every pair of faces $F,F'$ the functions $\lambda_F,\lambda_{F'}$ differ by a global constant and the appropriate sign change.  Setting $\lambda=\lambda_{F_0}$ for any face $F_0$ of codimension $r$, this proves the claim.
  \end{proof}

  This completes the proof of Lemma~\ref{lem:gen-cocycle}.  
\end{proof}

\subsection{Finding and eliminating redundancy}%
\label{subsec:redundancy}

Next, we formalize what it means to find redundancies in a polynomial hierarchy, and recursively replace the hierarchy with a simpler one to eliminate them.  We introduce some definitions.

\begin{definition}%
  \label{def:reduce}
  Given non-negative integers $s$, $d_0,\dots,d_s$ and $d_0',\dots,d_s'$, and two tuples $(f_{i,j})$ for $0 \le i \le s$, $j \in [d_i]$ and $(f'_{i,j})$ for $0 \le i \le s$, $j \in [d'_i]$, consisting of functions $X \to \RR$ on some set $X$, we say \emph{$f'$ reduces $f$} on $X$ with parameter $K \ge 1$ if the following holds: for each function $f_{i,j}$ and each $x \in X$ there exists $v(x) \in \bigoplus_{r=0}^i \ZZ^{d'_r}$ such that $\|v(x)\|_1 \le K$ and $f_{i,j}(x) = v(x) \cdot f'(x)$.
\end{definition}

The motivation for this definition is that $f'$ carries at least as much information as $f$ for the purposes of derivatives conditions, in the sense of the following result.

\begin{lemma}%
  \label{lem:use-reduction}
  If $X \subseteq H$, $(f_{i,j})$ \uppar{$0 \le i \le s$, $j \in [d_i]$} is a tuple of functions $X \to \RR$, $g \colon X \to \RR$ is a function satisfying the $(f,k, t, M)$-derivatives condition on $S \subseteq C^k(H)$ \uppar{where $k \ge 0$ and $1 \le t \le s+1$}, and $(f'_{i,j})$ \uppar{$0 \le i \le s$, $j \in [d'_i]$} is some other tuple of functions $X \to \RR$ such that $f'$ reduces $f$ on $X$ with parameter $K$;
  then $g$ satisfies the $(f',k, t, KM)$-derivatives condition.
\end{lemma}
\begin{proof}
  This follows simply by substituting the expressions for $f_{i,j}(x)$ in terms of $f'$ given by Definition~\ref{def:reduce}, into the derivatives condition~\eqref{eq:deriv-condition}.
\end{proof}

We are now in a position to state the main result of this subsection, which shows that a polynomial hierarchy can be upgraded to a strong one without losing information, at the expense of a loss in parameters and a small further error set.

\begin{theorem}%
  \label{thm:strong-derivs}
  Let $s \ge t \ge 1$ be integers, let $S_0,\dots,S_{s+1}$ be a system of cubes of parameter $1-\eps$, let $f$ be an $(s-1,d,M)$-polynomial hierarchy on $S_0,\dots,S_s$, and suppose $g \colon S_0 \to \RR$ obeys the $(f,s+1,t,M)$-derivatives condition on $S_{s+1}$.  Write $D = \sum_{i=0}^s d_i$, and let $\delta > 0$ be an arbitrary parameter.

  Suppose finally that $\eps \le \eps_0(s)$ is sufficiently small.   Then there exists a system of cubes $S'_0,\dots,S'_{s+1}$ with parameter at least $1-\eps-\delta$, and with $S'_k \subseteq S_k$ and $\mu(S_k \setminus S'_k) \le \delta$ for each $k$; and a parameter $M'$, where
  \[
    M' \le O_s(M / \delta)^{O_s(D)^{O(1)}}
  \]
  if $t \le 2$, and
  \[
    M' \le O_s(M / \delta)^{O_s(D)^{O_s(D)}}
  \]
  if $t \ge 3$, such that the following holds.  Either $H$ has a subgroup of size at least $|H|/M'$, or 
  there is an $(s,d',M')$-polynomial hierarchy $f'$ which reduces $f$ with parameter $M'$, such that $g$ obeys the $(f',s+1,t,M',\delta)$-strong derivatives condition on $S'_{s+1}$.
\end{theorem}

\begin{remark}
  The weaker bound when $t \ge 3$ is unsatisfactory, but the author has been unable to improve it.
  In the proof, we argue the two cases separately, although using all the same ingredients: roughly, the difference is that when $t \le 2$ we can remove all redundancies at the same time, whereas when $t \ge 3$ we resort to applying induction on the dimensional quantity $\sum_{r=0}^{s-1} (r+1) d_r$.
  
  It is possible to adapt the former strategy to the case $t \ge 3$, by removing all redundancies in the same degree in one go, and then appling induction on $t$.
  However, at one specific point in the argument (the proof of Lemma~\ref{lem:1pc-implies-derivs}) something new and bad happens, in that we need to assume that a strong derivatives condition holds recursively in lower-degree parts of the hierarchy $f$ for the argument to go through.
  It is still possible to structure the argument to avoid cyclic dependencies in the parameters, but doing so nonetheless causes a fairly disastrous explosion in the bounds, and they end up much worse than using the more direct induction on dimensions.
\end{remark}

The upgraded hierarchy $f'$ in Theorem~\ref{thm:strong-derivs}, which should be thought of as $f$ with enough redundancies removed, will be constructed from $f$ using a combination of the following steps.
\begin{enumerate}[label=(\roman*)]
  \item Suppose some function $f_{i,j}$, assumed by hypothesis to satisfy the $(f,i+1,i,M)$-derivatives condition, in fact satisfies the stronger $(f,i,i-1,M')$-derivatives condition.  Then moving $f_{i,j}$ out of $f_{=i}$ and into $f_{=i-1}$ gives a tuple that reduces $f$.
  \item Suppose instead that $f_{i,j}$ has the form $f_{i,j}(x) = k(x) \cdot f_{\le i-1}$, where $k \colon S_0 \to \bigoplus_{r=0}^{i-1} \ZZ^{d_r}$ is an arbitrary function taking bounded values.  This could happen: in fact, such a function even satisfies the stronger $(f,0,i,M)$-derivatives condition.  However, it is also clear that just deleting $f_{i,j}$ gives a tuple that reduces $f$.
  \item If $f_{i,j}$ satisfies the $(f,i,i,M)$-derivatives condition (stronger than the $(f,i+1,i,M)$-condition from the hypothesis, but weaker than (i)) it turns out we can decompose $f_{i,j} = \phi_0 + \phi_1$ where $\phi_0$ has the properties from (i) and $\phi_1$ has the properties from (ii), and deal with each accordingly.
  \item We can perform any of these operations, and in particular (iii), replacing $f_{i,j}$ (a function in the tuple) with a bounded linear combination $v \cdot f_{\le i}$ of such functions (with $v_{=i} \ne 0$), by first applying a suitable change of basis.
\end{enumerate}

The dichotomy we need to establish is that either $g$ already obeys a strong derivatives condition on the hierarchy $f$, or if not then have an opportunity to apply one of these reduction steps to $f$.

This is achieved in two further stages.  First we first show that whenever $g$ fails to obey a strong derivatives condition, some particular linear combination $v \cdot f_{\le i}$ is an approximate polynomial of lower degree than expected (i.e., $i-1$).  Then, we argue that this approximate polynomial condition can be boosted to a derivatives condition for $v \cdot f_{\le i}$ of lower degree than expected as in (iv) above; or in some cases that, if not, some $v' \cdot f_{\le i'}$ is similarly an approximate polynomial of lower degree than expected, with $i' < i$.

We now state the necessary ingredients precisely.  First, we show that if $g$ and $b$ obey a derivatives condition but the compatibility condition (Definition~\ref{def:compatible}) fails to hold for $b$, then we can find some approximate polynomial $v \cdot f_{\le i}$ as discussed above.  We first give a technical statement and then its generic consequence.
\begin{lemma}%
  \label{lem:incompatible-implies-1pc}
  Fix $t \ge 2$ and $k \ge 0$; let $X \subseteq H$; let $f$ be a tuple of functions $X \to \RR$, of shape $d_0,\dots,d_{t-1}$ with $\sum_{r=0}^{t-1} d_r = D$; and suppose that $g \colon X \to \RR$ and $b$ obey the $(f,k,t,M)$-derivatives condition on $S \subseteq C^k(H)$ and that $b$ is in normal form.  Suppose also that $\gcd(|H|, (t-2)!) = 1$.
  
  Suppose moreover that for some $R$, $1 \le R \le t-1$, some $i \in [k]$ and $\omega \in \llbracket k \rrbracket$ with $\omega(i) = 1$, and some $v \in \bigoplus_{i=1}^{R} \ZZ^{d_i}$, there are at least $\eps |H|^{k+2}$ triples $(c_0,c_1,c) \in T_{k,i}$ \uppar{defined as in~\eqref{eq:tk}} such that:---
  \begin{itemize}
    \item $c_0,c_1 \in S$;
    \item $b_{\le R}(c_0,\omega) - b_{\le R}(c_1,\omega) = v$; and
    \item $b_{>R}(c_0,\omega') = b_{>R}(c_1,\omega')$ for all $\omega' \in F^i$.
  \end{itemize}
  Then the function $v \cdot f_{\le R}$ satisfies
  \[
    \big|\big\{c \in C^{R}(H) \colon \partial^{R} (v \cdot f_{\le R})(c) = 0 \big\}\big| \ge O_{k,t}(M)^{-O_{k,t}(D)} \eps^{O_{k,t}(1)} |C^{R}(H)| ,
  \]
  i.e.\ is an approximate polynomial with the parameter shown.  \uppar{As usual, we interpret expressions such as the one above only to count those $c \in C^R(H)$ where $\partial^R (v \cdot f_{\le R})$ is actually well-defined.}
\end{lemma}

If $v_{=R} \ne 0$, this is meant to be surprising: a priori, $v \cdot f_{\le R}$ is a linear combination of objects of degree (at most) $R$, and so would typically be an approximate polynomial of degree $R$, but under these hypotheses turns out to be one of degree $R-1$.  On the other hand, if $b$ is upper compatible then $v$ is necessarily zero and the conclusion is not surprising at all.

We recall (Remark~\ref{rem:b0-compat}) that the issue of upper compatibility does not arise for $b_{=0}$, and hence the restriction to $t \ge 2$ is natural.

The more simply stated consequence of Lemma~\ref{lem:incompatible-implies-1pc} is the following.

\begin{corollary}%
  \label{cor:incompatible-implies-1pc}
  Suppose $t$, $k$, $d_0,\dots,d_{t-1}$, $X$, $S$, $f$, $g$, $b$ and $D$ are as in the first paragraph of Lemma~\ref{lem:incompatible-implies-1pc} \uppar{with $\gcd(|H|,(t-2)!) = 1$}, and that $b_{=r}$ fails to be $\delta$-upper compatible on $S$, for some $1 \le r \le t-1$ and some $\delta > 0$.

  Then there exists $R$, $1 \le R \le t-1$, and some $v \in \bigoplus_{i=1}^R \ZZ^{d_i}$ with $v_{=R} \ne 0$ and $\|v\|_1 \ll_t M^{O_t(1)}$, such that
  \[
    \big|\big\{c \in C^{R}(H) \colon \partial^{R} (v \cdot f_{\le R})(c) = 0 \big\}\big| \ge O_{k,t}(M)^{-O_{k,t}(D)} \delta^{O_{k,t}(1)} |C^{R}(H)|  .
  \]
\end{corollary}

A similar pair of statements holds if the function $b_{=t-1}$ fails to satisfy the other requirements of a generalized cocycle.
\begin{lemma}%
  \label{lem:no-cocycle-implies-1pc}
  Suppose $t \ge 1$, $k \ge  0$, $f$ is a $(t-1,d,M)$-polynomial hierarchy on a system of cubes $S_0,\dots,S_t$, and $g$ and $b$ obey the $(f,k,t,M)$-derivatives condition on $S \subseteq C^k(H)$, where $c|_F \in S_{\dim F}$ for each $c \in S$ and each face $F$ of $\llbracket k \rrbracket$ of dimension at most $t$, and $b$ is in normal form.  Suppose also that $\gcd(|H|, (t-1)!) = 1$ and write $D = \sum_{r=0}^{t-1} d_r$.

  Now suppose that for some $i \in [k]$, some $\omega \in \llbracket k \rrbracket$ and some $z \in \ZZ^{d_{t-1}}$, the following holds: there are at least $\eps |H|^{k+2}$ triples $(c_0,c_1,c) \in T_{k,i}$ such that $c_0,c_1,c \in S$, $b(c_0,\omega') = b(c_1,\omega')$ for each $\omega' \in F^i$, and
  \[
    b_{=t-1}(c,\omega) - \sum_{\omega' \in \llbracket k \rrbracket} Z_{t-1}(\omega,\omega') \fb_{=t-1}(\omega') = z
  \]
  where $Z_{t-1}$ is as in~\eqref{eq:zr} and
  \[
    \fb(\omega') = \begin{cases} b(c_0, \omega') &\colon \omega'(i) = 0 \\ b(c_1, \omega' \setminus \{i\}) &\colon \omega'(i) = 1 \end{cases}
  \]
  as in Definition~\ref{def:gen-cocycle}.  Then there exists $v \in \bigoplus_{r=0}^{t-1} \ZZ^{d_r}$ with $v_{=t-1}=z$, $\|v\|_1 \ll_t M^{O_t(1)}$, and such that
  \[
    \big|\big\{c \in C^{t-1}(H) \colon \partial^{t-1} (v \cdot f_{\le t-1})(c) = 0 \big\}\big| \gg_{k,t} O(M)^{-O_{k,t}(D)} \eps^{O_{k,t}(1)} |C^{t-1}(H)|  .
  \]
\end{lemma}

Again, the conclusion is surprising whenever $z \ne 0$, which exactly corresponds to a failure of the relevant part of Definition~\ref{def:gen-cocycle}.  In particular we again record a simple consequence.

\begin{corollary}%
  \label{cor:no-cocycle-implies-1pc}
  Suppose $t$, $k$, $f$, $S_0,\dots,S_t$, $S$, $g$, $b$ and $D$ are as in the first paragraph of Lemma~\ref{lem:no-cocycle-implies-1pc} \uppar{and $\gcd(|H|,(t-1)!) = 1$}, and $b_{=t-1}$ is not a generalized $k$-cocycle of type $t-1$ and loss $\delta$.

  Then there exists $v \in \bigoplus_{r=0}^{t-1} \ZZ^{d_r}$ with $v_{=t-1} \ne 0$, $\|v\|_1 \ll_t M^{O_t(1)}$, and such that
  \[
    \big|\big\{c \in C^{t-1}(H) \colon \partial^{t-1} (v \cdot f_{\le t-1})(c) = 0 \big\}\big| \gg_{k,t} O(M)^{-O_{k,t}(D)} \delta^{O_{k,t}(1)} |C^{t-1}(H)|  .
  \]
\end{corollary}

Next, we show that we can bootstrap these approximate polynomial conclusions to improved derivatives conditions.  Crucially, we now need some hypothesis on the subgroups of the ambient group $H$.

\begin{lemma}%
  \label{lem:1pc-implies-derivs}
  Let $t \ge 1$, let $S_0,\dots,S_{t+1}$ be a system of cubes of parameter $1-\eps$ where $\eps \le \eps_0(t)$ is sufficiently small; let $f$ be an $(t-1,d,M)$-polynomial hierarchy on $S_0,\dots,S_{t}$; and suppose $\phi \colon S_0 \to \RR$ satisfies the $(f,t+1,t,M)$-derivatives condition on $S_{t+1}$.  Suppose also that
  \[
    \big|\big\{c \in C^{t}(H) \colon \partial^{t} \phi(c) = 0 \big\}\big| \ge \delta |C^t(H)|.
  \]
  Also, let $\eta > 0$ be another parameter; suppose $H$ has no proper subgroups of order at least $\delta |H|/J$ where $J=J(t)$ is some absolute constant; suppose $\gcd(|H|,(t-1)!) = 1$; and write $D = \sum_{i=0}^{t-1} d_i$.
  
  Then either:---
  \begin{enumerate}[label=(\Alph*)]
    \item $\phi$ satisfies the $\big(f,t,t,M'\big)$-derivatives condition on a set $S \subseteq S_t$ with $\mu(S_t \setminus S) \le \eta$, where $M' \ll_t M \delta^{-O_t(1)}$; or
    \item for some $R$, $1 \le R \le t-1$, and some $v \in \bigoplus_{r=1}^{R} \ZZ^{d_r}$ with $v_{=R} \ne 0$ and $\|v\|_1 \ll_t (M/\delta)^{O_t(1)}$, we have
      \[
        \big|\big\{c \in C^{R}(H) \colon \partial^R (v \cdot f_{\le R})(c) = 0 \big\}\big| \ge O(M/\delta)^{-O_t(D)} \eta^{O_t(1)} |C^R(H)| .
      \]
  \end{enumerate}
  Note in particular that if $t=1$ then option (B) is nonsensical and so necessarily (A) holds.
\end{lemma}

Finally, we show that a function obeying an $(f,t,t,M')$-derivatives condition can be split into a piece obeying the more usual $(f,t,t-1,M'')$-derivatives condition, and a function that is pointwise a bounded linear combination of the values $f_{\le t-1}$ (as in part (iii) above).

\begin{lemma}%
  \label{lem:derivs-implies-derivs}
  Let $t \ge 1$, let $f$ be an $(t-1,d,M)$-polynomial hierarchy on a system of cubes $S_0,\dots,S_{t}$ of parameter $\eps$ where $\eps \le \eps_0(t)$ is sufficiently small, and suppose $\phi \colon S_0 \to \RR$ and $b$ satisfy the $(f,t,t,M)$-derivatives condition on $S_{t}$.  Also let $\eta > 0$ be a parameter.

  Then either:---
  \begin{enumerate}[label=(\Alph*)]
    \item there exists a set $S \subseteq S_{t}$ with $\mu(S_t \setminus S) \le \eta$, and a decomposition $\phi = \phi_0 + \phi_1$ such that (i) $\phi_0$ satisfies the $(f,t,t-1,O_t(M))$-derivatives condition on $S$, and (ii) for each $x \in S_0$ there exists $v(x) \in \ZZ^{d_{t-1}}$ with $\phi_1(x) = v(x) \cdot f_{=t-1}(x)$ and $\|v(x)\|_1 \ll_t M + d_{t-1}$;

    \item or, for some $R$, $1 \le R \le t-1$, and some $v \in \bigoplus_{r=1}^{R} \ZZ^{d_r}$ with $v_{=R} \ne 0$ and $\|v\|_1 \ll_t M^{O_t(1)}$, we have
    \[
      \big|\big\{c \in C^{R}(H) \colon \partial^R (v \cdot f_{\le R})(c) = 0 \big\}\big| \ge O_t(M)^{-O_{t}(D)} \eta^{O_{t}(1)} |C^R(H)| .
    \]
  \end{enumerate}
  Again, when $t=1$ statement (A) holds unconditionally.
\end{lemma}

We briefly record the deductions of Corollary~\ref{cor:incompatible-implies-1pc} and Corollary~\ref{cor:no-cocycle-implies-1pc}, which are direct pigeonholing arguments.

\begin{proof}[Proof of Corollary~\ref{cor:incompatible-implies-1pc}]
  The hypothesis implies (\emph{a fortiori}) that there is some $r$, $1 \le r \le t-1$, such that $b_{=r}$ is not $\delta 2^{-(k+1)r}$-almost upper compatible. We let $R$ denote the largest $r$ for which this holds.
  
  By a union bound, this implies that there exists $\omega \in \llbracket k \rrbracket$ and $i \in [k]$ with $\omega(i) = 1$, such that the number of triples $(c_0,c_1,c) \in T_{k,i}$ such that $c_0,c_1 \in S$, $b_{=R}(c_0,\omega)-b_{=R}(c_1,\omega) \ne 0$, but $b_{>R}(c_0,\omega') = b_{>R}(c_1,\omega')$ for all $\omega' \in F^i$, is at least
  \[
    2^{-(k+1)R} \delta\, |H|^{k+2} - \sum_{r=R+1}^{t-1} \sum_{\omega' \in F^i} 2^{-(k+1)r } \delta\, H^{k+2} \gg_{k,t} \delta\, |H|^{k+2} \ .
  \]
  Since $\|b(c_0,\omega)\|_1,\ \|b(c_1,\omega)\|_1 \ll_t M^{O_t(1)}$, there are at most $O_t(M)^{O_t(D)}$ possible values $v = b(c_0,\omega)-b(c_1,\omega)$ with $v_{=R} \ne 0$, and so one such value $v$ occurs for $\gg_{k,t} \delta O(M)^{-O_t(D)} |H|^{k+2}$ triples.  

  This choice of $\omega$, $i$ and $v$ (with $v_{=R} \ne 0$) then obeys the hypotheses of Lemma~\ref{lem:incompatible-implies-1pc}, and the result follows.
\end{proof}

\begin{proof}[Proof of Corollary~\ref{cor:no-cocycle-implies-1pc}]
  If $b_{=t-1}$ fails to be $2^{-2k}\delta$ upper-compatible on $S$, we are done by Corollary~\ref{cor:incompatible-implies-1pc}, so we can assume that it is.
  Since $b_{=t-1}$ is not a generalized $k$-cocycle of type $t-1$ and loss $\delta$, part (ii) of Definition~\ref{def:gen-cocycle} must fail.
  By pigeonholing in the choice of $\omega \in \llbracket k \rrbracket$ in that definition, there exists $i \in [k]$ and $\omega \in \llbracket k \rrbracket$ such that for at least $2^{-k} \delta |H|^{k+2}$ triples $(c_0,c_1,c) \in T_{k,i}$, we have
  \begin{equation}
    \label{eq:no-cocycle-neq}
    b_{=t-1}(c,\omega) - \sum_{\omega' \in \llbracket k \rrbracket} Z_{=t-1}(\omega,\omega') \fb_{=t-1}(\omega') \ne 0
  \end{equation}
  where $\fb$ and $Z_{t-1}$ are as in the statement of Lemma~\ref{lem:no-cocycle-implies-1pc}.
  Since $b_{=t-1}$ is $2^{-2k}\delta$ upper-compatible, for at least $2^{-k-1} \delta |H|^{k+2}$ of these triples $(c_0,c_1,c)$ we simultaneously have $b_{=t-1}(c_0,\omega')=b_{=t-1}(c_1,\omega')$ for all $\omega' \in F^i$.
  
  Finally, since the left-hand side of~\eqref{eq:no-cocycle-neq} takes $O(M)^{O_{k,t}(D)}$ distinct values, we can find a fixed $z \in \ZZ^{d_{t-1}}$, $z \ne 0$, such in addition to these previous conditions,
  \[
    b_{=t-1}(c,\omega) - \sum_{\omega' \in \llbracket k \rrbracket} Z_{t-1}(\omega,\omega') \fb_{=t-1}(\omega') = z
  \]
  holds for at least $2^{-k-1} \delta M^{-O_{k,t}(D)}|H|^{k+2}$ such triples $(c_0,c_1,c)$.
  At this point, the hypotheses of Lemma~\ref{lem:no-cocycle-implies-1pc} are satisfied, and applying that lemma gives the result.
\end{proof}

We now prove Theorem~\ref{thm:strong-derivs}, assuming all of these statements.  The proofs of Lemma~\ref{lem:incompatible-implies-1pc}, Lemma~\ref{lem:no-cocycle-implies-1pc}, Lemma~\ref{lem:1pc-implies-derivs} and Lemma~\ref{lem:derivs-implies-derivs} are given in the next two sections.

\begin{proof}[Proof of Theorem~\ref{thm:strong-derivs}]
  We first argue the weaker bound that holds for large $t$, and which is more direct.  Since we treat the case $t=1$ in detail below, to avoid corner cases we assume $t \ge 2$ here.  We work by induction on the quantity $I = \sum_{r=0}^{t-1} (r+1) d_r$.

  Fix $b$ such that $g$ and $b$ obey the $(f,s+1,t,M)$-derivatives condition.  Clearly if the same $g$ and $b$ obey the $(f,s+1,t,M,\delta)$-strong derivatives condition there is nothing to do.
  
  If not, either some $b_{=r}$ is not $\delta$-almost upper compatible (for $1 \le r \le t-1$), or $b_{=t-1}$ fails to be a generalized $(s+1)$-cocycle of type $t-1$ and loss $\delta$.
  By Corollary~\ref{cor:incompatible-implies-1pc} or Corollary~\ref{cor:no-cocycle-implies-1pc} respectively, in either case we deduce that there is some $R$, $1 \le R \le t-1$, and some $v \in \bigoplus_{r=0}^R \ZZ^{d_r}$ such that $v_{=R} \ne 0$, $\|v\|_1 \ll_t M^{O_t(1)}$ and
  \begin{equation}
    \label{eq:1pc-cond}
    \big|\big\{c \in C^{R}(H) \colon \partial^{R} (v \cdot f_{\le R})(c) = 0 \big\}\big| \ge \nu |C^{R}(H)|
  \end{equation}
  where $\nu \ge O_{s}(M)^{-O_{s}(D)} \delta^{O_{s}(1)}$.
  We note that these bounds on $\nu$ and $\|v\|_1$ are stronger than the bound
  \begin{equation}
    \label{eq:new-v-nu-bound}
    \|v\|_1,\ 1/\nu \le O_s(M/\delta)^{O_s(D)^{O_s(1)}}.
  \end{equation}
  
  We then invoke Lemma~\ref{lem:1pc-implies-derivs} on $v \cdot f_{\le R}$ (noting that $v \cdot f_{\le R}$ obeys an $\big(f,R+1,R, \|v\|_1 M\big)$-derivatives condition), as well as Lemma~\ref{lem:derivs-implies-derivs}, with parameter $\eta=(\delta/I)^{C}/C$ for some suitable absolute constant $C=C(s)$ to be specified shortly.
  There are two possible outcomes.
  
  One (option (B) in either case) is that we deduce~\eqref{eq:1pc-cond} again but with a smaller value $R'$, $1 \le R' < R$, and new values $\nu'$ and $v'$.
  Moreover the bounds this gives on $\nu'$ and $\|v'\|_1$ are of the same form as~\eqref{eq:new-v-nu-bound}, albeit with worse implied constants.
  If this happens, we try again with these new parameters until this case no longer occurs, which must happen after at most $t-2$ attempts.
  In what follows, we abuse notation to write $v$, $R$ and $\nu$ for the final choices obtained, which still obey~\eqref{eq:new-v-nu-bound}.\footnote{Although~\eqref{eq:new-v-nu-bound} is initially pessimistic, it becomes roughly the best we can say after iterating option (B) several times.}
  
  Given the other outcome (option (A) both times), with the help of Corollary~\ref{cor:cube-system-patch}, we find: a subsystem of cubes $S'_{0},\dots,S'_{s+1}$ with $S'_k \subseteq S_k$ for each $0 \le k \le s+1$; a decomposition $v \cdot f_{\le R} = \phi_0 + \phi_1$; and a parameter $M'$; such that the following hold:---
  \begin{itemize}
    \item $S'_0,\dots,S'_{s+1}$ has parameter $1-\eps-O_s(\eta^{1/O_s(1)})$ and $\mu(S_k \setminus S'_k) \ll_s \eta^{1/O_s(1)}$ for each $0 \le k \le s+1$;
    \item the function $\phi_0$ obeys the $(f,R,R-1,M')$-derivatives condition on $S'_R$;
    \item $\phi_1(x) = u(x) \cdot f_{\le R-1}(x)$ for each $x \in S'_0$, where $u \in \bigoplus_{i=0}^{R-1} \ZZ^{d_i}$ has $\|u\|_1 \le M'$; and
    \item $M' \le O_s(M/\delta)^{-O_s(D)^{O_s(1)}}$.
  \end{itemize}
  By setting the constant $C=C(s)$ in the choice of $\eta$ appropriately, we can assume that $S'_0,\dots,S'_{s+1}$ has parameter at least $1-\eps-\delta/2I$ and that $\mu(S_k \setminus S'_k) \le \delta/2I$ for each $k$.

  To conclude, we define a new hierarchy $f'$ by deleting some $f_{R,j}$ from $f_{=R}$ and adding $\phi_0$ to level $f_{=R-1}$.
  Specifically, pick any $j \in [d_R]$ such that $v_{R,j} \ne 0$ (possible by hypothesis); then set
  \[
    d'_r = \begin{cases} d_r - 1 &\colon r = R \\ d_r + 1 &\colon r=R-1 \\ d_r &\colon \text{otherwise} \end{cases}
  \]
  and
  \[
    f'_{=r} = \begin{cases}
      f_{=r} &\colon r > R \\
      \frac1{v_{R,j}} (f_{R,1},\dots,f_{R,j-1},f_{R,j+1},\dots,f_{R,d_R}) &\colon r=R \\
      \frac1{v_{R,j}} (f_{R-1,1},\dots,f_{R-1,d_{R-1}}, \phi_0)  &\colon r = R-1 \\
      \frac1{v_{R,j}} f_{=r} &\colon r < R-1  .
    \end{cases}
  \]
  By the properties above, this is an $(s-1,d',M'')$-polynomial hierarchy on $S'_0,\dots,S'_s$, which reduces $f$ with parameter $K$, where $K, M'' \le (M' |v_{R,j}|)^{O_s(1)}$.
  Hence, using Lemma~\ref{lem:use-reduction}, the original hypotheses of the theorem are satisfied by $g$, $f'$, $\eps'=\eps+\delta/2I$ and $K M''$.
  Moreover, the quantity $I$ has decreased by one.

  The result now follows by inductive hypothesis, applied with $\delta' = (1-1/2I)\delta$.
  In the inductive step, we use the fact (whose proof is straightforward) that if $f'$ reduces $f$ with parameter $K_1$ and $f''$ reduces $f'$ with parameter $K_2$ then $f''$ reduces $f$ with parameter $K_1 K_2$.
  
\vspace{\baselineskip}

  To achieve the better bound when $t \le 2$, we adapt this argument slightly.
  First, note that when $t \le 2$ in the theorem, in the above argument we only ever apply Lemma~\ref{lem:1pc-implies-derivs} or Lemma~\ref{lem:derivs-implies-derivs} with $t=1$, and so the recursive step of reducing $R$ (i.e., option (B) in the dichotomies) never arises.

  Next, instead of finding just one $v$ that represents a failure of the strong derivatives condition for $g$, we record a list of all such $v$ simultaneously.
  Firstly, when $t=2$ we define a set of vectors $V \subseteq \ZZ^{d_1}$ to consist of all $v$ such that, for some $\omega \in \llbracket s+1 \rrbracket$ and $i \in [s+1]$ with $\omega(i)=1$, there are at least $\sigma |H|^{s+3}$ triples $(c_0,c_1,c) \in T_{s+1,i}$ such that $c_0,c_1 \in S_{s+1}$ and $b(c_0,\omega)_{=1}-b(c_1,\omega)_{=1}= v$ (i.e., $v$ obeys the hypothesis of Lemma~\ref{lem:incompatible-implies-1pc} in the special case $t=2$, $k=s+1$, with these values of $\omega$, $i$ and with $\eps$ replaced by $\sigma$).
  Similarly, for $t=1$ or $t=2$ we define $W \subseteq \ZZ^{d_{t-1}}$ to consist of all $z$ that obey the hypothesis of Lemma~\ref{lem:no-cocycle-implies-1pc} (with $k=s+1$ and $\eps$ replaced by $\sigma$).
  Here $\sigma>0$ is some parameter to be determined.
  We then set $U = V \cup W$ (if $t=2$) or $U = W$ (if $t=1$), so $U \subseteq \ZZ^{d_{t-1}}$.

  Note that it follows from the definitions of $V$ and $W$ that any $u \in U$ has $\|u\|_1 \ll M^{O_s(1)}$.
  Also, we now fix a subset $U' \subseteq U$ of size at most $d_{t-1}$ such that $\spn_\QQ(U') = \spn_\QQ(U)$ as vector subspaces of $\QQ^{d_{t-1}}$.

  If $u \in V$, Lemma~\ref{lem:incompatible-implies-1pc} tells us that $u \cdot f_{=1}$ is an approximate polynomial of degree $t-1$ and parameter at least $O_s(M)^{-O_s(D)} \sigma^{O_s(1)}$.
  If $u \in W$ and $t=1$, by Lemma~\ref{lem:no-cocycle-implies-1pc} we get the same conclusion; if $t=2$ then we get that $v \cdot f_{\le 1}$ is an approximate polynomial of degree $1$ with the same parameter for some $v \in \ZZ^{d_0} \oplus \ZZ^{d_1}$ with $v_{=1}=u$; however, since $f_{=0}$ consists of constant functions (by Remark~\ref{rem:hierarchy-t0}) this implies the same conclusion for $u \cdot f_{=1}$.
  Hence, for all $u \in U$, $u \cdot f_{=1}$ is an approximate polynomial of degree $t-1$ and parameter at least $O_s(M)^{-O_s(D)} \sigma^{O_s(1)}$.

  When $t=1$, the statement that $u \cdot f_{=0}$ as an approximate polynomial of degree $0$ states that it is zero a positive fraction of the time (and in fact we did not need the machinery of Lemma~\ref{lem:incompatible-implies-1pc} or Lemma~\ref{lem:no-cocycle-implies-1pc} to tell us this).  Moreover, each $f_{0,j}$ is a constant function (again, see Remark~\ref{rem:hierarchy-t0}), so $u \cdot f_{=0}$ is identically zero.

  We can therefore build a new system of cubes as follows.
  We pick a subset $J \subseteq [d_0]$, $|J|=d_0 - |U'|$ such that the standard basis vectors $\{e_j \in \ZZ^{d_0} \colon j \in J\}$ together with $U'$ form a basis $\cB=(\beta_1,\dots,\beta_{d_0})$ for $\QQ^{d_0}$.
  We then define $f'_{=0}$ to be $\frac1{\det A} (f_{0,j})_{j \in J}$.
  Write $A$ for the $d_0 \times d_0$ matrix with the vectors $v \in \cB$ as rows, so that
  \[
    \beta_j = \sum_{i=1}^{d_0} A_{ji} e_i, \qquad\qquad e_i = \sum_{j=1}^{d_0} A^{-1}_{ij} \beta_j.
  \]
  Hence, each original function $f_{0,i}$ for $i \in [d_0]$ can be expressed as an integer linear combination
  \begin{equation}
    f_{0,i} = \frac1{\det A} \sum_{j=1}^{d_0} \cA_{ij} (\beta_j \cdot f_{=0})
            = \sum_{j=1}^{|J|} \cA_{ij} f'_{0,j}
               \label{eq:t1-f-exp}
  \end{equation}
  where $\cA = (\det A) A^{-1}$ is the adjugate matrix of $A$, and the second equality holds because $\beta_j \cdot f_{=0}$ is just $(\det A) f'_{0,j}$ if $j \le |J|$ and so $\beta_j$ is one of the standard basis vectors, or identically zero if $\beta_j \in U'$.
  Since the rows of $A$ all have $\|b_j\|_1 \ll M^{O_s(1)}$, by the Hadamard bound (say) each coefficient $\cA_{ij}$ has size $O(M)^{O_s(D)}$, and so (absorbing a factor of $d_0$ into the term $O(1)^D$) we find that $f'$ reduces $f$ with parameter $O(M)^{O_s(D)}$.

  When $t=2$, we further apply Lemma~\ref{lem:1pc-implies-derivs} and Lemma~\ref{lem:derivs-implies-derivs} (with $t$ replaced by $t-1=1$) to $u \cdot f_{=1}$ for each $u \in U'$, again with a suitably small $\eta = \delta^{C}/d_1 C$, where $C=C(s)$ is an absolute constant to be chosen shortly. 
  Recall that option (B) cannot occur as we are in the case $t=1$ of these lemmas.
  Hence, for each $u \in U'$ we find:
  \begin{itemize}
    \item a set $S_u \subseteq S_1$ with $\mu(S_1 \setminus S_u) \le \eta$;
    \item a decomposition $u \cdot f_{=1} = \phi_{0,u} + \phi_{1,u}$ on $S_0$, where
    \item $\phi_{0,u}$ satisfies $\partial^1 \phi_{0,u}(c) = 0$ for all $c \in S_u$, and
    \item for each $x \in S_0$, $\phi_{1,u}(x) = k_u(x) \cdot f_{=0}$ for some $k_u(x) \in \ZZ^{d_0}$ with $\|k_u(x)\|_1 \le O(M)^{O(D)} \sigma^{-O(1)}$.
  \end{itemize}
  By taking $S = \bigcap_{u \in U'} S_u$, which satisfies $\mu(S_1 \setminus S) \le |U'| \eta \le \delta^{C(s)}/C(s)$, applying Corollary~\ref{cor:cube-system-patch}, and choosing $C(s)$ suitably, we obtain a system of cubes $S'_0,\dots,S'_{s+1}$ with parameter at least $1-\eps-\delta$ (or $0.51$, whichever is larger) such that $S_u \subseteq S'_1$ for each $u \in U'$, $S'_k \subseteq S_k$ for each $k$ and $\mu(S_k\setminus S'_k) \le \delta$ for each $k$.
  In particular, $\phi_{0,u}$ is a constant function on $S'_0$ (by Remark~\ref{rem:hierarchy-t0}).
  
  The new system of cubes is then built as follows.
  Again pick a subset $J \subseteq [d_1]$ such that $\cB = \{e_j \in \ZZ^{d_1} \colon j \in J \} \cup U'$, $\cB = (\beta_1,\dots,\beta_{d_1})$ is a basis for $\QQ^{d_1}$, and let $A$ be the matrix with $\cB$ as its rows.
  Set
  \begin{align*}
    f'_1 &= \frac1{\det A} (f_{1,j})_{j \in J} \\
    f'_0 &= \frac1{\det A} \bigl( f_{=0} \cup (\phi_{0,u})_{u \in U'} \bigr).
  \end{align*}
  For $i \in [d_1]$ we have
  \begin{align}
    \nonumber
    f_{1,i}(x) &= \frac1{\det A} \sum_{j=1}^{d_1} \cA_{i,j} (\beta_j \cdot f_{=1}(x)) \\
               &= \sum_{j=1}^{|J|} \cA_{ij} f'_{0,j} + \left( \sum_{j = |J|+1}^{d_1} \cA_{ij} k_{\beta_j}(x) \right) \cdot f_{=0} + \sum_{j = |J|+1}^{d_1} \cA_{ij} \phi_{0,\beta_j}(x)
               \label{eq:t2-f-exp}
  \end{align}
  (recalling $\beta_j \in U'$ for $|J|+1 \le j \le d_1$).
  It follows that this new hierarchy $f'$ reduces $f$ with parameter at most $O(M/\sigma)^{O(D)}$ (bounding the coefficients as in the case $t=1$).

  Finally we claim in both cases $t=1$ and $t=2$ that the new $(f',s+1,t,M')$-derivatives condition on $g$ (obtained from the calculations above and Lemma~\ref{lem:use-reduction}) is in fact an $(f',s+1,t,M',\delta)$-strong derivatives condition.

  Let $b$ be the original coefficients (so that $g$ and $b$ obey an $(f,s+1,t,M)$-derivatives condition) and
  let $b'$ be the new coefficients provided by Lemma~\ref{lem:use-reduction} (so that $g$ and $b'$ obey an $(f',s+1,t,M')$-derivatives condition).
  In both~\eqref{eq:t1-f-exp} and~\eqref{eq:t2-f-exp} above, the functions $f_{t-1,j}$ are expressed as fixed linear combinations $\sum_{j=1}^{|J|} \cA_{ij} f'_{t-1,r}$ plus lower-order terms.
  It follows that dually
  \begin{equation}
    \label{eq:old-b-new-b}
    b'_{t-1,j}(c,\omega) = \sum_{i=1}^{d_{t-1}} \cA_{ij} b_{t-1,i}(c,\omega)
  \end{equation}
  for all $j \in [d'_{t-1}]$, $c \in S'_{s+1}$ and $\omega \in \llbracket s+1 \rrbracket$.

  Suppose then that $g$ and $b'$ fail to obey the strong $(f',s+1,t,M',\delta)$-derivatives condition on $S'_{s+1}$.
  First, suppose $b'$ fails to be $\delta$ upper-compatible.
  Arguing as in the proof of Corollary~\ref{cor:incompatible-implies-1pc}, there exist $i \in [s+1]$, $\omega \in F^i$ and $\gg_s \delta |H|^{s+3}$ triples $(c_0,c_1,c) \in T_{s+1,i}$ such that $b'_{=t-1}(c_0,\omega) \ne b'_{=t-1}(c_1,\omega)$.
  By~\eqref{eq:old-b-new-b}, it follows that for these triples,
  \[
    \left( \sum_{i=1}^{d_t-1} \cA_{ij} \bigl( b_{t-1,i}(c_0,\omega) - b_{t-1,i}(c_0,\omega) \bigr) \right)_{j=1}^{|J|} \ne 0.
  \]
  Equivalently, this states that $b_{=t-1}(c_0,\omega) - b_{=t-1}(c_1,\omega)$ is \emph{not} supported on the latter basis vectors $\beta_{|J|+1},\dots,\beta_{d_{t-1}}$, i.e., is not contained in $\spn_\QQ(U')$.
  By pigeonholing as in the proof of Corollary~\ref{cor:incompatible-implies-1pc}, we may find $v \in \ZZ^{d_1} \setminus \spn_\QQ(U')$ such that
  \[
    b_{=t-1}(c_0,\omega) - b_{=t-1}(c_1,\omega) = v
  \]
  for $\gg_s \delta M^{-O(D)} |H|^{s+3}$ such triples $(c_0,c_1,c)$.
  Setting $\sigma = O_s(M)^{-O_s(D)} \delta$ for suitable implied constants, there are at least $\sigma |H|^{s+3}$ such triples, and hence $v \in V$ by definition, which is a contradiction as by construction $v \notin \spn_\QQ(U)$.
  
  If $b'$ fails to be a generalized $k$-cocycle of type $t-1$ and loss $\delta$, we argue similarly: pigeonholing as in the proof of Corollary~\ref{cor:no-cocycle-implies-1pc}, and again taking $\sigma = O(M)^{-O_s(D)} \delta$, we obtain $z \in \ZZ^{d_1} \setminus \spn_\QQ(U')$ obeying the criterion to be in the set $W$, which is a contradiction.
  The details are essentially a reiteration of the proof of Corollary~\ref{cor:no-cocycle-implies-1pc} and the argument above.
\end{proof}

\subsection{Locating approximate polynomials by Cauchy--Schwarz}

The main ingredient in both Lemma~\ref{lem:incompatible-implies-1pc} and Lemma~\ref{lem:no-cocycle-implies-1pc} is repeated application of the Cauchy--Schwarz inequality.  It will be convenient to systematize this process using the notion of Cauchy--Schwarz complexity; see~\cite[Definition 1.5]{gt-primes}. We recall a slight variant of the definition.

\begin{definition}%
  \label{def:cauchy-schwarz-complexity}
  Let $\phi_1,\dots,\phi_k \colon \ZZ^d \to \ZZ$ be linear forms, and $j \in [k]$ some index.  Say $(\phi_1,\dots,\phi_k)$ have \emph{Cauchy--Schwarz complexity $\le t$ with denominator $Q$ at position $j$} if there are sets $\Sigma_1,\dots,\Sigma_{t+1} \subseteq [k] \setminus \{j\}$ with union $[k] \setminus \{j\}$, and elements $\sigma_1,\dots,\sigma_{t+1} \in \ZZ^d$, such that for each $r \in [t+1]$ we have $\phi_i(\sigma_r) = 0$ for all $i \in \Sigma_r$, but $q_r := \phi_j(\sigma_r)$ satisfies $q_r \ne 0$ and moreover $q_r | Q$.
\end{definition}

\begin{remark}
  The more usual way of stating this is that $\phi_j \notin \spn(\phi_i \colon i \in S_r)$ for each $r$.  This is equivalent up to getting quantitative control on the parameter $Q$, which becomes significant in groups with torsion.
\end{remark}

This notion is usually applied to bounding multilinear averages of real or complex valued functions by $U^k$-norms (such as~\cite[Theorem 2.3]{gowers-wolf}, or the original~\cite[Proposition 7.1]{gt-primes} which works in a more difficult setting).  The following lemma can be thought of as a discrete analogue of this result.

\begin{lemma}[Discrete Cauchy--Schwarz complexity]%
  \label{lem:gvn}
  Suppose $\phi_1,\dots,\phi_k \colon \ZZ^d \to \ZZ$ is a system of linear forms of Cauchy--Schwarz complexity $\le t$ with denominator $Q$ at position $j$, and suppose that $f_1,\dots,f_k \colon H \to \RR$ are functions such that
  \[
    \Big|\Big\{ x \in H^d \colon \sum_{i=1}^k f_i(\phi_i(x)) = 0 \Big\}\Big| \ge \delta |H|^d.
  \]
  Then provided $\gcd(|H|,Q) = 1$, we have
  \[
    \big|\big\{ c \in C^{t+1}(H) \colon \partial^{t+1} f_j(c) = 0 \big\}\big| \ge \delta^{2^{t+1}} |C^{t+1}(H)| .
  \]
  \uppar{Note we abuse notation to write $\phi_i \colon H^d \to H$ for the induced homomorphism.}
\end{lemma}
The proof is given in Appendix~\ref{app:cs}, and as might be expected consists of multiple applications of the Cauchy--Schwarz inequality.

\begin{proof}[Proof of Lemma~\ref{lem:incompatible-implies-1pc}]
  If $v=0$ the conclusion is trivial, so we may assume $v \ne 0$.
  Since $b$ is in normal form, if $|\omega| \ge R+1$ we have $b_{\le R}(c,\omega) = 0$ for all $c \in S$, which would force $v=0$; hence $|\omega| \le R$.

  Since $b(c_0,\omega')$ and $b(c_1,\omega')$ are bounded by $M^{O_k(1)}$ in $\ell^1$-norm, each takes at most $O(M)^{O_k(D)}$ distinct values, so we may pigeonhole to assume these functions are constant in $c_0$ and $c_1$.  
  That is, we choose global values $B_0(\omega')$ and $B_1(\omega')$ for each $\omega' \in \llbracket k \rrbracket$, such that at least $\eps\, O(M)^{-O_{k,t}(D)}$ triples $(c_0,c_1,c) \in T_{k,i}$ satisfy the hypotheses from the statement as well as $b(c_0,\omega') = B_0(\omega')$ and $b(c_1,\omega') = B_1(\omega')$ for each $\omega' \in \llbracket k \rrbracket$.
  Let $T'$ denote this set of triples, with $\mu(T') \ge \eps\, O(M)^{-O_{k,t}(D)}$.

  Recalling~\eqref{eq:trying-to-glue}, we have that
  \[
    \partial^k g(c) = \sum_{\omega' \in \llbracket k \rrbracket} (-1)^{|\omega'|} B_0(\omega') \cdot f_{\le t-1}(c_0(\omega')) - 
    \sum_{\omega' \in \llbracket k \rrbracket} (-1)^{|\omega'|} B_1(\omega') \cdot f_{\le t-1}(c_1(\omega')) 
  \]
  for all $(c_0,c_1,c) \in T'$.
  %Since $b$ is in normal form, so are $B_0,B_1$, and by hypothesis $b_{>R}(c_0,\omega') = b_{>R}(c_1,\omega')$ for all $\omega' \in F^i$; therefore, only terms with $\omega' \in F_i$ or $|\omega'| \le R$ make a non-zero contribution.
  Recalling that $c_0(\omega') = c_1(\omega')$ for $\omega' \in F^i$ and
  regrouping terms, we get that
  \begin{equation}
    \label{eq:1st-sum}
    \sum_{\eta \in \llbracket k \rrbracket} (-1)^{|\eta|} G_{\eta}(c(\eta)) + \sum_{\omega' \in F^i} (-1)^{|\omega'|} (B_0(\omega') - B_1(\omega')) \cdot f_{\le R}(c_0(\omega')) = 0
  \end{equation}
  on $T'$, where $G_\eta$ for $\eta \in \llbracket k \rrbracket$ are some functions $H \to \RR$ absorbing all terms of the form $f_{\le t-1}(c_0(\omega))$ and $f_{\le t-1}(c_1(\omega))$ for $\omega \in F_i$, as well as $g(c(\omega))$ for $\omega \in \llbracket k \rrbracket$, and whose precise forms are not important.

  Recall that $b$ is in normal form, and so $b_{\le R}(c_0,\omega')=b_{\le R}(c_1,\omega')=0$ whenever $|\omega| > R$.
  Moreover, by hypothesis $b_{>R}(c_0,\omega')=b_{>R}(c_1,\omega')$ for all $\omega' \in F^i$ and all $(c_0,c_1,c) \in T'$.
  Combining these two statements, if $\omega' \in F^i$ and $|\omega'| > R$ then $b(c_0,\omega')=b(c_1,\omega')$.
  It follows that $B_0(\omega')=B_1(\omega')$ whenever $\omega' \in F^i$, $|\omega'| > R$, and so~\eqref{eq:1st-sum} implies
  \begin{equation}
    \label{eq:2nd-sum}
    \sum_{\eta \in \llbracket k \rrbracket} (-1)^{|\eta|} G_{\eta}(c(\eta)) + \sum_{\substack{\omega' \in F^i \\ |\omega'| \le R}} (-1)^{|\omega'|} (B_0(\omega') - B_1(\omega')) \cdot f_{\le R}(c_0(\omega')) = 0
  \end{equation}
  on $T'$.
  Note the summand with $\omega'=\omega$ is exactly $v \cdot f_{\le R}(c_0(\omega))$.

  It therefore suffices to show that the system of linear forms $\ZZ^{k+2} \to \ZZ$ consisting of
  \[
    \phi_{\eta}(x,h_1,\dots,h_k,h_i') = x + \eta \cdot (h_1,\dots,h_k)
  \]
  for $\eta \in \llbracket k \rrbracket$, together with
  \[
    \psi_{\nu}(x,h_1,\dots,h_k,h_i') = x + h_i' + \nu \cdot (h_1,\dots,h_k)
  \]
  for $\nu \in F^i$ with $|\nu| \le R$, has Cauchy--Schwarz complexity at most $R-1$ at $\psi_{\omega}$, with denominator $(t-2)!$: the result then follows directly from Lemma~\ref{lem:gvn} applied to~\eqref{eq:2nd-sum}.    %chktex 40

  To check the Cauchy--Schwarz complexity we define $R$ sets of forms $\Sigma_r$ and corresponding vectors $\sigma_r \in \ZZ^{k+2}$ as in Definition~\ref{def:cauchy-schwarz-complexity}, as follows:---
  \begin{enumerate}[label=(\alph*)]
    \item $\Sigma_1 = \{ \phi_{\eta} \colon \eta \in \llbracket k \rrbracket \} $ and $\sigma_1 = (0,0,\dots,0,1)$ (i.e., $h'_i = 1$ and all other coordinates are zero);
    \item $\Sigma_2,\dots,\Sigma_{|\omega|}$ are chosen according to the following scheme: for each coordinate $j\ne i$ with $\omega(j) = 1$, we take a set
      \[
        \Sigma = \bigl\{ \psi_{\nu} \colon \nu \in F^i,\ |\nu| \le R,\ \nu(j) = 0\bigr\}
      \]
      and a vector $\sigma = (x,h_1,\dots,h_k,h'_i)$ where $h_j = 1$ and all other entries are zero; and
    \item for $|\omega|+1\le r \le R$ we take
      \[
        \Sigma_r = \bigl\{ \psi_{\nu} \colon \nu \in F^i,\ |\nu| = r \bigr\}
      \]
      and the vector $\sigma_r = (-r-1,1,1,\dots,1)$.
  \end{enumerate}
  Note that if $\nu \in F^i$, $|\nu| \le R$ and $\nu \ne \omega$, then either $|\nu| > |\omega$ (in which case $\psi_\nu$ is covered by one of the sets in case (c)) or there is some $j \ne i$ with $\nu(j)=0$ but $\omega(j)=1$ (in which case $\psi_\nu$ is covered by one of the sets in case (b)).
  Finally, $\psi_{\omega}(\sigma_r) \in \{1,-1,-2,\dots,-R+|\omega|\}$, which are all non-zero divisors of $(t-2)!$. %chktex 40
\end{proof}

\begin{proof}[Proof of Lemma~\ref{lem:no-cocycle-implies-1pc}]
  The strategy is very similar to the proof of the previous lemma.
  Recalling the hypothesis that $b(c_0,\omega') = b(c_1,\omega')$ for all $\omega' \in F^i$ and all triples $(c_1,c_1,c)$ in our set, we deduce from~\eqref{eq:really-trying-to-glue} that
  \begin{align*}
    \partial^k g(c) &= \sum_{\omega' \in \llbracket k \rrbracket} (-1)^{|\omega'|} \fb(c_0,c_1,c; \omega') \cdot f_{\le t-1}(c(\omega')) \\
    &= \sum_{\omega' \in \llbracket k \rrbracket} (-1)^{|\omega'|} b(c; \omega') \cdot f_{\le t-1}(c(\omega'))
  \end{align*}
  where
  \[
    \fb(c_0,c_1,c;\omega') = \begin{cases} b(c_0, \omega') &\colon \omega'(i) = 0 \\ b(c_1, \omega' \setminus \{i\}) &\colon \omega'(i) = 1  . \end{cases}
  \]
  By Lemma~\ref{lem:canonical-b} we may define $\fb'(c_0,c_1,c;-) \colon \llbracket k \rrbracket \to \bigoplus_{r=0}^{t-1} \ZZ^{d_i}$ such that
  \[
    \sum_{\omega' \in \llbracket k \rrbracket} (-1)^{|\omega'|} \fb(c_0,c_1,c; \omega') \cdot f_{\le t-1}(c(\omega'))
    =
    \sum_{\omega' \in \llbracket k \rrbracket} (-1)^{|\omega'|} \fb'(c_0,c_1,c; \omega') \cdot f_{\le t-1}(c(\omega'))
  \]
  and moreover $\fb'(c_0,c_1,c;-)$ is in normal form and $\|\fb'(c_0,c_1,c;\omega)\|_1 \ll_{k,t} M^{O_{k,t}(1)}$.
  By that lemma, we also get that
  \[
    \fb'_{=t-1}(c_0,c_1,c; \omega) = \sum_{\omega' \in \llbracket k \rrbracket} Z_{t-1}(\omega,\omega') \fb_{=t-1}(c_0,c_1,c;\omega')
  \]
  and hence $b_{=t-1}(c,\omega) - \fb'_{=t-1}(c_0,c_1,c;\omega) = z$ for all triples in our set.
  
  As before, we now pigeonhole the values of $b(c,\omega')$ and $\fb'(c_0,c_1,c;\omega')$: i.e., we define a set of triples $T' \subseteq T_{k,i}$ to consist of all $(c_0,c_1,c)$ which obey all the hypotheses from the statement together with $b(c,\omega') = B(\omega')$ and $\fb'(c_0,c_1,c; \omega')= B'(\omega')$ for each $\omega' \in \llbracket k \rrbracket$, for some choice of $B$ and $B'$ such that $|T'| \ge \eps\, O(M)^{-O_{k,t}(D)} |H|^{k+2}$
  (and $\|B\|_1,\|B'\|_1 \ll_{k,t} M^{O_{k,t}(1)}$).

  We conclude that
  \[
    \sum_{\omega' \in \llbracket k \rrbracket} (-1)^{|\omega'|} (B(\omega') - B'(\omega')) \cdot f(c(\omega')) = 0
  \]
  for all $(c_0,c_1,c) \in T'$; or rather at this point, for all $c \in C^k(H)$ which appear in at least one triple $(c_0,c_1,c) \in T'$, which constitute a subset of $C^k(H)$ with density at least $\eps\, O(M)^{-O_{k,t}(D)}$.  Note that because $b$ and $\fb'$ were in normal form, $B(\omega') = B'(\omega') = 0$ whenever $|\omega'| > t-1$.  In particular if $|\omega| > t-1$ then $z=0$ and there is nothing to prove, so we may assume $|\omega| \le t-1$.

  It therefore suffices to show that the system of linear forms $\ZZ^{k+1} \to \ZZ$ consisting of
  \[
    \psi_{\eta}(x,h_1,\dots,h_k) = x + \eta \cdot (h_1,\dots,h_k)
  \]
  for all $\eta \in \llbracket k \rrbracket$ with $|\eta| \le t-1$, has Cauchy--Schwarz complexity at most $t-2$ at index $\eta=\omega$, with denominator ${(t-1)!}$.
  If this holds, by Lemma~\ref{lem:gvn} we deduce that $(B(\omega) - B'(\omega)) \cdot f$ is an approximate polynomial with parameter $O(M)^{-O_{k,t}(D)} \eps^{O_{t}(1)}$, and since $(B(\omega) - B'(\omega))_{=t-1} = z$ this implies the result, setting $v= B(\omega)-B'(\omega)$.

  To verify this, we use a $t-1$ sets of linear forms similar to those in the proof of Lemma~\ref{lem:incompatible-implies-1pc}:---
  \begin{itemize}
    \item choose $\Sigma_1,\dots,\Sigma_{|\omega|}$ by considering each coordinate $j\ne i$ with $\omega(j) = 1$ and defining a set
      \[
        \Sigma = \big\{ \psi_{\eta} \colon |\nu| \le t-1,\ \eta(j) = 0\big\}
      \]
      and a vector $\sigma = (x,h_1,\dots,h_k)$ where $h_j = 1$ and all other entries are zero; and
    \item for $|\omega|+1\le r \le t-1$ we take
      \[
        \Sigma_r = \big\{ \psi_{\eta} \colon |\eta| = r \big\}
      \]
      and set $\sigma_r = (-r,1,1,\dots,1)$.
  \end{itemize}
  As in the proof of Lemma~\ref{lem:incompatible-implies-1pc}, these sets cover all forms $\psi_\eta$ with $\eta \in \llbracket k \rrbracket$, $|\eta| \le t-1$ and $\eta \ne \omega$. 
  Finally, we have $\psi_\omega(\sigma_r) \in \{ 1,-1,-2,\dots,-(t-1)-|\omega|\}$, which are all integers dividing ${(t-1)!}$.
\end{proof}

\subsection{Boosting approximate polynomials to derivatives conditions}

We now tackle Lemma~\ref{lem:1pc-implies-derivs} and Lemma~\ref{lem:derivs-implies-derivs}.

The proof of Lemma~\ref{lem:1pc-implies-derivs} is perhaps the most technically involved part of the paper.  There are some similarities between its proof and that of Lemma~\ref{lem:extend-global}.  We give a brief outline before starting on the details.

Recall we assume that $\phi$ obeys an $(f,t+1,t,M)$-derivatives condition almost everywhere, and is an approximate polynomial of degree $t-1$ and parameter $\delta>0$.  We can interpret the second statement as saying that $\phi$ obeys a $(f,t,t,0)$-derivatives condition, but on a small set of cubes of density $\delta$.

The basic approach is to grow the $t$-derivatives condition from this small set of cubes to one on progressively larger sets; or, find along the way that part (B) of the dichotomy holds and stop.  In doing so, we use in an essential way that we already have a global $(t+1)$-derivatives condition in play.

There are two basic mechanisms for ``growing'' the set of cubes.  One is to observe, as we essentially already did in~\eqref{eq:really-trying-to-glue}, that if $c_0,c_1 \in C^t(H)$ have the same upper face $c_0|_{F^i}=c_1|_{F^i}$, and $t$-derivatives conditions holds at both $c_0$ and $c_1$ and are compatible on the shared upper face, then we can deduce a derivatives condition on the cube $c=[c_0|_{F_i},c_1|_{F_i}]_i$ obtained by glueing $c_0$ and $c_1$ along their shared upper face.

The other is to note that if $c_0 \in C^t(H)$, $h \in H$ and $c_1 = c_0 + \square^t(h)$ is a translate of $c_0$ on which a $t$-derivatives condition holds, and $[c_0,c_1]_{t+1} \in C^{t+1}(H)$ has a $(t+1)$-derivatives condition, then assuming again some kind of compatibility on $c_1$ we can deduce a $t$-derivatives condition on $c_0$.

The conclusion of applying these two arguments is that we can assume the set of cubes at which a $t$-derivatives condition holds is---roughly---closed under glueing and under translation.  We would like to deduce that such a set of cubes is close to being all of $C^t(H)$, and this turns out to be true, unless the process gets stuck inside a large subgroup of $H^t$ in some sense.  To locate such a subgroup we use a standard result on sets of very large additive energy, due to Fournier~\cite{fournier}.  Finally, we can use our hypothesis that $H$ has no large subgroups to rule out this case.

This approach is complicated in practice in several ways.  One is that the ambient system of cubes $S_0,\dots,S_{t+1}$ has parameter $1-\eps$ where $\eps$ is small in absolute terms but large compared to the other parameters $\delta,\eta$, meaning we cannot afford to lose sets of density $\eps$ at any stage and must work around that.  Another is the need to address the compatibility statements alluded to vaguely above. This is handled by part (B) and Cauchy--Schwarz arguments (old or new), but we still only get to assume compatibility holds almost all the time, which adds another layer of small losses and parameters.

\begin{proof}[Proof of Lemma~\ref{lem:1pc-implies-derivs}]
  Because it will come up a lot, we say \emph{option (B) holds with parameters $K$, $\alpha$} if (as in (B) in the statement) for some $R$, $1 \le R \le t-1$, and some $v \in \bigoplus_{r=1}^{R} \ZZ^{d_r}$ with $v_{=R} \ne 0$ and $\|v\|_1 \le K$, we have
  \[
    \big|\big\{c \in C^{R}(H) \colon \partial^R (v \cdot f_{\le R})(c) = 0 \big\}\big| \ge \alpha |C^R(H)|.
  \]
  If this holds for parameters $K \ll_t (M/\delta)^{O_t(1)}$ and $\alpha \ge O(M/\delta)^{-O_t(D)} \eta^{O_t(1)}$, we will be done; however, we also use this terminology in cases where it is not yet obvious that $K,\alpha$ obey these bounds.
  In the case $t=1$ of the lemma, this statement should be taken to be false whenever it appears.

  As stated above, the first key argument is that, given a derivatives condition on a set of cubes $S$, we can deduce a derivatives condition on the set of cubes obtained by glueing cubes in $S$ together along common faces (or deduce that option (B) holds).  The mechanism was discussed in~\eqref{eq:really-trying-to-glue}, assuming some kind of upper compatibility.  We now state a result of this form precisely.

  \begin{definition}
    If $S \subseteq C^k(H)$, $i \in [k]$ and $\gamma > 0$, write $S +_{i,\gamma} S$ for the set of all $c \in C^k(H)$ such that
    \[
      \big|\big\{ (c_0,c_1,c') \in T_{k,i} \colon c = c',\ c_0 \in S, c_1 \in S \big\}\big| \ge \gamma |H|  ;
    \]
    i.e., all $c \in C^k(H)$ which can be obtained by glueing together two elements of $S$ along a common upper face in direction $i$, in many ways.
  \end{definition}

  We note that if $S' \subseteq S$ and $\gamma > \gamma'$ then
  \begin{equation}
    \label{eq:concat-loss}
    |(S +_{i,\gamma} S) \setminus (S' +_{i,\gamma'} S')| \le \frac{ 2 |S \setminus S'| }{\gamma - \gamma'}
  \end{equation}
  since each element $c$ of the left hand set accounts for at least $(\gamma - \gamma')|H|$ triples $(c_0,c_1,c) \in T_{k,i}$ for which either $c_0 \in S \setminus S'$ or $c_1 \in S \setminus S'$.

  \begin{claim}%
    \label{claim:stitch-derivs}
    Let $k \ge 1$ and $i \in [k]$, suppose $\phi$ and $b$ obey an $(f,k,t,m)$-derivatives condition on a set $S \subseteq C^k(H)$, and let $\nu, \gamma > 0$ be parameters.  Then either $\phi$ obeys an $(f,k,t,O_{k,t}(m))$-derivatives condition on a subset $S' \subseteq S +_{i,\gamma} S$ with $|(S +_{i,\gamma} S) \setminus S'| \le \nu |C^k(H)|$; or, option (B) holds with parameters
    $O(m)^{O_{k,t}(1)}$ and $O(m)^{-O_{k,t}(D)} (\nu \gamma)^{O_{k,t}(1)}$.
  \end{claim}
  \begin{proof}[Proof of claim]
    By Lemma~\ref{lem:canonical-b}, we are free to assume $b$ is in normal form, at the expense of replacing $m$ by $O_{k,t}(m)$.
    
    By Corollary~\ref{cor:incompatible-implies-1pc} we have that either option (B) holds with the parameter stated, or we may assume $b$ is $(\nu \gamma)$-almost upper compatible.

    Define $S'$ to consist of those $c \in S +_{i,\gamma} S$ such that at least one triple $(c_0,c_1,c) \in T_{k,i}$ with $c_0,c_1 \in S$ has $b(c_0,\omega) = b(c_1,\omega)$ for all $\omega \in F^i$ (i.e., an upper compatible triple).  So, for each $c$ in $S +_{i,\gamma} S$ but not in $S'$ there are at least $\gamma |H|$ distinct incompatible triples by hypothesis; hence,
    \[
      \gamma |H|\, |(S +_{i,\gamma} S) \setminus S'| \le \nu \gamma\, |H|^{k+2}
    \]
    and the bound on $S'$ follows.

    Finally, if $c \in S'$ and $(c_0,c_1,c)$ is any compatible triple with $c_0,c_1 \in S$, then~\eqref{eq:really-trying-to-glue} implies a derivatives condition at $c$ as required.
  \end{proof}

  By hypothesis, $\phi$ obeys an underlying $(f,t+1,t,M)$-derivatives condition on $S_{t+1}$.  Although $S_{t+1}$ contains almost all cubes with vertices in $S_0$, as measured by $\eps$, it is likely that $\eps$ is much larger than $\delta$ or $\eta$, and so this notion of ``almost all'' is quite weak.
  Our first use for Claim~\ref{claim:stitch-derivs} is to boost this $(f,t+1,t,M)$-derivatives condition to a larger set of cubes (although, throughout, the underlying the vertex set $S_0 \subseteq H$ does not grow).
  In fact we state a result in slightly greater generality.
  \begin{claim}%
    \label{claim:augment-derivs}
    Let $k \ge 1$, $S'_0,\dots,S'_{k}$ be a system of cubes with parameter $\eps'$, where $\eps' \le \eps'_0(k,t)$ is sufficiently small, and suppose $\phi$ obeys an $(f,k,t,m)$-derivatives condition on $S'_k$.
    Also let $\nu > 0$ be a further parameter.
    
    Then either there exists a set $S \subseteq C^{k}(H) \cap S_0'^{\llbracket k \rrbracket}$ such that
    \[
      |(C^k(H) \cap S_0'^{\llbracket k \rrbracket}) \setminus S| \le \nu |C^{k}(H)|
    \]
    and such that $\phi$ obeys an $\big(f,k,t,O_{k,t}(m)\big)$-derivatives condition on $S$; or, option (B) holds with parameters $O(m)^{O_{k,t}(1)}$ and $O(m)^{-O_{k,t}(D)} \nu^{O_{k,t}(1)}$.
  \end{claim}
  \begin{proof}[Proof of claim]
    We apply Claim~\ref{claim:stitch-derivs} for each $i \in [k]$ in turn, with parameters $\nu'$ and $\gamma'$ to be determined.  Specifically, define a sequence of sets $S^{(i)} \subseteq C^{k}(H)$ for $0 \le i \le k$ by taking $S^{(0)} = S'_{k}$ and setting $S^{(i)}$ to be the subset of $S^{(i-1)} +_{i,\gamma'} S^{(i-1)}$ produced by applying the claim.

    Also define sets $Z^{(i)}$ for $0 \le i \le k$ to consist of all $c \in C^{k}(H)$ such that $c|_F \in S'_{k-i}$, for each of the $2^{i}$ faces of codimension $i$ with active coordinates $[i]$; i.e., for all faces
    \[
      F = \{\omega \in \llbracket k \rrbracket \colon \omega_{j} = \sigma_j \ \forall j \in [i]\}  ,
    \]
    where $\sigma \in \llbracket i \rrbracket$. So, $Z^{(0)} = S'_{k}$, and $Z^{(k)}$ consists of all $c \in C^{k}(H)$ whose vertices $c(\omega)$ all lie in $S'_0$.  Then we claim that for each $i$, $1 \le i \le k$,
    \[
      Z^{(i)} = Z^{(i-1)} +_{i,1-2^i \eps'} Z^{(i-1)} .
    \]
    The inclusion $\supseteq$ is immediate.  Conversely, suppose $c \in S^{(i)}$ and $F,F'$ are a pair of opposite codimension $i$ faces in direction $i$,
    \begin{align*}
      F &= \{ \omega \in \llbracket k \rrbracket \colon \omega_{j} = \sigma'_j \ \forall j \in [i-1],\ \omega(i) = 0 \} \\
      F' &= \{ \omega \in \llbracket k \rrbracket \colon \omega_{j} = \sigma'_j \ \forall j \in [i-1],\ \omega(i) = 1 \}
    \end{align*}
    where $\sigma' \in \llbracket i-1 \rrbracket$; so, $c|_F$ and $c|_{F'}$ are in $S_{k-i}$ by hypothesis. Then by definition of a system of cubes there are at least $(1-2\eps')|H|$ values $h \in H$ such that, writing $c' = c|_F + \square^{k-i}(h)$, both $[c|_F, c']_i$ and $[c|_{F'}, c']_i$ are cubes in $S_{k+1-i}$.  Taking a union bound over all $2^{i-1}$ such pairs $(F,F')$, there are at least $(1-2^i \eps')|H|$ elements $h \in H$ such that $[c|_{F_i}, c|_{F_i} + \square^s(h)] \in Z^{(i-1)}$ and $[c|_{F^i}, c|_{F_i} + \square^s(h)] \in Z^{(i-1)}$, and hence $c \in Z^{(i-1)} +_{i,1-2^i \eps'} Z^{(i-1)}$ as required.

    By the definition of $S^{(i)}$ and~\eqref{eq:concat-loss} we have that for each $i$, $1 \le i \le k$,
    \begin{align*}
      \big|Z^{(i)} \setminus S^{(i)}\big| &= \big|(Z^{(i-1)} +_{i,1-2^i \eps'} Z^{(i-1)}) \setminus (S^{(i-1)} +_{i,\gamma'} S^{(i-1)})\big| + \big| (S^{(i-1)} +_{i,\gamma'} S^{(i-1)}) \setminus S^{(i)}\big| \\
      &\le \frac2{1-2^i\eps' - \gamma'} \big|Z^{(i-1)} \setminus S^{(i-1)}\big| + \nu' |C^{s+1}(H)| .
    \end{align*}
    We may assume $\eps' \le 2^{-k-2}$, and set $\gamma' = 1/4$ and $\nu' = 4^{-k} \nu$. Since $Z^{(0)} = S^{(0)}$, the resulting bound on $|Z^{(k)} \setminus S^{(k)}|$ proves the claim.
  \end{proof}

  For the remainder of the argument, we assume this claim has been applied with $k=t+1$, $m=M$, the system $S_0,\dots,S_{t+1}$ with parameter $\eps$ from the statement, and some $\nu_0 > 0$ to be determined.
  We rename the set of cubes produced this way as $Y \subseteq C^{t+1}(H)$.
  So, either $\phi$ obeys the $\big(f,t+1,t,O_t(M)\big)$-derivatives condition on $Y$ and $\big|(C^{t+1}(H) \cap S_0^{\llbracket t+1 \rrbracket}) \setminus Y\big| \le \nu_0 |C^{t+1}(H)|$; or, option (B) holds with parameters $O(M)^{O_t(1)}$ and $O(M)^{-O_{t}(D)} \nu_0^{O_{t}(1)}$.
  We may assume the former, provided we check at the end that our final choice of $\nu_0$ is not too small for part (B) of the statement.

  Our other key component in expanding the derivatives condition on subsets of $C^t(H)$ is to pass from a derivatives condition on $c \in C^t(H)$ to a derivatives condition on translates $c + \square^t(h)$ of $c$.

  \begin{claim}%
    \label{claim:translate}
    Suppose $\phi$ and $b$ obey an $(f,t,t,m)$-derivatives condition on some set $S \subseteq C^t(H) \cap S_0^{\llbracket t \rrbracket}$, and let $\alpha > 0$ be a further parameter.
    Consider the set
    \begin{equation}
      \label{eq:def-w}
      W = \big\{ c \in C^t(H) \cap S_0^{\llbracket t \rrbracket}  \colon |\{ h \in H \colon c + \square^t(h) \in S \}| \ge \alpha |H| \big\} ;
    \end{equation}
    i.e., those cubes which have many translates in $S$.  Then either there exists a subset $W' \subseteq W$ such that $|W \setminus W'| \le (3 \nu_0 / \alpha) |C^t(H)|$ and $\phi$ obeys an $\big(f,t,t,O_t(m+M)\big)$-derivatives condition on $W'$; or, option (B) holds with parameters $O(m+M)^{O_t(1)}$ and $O(m + M)^{-O_{t}(D)} \nu_0^{O_t(1)}$.
  \end{claim}
  \begin{proof}[Proof of claim]
    We pick functions $b \colon Y \times \llbracket t+1 \rrbracket \to \bigoplus_{i=0}^{t-1} \ZZ^{d_i}$ and $b' \colon S \times \llbracket t \rrbracket \to \bigoplus_{i=0}^{t-1} \ZZ^{d_i}$ such that $\phi$ and $b$ obey the $(f,t+1,t,O_t(M))$-derivatives condition on $Y$, and $\phi$ and $b'$ obey the $(f,t,t,m)$-derivatives condition on $S$, as hypothesized.

    First we set
    \[
      W'' = \big\{ c \in C^t(H) \cap S_0^{\llbracket t \rrbracket} \colon |\{h \in H \colon c+\square^t(h) \in S,\ [c, c+\square^t(h)]_{t+1} \in Y \}| \ge \alpha / 2\big\} .
    \]
    Each $c \in W \setminus W''$ contributes at least $(\alpha/2) |H|$ elements to $\big(C^{t+1}(H) \cap S_0^{\llbracket t+1 \rrbracket}\big) \setminus Y$, and therefore
    \[
      |W \setminus W''| \le (2/\alpha) \nu_0 |C^t(H)| .
    \]

    Now suppose $c_0 \in W''$ and take any $h \in H$ such that $c_1 = c_0 + \square^t(h) \in S$ and $c = [c_0, c_1]_{t+1} \in Y$, as in the definition.  We have
    \begin{align}
      \nonumber
      \partial^t \phi(c_0) &= \partial^{t+1} \phi(c) + \partial^t \phi(c_1) \\
      \nonumber
      &= \sum_{\omega \in \llbracket t+1 \rrbracket} (-1)^{|\omega|} b(c, \omega) \cdot f(c(\omega))+ \sum_{\omega' \in \llbracket t \rrbracket} (-1)^{|\omega'|} b'(c_1, \omega') \cdot f(c_1(\omega')) \\
      &= \sum_{\omega \in \llbracket t+1 \rrbracket} (-1)^{|\omega|} \fb(c, \omega) \cdot f(c(\omega))
      \label{eq:nice-deriv}
    \end{align}
    where $\omega \mapsto \fb(c,\omega)$ is the configuration obtained by putting
    \[
      \omega \mapsto \begin{cases} 
        b(c,\omega) &\colon \omega(t+1) = 0 \\
        b(c,\omega) - b'(c_1,(\omega_1,\dots,\omega_t)) &\colon \omega(t+1) = 1
      \end{cases}
    \]
    into normal form (by Lemma~\ref{lem:canonical-b}).

    If $\fb(c,\omega)$ has the property that $\fb(c, \omega) = 0$ for all $\omega \in F^{t+1}$, then
    \[
      \partial^t \phi(c_0) = \sum_{\omega \in \llbracket t \rrbracket} (-1)^{|\omega|} \fb(c_0,\omega) \cdot f(c_0(\omega))
    \]
    and so an $\big(f,t,t,O_t(m+M)\big)$-derivatives condition holds at $c_0$.
    Also note that this conclusion holds if the hypothesis on $\fb(c,\omega)$ is met for any value of $h$.
    Finally, if $t=1$ then this property always holds, as in this case since $\fb$ is in normal form we have $\fb(c,\omega)=0$ for any $\omega \ne 0$.

    So, we define $W' \subseteq W''$ to consist of all $c_0 \in W''$ such that, in the above notation, there is at least one $h \in H$ and corresponding $c_1$, $c$ and $\fb$ with the properties above such that $\fb(c, \omega) = 0$ for all $\omega \in F^{t+1}$.  Hence, if $|W'' \setminus W'| \le (\nu_0 / \alpha) |C^t(H)|$, the set $W'$ has all the desired properties and we are done.

    If on the contrary $|W'' \setminus W'| \ge (\nu_0 / \alpha) |C^t(H)|$, we must show that option (B) holds with the indicated parameters.
    To do this, we again use a Cauchy--Schwarz argument, sufficiently similar to the proof of Lemma~\ref{lem:no-cocycle-implies-1pc} that we will describe it fairly concisely.

    Note first that for each $c_0 \in W'' \setminus W'$ there are at least $\alpha/2$ values $h \in H$ such that $c_1 = c_0 + \square^t(h) \in S$, $c = [c_0,c_1]_{t+1} \in Y$ and, in the above notation, $\fb(c,\omega) \ne 0$ for at least one $\omega \in F^{t+1}$.
    Hence there at least $(\nu_0/2) |C^{t+1}(H)|$ such cubes $c$ (such that $c \in Y$, $c=[c_0,c_1]_{t+1}$ where $c_0 \in W''$, and $\fb(c,\omega) \ne 0$ for at least one $\omega \in F^{t+1}$) in total.
    
    By pigeonholing, there exists some $R$, $1 \le R \le t-1$, some $\omega \in \llbracket t+1 \rrbracket$ with $\omega(t+1) = 1$, and $\gg_t \nu_0 |C^{t+1}(H)|$ cubes $c \in C^{t+1}(H)$ as above, such that $\fb_{=R}(c, \omega) \ne 0$ but $\fb_{>R}(c,\omega') = 0$ for all $\omega' \in F^{t+1}$.
    Since $\fb$ is in normal form and $\fb_{=R}(c,\omega) \ne 0$, necessarily $|\omega| \le R$.
    Similarly, if $\omega' \in F^{t+1}$ and $|\omega'| > R$ then $\fb_{>R}(c,\omega')=0$ (by hypothesis) and $\fb_{\le R}(c,\omega')=0$ (by normal form) so more simply $\fb(c,\omega')=0$ for such $\omega'$.
    
    By further pigeonholing to make the values $\fb(c,-)$ constant, we may pick $B \colon \llbracket t+1 \rrbracket \to \bigoplus_{i=0}^{t-1} \ZZ^{d_i}$ such that $O(M+m)^{-O_t(D)} \nu_0 |C^{t+1}(H)|$ of these cubes $c$ have $\fb(c,\omega') = B(\omega')$ for all $\omega' \in \llbracket t+1 \rrbracket$.
    For these $c=[c_0,c_1]_{t+1}$, noting from the discussion above that $B(\omega')=0$ if $\omega' \in F^{t+1}$ and $|\omega'|>R$, by~\eqref{eq:nice-deriv} we have
    \[
      - \partial^t \phi(c_0) + \sum_{\substack{\omega' \in \llbracket t+1 \rrbracket \\ \omega'(t+1) = 0 \text{ or } |\omega'| \le R}} (-1)^{|\omega'|} B(\omega') \cdot f(c(\omega')) = 0.
    \]
    Hence, Lemma~\ref{lem:gvn} will show that option (B) holds with parameters $O(M+m)^{O_t(1)}$ and 
      $O(M+m)^{-O_t(D)} \nu_0^{O_t(1)}$
    as required; provided the system of linear forms $\ZZ^{t+2} \to \ZZ$ consisting of
    \[
      \phi_\eta(x,h_1,\dots,h_{t+1}) = x + \eta \cdot (h_1,\dots,h_{t+1})
    \]
    for $\eta \in \llbracket t+1 \rrbracket$ such that $\eta(t+1) = 0$ or $|\eta| \le R$, has Cauchy--Schwarz complexity at most $R$ at $\eta = \omega$ with denominator $(t-1)!$.  Indeed, much as before we can take:--- %chktex 40
    \begin{itemize}
      \item sets $\Sigma_1,\dots,\Sigma_{|\omega|}$ given for each index  $j \in [t+1]$ with $\omega(j) = 1$ by
        \[
          \Sigma = \big\{ \phi_\eta \colon \eta(j) = 0 \big\}
        \]
        and with corresponding vector $\sigma = (0,\dots,0,1,0,\dots,0)$ (i.e., $\sigma_j = 1$ and all other entries are zero); and
      \item $\Sigma_{|\omega|+1,\dots,\Sigma_{|R|}}$ are given by $\Sigma_i = \{ \phi_\eta \colon |\eta| = i\}$, with $\sigma_i = (-i,1,1,\dots,1)$;
    \end{itemize}
    which have the required properties.
  \end{proof}

  Combining translation (Claim~\ref{claim:translate}) and glueing (Claim~\ref{claim:stitch-derivs}), we can almost always pass from a derivatives condition on $S \subseteq C^t(H)$ to one on a larger set.

  \begin{claim}%
    \label{claim:anything-grows}
    For any constant $c_1>0$, there exist $c_2,c_3 \gg_{t,c_1} 1$ such that the following holds.  Let $S_0,\dots,S_{t+1}$ be the system of cubes of parameter $1-\eps$ from the statement, and let $S \subseteq C^t(H) \cap S_0^{\llbracket t \rrbracket}$ be any set of cubes of size $\delta' |C^t(H)|$.  There exists a parameter $\gamma \gg_{t,c_1} \delta'^2$, such that if we define
    \[
      W = \big\{ c \in C^t(H) \cap S_0^{\llbracket t \rrbracket}  \colon |\{ h \in H \colon c + \square^t(h) \in S \}| \ge (\delta'/2) |H| \big\}
    \]
    as above, then one of the following occurs:---
    \begin{enumerate}[label=(\Roman*)]
      \item we get a size increase, in that for some $i \in [t]$ we have $|W +_{i,\gamma} W| \ge \big(1 + c_2 - O_t(\eps) \big) |S|$;
      \item $H$ has a proper subgroup of order at least $c_3 \delta' |H|$; or
      \item the set $S$ was almost everything to start with, in that $\delta' \ge 1 - c_1$.
    \end{enumerate}
  \end{claim}
  \begin{proof}[Proof of claim]
    We write
    \begin{equation}
      \label{eq:v-def}
      V = \big\{ (h_1,\dots,h_t) \in H^t \colon |\{ x \in H \colon \angle(x;h_1,\dots,h_t) \in S \}| \ge (\delta'/2) |H| \big\}
    \end{equation}
    which means that
    \[
      W = \big\{ \angle(x;h_1,\dots,h_t) \in S_0^{\llbracket t \rrbracket} \colon (h_1,\dots,h_t) \in V \big\}
    \]
    by definition.  We record the estimate
    \begin{equation}
      \label{eq:v-bound}
      |S| = \sum_{(h_1,\dots,h_t) \in H^t} \big|\big\{ x \in H \colon \angle(x;h_1,\dots,h_t) \in S \big\}\big|  \le |V|\,|H| + (\delta'/2) |H|^{t+1}
    \end{equation}
    which implies $|V| \ge (\delta'/2) |H|^t = |S|/2|H|$.
    
    For each $i \in [t]$ and each choice $\vec h = (h_1,\dots,h_{i-1},h_{i+1},\dots,h_t) \in H^{t-1}$, we define the fiber $V_{\vec h, i} = \big\{ h_i \in H \colon (h_1,\dots,h_i,\dots,h_t) \in V\big\}$, as well as a directed graph (with self-loops) with vertex set
    \[
      X = \big\{ x \in H \colon \angle(x;h_1,\dots,h_{i-1},h_{i+1},\dots,h_t) \in S_0^{\llbracket t-1 \rrbracket} \big\}
    \]
    and edge set
    \[
      \cE = \big\{ (x,y) \in X \times X \colon y-x \in V_{\vec h,i} \big\}  .
    \]
    It follows that $\angle(x;h_1,\dots,h_t) \in W$ if and only if $x \in X$, $x+h_i \in X$ and $(x,x+h_i) \in \cE$.  Moreover, $\angle(x;h_1,\dots,h_t) \in W +_{i,\gamma} W$ if and only if $x$ and $y=x+h_i$ are in $X$ and there are at least $\gamma |H|$ elements $z \in X$ such that $xz, yz \in \cE$.

    We note that $|X| \ge (1-2^{t-1} \eps) |H|$: as $x$ ranges over $H$, each vertex of $\angle(x; h_1,\dots,h_{i-1},h_{i+1},\dots,h_t)$ is excluded from $S_0$ with probability $\eps$, and we take a union bound.
    
    Similarly, $|\cE| \ge |V_{\vec h, i}| (2|X| - |H|)$, since each pair $h \in V_{\vec h, i}$ and $x \in H$ yields an edge $(x,x+h) \in \cE$ unless either $x \in H \setminus X$ or $x+h \in H \setminus X$, each of which occurs with probability exactly $1-|X|/|H|$ as $x$ varies over $H$.  Combining these estimates, we deduce $|\cE| \ge |V_{\vec h,i}| |H| (1 - 2^t \eps)$.

    Given $x,y \in X$, write $r(x,y) = |\{ z \in X \colon xz,yz \in \cE\}|$.  Note
    \begin{align*}
      \big|\big\{ (x,y) \in X \times X \colon r(x,y) \ge \gamma |H|\big\}\big|
      &\ge \frac{\left(\sum_{(x,y) \in X \times X} r(x,y)\, [ r(x,y) \ge \gamma |H| ]\right)^2}
      {\sum_{(x,y) \in X \times X} r(x,y)^2\, [ r(x,y) \ge \gamma |H| ]}
    \end{align*}
    by Cauchy--Schwarz.  Bounding based on the cases $r(x,y) < \gamma |H|$ and $r(x,y) \ge \gamma |H|$, we have
    \[
      \sum_{(x,y) \in X \times X} r(x,y) \le \gamma |H| |X|^2 + \sum_{(x,y) \in X \times X} r(x,y)\, [ r(x,y) \ge \gamma |H| ]
    \]
    and conversely by double-counting and Cauchy--Schwarz we have the lower bound
    \[
      \sum_{(x,y) \in X \times X} r(x,y) = \sum_{z \in X} |\{x \in X \colon xz \in \cE\}|^2 \ge |\cE|^2 / |X| .
    \]
    Combining these, and using the lower bound on $|\cE|$ above, gives
    \[
      \sum_{(x,y) \in X \times X} r(x,y)\, [r(x,y) \ge \gamma |H|] \ge |\cE|^2 / |X| - \gamma |H| |X|^2 \ge |V_{\vec h, i}|^2 |H| (1 - 2^{t+1} \eps) - \gamma |H|^3 .
    \]
    Meanwhile, the second moment of $r(x,y)$ can be controlled in terms of the additive energy of $V_{\vec h,i}$:
    \begin{align*}
      \sum_{(x,y) \in X \times X} r(x,y)^2 &= 
      \big|\big\{(x,y,z,z') \in X^4 \colon xz, xz', yz, yz' \in \cE\big\}\big| \\
      &\le \big|\big\{ x\in X,\ (y,z,z') \in H^3 \colon z-x, z'-x, z-y, z'-y \in V_{\vec h, i} \big\}\big| \\
      &= |X|\, E(V_{\vec h, i})
    \end{align*}
    where $E(A)$ denotes the additive energy $|\{ (a,a',b,b') \in A^4 \colon a-a'=b-b'\}|$ of a set $A \subseteq H$.

    Putting this together, we deduce that
    \begin{align}
      \label{eq:local-triples-lower}
      \nonumber
      \big|\big\{ (x,y) &\in X \times X \colon r(x,y) \ge \gamma |H|\big\}\big|
      \ge \frac{\big( |V_{\vec h, i}|^2 |H| (1 - 2^{t+1} \eps) - \gamma |H|^3 \big)^2}{|H| E(V_{\vec h, i})} \\
      &= |H|\,|V_{\vec h,i}|\ \big(E(V_{\vec h,i}) / |V_{\vec h,i}|^3 \big)^{-1} \big( 1 - 2^{t+1} \eps - \gamma |H|^2 / |V_{\vec h, i}|^2\ \big)^2 .
    \end{align}
      
    We now consider some cases based on the additive energy of $V_{\vec h, i}$.  Let $\beta > 0$ be a further parameter.  If $E(V_{\vec h,i }) \ge (1 - \beta) |V_{\vec h,i }|^3$, a well-known argument of Fournier (\cite{fournier}, or see e.g.\ \cite{green-barbados}) shows that there is some subgroup $H' \le H$ and a coset $x + H'$ such that
    \[
      \big|(x + H') \cap V_{\vec h, i}\big| \ge \big(1 - 10 \beta^{1/2}\big) \max(|H'|, |V_{\vec h, i}|) .
    \]
    In particular, $|H'| \ge (1 - 10 \beta^{1/2}) |V_{\vec h, i}|$; but then either $H'=H$; or $H'$ is a proper subgroup of order at least $c_3 \delta' |H|$ and (II) holds, which we may assume does not happen; or $|V_{\vec h, i}| \le (1 - 10\beta^{1/2})^{-1} c_3 \delta' |H|$ must be rather small.  On the other hand, when $H' = H$ we have that $|V_{\vec h, i}| \ge \big(1 - 10 \beta^{1/2}\big) |H|$, i.e.\ $V_{\vec h,i}$ is almost all of $H$.

    Accordingly, for each $i \in [t]$ we distinguish three kinds of $\vec h \in H^{t-1}$:
    \begin{align*}
      L_{\operatorname{sml},i} &= \big\{ \vec h \in H^{t-1} \colon |V_{\vec h, i}| \le 2 \delta' c_3 |H| \big\} , \\
      L_{\operatorname{big},i} &= \big\{ \vec h \in H^{t-1} \colon |V_{\vec h, i}| \ge (1 - \rho)|H| \big\}  \setminus L_{\operatorname{sml},i} , \\
      L_{\operatorname{med},i} &= \big\{ \vec h \in H^{t-1} \colon E(V_{\vec h, i}) \le (1 - \beta) |V_{\vec h, i}|^3 \big\} \setminus L_{\operatorname{sml},i} \setminus L_{\operatorname{big},i}  ;
    \end{align*}
    where $\rho>0$ is some further parameter to be determined.  Provided we choose $\beta < \min(1/2, \rho^2)/100$ and so $(1 - 10 \beta^{1/2})^{-1} \le 2$ and $1 - 10 \beta^{1/2} \ge 1-\rho$, the discussion above tells us exactly that these three sets partition $H^{t-1}$.
    
    Note also that whenever $\vec h \notin L_{\operatorname{sml},i}$ we may modify~\eqref{eq:local-triples-lower} to 
    \begin{equation}
      \label{eq:better-triples-local}
      \big|\big\{ (x,y) \in X \times X \colon r(x,y) \ge \gamma |H| \big\}\big| \ge |H|\,|V_{\vec h,i}|\ \big(E(V_{\vec h,i}) / |V_{\vec h,i}|^3 \big)^{-1} \big(1 - 2^{t+1} \eps - \gamma / 2 \delta'^2 c_3^2 \big) .
    \end{equation}
    We can bound the total number of $v \in V$ that lie over fibers $L_{\operatorname{sml},i}$ by
    \[
      \sum_{\vec h \in L_{\operatorname{sml},i}} |V_{\vec h,i}| \le 2 \delta' c_3 |H|^s \le 4 c_3\, |V|
    \]
    where we used the lower bound on $|V|$ above.
    Also, we denote the proportion of $v \in V$ that lie over fibers $L_{\operatorname{med},i}$ by
    \[
      \sigma_i = \frac1{|V|} \sum_{\vec h \in L_{\operatorname{med},i}} |V_{\vec h,i}| .
    \]

    We now bound $|W +_{i,\gamma} W|$ below by summing~\eqref{eq:better-triples-local} over $\vec h \in H^{t-1}$, giving
    \begin{align}
      \nonumber
      |W +_{i,\gamma} W| &\ge \begin{aligned}[t] & \sum_{\vec h \in L_{\operatorname{med},i}} |H|\,|V_{\vec h,i}|\ (1 - \beta)^{-1} \big(1 - 2^{t+1} \eps - \gamma / 2 \delta'^2 c_3^2  \big) \\
        + & \sum_{\vec h \in L_{\operatorname{big},i}} |H|\,|V_{\vec h,i}|\ \big(1 - 2^{t+1} \eps  - \gamma / 2 \delta'^2 c_3^2 \big) 
      \end{aligned} \\
      \label{eq:best-triples}
      &\ge |H|\, |V|\, \big(1 - 2^{t+1} \eps - \gamma / 2 \delta'^2 c_3^2 - 4 c_3 \big) (1 + \beta \sigma_i) \ .
    \end{align}
    We would like to deduce the size increase (I).  Recalling from above that $|H|\,|V| \ge (1 - \alpha/\delta')|S|$, and since clearly $\sigma_i \ge 0$ for each $i$, the only way (I) can fail to happen is if both of these inequalities are nearly tight (up to small multiplicative factors), in which case we will show that $\delta' \approx 1$ and so (III) holds.

    For this last part, consider the projections $\pi_i \colon H^t \to H^{t-1}$ given by
    \[
      \pi_i(h_1,\dots,h_t) = (h_1,\dots,h_{i-1},h_{i+1},h_s)
    \]
    for each $i \in [t]$, and define
    \[
      B = \big\{ v \in V \colon \pi_i(v) \in L_{\operatorname{big},i} \text{ for all } i \in [t] \big\} .
    \]
    Since any elements $v \in V \setminus B$ must have either $\pi_i(v) \in L_{\operatorname{med},i}$ or $\pi_i(v) \in L_{\operatorname{sml},i}$ for some $i \in [t]$, a union bound gives
    \[
      |B|/|V| \ge 1 - \sum_{i=1}^t \sigma_i - 4 t c_3 .
    \]
    It is also clear that $|L_{\operatorname{big},i}| \le (1-\rho)^{-1} |V| / |H|$ for each $i \in [t]$, as every $\vec h \in L_{\operatorname{big},i}$ accounts for at least $(1-\rho) |H|$ elements of $V$.   Also, the image $\pi_i(B)$ is a subset of $L_{\operatorname{big},i}$ for each $i \in [t]$, so by the Loomis--Whitney inequality we deduce
    \[
      |B| \le \prod_{i=1}^t |\pi_i(B)|^{1/(t-1)} \le (1-\rho)^{-t/(t-1)} (|V|/|H|)^{t/(t-1)}
    \]
    and rearranging gives
    \begin{equation}
      \label{eq:v-is-large}
      |V| \ge (|B|/|V|)^{t-1} (1 - \rho)^{t} |H|^t \ge |H|^t \left(1 - t \sum_{i=1}^t \sigma_i - t \rho - 4 t^2 c_3 \right)
    \end{equation}
    i.e.\ if $\sigma_i \approx 0$ for each $i$ then $V$ is almost all of $H^t$.  If furthermore $|V| \approx |S|/|H|$ then this will imply $\delta' \approx 1$ as required, in a way to be made precise.

    It remains only to bolt the quantitative statements together and pick values for all the parameters.
    We are free to assume $c_1 \le 1$, or else (III) holds vacuously.

    We start by setting $\rho = c_1 / 10 t$, $\beta = c_1^2 / 10^4 t^2$ and $\gamma = 2 \delta'^2 c_3^3$.
    We will also assume that our final choice of $c_3$ obeys $c_3 \le c_1 / 40 t^2$ and $c_3 \le c_1/200$.

    In the case that $|V| \le \delta' (1 + c_1/10) |H|^t$ and $\sigma_i \le c_1 / 10 t^2$ for all $i \in [t]$, from~\eqref{eq:v-is-large} and these parameter choices we get $|V| \ge |H|^s (1 - 3c_1/10)$ and so $\delta' \ge 1 - c_1$, giving (III).
    We may therefore assume one of these inequalities fails.
    
    If $|V| \ge \delta' (1 + c_1/10) |H|^t$ then~\eqref{eq:best-triples} gives
    \[
      |W +_{i,\gamma} W| \ge |S| \big(1 - 2^{t+1} \eps - 5 c_3\big) \big(1 + c_1 / 10 \big)
    \]
    for any $i \in [t]$. The right hand side is at least $|S| (1 + c_1 / 20 - O_t(\eps))$, which is acceptable for (I) provided our final choice of $c_2$ has $c_2 \le c_1 / 20$.
    
    Finally, if $\sigma_i \ge c_1 / 10 t^2$ for some $i \in [t]$, then by~\eqref{eq:best-triples} again we get
    \[
      |W +_{i,\gamma} W| \ge |S| \big(1 - 2^{t+2} \eps - 5 c_3\big) \big(1 + (c_1 / 10 t^2)(c_1^2 / 10^4) \big)
    \]
    and so if we set $c_2 = c_3 = c_1^3 / 10^6 t^2$, the right hand side is again at least $|S|(1 + c_2 - O_t(\eps))$, giving (I).
  \end{proof}

  The proof of Lemma~\ref{lem:1pc-implies-derivs} now proceeds as follows.
  First consider the set
  \[
    U_0 = \big\{ c \in C^t(H) \cap S_0^{\llbracket t \rrbracket} \colon \partial^t \phi(c) = 0 \big\}
  \]
  which by hypothesis has $\mu(U_0) \ge \delta$.
  It is clear that in particular $\phi$ (together with the zero function) obeys an $(f,t,t,0)$-derivatives condition on $U_0$.
  
  We recursively use Claim~\ref{claim:anything-grows} (and Claim~\ref{claim:translate}, Claim~\ref{claim:stitch-derivs}, and~\eqref{eq:concat-loss}) to define further subsets $U_j \subseteq C^t(H) \cap S_0^{\llbracket t \rrbracket}$ for $1 \le j \le K$ on which $\phi$ obeys an $\big(f,t,t,m_j\big)$-derivatives condition (for some parameters $m_j$ to be determined), as follows.
  
  Suppose inductively that $U_j$ is defined, that $\mu(U_j) \ge \delta$ and that $\phi$ obeys an $(f,t,t,m_j)$-derivatives condition on $U_j$.
  We apply Claim~\ref{claim:anything-grows} with $S = U_j$, and with a parameter $c_1$ depending only on $t$ to be determined.
  Provided the choice of $J(t)$ in the statement of the lemma obeys $J(t) > 1/c_3$, option (II) cannot occur.
  
  If option (I) occurs, let $W$, $i$ and $\gamma \gg_{t,c_1} \mu(U_j)^2$ be as in the claim, and set $U'_{j+1} = W +_{i,\gamma} W$, so that $\mu(U'_{j+1}) \ge (1 + c_2 - O_t(\eps))\, \mu(U_j)$.
  By Claim~\ref{claim:translate} there is a subset $W' \subseteq W$ such that $\mu(W \setminus W') \le 6 \nu_0 / \delta$ and $\phi$ obeys an $(f,t,t,O_t(M+m_j))$-derivatives condition\footnote{In a previous version of this paper, it was stated that $\phi$ obeys an $(f,t,t,O_t(M))$-derivatives condition where the parameters do not depend on $m_j$; however, this was not correct and the dependence on $m_j$ here is necessary.  The effect on the bounds is ultimately not significant.  The author is grateful to Thomas Bloom for pointing this out.} on $W'$ (unless option (B) holds with parameters $O(M+m_j)^{O_t(1)}$ and $O(M+m_j)^{-O_t(D)} \nu_0^{O_t(1)}$).
  
  We then let $U_{j+1}$ be the set obtained by applying Claim~\ref{claim:stitch-derivs} to $W' +_{i,\gamma/2} W'$ with parameter $\nu = \nu_0/\delta$, so that $\phi$ again obeys the $\bigl(f,t,t,O_t(M+m_j)\bigr)$-derivatives condition on $U_{j+1}$ (or, noting the bound on $\gamma$, option (B) occurs with parameters $O(M+m_j)^{O_t(1)}$ and $O(M)^{-O_t(D)} (\nu_0 \delta)^{O_t(1)}$), and
  \begin{align*}
    \mu(U'_{j+1} \setminus U_{j+1}) &\le \mu\bigl(U'_{j+1} \setminus (W' +_{i,\gamma/2} W')\bigr) + \mu\bigl((W' +_{i,\gamma/2} W') \setminus U_{j+1}\bigr) \\
                                    &\le 24 \nu_0/\delta + \nu_0 / \delta = 25 \nu_0 / \delta
  \end{align*}
  where we used the definition $U'_{j+1}=W +_{i,\gamma} W$ and~\eqref{eq:concat-loss}.
  If we demand that our final choice of $\nu_0$ obeys $\nu_0 \le c_2 \delta^2 /100$, and that $\eps_0 \le c_2 / C$ where $C=C(t)$ is a constant chosen such that the term $O_t(\eps)$ in Claim~\ref{claim:anything-grows}(I) is bounded by $(C/4)\eps$ (and hence by $c_2/4$), then we establish that
  \[
    \mu(U_{j+1}) \ge (1 + c_2 - O_t(\eps))\, \mu(U_j) - 25 \nu_0 / \delta \ge (1 + c_2/2)\, \mu(U_j),
  \]
  and particular we still have $\mu(U_{j+1}) \ge \delta$.
  We may also choose a parameter $m_{j+1}=O_t(M+m_j)$ so that $\phi$ obeys an $(f,t,t,m_{j+1})$-derivatives condition on $U_{j+1}$.
  
  If on the other hand option (III) in Claim~\ref{claim:anything-grows} occurs, we terminate the process and set $K=j$ for this final index.
  As $\mu(U_j)$ is increasing by a factor of at least $1+c_2/2$ at each step in the process, this must happen eventually, and in fact for some $K \ll_{c_2} \log (1 / \delta)$.
  
  In conclusion, we have located a set $U_K \subseteq C^t(H) \cap S_0^{\llbracket t \rrbracket}$ with $\mu(U_K) \ge 1 - c_1$ and such that $\phi$ obeys an $\big(f,t,t,m_K\big)$-derivatives condition on $U_K$, where our recursive bound $m_{j+1} \ll_t m_j+M$ (and $m_0=0$) implies that
  \begin{equation}
    \label{eq:mk-bound}
    m_K \ll_t \delta^{-O_{t,c_2}(1)} M.
  \end{equation}

  We are again in the situation that a derivatives condition holds for almost all cubes with vertices in $S_0$, but only ``almost all'' in a weak sense, since $c_1$ is still a constant depending only on $t$ (to be determined).
  However, we need the derivatives condition to hold on all but $\eta |C^t(H)|$ cubes in $S_t$, where $\eta>0$ is the parameter from the statement and could be very small.
  
  We address this in two stages.
  First, we will use Claim~\ref{claim:translate} to expand our derivatives condition to a system of cubes $S'_0,\dots,S'_t$ whose vertex set $S'_0$ is almost all of $S_0$ (in a strong sense); however, $S'_t$ may only be almost all of $S_t$ in a weak sense.
  Then, we apply Claim~\ref{claim:augment-derivs} to upgrade this to almost all of $S_t$ in a strong sense also.

  The first step is captured by the following claim.
  During the course of the proof, we will finally fix the choices of $c_1$, $\nu_0$ and $\eps_0(t)$ to suitable values.
  \begin{claim}
    There exists a system of cubes $S'_0,\dots,S'_t$ with parameter $1-\eps'$ for some $\eps' \le \eps'_0(t)$ \uppar{where $\eps'_0(t)$ is the quantity from Claim~\ref{claim:augment-derivs}}, with $S'_k \subseteq S_k$ for each $k$ and $\mu(S_0 \setminus S'_0) \le 2^{-t-1} \eta$ \uppar{where $\eta>0$ is the parameter from the statement}, and such that $\phi$ obeys an $\big(f,t,t,O_t(M+m_K)\big)$ derivatives condition on $S'_t$; or, option (B) holds with parameters $O(M+m_K)^{O_t(1)}$ and $O(M+m_K)^{-O_t(D)} \nu_0^{O_t(1)}$.
  \end{claim}
  By~\eqref{eq:mk-bound}, the parameters in option (B) are acceptable.

  Given the claim, we can finish the argument as follows.
  Applying Claim~\ref{claim:augment-derivs} to $S'_0,\dots,S'_t$ with parameter $\nu = \eta/2$, we find a set $S \subseteq {S'_0}^{\llbracket t \rrbracket}$ with $\mu\bigl((C^t(H) \cap {S'}_0^{\llbracket t \rrbracket}) \setminus S\bigr) \le \eta/2$ and such that $\phi$ obeys an $\big(f,t,t,O_t(M+m_K)\big)$-derivatives condition on $S$, as required; or option (B) holds with parameters $O(M+m_K)^{O_t(1)}$ and $O(M+m_K)^{O_t(D)} \eta^{O_t(1)}$.
  Since
  \[
    \mu\bigl(C^t(H) \cap (S_0^{\llbracket t \rrbracket} \setminus {S'}_0^{\llbracket t \rrbracket})\bigr) \le 2^t \mu(S_0 \setminus S'_0) \le \eta / 2
  \]
  we deduce $\mu\bigl((C^t(H) \cap S_0^{\llbracket t \rrbracket}) \setminus S\bigr) \le \eta$ (and hence $\mu(S_t \setminus S) \le \eta$) as required.

  \begin{proof}[Proof of claim]
    Define a set $W$ as in~\eqref{eq:def-w}:
    \[
      W = \big\{ c \in C^t(H) \cap S_0^{\llbracket t \rrbracket}  \colon |\{ h \in H \colon c + \square^t(h) \in U_K \}| \ge c_1 |H| \big\}
    \]
    and apply Claim~\ref{claim:translate} to find a set $W' \subseteq W$ such that $\phi$ obeys an $\big(f,t,t,O_t(m_K)\big)$-derivatives condition on $W'$ and $|W \setminus W'| \le (3 \nu_0 / c_1) |C^t(H)|$.

    It remains only to find $S'_0,\dots,S'_t$ obeying the conditions in the statement and with $S'_t \subseteq W'$.
    As in~\eqref{eq:v-def}, we define
    \[
      V = \big\{ (h_1,\dots,h_t) \in H^t \colon |\{ x \in H \colon \angle(x;h_1,\dots,h_t) \in U_K \}| \ge c_1 |H| \big\}
    \]
    and consider the set of cubes
    \[
      \wt{V} = \big\{ \angle(x; h_1,\dots,h_t) \in C^t(H) \colon (h_1,\dots,h_t) \in V \big\} .
    \]
    By definition, $W = \wt{V} \cap S_0^{\llbracket t \rrbracket}$.
    Just as in~\eqref{eq:v-bound}, we can bound $|U_K| \le c_1 |H|^{t+1} + |V|\,|H|$, and note $|\wt V| = |V|\,|H|$; so as $|U_K| \ge (1-c_1) |C^t(H)|$ we deduce $|\wt V| \ge (1 - 2 c_1) |C^t(H)|$.
    
    Applying Lemma~\ref{lem:cube-system-ae} to $\wt{V}$, we may define a (non-empty) system of cubes $S^{(1)}_0,\dots,S^{(1)}_{t}$ with parameter $1-\eps^{(1)}$ where $\eps^{(1)} \ll_t c_1^{1/O_t(1)}$, such that $S^{(1)}_t \subseteq \wt{V}$.

    Since $\wt{V}$ is translation-invariant (i.e., if $c \in \wt{V}$ then so is $c + \square^t(h)$ for any $h \in H$), it is fairly clear from the proof of Lemma~\ref{lem:cube-system-ae} that the sets $S^{(1)}_k \subseteq C^k(H)$ inherit the same property; or alternatively, we can force this directly by replacing $S^{(1)}_k$ with
  \[
    \big\{ c + \square^k(h) \colon c \in S^{(1)}_k,\; h \in H \big\}
  \]
    for each $k$, and noting that this defines another system of cubes with the same parameter $\eps^{(1)}$ that is definitely translation-invariant, and still satisfies $S^{(1)}_t \subseteq \wt V$.
    We assume this has been done (if necessary).
    In particular, $S^{(1)}_0$ is translation-invariant, or in other words $S^{(1)}_0 = C^0(H)$ is everything.

    Also define a (non-empty) system of cubes $S^{(2)}_0,\dots,S^{(2)}_t$ by removing elements of $W \setminus W'$ from $C^t(H)$ and applying Lemma~\ref{lem:cube-system-ae}.
    This has parameter $1-\eps^{(2)}$ where $\eps^{(2)} \ll_t (\nu_0/c_1)^{1/O_t(1)}$, and $S^{(2)}_t \cap (W \setminus W') = \emptyset$.

    We finally define the system of cubes $S'_k = S_k \cap S^{(1)}_k \cap S^{(2)}_k$ for $0 \le k \le t$.
    It has parameter $1-\eps'$ where $\eps' = \eps + \eps^{(1)} + \eps^{(2)}$, and is non-empty provided $\eps'<1$.
    Moreover:---
    \begin{itemize}
      \item $S'_t \subseteq S_t \subseteq S_0^{\llbracket t \rrbracket}$,
      \item $S'_t \subseteq S^{(1)}_t \subseteq \wt V$ and so $S'_t \subseteq S_0^{\llbracket t \rrbracket} \cap \widetilde{V} = W$; and
      \item $S'_t \subseteq S^{(2)}_t$ which is disjoint from $W \setminus W'$, so $S'_t \subseteq W'$.
    \end{itemize}
    Finally, since $S^{(1)}_0 = H$ we have $\mu(S_0 \setminus S'_0) \le \eps^{(2)}$ (using Remark~\ref{rem:cube-system-large}).

    Lastly, we choose parameters.
    Since $\eps^{(1)} \ll_t c_1^{1/O_t(1)}$, we may choose $c_1$ depending only on $t$ such that $\eps^{(1)} \le \eps'_0(t)/3$, where $\eps'_0(t)$ is the quantity from the statement of Claim~\ref{claim:augment-derivs}.
    By choosing $\nu_0 \le (\eta/c_1)^C/C$ for a suitable constant $C=C(t)$, we may ensure that
    \[
      \eps^{(2)} \le \min\bigl( \eps'_0(t)/3,\ 2^{-t-1} \eta \bigr)
    \]
    and so $\mu(S_0 \setminus S'_0) \le \eps^{(2)} \le 2^{-t-1} \eta$ as required.
    Finally, we may insist $\eps_0(t)$ from the statement obeys $\eps_0(t) \le \eps'_0(t)/3$, so that $\eps' \le \eps'_0(t)$ as required.
  \end{proof}
  We can now choose any $\eta_0$ obeying $\eta_0 \le c_2 \delta^2/100$ and $\eta_0 \le (\eta/c_1)^C/C$ for some given constant $C=C(t)$; clearly (as we have now chosen $c_1$ and hence $c_2$ to be fixed constants depending on $t$) both bounds have the form $O(\eta \delta)^{O_t(1)}$, and this completes our choice of parameters.
  We retrospectively note that this means any outstanding appeals to option (B) had acceptable bounds.
  This completes the proof of Lemma~\ref{lem:1pc-implies-derivs}.
\end{proof}

It remains to prove Lemma~\ref{lem:derivs-implies-derivs}. This is comparatively straightforward, provided we borrow a special case of the theory of cocycles (Lemma~\ref{lem:gen-cocycle}).

\begin{proof}[Proof of Lemma~\ref{lem:derivs-implies-derivs}]
  As always, we may assume $b$ is in normal form, at the expense of replacing $M$ with $O_t(M)$.

  We first consider the case $t=1$ separately.  Recall (Remark~\ref{rem:hierarchy-t0}) that $f_{=0}$ are constant functions (if $\eps < 1/2$), and so we may write the hypothesis as $\phi(x)-\phi(y) = b([x,y], 0) \cdot f_{=0}$ for all $[x,y] \in S_1$.
  
  We will show that $\phi = \phi_0 + \phi_1$, where (i) $\phi_0$ is a constant function, and (ii) for all $x \in S_0$ we have $\phi_1(x) = v(x) \cdot f_{=0}$ for some $v(x) \in \ZZ^{d_0}$ with $\|v(x)\|_1 \ll M$.  Indeed, we pick any $x_0 \in S_0$ and set $\phi_0(x) = \phi(x_0)$ for all $x \in S_0$.  Setting $\phi_1(y) := \phi(y) - \phi_0(y)$ for any $y \in S_0$, we note that we can pick $z \in S_0$ such that $[x_0,z], [y,z] \in S_1$ (provided $\eps < 1/2$) and so
  \[
    \phi_1(y) = \phi(y) - \phi(x_0) = \partial^1 \phi ([y,z]) - \partial^1 \phi ([x_0,z]) = \big(b([y,z], 0) - b([x_0,z], 0)\big) \cdot f_{=0}
  \]
  which has the required form.

  When $t > 1$, we instead apply Corollary~\ref{cor:incompatible-implies-1pc} and Corollary~\ref{cor:no-cocycle-implies-1pc} to deduce that if option (B) does not hold then $b$ obeys the $\big(f,t,t,O_t(M),\delta\big)$-strong derivatives condition, where $\delta \gg_t \eta^{O_t(1)}$ is a parameter to be determined.  In particular, we may assume $b_{=t-1}$ is a generalized $t$-cocycle of type $t-1$ and loss at most $\delta$.
  
  By Lemma~\ref{lem:gen-cocycle}, we deduce that there is a subset $S \subseteq S_t$ with $\mu(S_t \setminus S) \ll_t \delta^{1/O_t(1)}$, and a function $\lambda \colon H \to \RR^{d_{t-1}}$ such that
  \begin{equation}
    \label{eq:Lambda}
    \Lambda(c,\omega') := \sum_{\omega \in \llbracket t \rrbracket} Z_{t-1}(\omega,\omega') \lambda(c(\omega))
  \end{equation}
  as in the statement\footnote{\label{footnote:diff}In the case of a $t$-cocycle of type $t-1$, an inspection of the definition of $Z_{t-1}(\omega,\omega')$ in~\eqref{eq:zr} together with~\eqref{eq:cube-orthog} shows that this identity reduces to $\Lambda(c,\omega') = \lambda(\omega')-\lambda(\vec 1)$.  However, this does not greatly simplify the calculations: what is really important is that the right-hand side is what you get by putting $\omega \mapsto \lambda(c(\omega))$ into normal form at level $t-1$.} satisfies $\Lambda(c,\omega) = b_{=t-1}(c,\omega)$ for all $c \in S$ and $\omega \in \llbracket t \rrbracket$.
  Also, $\|\lambda(x)\|_1 \ll_t M + d_{t-1}$ for each $x \in H$, and $\partial^1 \lambda([x,y]) \in \ZZ^{d_{t-1}}$ for all $x,y \in H$.
  By a suitable choice $\delta=\eta^{O_t(1)}/O_t(1)$ we can ensure $\mu(S_t \setminus S) \le \eta$.
  
  It is easy to check that if we change $\lambda$ by adding a global constant, e.g.\ to $x \mapsto \lambda(x) - \lambda(0)$, then~\eqref{eq:Lambda} still holds (see~\eqref{eq:zr}, or Remark~\ref{rem:normal-form}, or Footnote~\ref{footnote:diff}), so we may assume without loss of generality that we have done this so that $\lambda(0)=0$.
  In particular,
  \[
    \lambda(x)=\lambda(x)-\lambda(0)=\partial^1\lambda([x,0]) \in \ZZ^{d_{t-1}}
  \]
  and so $\lambda$ itself takes values in $\ZZ^{d_{t-1}}$.

  Having done this, we can define $\phi_1(x) := \lambda(x) \cdot f_{=t-1}(x)$ and $\phi_0(x) := \phi(x) - \phi_1(x)$, which we claim have the desired properties.
  This is clear for $\phi_1$.
  It remains to show that $\phi_0$ obeys an $(f,t,t-1,O_t(M))$-derivatives condition on $S$.
  Taking any $c \in S$, we certainly have
  \begin{align*}
    \partial^t \phi_0(c) = \partial^t \phi(c)- \partial^t \phi_1(c) &= \sum_{\omega \in \llbracket t \rrbracket} (-1)^{|\omega|} b(c,\omega) \cdot f(c(\omega)) - \sum_{\omega \in \llbracket t \rrbracket} \lambda(c(\omega)) \cdot f_{=t-1}(c(\omega)) \\
                                                                    &= \sum_{\omega \in \llbracket t \rrbracket} (-1)^{|\omega|} \fb(c,\omega) \cdot f(c(\omega))
  \end{align*}
  where $\fb_{=t-1}(c,\omega) = b_{=t-1}(c, \omega) - \lambda(c(\omega))$ and $\fb_{=i} = b_{=i}$ for all $0 \le i < t-1$.
  Let $\fb'(c,\omega)$ be the configuration obtained by placing $\fb(c,\omega)$ in normal form, by Lemma~\ref{lem:canonical-b}: hence, we still have
  \[
    \partial^t \phi_0(c) = \sum_{\omega \in \llbracket t \rrbracket} (-1)^{|\omega|} \fb'(c,\omega) \cdot f(c(\omega))
  \]
  and using the explicit formula for $\fb'_{=t-1}$ in Lemma~\ref{lem:canonical-b} gives 
  \begin{align*}
    \fb'_{=t-1}(c,\omega') &= \sum_{\omega \in \llbracket t \rrbracket} Z_{t-1}(\omega,\omega') \fb_{=t-1}(c,\omega) \\
                           &= \sum_{\omega \in \llbracket t \rrbracket} Z_{t-1}(\omega,\omega') b_{=t-1}(c,\omega) -  \sum_{\omega \in \llbracket t \rrbracket} Z_{t-1}(\omega,\omega') \lambda(c(\omega)) \\
                           &= b_{=t-1}(c,\omega) - \Lambda(c,\omega) = 0
  \end{align*}
  where we used the definition of $\fb$,~\eqref{eq:Lambda} and the fact that $b_{=t-1}(c,\omega)$ is already in normal form.
  Since $\fb'_{=t-1}$ is identically zero, we deduce that $\phi_0$ and $\fb'_{\le t-2}$ obey an $\big(f,t,t-1,O_t(M)\big)$-derivatives condition on $S$, as required.
\end{proof}

\section{Recovering algebraic structure}%
\label{sec:algebraic-poly-system}

\subsection{The structure theorem for polynomial hierarchies}%
\label{subsec:poly-structure}

We will now use the improved structure obtained in Section~\ref{sec:hierarchy-improve} to show that any function obeying a derivatives condition on a polynomial hierarchy, such as the function $\wt g$ in the conclusion of Corollary~\ref{cor:poly-hierarchy-everything}, agrees almost everywhere with a nil-polynomial.

For inductive reasons it is necessary to work over a slightly larger class of groups $H$ than we traditionally care about.  We also handle $(f,s+1,t,M)$-derivatives conditions for $1 \le t \le s$, again for induction. The precise statement is as follows.

\begin{theorem}%
  \label{thm:poly-structure}
  Let $H = (\ZZ/N^\kappa\ZZ)^n$ for some $N$ prime and $n,\kappa \ge 1$.  For $s \ge 1$, let $S_0,\dots,S_{s+1}$ be a system of cubes of parameter $1-\eps$, let $f$ be an $\big(s-1,d,M)$-polynomial hierarchy on $S_0,\dots,S_{s}$, and suppose $g \colon S_0 \to \RR$ obeys the $(f,s+1,t,M)$-derivatives condition on $S_{s+1}$ for some $t$ \uppar{$1 \le t \le s$}.  Write $D = \sum_{i=0}^{s-1} d_i$, and let $\eta > 0$ be a further parameter.
  
  Then there exists a quantity $M'$ with
  \[
    M' \le O(M/\eta)^{O_{n,s}(D)^{O_{s}(D)}}
  \]
  if $s \ge 3$, or
  \[
    M' \le O_n(M/\eta)^{O_n(D)^{O(1)}}
  \]
  if $s \le 2$, such that the following holds: provided $\eps \le \eps_0(s)$ is sufficiently small and $N \ge M'$, then
  there exists a nil-polynomial $g' \colon H \to \RR$, with degree $s$, dimension $O(D)^{O_{n,s}(1)}$ and complexity at most $M'$, such that $g(x) = g'(x)$ for all but at most $\eta |H|$ values $x \in S_0$.
\end{theorem}

This is one of the two places in the work where serious discussion of nilmanifolds is necessary, and we will certainly need to use discuss notions from Appendix~\ref{app:nilmanifolds} (polynomial maps, filtrations, etc.) a fair amount.

The more serious nilmanifold content is in Section~\ref{subsec:nilmanifold-constructions}, and quotes more technical content from Appendix~\ref{app:nilmanifolds}. An attempt has been made to limit the amount of Appendix~\ref{app:nilmanifolds} the reader would actually have to read to follow the proofs, and to make it clear when we are merely performing technical checks of technical properties as opposed to anything conceptual. However, some engagement with the theory is inevitably required for a detailed reading of that section.

We make some remarks which will be useful in the proof of Theorem~\ref{thm:poly-structure}.
\begin{remark}%
  \label{rem:nil-vs}
  If $g,g' \colon H \to \RR$ are nil-polynomials of degrees $s$ and $s'$, dimensions $D$ and $D'$ and complexities $M$ and $M'$ respectively, then for $\lambda \in \RR$ we have that $x \mapsto \lambda g(x)$ is again a nil-polynomial of degree $s$, dimension $D$ and complexity $M$, and $g+g'$ is a nil-polynomial of degree $\max(s,s')$, dimension $D+D'$ and complexity $\max(M,M')$.
  
  Indeed, if $g = F \circ r$, defined with respect to a nilmanifold $G/\Gamma$, and similarly $g' = F' \circ r'$ on $G'/\Gamma'$, as in Definition~\ref{def:nil-poly}, then $\lambda g = (\lambda F) \circ r$ has the same form. For $g+g'$,  we work on the direct product $\wt{G} = G \times G'$, $\wt{\Gamma} = \Gamma \times \Gamma'$ (which inherits complexity $\max(M,M')$ in the sense of Definition~\ref{def:nilmanifold-complex}) and define $\wt{r}(x) = (r(x), r'(x))$ and $\wt{F}(y,y') = F(y) + F'(y')$, giving $g+g' = \wt{F} \circ \wt{r}$.
\end{remark}

\begin{remark}%
  \label{rem:t=0}
  Suppose instead that $g \colon H \to \RR$ obeys the hypothesis $\partial^{s+1} g(c) = 0$ for all $c \in S_{s+1}$, where $S_0,\dots,S_{s+1}$ is as in Theorem~\ref{thm:poly-structure}.  This corresponds to the missing case $t=0$ of the theorem, and is useful for inductive purposes.
  
  In this case, we can deploy an argument we made in the proof of Claim~\ref{claim:lambda-compat}. By Lemma~\ref{lem:99pc-poly}, we know that there is some $\wt{g}\colon H \to \RR$ such that $\partial^{s+1} \wt{g} \equiv 0$ and $g(x) = \wt{g}(x)$ for all $x \in S_0$ (again arguing that the set $X$ in~\eqref{eq:bigX} contains all of $S_0$).  Then,~\eqref{eq:triv-poly} again shows $\wt{g}$ must be constant, and hence so is $g$.
\end{remark}

The key step in the proof of Theorem~\ref{thm:poly-structure} is as follows: given $g$ and $b$ obeying a strong $(f,s+1,t,M,\delta)$-derivatives condition, we consider the top-level part $b_{=t-1}$, whose structure we understand in terms of generalized cocycles, and construct a nil-polynomial obeying a derivatives condition with the same top-level part $b_{=t-1}$.  Subtracting the two gives a function obeying an $(f,s+1,t-1,M')$-derivatives condition, which we can handle by induction.

The technical statement encoding this step is the following.

\begin{lemma}%
  \label{lem:poly-construction}
  Let $H = (\ZZ/N^\kappa\ZZ)^n$ for some $n,\kappa \ge 1$ and some prime $N$.  Let $s \ge 1$ and $t$, $1 \le t \le s$, be integers, and let $f \colon H \to \RR$ be a nil-polynomial of degree $t-1$, dimension $D$ and complexity $M$ on a nilmanifold $G/\Gamma$.  Further let $\lambda \colon H \to \RR$ be a function with $|\lambda(x)| \le M$ for each $x \in H$ and such that $\lambda \bmod \ZZ \colon H \to \RR/\ZZ$ is a polynomial map of degree $s+1-t$.  For $c \in C^{s+1}(H)$, write
  \[
    \Lambda(c,\omega') = \sum_{\omega \in \llbracket k \rrbracket} Z_r(\omega,\omega') \lambda(c(\omega))
  \]
  as in~\eqref{eq:big-lambda}.  Finally suppose $N > s$.

  Write $H' = (\ZZ/N^{2\kappa}\ZZ)^n$.  Then there exists a nil-polynomial $g \colon H' \to \RR$ of degree $s$, dimension at most $D'$ and complexity at most $M'$, and a $(t-1,d,M')$-polynomial hierarchy $f'$ on the full system $C^0(H'),\dots,C^{t}(H')$, such that:---
  \begin{enumerate}[label=(\roman*)]
    \item $d_{t-1}=1$ and $f'_{t-1,1}$ is the function $x \mapsto f(x \bmod N^\kappa)$; and
    \item $g$ obeys an $(f', s+1, t, M')$-derivatives condition with some $b$, such that $b_{=t-1}(c,\omega) = \Lambda(c \bmod N^\kappa, \omega)$ for all $c \in C^{s+1}(H')$ and $\omega \in \llbracket s+1 \rrbracket$.
  \end{enumerate}
  Here, $D' \ll_s D^{O_s(1)}$ and $M' \ll_s (DM)^{O_s(1)}$ are parameters, and we have $\sum_{r=0}^{t-1} d_r \ll_s D^{O_s(1)}$.
\end{lemma}

Note that we are forced to pass from $H$ to the group extension $H'$.  In a sense this is necessary: under the hypotheses given, there may not exist a function $g \colon H \to \RR$ obeying a derivatives condition whose top piece $b_{=t-1}$ is $\Lambda$, on a hierarchy of bounded complexity whose top piece $f'_{=t-1}$ is the single function $f$, irrespective of whether $g$ is a nil-polynomial: the requirements this imposes on $g$ may just be inconsistent.  In other words, $\Lambda$ being a generalized cocycle is not a sufficient condition for it to appear as $b_{=t-1}$ alongside $f$ in a derivatives condition, as above: there may be further cohomological obstructions.

We can rectify this runaway group extension after the fact using the following construction, which is closely related to~\cite[Theorem 1.5]{me-periodic} or~\cite[Proposition C.2]{gtz}.

\begin{lemma}%
  \label{lem:extension-fix}
  Let $H = (\ZZ/N^\kappa\ZZ)^n$ and $H' = (\ZZ/N^{\kappa'}\ZZ)^n$, where $N$ is prime and $n \ge 1$, $\kappa' \ge \kappa \ge 1$ are integers.  Suppose $g' \colon H' \to \RR$ is a nil-polynomial on $H'$, with degree $s \ge 0$, dimension $D$ and complexity $M$.  Define the inclusion $\imath \colon H \to H'$ by $\imath(x_1,\dots,x_n) = (x_1,\dots,x_n)$ whenever $x_1,\dots,x_n$ are integers with $0 \le x_i < N^\kappa$.
  
  Then the function $g \colon H \to \RR$ given by $g = g' \circ \imath$ is a nil-polynomial of degree $s$, dimension $\ll_{s,n} D^{O_{s,n}(1)}$ and complexity $\ll_{s,n} (D M)^{O_{s,n}(1)}$.
\end{lemma}

\begin{proof}[Proof of Theorem~\ref{thm:poly-structure} assuming Lemma~\ref{lem:poly-construction} and Lemma~\ref{lem:extension-fix}]
  We proceed by induction on $s$ and $t$; i.e., we will assume that the theorem holds for all smaller values of $s$, and for the same $s$ but smaller values of $t$.

  We begin by applying Theorem~\ref{thm:strong-derivs} to $g$ and $f$ with some parameter $\delta$ to be determined. So, we may assume that $g$ and $b$ obey an $(f',s+1,t,M_1,\delta)$-strong derivatives condition on $S'_{s+1}$, where $f'$ is a modified $(s,d',M_1)$-polynomial hierarchy on a subsystem $S'_0,\dots,S'_{s+1}$, with bounds on $d'$, $S'_k$ and $M_1$ as stated in Theorem~\ref{thm:strong-derivs}.  Here $\delta>0$ is a parameter to be determined.

  By our hypotheses on $b$ and Lemma~\ref{lem:gen-cocycle}, we may find $\lambda \colon H \to \RR^{d'_{t-1}}$ such that $b(c,\omega) = \Lambda(c,\omega)$ for all $\omega \in \llbracket s+1 \rrbracket$, for all but $O_s\big(\delta^{1/O_s(1)}\big) |C^{s+1}(H)|$ cubes $c \in S'_{s+1}$, where $\Lambda(c,\omega)$ is defined as in~\eqref{eq:big-lambda}.  We will write $\lambda_j$ and $\Lambda_j$ to refer to individual coordinates of these functions, for $j \in [d'_{t-1}]$.

  We recall that each function $f_{0,j}$ for $j \in [d'_0]$ is constant (provided $\eps < 1/2$; see Remark~\ref{rem:hierarchy-t0}), and a constant function is a nil-polynomial of degree $0$, with $G = \Gamma = \{\ast\}$ being the trivial group, $r \colon H \to G$ the unique map and $F \colon G \to \RR_{(0)}$ taking the constant value.

  By induction on $s$ (when $t \ge 2$) or the previous paragraph (when $t=1$), each function $f'_{t-1,j} \colon S'_0 \to \RR$ for $j \in [d'_{t-1}]$ agrees on all but $\eta_1 |H|$ values $x \in S'_0$ with a nil-polynomial $\ff_j \colon H \to \RR$ of degree $t-1$, whose dimension $D_1$ and complexity $M_2$ obey the bounds from the statement of the theorem.  Here $\eta_1 > 0$ is another a parameter to be determined.

  For each $\ff_j$, we apply Lemma~\ref{lem:poly-construction} to find a nil-polynomial $g_j \colon H' \to \RR$ of degree $s$, dimension $D_2 \ll_s D_1^{O_s(1)}$ and complexity $M_3 \ll_s (D_1 M_2)^{O_s(1)}$, with all the properties given in the statement of the lemma, where $H' = (\ZZ/N^{2 \kappa}\ZZ)$.

  Note that we may lift $S'_0,\dots,S'_{s+1}$ to a system of cubes $\wt{S'}_0,\dots,\wt{S'}_{s+1}$ on $H'$, $g$ to a function $H' \to \RR$ and $f'$ to a polynomial hierarchy on $H'$ (by composing with the projection $H' \to H$), without any loss of parameters.
  
  We therefore combine each of the new polynomial hierarchies $f'^{(j)}$ produced by Lemma~\ref{lem:poly-construction}, together with our original hierarchy $f'$, into a new hierarchy $f''$ on $\wt{S'}_0,\dots,\wt{S'}_{s+1}$, whose dimension and complexity are still controlled polynomially in terms of the other parameters.  Since $f'_{t-1,j}$ and $\ff_j$ agree on all but $\eta_1 |H'|$ values each, we may discard each $f'_{t-1,j}$ from the hierarchy and replace all references to it with $\ff_j$, at the expense of replacing $\wt{S'}_0,\dots,\wt{S'}_{s+1}$ with a subsystem $\wt S''_0,\dots,\wt S''_{s+1}$ that loses another $O_s\big((d'_{t-1} \eta_1)^{1/O_s(1)}\big)$ in its parameter (by Corollary~\ref{cor:cube-system-patch}).
  
  So, $g$ and $g_j$ all obey $(f'',s+1,t, M_3)$-derivatives conditions on $\wt S''_{s+1}$.  Write (with some abuse of notation) $b$ and $b^{(j)}$ for the other functions appearing in these derivatives conditions.  By the definition of $\Lambda$ and the conclusion of Lemma~\ref{lem:poly-construction} again, we have
  \[
    b_{=t-1} - \sum_{j=1}^{d'_{t-1}} b^{(j)}_{=t-1} = 0
  \]
  and hence $\wt{g} = g - \sum_{j=1}^{d'_{t-1}} g_j$ actually obeys an $\big(f'', s+1, t-1, M_3(1 + d'_{t-1})\big)$-derivatives condition on $\wt{S''}_{s+1}$, as the terms in the $t-1$ level cancel.

  By induction on $t$ (or Remark~\ref{rem:t=0}, if $t=1$) we have that $\wt{g}$ agrees with a nil-polynomial on $H'$ of degree $s$, complexity $M_4$ and dimension $D_3$ on all but $\eta_1 |H'|$ values $x \in \wt S''_0$ (again obeying the bounds in the statement).  By Remark~\ref{rem:nil-vs}, $g' = \wt{g} + \sum_{j=1}^{d'_{t-1}} g_j$ is therefore also a nil-polynomial of degree $s$, complexity $M_4$ and dimension $D_4 = D_3 + d'_{t-1} D_2$ on $H'$, and has the property that $g'(x) = g(x \bmod N)$ for all but $\eta_1 |H'|$ values $x \in \wt S''_0$, and hence all but $O_s\big(\eta_1^{1/O_s(1)}\big) |H'|$ of those values $x \in H'$ with $x \bmod N^\kappa \in S_0$.

  It is clear that these properties of $g'$ also hold for its translates $x \mapsto g'(x+a)$ for $a \in \ker(H' \to H)$.  Writing $\imath \colon H \to H'$ as in Lemma~\ref{lem:extension-fix}, by averaging there is some such $a$ such that $g(x) = g'(\imath(x) + a)$ for all but $O_s\big(\eta_1^{1/O_s(1)}\big) |H|$ values $x \in S_0$.  Applying Lemma~\ref{lem:extension-fix} to $x \mapsto g'(x+a)$, we obtain a nil-polynomial $x \mapsto g'(\imath(x)+a)$ with the degree $s$, dimension $D_5 \ll_{s,n} D_4^{O_{s,n}(1)}$ and complexity $M_5 \ll_{s,n} (D_4 M_4)^{O_{s,n}(1)}$.

  Finally, without laboring the point, we note that it suffices to take $\eta_1 = (\eta / C d'_{t-1})^C$ and $\delta = \eta^{C} / C$ for some suitable $C=C(s)$, and that then all the other various quantities $D_1,\dots,D_5$ and $M_1,\dots,M_5$, and in particular the final ones, have bounds of the form required.
\end{proof}

\subsection{Nilmanifold constructions}%
\label{subsec:nilmanifold-constructions}

We now turn to proving the two construction lemmas above, Lemma~\ref{lem:poly-construction} and Lemma~\ref{lem:extension-fix}.  We first record some generalities concerning spaces of polynomial maps $G \to G'$ between filtered groups, which will be relevant throughout.  Again see Appendix~\ref{app:nilmanifolds} for the fundamental definitions.

\begin{proposition}%
  \label{prop:poly-space}
  Let $G_\bullet$ and $G'_\bullet$ be explicitly presented filtered groups, with degrees $s$, $s'$, dimensions $D$, $D'$, complexities $M,M'$ and integral subgroups $\Gamma$, $\Gamma'$ respectively.  We consider the space $\poly(G, G')$ of polynomial maps $G \to G'$ with these filtrations.
  \begin{enumerate}
    \item The space $\poly(G, G')$ is a group under pointwise multiplication.
    \item We can put a filtration on $\poly(G, G')$ by $\poly(G, G')_i = \poly\big(G, {G'}^{+i}\big)$, where ${G'}_\bullet^{+i}$ is the filtered group with ${G'}^{+i}_j = G'_{i+j}$.  This is indeed a filtration on $\poly(G, G')$.
    \item Any polynomial map $\Gamma \to G'$ can be extended uniquely to a polynomial map $G \to G'$, and hence $\poly(\Gamma, G') \cong \poly(G, G')$ are interchangeable.
    \item There are one-parameter subgroups $p_{i,j} \colon \RR \to \poly(G, G')_i$, denoted $\alpha \mapsto p_{i,j}^\alpha$ \uppar{writing $p_{i,j}$ for $p_{i,j}^1$}, for $0 \le i \le s'$ and $j \in [\fd_i]$, where $\fd_i \ll_{s,s'} (D' D)^{O_{s,s'}(1)}$ are non-negative integers, with the property that for any $p \in \poly(G,G')$ there are unique coefficients $\alpha_{i,j} \in \RR$ such that $p = \prod_{i=0}^{s'} \prod_{j=1}^{\fd_i} p_{i,j}^{\alpha_{i,j}}$ \uppar{with products ordered from left to right}.

      Moreover, $p \in \poly(G,G')_{i'}$ if and only if $\alpha_{i,j} = 0$ whenever $i < i'$; and $p \in \poly(\Gamma, \Gamma')$ \uppar{equivalently, $p(\Gamma) \subseteq \Gamma'$} if and only if $\alpha_{i,j} \in \ZZ$ for each $i,j$.
      
      In particular, $(p_{i,j})$ for $i \ge 1$ make $\poly(G, G')_1$ into an explicitly presented filtered group of degree $s'$, dimension $\ll_{s,s'} (D' D)^{O_{s,s'}(1)}$ and complexity $\ll_{s,s'} (D M D'M')^{O_{s,s'}(1)}$,  with integral subgroup exactly $\poly(\Gamma, \Gamma')_1$.
    \item There are left and right actions of $G$ by homomorphisms of $\poly(G, G')$, by translation: for $g \in G$ and $p \in \poly(G, G')$ we let $\fs(g, p)$ denote the map $x \mapsto p(g x)$ and similarly $\tau(p, g)$ is the map $x \mapsto p(x g)$.
      
      If $p \in \poly(G, G')_i$ then so are $\fs(g, p)$ and $\tau(p,g)$, and if furthermore $g \in G_{i'}$ then $p \cdot \fs(g,p)^{-1} \in \poly(G, G')_{i+i'}$, and similarly for $\tau$.
    \item If $g \in G$, $p_{i',j'}$ is one of the polynomials in (iv) and we expand $\fs(g, p_{i',j'})$ or $\tau(p_{i',j'}, g)$ as $\prod_{i=0}^{i'} \prod_{j=1}^{\fd_i} p_{i,j}^{\alpha_{i,j}}$ as in (iv), then $\alpha_{i,j}$ satisfy 
      \[
        \sum_{i,j} |\alpha_{i,j}| \ll_{s,s'} (D' D M' M)^{O_{s,s'}(1)} (1 + d_G(\id_G, g))^{O_{s,s'}(1)}.
      \]
      If $g \in \Gamma$ then $\alpha_{i,j}$ as above are integers.
    \item If $p = \prod_{i=0}^{s'} \prod_{j=1}^{\fd_i} p_{i,j}^{\alpha_{i,j}}$, then for some $R = O_{s,s'}(1)$ the sum $\sum_{i=0}^{s'} \sum_{j=1}^{\fd_i} |\alpha_{i,j}|$ is bounded by
      \[
        O(D' D M' M)^{O_{s,s'}(1)} \big(1 + \sup \big\{ d_{G'}(\id_{G'}, p(\gamma)) \colon \gamma \in \Gamma,\, d_G(\id_G,\gamma) \le R \big\} \big)^{O_{s,s'}(1)};
      \]
      i.e., the size of the coefficients in the expansion of $p$ are controlled by the size of elements $p(\gamma)$ for bounded $\gamma \in \Gamma$.
    \item For $x_1,x_2 \in G$ and any $p_{i,j}$ the quantity $d_{G'}\big(p_{i,j}(x_1), p_{i,j}(x_2)\big)$ is bounded by
      \[
        O(D' D M' M)^{O_s(1)} d_G(x_1, x_2) \big(1 + d_G(\id_G, x_1) + d_G(\id_G, x_2) \big)^{O_{s,s'}(1)} .
      \]
  \end{enumerate}
\end{proposition}
The proof is given in Appendix~\ref{appsub:misc}.
\begin{remark}%
  \label{rem:poly-space-ab}
  In some applications $G'_\bullet$ is abelian, and more specifically $G'_\bullet = \RR_{(s')}$ (the abelian group $\RR$ with the degree $s'$ filtration).
  
  In this case we convert these statements to additive notation.
  Moreover, part (iv) can be stated more succinctly in terms of the basis $\{p_{i,j}\}$ for the vector space $\poly\big(G,\RR_{(s')}\big)$ and the full rank sub-lattice $\poly\big(\Gamma, \ZZ_{(s')}\big)$ it generates, and 
  in (ii) the shifted filtration ${G'}^{+i}_\bullet$ is just $\RR_{(s'-i)}$.
\end{remark}

Before we begin proving Lemma~\ref{lem:poly-construction}, we give a kind of converse to Theorem~\ref{thm:poly-structure} which will be needed in the proof.  In other words, this shows that nil-polynomials always obey a derivatives condition on a polynomial hierarchy.

\begin{proposition}%
  \label{prop:nil-is-hierarchy}
  Let $H$ be any abelian group and $g = F \circ r$ a nil-polynomial on $H$ of degree $s$, dimension $D$ and complexity $M$, where $r \colon H \to G$, $F \colon G \to \RR_{(s)}$ and $G/\Gamma$ are as in Definition~\ref{def:nil-poly}.
  Then there is a sequence of non-negative integers $(d_0,\dots,d_{s-1})$, and collection of polynomial maps $F_{i,j} \colon G \to \RR_{(i)}$ for $0 \le i \le s-1$ and $j \in [d_i]$, such that the following hold.

  \begin{enumerate}[label=(\roman*)]
    \item If we expand
      \[
        F = \sum_{i_1=0}^s \sum_{j_1=1}^{\fd_{i_1}} \alpha_{i_1,j_1} p_{i_1,j_1}
      \]
      as in Proposition~\ref{prop:poly-space}(iv), then $F_{i,j}$ are be exactly the functions $\alpha_{i_1,j_1} p_{i_2,j_2}$, where $0 \le i_1 < i_2 \le s$, $j_1 \in [\fd_{i_1}]$, $j_2 \in [\fd_{i_2}]$, and $i_2 = s-i$.
    \item For every $\gamma \in \Gamma$ we may write
      \[
        F - \tau(F,\gamma) = \sum_{i=0}^{s-1} \sum_{j=1}^{d_i} b^{(\gamma)}_{i,j} F_{i,j}
      \]
      where $b^{(\gamma)}_{i,j} \in \ZZ$ are integers with $\|b\|_1 \ll_s (D M)^{O_s(1)} (1 + d_G(\id_G,\gamma))^{O_s(1)}$.

    \item Write $f_{i,j} = F_{i,j} \circ r$. Then the tuple $f=(f_{i,j})$ is an $(s-1,d,M')$-polynomial hierarchy on $C^0(H),\dots,C^s(H)$, and there is some $b$ such that $g$ and $b$ obey an $(f, s+1, s, M')$-derivatives condition on $C^{s+1}(H)$, where $M' \ll_s (D M)^{O_s(1)}$.
  
  \end{enumerate}
\end{proposition}
We first show the following.

\begin{lemma}%
  \label{lem:nil-derivs}
  If $F' \circ r$ be a nil-polynomial $H \to \RR$ on $G/\Gamma$, of degree $t$, dimension $D$ and complexity $M$. For any $c \in C^{t+1}(H)$ there exist values $\gamma(c,\omega) \in \Gamma$ with $d_G(\id_G, \gamma(c,\omega)) \ll_t (MD)^{O_t(1)}$ such that
  \[
    \partial^{t+1}(F' \circ r) (c) = \sum_{\omega \in \llbracket t+1 \rrbracket} (-1)^{|\omega|} \big(F' - \tau(F', \gamma(c,\omega))\big)(r(c(\omega))).
  \]
\end{lemma}
\begin{proof}
  We choose $\gamma(c,\omega)$ to ``correct'' the configuration $\omega \mapsto r(c(\omega))$ to a Host--Kra cube.  Specifically, since $\omega \mapsto r(c(\omega)) \Gamma$ lies in $\HK^{t+1}(G/\Gamma)$ by the hypothesis on $r$, by Proposition~\ref{prop:effective-cubes} we can find $\gamma(c,\omega)$ with $d_G(r(c(\omega)), \gamma(c,\omega)) \ll_t D$ (and hence $d_G(\id_G,\gamma(c,\omega)) \ll_t (M D)^{O_t(1)}$), such that the configuration $\wt c(\omega) = r(c(\omega)) \gamma(c,\omega)$ lies in $\HK^{t+1}(G)$.

  Since $F'$ is a polynomial map, we have that $\omega \mapsto F(\wt c(\omega))$ lies in $\HK^{t+1}(\RR_{(t)})$, and hence
  \[
    \sum_{\omega \in \llbracket t+1 \rrbracket} (-1)^{|\omega|} F(\wt c(\omega)) = 0 .
  \]
  It follows immediately that
  \[
    \partial^{t+1}(F' \circ r) (c) = \sum_{\omega \in \llbracket r+1 \rrbracket} (-1)^{|\omega} \Big(F'(r(c(\omega))) - F'(\wt c(\omega)) \Big)
  \]
  and on unwrapping the definitions of $\wt c$ and $\tau$ this completes the proof.
\end{proof}

\begin{proof}[Proof of Proposition~\ref{prop:nil-is-hierarchy}]
  Assuming (ii) and noting the choice (i), applying Lemma~\ref{lem:nil-derivs} with $F'=F$ proves that $F \circ r$ obeys the derivatives condition on $(f_{i,j})$ required in (iii).

  To finish the proof we need the following.
  \begin{claim}
    For any polynomial $p_{i',j'}$ as in Proposition~\ref{prop:poly-space}, where $0 \le i' \le s$ and $j' \in [\fd_{i'}]$, and any $\gamma \in \Gamma$, we have
    \[
      p_{i',j'} - \tau(p_{i',j'}, \gamma) = \sum_{i=i'+1}^{s} \sum_{j=1}^{d_i} b_{i,j} p_{i,j}
    \]
    where $b_{i,j} \in \ZZ$ depend on $\gamma$ and $\|b\|_1 \ll_s (DM)^{O_s(1)} (1 + d_G(\id_G,\gamma))^{O_s(1)}$.
  \end{claim}
  With the choice of $F_{i,j}$ in (i) this proves (ii), and combined with Lemma~\ref{lem:nil-derivs} applied with $F'=F_{i,j}$ we get a derivatives condition for each function $F_{i,j} \circ r$, proving $f$ is a polynomial hierarchy as required in (iii).
  \begin{proof}[Proof of claim]
    The fact that integers $b_{i,j}$ of this form exist is immediate by Proposition~\ref{prop:poly-space}(iv,v) (converting to additive notation).

    In this abelian setting, the bound on $\|b\|_1$ is immediate from Proposition~\ref{prop:poly-space}(vi).
  \end{proof}
  This completes the proof of Proposition~\ref{prop:nil-is-hierarchy}.
\end{proof}

\begin{proof}[Proof of Lemma~\ref{lem:poly-construction}]
  Let $r \colon H \to G$ and $F \colon G \to \RR_{(t-1)}$ be the functions associated to the nil-polynomial $f$, as in Definition~\ref{def:nil-poly}.  Also let $D_1,\dots,D_{t-1}$ and $\gamma_{i,j}$ for $1 \le i \le t-1$ and $j \in [D_i]$ be the data associated to the explicitly presented filtered group $G$.

  Our task is twofold.  First we construct the nilmanifold $G'/\Gamma'$ on which the nil-polynomial $g$ from the statement is defined, and followed by the remaining data of $g$.  Second, we use Proposition~\ref{prop:nil-is-hierarchy} to construct a polynomial hierarchy, and verify that it has the required properties with respect to $g$.

  {\bf \noindent Preliminaries.}  
  We consider the real vector space $\poly\big(G, \RR_{(t-1)}\big)$, and note the conclusions of Proposition~\ref{prop:poly-space} (and Remark~\ref{rem:poly-space-ab}) in this setting.  In particular we write $e_{i,j}$ (where $0 \le i \le t-1$ and $j \in [\fd_i]$) for the basis given in Proposition~\ref{prop:poly-space}(iv), which has the property that $\{e_{i,j} \colon i \ge i'\}$ is a basis for $\poly\big(G, \RR_{(t-1)}\big)_{i'}$, which is the same as $\poly\big(G, \RR_{(t-1-i')}\big)$.

  We also consider the linear algebraic dual $V = \poly\big(G, \RR_{(t-1)}\big)^\ast$.  This again has the structure of a filtered vector space, by defining $V_i \subseteq V$ for $i=0,1,\dots$ by $V_0 = V$ and
  \[
    V_i = \big\{ \phi \in V \colon \phi(z) = 0 \ \forall z \in \poly\big(G, \RR_{(i-1)}\big) \big\}
  \]
  when $i \ge 1$; i.e., $V_i = \poly\big(G, \RR_{(i-1)}\big)^\perp$.  It follows that $V_0 \supseteq V_1 \supseteq \dots \supseteq V_{t} = 0$.

  The left action $\fs$ of $G$ on $\poly(G, \RR_{(t-1)})$ induces a dual action on $V$, i.e.\ sending a linear map $\phi \colon \poly\big(G, \RR_{(t-1)}\big) \to \RR$ in $V$ to $\fs(z,\phi) \colon F \mapsto \phi\big(\fs(z, F)\big)$, for $z \in G$.
  It follows from Proposition~\ref{prop:poly-space}(v) that for $\phi \in V_i$ and $z \in G_{i'}$ we have $\fs(z,\phi) \in V_i$ and $\phi - \fs(z,\phi) \in V_{i+i'}$.

  The basis $\{e_{i,j}\}$ above induces a dual basis $\{e_{i,j}^\ast\}$ for $V$.  If we change the labels, writing $e'_{i,j} = e_{t-1-i,j}^\ast$ for $j \in [\fd'_i] = [\fd_{t-1-i}]$, the remarks above imply that $\{e'_{i,j} \colon i \ge i'\}$ is a basis for $V_{i'}$, for each $i'$, $0 \le i' \le t-1$.
  
  We define an $\ell^1$-norm $\|\cdot\|_1$ on $V$ with respect to this basis.  By Proposition~\ref{prop:poly-space}(vi), for $z \in G$ the linear map $\fs(z) \colon V \to V$ obeys the operator norm bound
  \begin{equation}
    \label{eq:tau-l1-l1}
    \|\fs(z)\|_{\ell^1 \to \ell^1} \ll_s (M D)^{O_s(1)} (1 + d_G(\id_G, z))^{O_s(1)} .
  \end{equation}

  We also note that the lattice $\Upsilon \subseteq V$ generated by $\{e'_{i,j}\}$ is exactly the dual lattice of $\poly\big(\Gamma, \ZZ_{(t-1)}\big)$; that is,
  \[
    \Upsilon = \big\{ v \in V \colon v(F) \in \ZZ \ \forall F \in \poly\big(\Gamma, \ZZ_{(t-1)}\big) \big\}.
  \]

  Given $z \in G$, there is an element $\ev_z \in V$ defined by $\ev_z(F) = F(z)$; so, $\ev_z = \fs(z,\ev_{\id_G})$.  It is immediate that $\ev_\gamma \in \Upsilon$ for all $\gamma \in \Gamma$.
  A close inspection of the definition of the basis vectors $e_{i,j}$ (in the proof of Proposition~\ref{prop:poly-space}) shows that $e_{t-1,1}$ is the constant function $1$ and $e_{i,j}(\id_G) = 0$ whenever $i < t-1$; hence, $\ev_{\id_G} = e'_{0,1}$ is already one of our standard basis vectors.

  {\bf \noindent Construction of $G'$.}
  The intuition, such as it is, is that we want to build a machine which stores (roughly) three pieces of information: an element of $z \in G$, corresponding roughly to $r(x)$; an element of $t \in \RR$, corresponding roughly to $\lambda(x)$; and an element of $V$ which acts as an ``accumulator''.  It should support the following operations: we should be able to move around $G$ by right-multiplying by elements of $\Gamma$; and also translate $\RR$ by elements $m \in \ZZ$, but in the process we do something that behaves like adding $m F(z)$ to the accumulator.
  
  The idea is that when we correct a configuration $\omega \mapsto (r(c(\omega)), \lambda(c(\omega)), \dots)$ for $c \in C^{t+1}(H')$ to a Host--Kra cube in $\HK^{t+1}(G')$, as in the proof of Proposition~\ref{prop:nil-is-hierarchy}, the top-level correction will be exactly one of the form $\Lambda(c,\omega) F(r(c(\omega)))$ at each vertex.

  We now define $G'$ formally. As a set, we have $G' = G \times \RR \times V \times V$.\footnote{The first copy of $V$ was not mentioned in the informal sketch above: roughly it exists because the object has to be a group, and so we need a spare copy of $V$ for $G$ to act on by translation as we move around.}  The group law of $G'$ is the somewhat complicated semi-direct product construction
  \begin{align*}
    (z, a, u, v) \ast (z', a', u', v') &= \big(z z',\, a + a',\, u+\fs(z, u') ,\, v+\fs(z,v') + a' u \big).
  \end{align*}
  The filtration on $G'$ is the following: if $\RR_i$ is taken to be $\RR$ for $i \le s+1-t$ and $\{0\}$ otherwise (i.e., $\RR$ is $\RR_{(s+1-t)}$) then
  \[
    G'_i = G_i \times \RR_i \times V_i \times V_{\max(i-s-1+t,0)} ;
  \]
  so the filtration on the last copy of $V$ has been shifted $s+1-t \ge 1$ places to the right (which is acceptable given $V$ is an abelian group).

  The filtration is not proper, as $V_0 \ne V_1$ in general.  As usual, we may consider the subgroup $G'_1$ with the filtration $G'_1 = G'_1 \supseteq G'_2 \supseteq \dots$, which is proper, and $g$ will ultimately be defined on this smaller group; however, it is convenient in passing to work on all of $G'$.
  
  We make the following technical claim.
  \begin{claim}
    This group $G'_\bullet$ is indeed a filtered group, and $G'_1$ can be made into an explicitly presented filtered group with dimension $\ll_s D^{O_s(1)}$ and complexity $\ll_s (M D)^{O_s(1)}$, with integral subgroup $\Gamma' = \Gamma \times \ZZ \times \Upsilon \times \Upsilon$.
  \end{claim}
  \begin{proof}[Proof of claim]
    It is straightforward to verify that the law above does in fact define a group, and to evaluate the commutators
    \begin{align*}
      [(z, 0, 0, 0), (\id, 0, u, 0)] &= (\id, 0, u - \fs(z^{-1},u), 0) \\
      [(z, 0, 0, 0), (\id, 0, 0, v)] &= (\id, 0, 0, v - \fs(z^{-1}, u)) \\
      [(\id, a, 0, 0), (\id, 0, u, 0)] &= (\id, 0, 0, -a u)
    \end{align*}
    with other single-entry elements commuting.
    These verify that the filtration described really is a filtration.
    We can define one-parameter subgroups
    \begin{align*}
      a &\mapsto (\gamma_{i,j}^a, 0, 0, 0) &
      a &\mapsto (\id, 0, a e'_{i',j'}, 0) \\
      a &\mapsto (\id, a, 0, 0) &
      a &\mapsto (\id, 0, 0, a e'_{i',j'})
    \end{align*}
    for $1 \le i \le t-1$ and $j \in [D_i]$, or $0 \le i' \le t-1$ and $j' \in [\fd'_i]$.  It is clear (given the properties of $\gamma_{i,j}$) that any element of $G'$ is a product of such elements in a unique way.  If we omit the only instance $a \mapsto (\id, 0, a e'_{0,1}, 0)$ which does not lie in $G'_1$ from the list, this gives the subgroup $G'_1$ the structure of an explicitly presented filtered group, of degree $s$, dimension $\ll_s D^{O_s(1)}$ and complexity $\ll_s (M D)^{O_s(1)}$ as required.  The identification of the integral subgroup $\Gamma'$ is immediate.
  \end{proof}

  {\bf \noindent Construction of $g$.}  We now locate a function $r' \colon H' \to G'_1$ with the properties we require.

  Recall $H' = (\ZZ/N^{2 \kappa}\ZZ)^n$.  In general we omit explicit references to the usual quotient maps $H' \to H$ or $\ZZ^d \to H'$.
  \begin{claim}
    There is some function $v \colon H' \to V_1$, with $\|v(x)\|_1 \ll_s D^{O_s(1)}$ for each $x \in H'$, such that if we define $r' \colon H' \to G'$ by
    \[
      r'(x) = \big(r(x),\, \lambda(x),\, \ev_{r(x)} - \ev_{\id_G},\, v(x) \big)
    \]
    then $r' \bmod \Gamma' \colon H' \to G'/\Gamma'$ is a polynomial map.
  \end{claim}
  \begin{proof}[Proof of claim]
    Since $r \bmod \Gamma \colon H \to G/\Gamma$ is a polynomial map, we may extend it to a polynomial map $\ZZ^n \to G/\Gamma$ given by $x \mapsto r(x \bmod N^\kappa) \Gamma$.  By Proposition~\ref{prop:z-lift} we can then find a polynomial map $p \colon \ZZ^n \to G$ such that $p(x) \Gamma = r(x \bmod N^\kappa) \Gamma$ for all $x \in \ZZ^n$.

    Similarly, $\lambda \bmod \ZZ \colon H \to \RR/\ZZ$ is a polynomial map of degree $s+1-t$, so by Proposition~\ref{prop:z-lift} we may choose a polynomial map $\wt{\lambda} \colon \ZZ^n \to \RR_{(s+1-t)}$ such that $\wt{\lambda}(x) \bmod \ZZ = \lambda(x \bmod N^\kappa) \bmod \ZZ$ for all $x \in \ZZ^n$.

    We next define a polynomial map $p' \colon \ZZ^n \to G'$ by
    \begin{align*}
      p'(x) &= \big(\id_G, \wt{\lambda}(x), 0, 0\big) \ast \big(\id_G, 0, -\ev_{\id}, 0\big) \ast \big(p(x), 0, 0, 0\big) \ast \big(\id_G, 0, \ev_{\id}, 0\big) \\
      &= \big(p(x),\, \wt{\lambda}(x),\, \ev_{p(x)} - \ev_{\id_G},\, 0 \big) .
    \end{align*}
    By the first line definition, the comments in Example~\ref{ex:poly-examples}, and our hypotheses, this is a polynomial map.  By the second line, we note that $\ev_{z} - \ev_{\id_G} \in V_1$ for any $z \in G$ (since it evaluates to zero on constant functions), and hence $p'(x) \in G'_1$ for all $x$ and $p'$ is a polynomial map $\ZZ^n \to G'_1$.

    We claim that $p' \bmod \Gamma' \colon \ZZ^n \to G'_1 / \Gamma'_1$ descends to a polynomial map on $H' = (\ZZ/N^{2\kappa}\ZZ)^n$; that is, that $p'(x+y)^{-1} p'(x) \in \Gamma'$ for all $x \in \ZZ^n$ and $y \in N^{2\kappa} \ZZ^n$.
    To see this, we compute
    \[
      p'(x+y)^{-1} p'(x) = \Big( \begin{aligned}[t] &p(x+y)^{-1} p(x),\; \wt{\lambda}(x) - \wt{\lambda}(x+y),\; \ev_{p(x+y)^{-1} p(x)} - \ev_{\id_G},\\ &\; \big(\wt{\lambda}(x) - \wt{\lambda}(x+y)\big) \big(\ev_{p(x+y)^{-1}} - \ev_{\id_G}\big) \Big) . \end{aligned}
    \]
    Whenever $y \in N^\kappa \ZZ^n$ we know that $p(x+y)^{-1} p(x) \in \Gamma$, $\wt \lambda(x) - \wt \lambda(x+y) \in \ZZ$ and $\ev_{\id_G} \in \Upsilon$, so it suffices to see when $(\wt{\lambda}(x) - \wt{\lambda}(x+y)) \ev_{p(x+y)^{-1}} \in \Upsilon$.  By Proposition~\ref{prop:taylor} we may expand
    \[
      \wt{\lambda}(x) = \sum_{w(\ell) \le s+1-t} \alpha_\ell \binom{x}{\ell}
    \]
    for some $\alpha_\ell \in \RR$, where the sum is over all $n$-tuples of non-negative integers $(\ell_1,\dots,\ell_n)$ with $w(\ell) \le s+1-t$; in this case $w(\ell) = \sum_{j=1}^n \ell_j$, and $\binom{x}{\ell} = \prod_{j=1}^n \binom{x_j}{\ell_j}$ is the multivariate binomial coefficient as in~\eqref{eq:gen-binom}.
    Similarly, since $x \mapsto \ev_{p(x)^{-1}}$ is a polynomial map $\ZZ^n \to V$ it may be expanded as
    \[
      \ev_{p(x)^{-1}} = \sum_{i=0}^{t-1} \sum_{j=1}^{\fd_i} \sum_{w(k) \le t-1} \beta_{i,j,k} \binom{x}{k} e'_{i,j}
    \]
    for some $\beta_{i,j,k} \in \RR$, by Proposition~\ref{prop:taylor} applied to each coordinate of $V$.
    Provided $N > s$ (say) and so the denominators of binomial coefficients are coprime to $N$, the maps $x \mapsto \ev_{p(x)^{-1}} \bmod \Upsilon$ and $x \mapsto \wt{\lambda}(x) \bmod \ZZ$ are $N^\kappa$-periodic if and only if $\alpha_\ell, \beta_{i,j,k} \in N^{-\kappa} \ZZ$ for each $\ell$ and $i,j,k$.  It follows that if $y \in N^{2\kappa} \ZZ^n$ then $\wt{\lambda}(x) - \wt{\lambda}(x+y) \in N^{\kappa} \ZZ$, $\ev_{p(x+y)^{-1}} \in N^{-\kappa} \Upsilon$ and hence their product lies in $\Upsilon$, as required.

    We will now choose $v \colon H' \to V$, and thereby $r' \colon H' \to G'_1$ with the form above, so that $p'(x) \Gamma' = r'(x) \Gamma'$ for every $x \in H'$.  We may compute
    \[
      p'(x)^{-1} r'(x) 
      = \begin{aligned}[t] \Big(&p(x)^{-1} r(x),\; \lambda(x) - \wt \lambda(x),\; \ev_{p(x)^{-1} r(x)} - \ev_{\id_G},\; \\
        &\fs\big(p(x)^{-1}, v(x)\big) + \big(\lambda(x) - \wt \lambda(x)\big)\big(\ev_{p(x)^{-1}} - \ev_{\id_G}\big) \Big)
      \end{aligned}
    \]
    and this lies in $\Gamma'$ if and only if
    \[
      \fs\big(p(x)^{-1}, v(x)\big) \in \big(\wt \lambda(x) - \lambda(x)\big)\big(\ev_{p(x)^{-1}} - \ev_{\id_G}\big) + \Upsilon .
    \]
    Using the fact that $\fs(p(x)^{-1}, u) - u \in V_{i+1}$ whenever $u \in V_i$, and induction on $i$, it is clear that there is an unique choice of $v(x)$ for each $x$ that obeys this condition and also has all its coordinates $v(x)_{i,j} \in [0,1)$ for $0 \le i \le s$ and $j \in [\fd'_i]$.  We take this as the definition of $v(x)$, and thereby of $r'$. %chktex 9

    Noting that $\ev_{p(x)^{-1}} - \ev_{\id_G} \in V_1$ again, this unique choice will necessarily have $v(x)_{=0} = 0$, so the construction implies that $v(x) \in V_1$ for all $x$.
    It is also clear that $r' \bmod \Gamma'_1 \colon H' \to G'_1 / \Gamma'_1$ is a polynomial map, since it is the same as $p' \bmod \Gamma'_1$.
  \end{proof}

  We have $d_G(\id_G, r(x)) \le M$ for all $x \in H$ by hypothesis, and therefore $\|\fs(r(x), \ev_{\id})\|_1 \ll_s (M D)^{O_s(1)}$ for each $x \in H'$ by~\eqref{eq:tau-l1-l1}.  Moreover, $|\lambda(x)|_1 \le M$ for all $x \in H$ by hypothesis, and $\|u(x)\|_1 \ll_s D^{O_s(1)}$ for all $x \in H'$.  It follows that $d_{G'}(\id_{G'}, r'(x)) \ll_s (MD)^{O_s(1)}$ for all $x \in H'$.

  We now define $F' \colon G' \to \RR_{(s)}$ by $F'(z, a, u, v) = v(F) + a F(\id_G)$, which is a polynomial map. Indeed, the second factor of $V$ carries a filtration of degree $s$, and any linear function from a degree $s$ filtered abelian group to $\RR_{(s)}$ is a polynomial map; and similarly a linear function $\RR_{(s+1-t)} \to \RR_{(s)}$ is always a polynomial map (noting $t \ge 1$).

  Hence, $g := F' \circ r' \colon H' \to \RR$ is a nil-polynomial of degree $s$ and with dimension and $\ll_s D^{O_s(1)}$ and complexity $\ll_s (D M)^{O_s(1)}$, as required.

  {\bf \noindent The derivatives condition on $g$.}  We can apply Proposition~\ref{prop:nil-is-hierarchy} to the original nil-polynomial to get a polynomial hierarchy $f''$ of degree $t-2$, dimension $\ll_s D^{O_s(1)}$ and complexity $\ll_s (D M)^{O_s(1)}$, such that $f$ obeys an $\big(f'', t,t-1, O_s(D M)^{O_s(1)}\big)$-derivatives condition.  If $t=1$ we can skip this step and take $f''$ to be the empty tuple.

  It follows in either case that the tuple $f'$ with $f'_{\le t-2} = f''$ and with $f'_{=t-1}$ the single function $f$, is again a polynomial hierarchy, having degree $t-1$, dimension $\ll_s D^{O_s(1)}$ and complexity $\ll_s (D M)^{O_s(1)}$, and of the form required by the statement.  So, it suffices to show that $g$ obeys a derivatives condition on $f'$.

  To do this we first apply Lemma~\ref{lem:nil-derivs}, or more accurately its proof, to $g$.  Given $c \in C^{s+1}(H')$, as before we can use Proposition~\ref{prop:effective-cubes} to find a configuration
  \[
    \gamma'(c,\omega) = \big(\gamma(c,\omega),\; \alpha(c; \omega),\; \beta(c,\omega),\; \theta(c,\omega) \big) \in \Gamma'
  \]
  with $d_{G'}\big(\id_{G'}, \gamma'(c,\omega)\big) \ll_s (D M)^{O_s(1)}$, such that $\omega \mapsto r(c(\omega)) \gamma'(c(\omega))$ lies in $\HK^{s+1}(G')$.  Given that $r(c(\omega)) \in G \times \RR \times V_1 \times V_1$ by construction, we can further insist that $\beta, \theta \in V_1$.  Moreover, since the configuration $\omega \mapsto \lambda(c(\omega)) - \Lambda(c,\omega)$ lies in $\HK^{s+1}(\RR_{(s+1-t)})$, by Remark~\ref{rem:normal-form}, multiplying on the right by $\big(\id_G, -\alpha(c,\omega) - \Lambda(c,\omega), 0,0\big)$ if necessary we can assume that $\alpha = -\Lambda$. Note this does not significantly change the bound on $d_{G'}\big(\id_{G'}, \gamma'(c,\omega)\big)$.
  
  Then as before we have that
  \[
    \partial^{s+1} g(c) = \sum_{\omega \in \llbracket s+1 \rrbracket} (-1)^{|\omega|} \Big[ F'\big(r(c(\omega))\big) - F'\big(r(c(\omega)) \gamma'(c,\omega)\big) \Big]
  \]
  and we can further compute
  \[
    F'\big(r(c(\omega)) \gamma'(c,\omega)\big) = \begin{aligned}[t] &\big(\lambda(c(\omega)) - \Lambda(c,\omega)\big) F(\id_G) \\
      +& \Big(v(c(\omega)) + \fs\big(r(c(\omega)), \theta(c,\omega)\big) - \Lambda(c,\omega) \big(\ev_{r(c(\omega))} - \ev_{\id_G}\big) \Big)(F)
    \end{aligned}
  \]
  and so
  \[
    F'\big(r(c(\omega))\big) - F'\big(r(c(\omega)) \gamma'(c,\omega)\big) = \Lambda(c,\omega) F(r(c(\omega))) - \theta(c,\omega)\big(\fs(r(c(\omega)),F)\big) .
  \]
  The first term is precisely $\Lambda(c,\omega) f(c(\omega))$, as desired.  For the second term, we can certainly expand
  \[
    F = \sum_{i=0}^{t-1} \sum_{j=1}^{\fd_{i}} \xi_{i,j} e_{i,j}
  \]
  for real numbers $\xi_{i,j}$, as in Proposition~\ref{prop:nil-is-hierarchy}(i), and
  \[
    \theta(c,\omega) = \sum_{i'=1}^{t-1} \sum_{j'=1}^{\fd'_{i'}} \theta(c,\omega)_{i',j'} e'_{i',j'}
  \]
  in the basis for $V_1$ (recalling $\theta(c,\omega) \in V_1 \cap \Upsilon$). Hence, to understand the functions $x \mapsto \theta(c,\omega)\big(\fs(r(x), F)\big)$ it suffices to understand the maps $G \to \RR$ of the form $z \mapsto e'_{i',j'} \big(\fs(z, e_{i,j})\big)$. In other words, these are the matrix coefficients of the left translation action of $G$ on $\poly\big(G, \RR_{(s)}\big)$, as a function of $z \in G$.
  
  Such functions do turn out to be polynomial maps $G \to \RR$ (of degree $i-i'$), and so can in turn be expressed as a linear combination of our standard basis vectors, with small integer coefficients.  However, this statement requires a small amount of theory to prove.
  \begin{claim}
    For any $i,j$ and $i',j'$, the map $\phi(z) = \fs(z, e'_{i',j'})(e_{i,j}) = e'_{i',j'}(\fs(z, e_{i,j}))$ is a polynomial  map $G \to \RR_{(i-i')}$, and admits a decomposition
    \[
      \phi = \sum_{i_1=0}^{i-i'} \sum_{j_1=1}^{\fd_{j_1}} p_{i_1,j_1}^{\alpha_{i_1,j_1}}
    \]
    where $\alpha_{i_1,j_1}$ are integers \uppar{depending on $i,j$ and $i',j'$} with $\sum_{i_1,j_1} |\alpha_{i_1,j_1}| \ll_{s} (D M)^{O_s(1)}$.
  \end{claim}
  \begin{proof}[Proof of claim]
    For any $h_1 \in G_{i_1},\dots,h_k \in G_{i_k}$ the function $p' = \partial_{h_1} \dots \partial_{h_k} e_{i,j}$ is in $\poly\big(G, \RR_{(i - i_1 - \cdots - i_k)}\big)$ by Proposition~\ref{prop:poly-equiv}, and hence so is $\fs(z, p')$ for any $z \in G$.  By definition of $V_{i'}$, $e'_{i',j'}$ vanishes on $\poly\big(G, \RR_{(i'-1)}\big)$, and so $e'_{i',j'}(\fs(z,e_{i,j}))$ is zero whenever $i_1+\cdots+i_k \ge i-i'+1$. By Proposition~\ref{prop:poly-equiv} again this proves that $\phi$ is a polynomial map of degree $i-i'$.

    It is clear $\phi(\Gamma) \subseteq \ZZ$.  Decomposing in the basis $\{e_{i_1,j_1}\}$ as in the statement, it follows that $\alpha_{i_1,j_1}$ are integers.

    Since the values $\|\phi(z)\|_1$ for $z \in \Gamma$ with $d_{G}(\id_G, \gamma) \le R$ are $\ll_{s} (R D M)^{O_{s}(1)}$---this is equivalent to Proposition~\ref{prop:poly-space}(vi)---we conclude by Proposition~\ref{prop:poly-space}(vii) that $\sum_{i_1,j_1} |\alpha_{i_1,j_1}| \ll_{s} (D M)^{O_s(1)}$.
  \end{proof}
  
  Hence, each value $\theta(c,\omega)(\fs(r(x), F))$ has the form $\sum_{i,j} b_{i,j} f_{i,j}(x)$, where $f_{i,j}$ are the functions in Proposition~\ref{prop:nil-is-hierarchy} and $b_{i,j}$ are integers with $\|b\| \ll_s (D M)^{O_s(1)}$.  This proves the required derivatives condition on $g$, and so completes the proof of Lemma~\ref{lem:poly-construction}.
\end{proof}

The proof of Lemma~\ref{lem:extension-fix} has many similar elements, but slightly less involved.

\begin{proof}[Proof of Lemma~\ref{lem:extension-fix}]
  Let $G_\bullet$, $r \colon H' \to G$ and $F \colon G \to \RR_{(s)}$ be the data associated to the nil-polynomial $g' \colon H' \to \RR$; so in particular $G_\bullet$ is an explicitly presented flitered group of degree $s$, dimension $D$ and complexity $M$.
  
  We consider $G' = \poly(\ZZ^n, G_\bullet)$ where $\ZZ^n$ carries the standard filtration.  Again this is a special case of the set-up of Proposition~\ref{prop:poly-space}, and the conclusions apply.
  
  In particular we recall that if $p \in G'$ and $t \in \RR^n$ then $\tau(p,t)$ denotes the shift $x \mapsto p(x+t)$.  (Since $\ZZ^n$ is abelian we need not distinguish between a left and a right action.)  We can therefore define a semi-direct product $\wt{G} = \RR^n \ltimes G'$ accordingly; i.e., $\wt{G}$ consists of pairs $\RR^n \times G'$ with the group law
  \[
    (t, p) \ast (t', p') = (t + t',\; \tau(p, t') p') .
  \]
  We give this the filtration $\wt{G}_i = (\RR^n)_i \ltimes \poly\big(\RR^n, G_\bullet\big)_i$, where $(\RR^n)_i$ is $\RR^n$ if $i=0,1$ and zero otherwise (i.e.\ the standard filtration).  By Proposition~\ref{prop:poly-space}(v) this is indeed a filtration, although again it is not proper.  Moreover, taking the one-parameter subgroups $(0, p_{i,j}^a)$ as in Proposition~\ref{prop:poly-space}(iv) together with $a \mapsto (a e_j, \id)$ for each basis vector $e_j$ of $\RR^n$ ($1 \le j \le n$), the subgroup $\wt{G}_1$ is an explicitly presented filtered group of degree $s$, dimension $\ll_s n^{O_s(1)} D$ and complexity $\ll_s (D M n)^{O_s(1)}$.  The integral subgroup is identified with $\ZZ^n \ltimes \poly(\ZZ^n, \Gamma)$.

  We now construct maps $\wt{r} \colon H \to \wt{G}_1$ and $\wt{F} \colon \wt{G}_1 \to \RR_{(s)}$ such that $g = \wt{F} \circ \wt{r}$ becomes a nil-polynomial, for $g$ as in the statement.

  First note that, since $r \bmod \Gamma \colon H' \to G/\Gamma$ is a polynomial map, we can define a polynomial map $\ZZ^n \to G/\Gamma$ (by $x \mapsto r(x \bmod N^\kappa) \Gamma$), and by Proposition~\ref{prop:z-lift} we can find a polynomial map $p' \colon \ZZ^n \to G$ such that $p'(x) \Gamma = r(x \bmod N^\kappa) \Gamma$ for all $x \in \ZZ^n$.

  We can extend $p'$ uniquely to a polynomial map $\RR^n \to G$ (Proposition~\ref{prop:poly-space}(iii)) and rescale to get $p \colon \RR^n \to G$ given by $p(x) = p'(N^\kappa x)$.  Finally, using Proposition~\ref{prop:fun-dom} and Proposition~\ref{prop:poly-space}(iv) again, we write $p = P p_0$ where $p_0 \in \poly(\ZZ^n, \Gamma)$ and $\sigma_{G'}(P) \in [0,1)^{\dim(G')}$ lies in the fundamental domain (see Definition~\ref{def:explicit-group} concerning the coordinate map $\sigma_{G'}$). %chktex 9

  We define a map $\wt{p} \colon \ZZ^n \to \wt{G}$ by
  \[
    \wt{p}(x) = (0, p_0) \ast \big(x/N^\kappa, \id_{G'}\big) \ast \big(0, p_0^{-1}\big) = \big(x/N,\; \tau(p_0,x/N)\, p_0^{-1}\big) ;
  \]
  this is a polynomial map by the comments in Example~\ref{ex:poly-examples} (in particular, constant functions and suitable group homomorphism are polynomial maps).  Moreover by Proposition~\ref{prop:poly-space}(v), $\wt{p}(x) \in \wt{G}_1$ for all $x$, so the same expression defines a polynomial map $\ZZ^n \to \wt{G}_1$.  Finally note
  \[
    \wt{p}(x) \wt{\Gamma} = (0, p_0) \ast \big(x/N^\kappa, \id_{G'}\big) \wt{\Gamma}
  \]
  which is unchanged under replacing $x$ by $x + N^\kappa \lambda$ for any $\lambda \in \ZZ^n$; i.e., $\wt{p}(x) \bmod \wt{\Gamma}$ descends to a polynomial map $H \to \wt{G} / \wt{\Gamma}$ or $\wt{G}_1 / \wt{\Gamma_1}$.

  We define $\wt{F} \colon \wt{G} \to \RR_{(s)}$ by setting $F'(t, \rho) = F(\rho(0))$ (which is a polynomial map) and then $\widetilde{F}(y) = F'((0, P) \ast y)$.

  Finally we define $\wt{r}$ as follows.  For any $g \in G$ let $\{x\}_{G}$ denote the element of the fundamental domain (in the sense of Proposition~\ref{prop:fun-dom}) with $\{x\}_G \Gamma = x \Gamma$, and define $\{\cdot\}_{G'}$ similarly.  Note $\{\rho\}_{G'}(0) = \{\rho(0)\}_{G}$.  Then for $x \in \{0,\dots,N^\kappa-1\}^n$ we set
  \[
    \wt{r}(x) = \big(x/N^\kappa,\; \{\tau(p_0, x/N^\kappa)\}_{G'}\; \{p_0(x/N^\kappa)\}_G^{-1}\; P(x/N^\kappa)^{-1}\; r(x)\big)
  \]
  where we allow elements of $G$ to appear as constant functions $\ZZ^n \to G$ in $\poly(\ZZ^n, G)$.
  It follows from our various bounds and the triangle inequality that $d_{\wt{G}}(\id, \wt{r}(x)) \ll_s (n M D)^{O_s(1)}$ (using the fact that $d_G(\id_G, P(t)) \ll_s D$ whenever $t \in [0,1)$).  Also, for all such $x$, %chktex 9
  \[
    (0, P) \ast \wt{r}(x) = \big(x/N^\kappa,\; \tau(P,x/N^\kappa)\; \{\tau(p_0,x/N^\kappa)\}_{G'}\ \{p_0(x/N^\kappa)\}_G^{-1}\ P(x/N^\kappa)^{-1}\ r(x)\big)
  \]
  and hence
  \[
    \wt{F}(\wt{r}(x)) = F\big(P(x/N^\kappa)\;\{p_0(x/N^\kappa)\}_G \;\{p_0(x/N^\kappa)\}_G^{-1}\; P(x/N^\kappa)^{-1} r(x)\big) = F(r(x))
  \]
  as required.

  Finally we verify that $\wt{r}(x) \wt{\Gamma} = \wt{p}(x) \wt{\Gamma}$ for every $x \in \{0,\dots,N^\kappa-1\}^{n}$; as we already know the latter defines a polynomial map $H \to \wt{G}_1/\wt{\Gamma}_1$, this complete the proof that $\wt{F} \circ \wt{r}$ is a nil-polynomial.  We recall that $r(x) \Gamma = p(x/N^\kappa) \Gamma = P(x/N^\kappa) p_0(x/N^\kappa) \Gamma$ for all $x$, and hence
  \[
    p_0(x/N^\kappa)^{-1} \, P(x/N^\kappa)^{-1} \, r(x) \in \Gamma . 
  \]
  By definition $\{p_0(x/N^\kappa)\}_{G} \Gamma = p_0(x/N^\kappa) \Gamma$, so
  \[
    \wt{r}(x) \wt{\Gamma} = \big(x/N^\kappa,\; \{\tau(p_0,x/N^\kappa)\}_{G'} \big) \wt{\Gamma} = \big(x/N^\kappa,\; \tau(x/N^\kappa, p_0)\, p_0^{-1}\big) \wt{\Gamma}
  \]
  since $(0, p_0) \in \wt{\Gamma}$.  But the right hand side matches $\wt p(x) \wt \Gamma$, as required, and this completes the proof.
\end{proof}

This completes the proof of Theorem~\ref{thm:poly-structure}.  Combining this with Corollary~\ref{cor:poly-hierarchy-everything} completes the proof of Theorem~\ref{thm:1pc}.

\section{Deduction of the inverse theorem}%
\label{sec:deduction}

We now prove Theorem~\ref{thm:main}. As stated above, the most significant ingredient is Theorem~\ref{thm:1pc}.

The argument deducing Theorem~\ref{thm:main} from Theorem~\ref{thm:1pc} is very much related to the original work of Gowers~\cite{gowers-kap}, and in the special case $s=2$ to~\cite{green-tao-u3}; the recent~\cite{gowers-milicevic} also employs a similar strategy in this phase.  The approach taken in~\cite{gtz} is rather different.

However, there are some technical differences between what we do here and what is done in all these other works.  Notably,~\cite{gowers-kap} and~\cite{gowers-milicevic} locate an object with approximate multilinear structure, in a suitable sense, whereas we work with approximate polynomial structure in the sense of Theorem~\ref{thm:1pc}.  This means we require a new Cauchy--Schwarz argument to locate approximate polynomial structure in Proposition~\ref{prop:norm-inequality} below, which is not entirely straightforward.

Another difference concerns the method for passing from this approximate polynomial object back to the required nilsequence.  The other works employ a ``symmetry argument'' and an anti-differentiation step, both of which seem hard to apply in this setting.  Here we instead use a sampling method and yet another Cauchy--Schwarz argument; see Section~\ref{subsec:antidiff}.

We also have to take more care to avoid poor bounds in the final step, where we apply the inverse theorem inductively to a lower-degree object.  A direct approach would give a final bound involving $O(s)$ iterated exponentials. To avoid this, we must maintain global rather than local control of certain functions, using a partition of unity argument.  This significantly adds to the work in Section~\ref{subsec:hs-nil}.

\subsection{Identifying a local polynomial object}

\newcommand{\fsig}[0]{\wt{\cS}}

For this subsection we let $H$ be any finite abelian group.  Suppose $f \colon H \to \CC$ is some function.

We define a function
\begin{align*}
  \cS_s f \colon H^s &\to \CC \\
  (h_1,\dots,h_s) &\mapsto \EE_{x \in H} \prod_{\omega \in \llbracket s \rrbracket} \cC^{|\omega|} f(x + \omega \cdot \vec{h})
\end{align*}
where $\cC$ is the complex conjugation operator ($\cC^m(f) = \overline{f}$ if $m$ is odd and $f$ if $m$ is even), $|\omega|$ denotes $\sum_{i=1}^s \omega_i$, and we write $\vec h = (h_1,\dots,h_s)$ and $\omega \cdot \vec h = \sum_{i \in \omega} h_i$.

We also define
\begin{align*}
  \fsig_s f \colon H^{s-1} \times \widehat{H} &\to \CC \\
  (h_1,\dots,h_{s-1},\chi) &\mapsto \EE_{h}\ \chi(-h) \cS_s f(h_1,\dots,h_{s-1},h) ,
\end{align*}
the Fourier transform of $\cS_s f$ in its last argument (although it doesn't matter which argument, as $\cS_s f$ is symmetric in $h_1,\dots,h_s$).

We record some basic properties of $\fsig_s$ (see~\cite{gowers-kap}).
\begin{proposition}%
  \label{prop:sig-facts}
  For any $s \ge 1$ and any function $f \colon H \to \CC$ with $\|f\|_\infty \le 1$ and $\fsig_s f$ defined as above, the following hold:
  \begin{enumerate}
    \item\label{item:sig-facts-1} for any $h_1,\dots,h_{s-1} \in H$ and $\chi \in \widehat{H}$, the value $\fsig_s f(h_1,\dots,h_{s-1},\chi)$ is real and non-negative;
    \item\label{item:sig-facts-2} for any fixed $h_1,\dots,h_{s-1} \in H$, the sum
      \[
        \sum_{\chi \in \widehat{H}} \fsig_s f(h_1,\dots,h_{s-1},\chi)
      \]
      lies in $[0,1]$;
    \item\label{item:sig-facts-3} there is an identity
      \[
        \EE_{h_1,\dots,h_{s-1} \in H} \sum_{\chi \in \widehat{H}} \fsig_s f(h_1,\dots,h_{s-1},\chi)^2 = \| f\|_{U^{s+1}}^{2^{s+1}} .
      \]
  \end{enumerate}
\end{proposition}
\begin{proof}
  Writing $\vec h = (h_1,\dots,h_{s-1})$, part (i) follows from the identity
  \begin{align*}
    \fsig_s f(\vec h,\chi) &= \EE_{h, x \in H}\ \chi(-h) \prod_{\omega \in \llbracket s \rrbracket} \cC^{|\omega|} f\big(x + \omega \cdot (\vec h,h)\big) \\
    &= \EE_{h, x \in H}\ \chi(-h) \prod_{\omega \in \llbracket s-1 \rrbracket} \cC^{|\omega|} f\big(x + \omega \cdot \vec h\big) \prod_{\omega \in \llbracket s-1 \rrbracket} \overline{\cC^{|\omega|} f\big(x + h + \omega \cdot \vec h\big)}  \\
    &= \left| \EE_{x \in H}\ \chi(x) \prod_{\omega \in \llbracket s-1 \rrbracket} \cC^{|\omega|} f\big(x + \omega \cdot \vec h \big)  \right|^2 .
  \end{align*}
  For part (ii), note by Fourier inversion that
  \[
    \sum_{\chi \in \widehat{H}} \fsig_s f(h_1,\dots,h_{s-1},\chi) = \cS_s f(h_1,\dots,h_{s-1},0) = \EE_{x \in H}   \prod_{\omega \in \llbracket s \rrbracket} \cC^{|\omega|} f(x + \omega \cdot (h_1,\dots,h_{s-1},0)) 
  \]
  and as the right hand side is an average of complex numbers of absolute value at most $1$, it is bounded by $1$.

  For (iii), by Parseval's identity,
  \[
    \EE_{h_1,\dots,h_{s-1} \in H} \sum_{\chi \in \widehat{H}} \left| \fsig_s f(h_1,\dots,h_{s-1},\chi) \right|^2 = \EE_{h_1,\dots,h_s} \left| \cS_s f (h_1,\dots,h_s)\right|^2
  \]
  and the right hand side can be expanded as
  \begin{align*}
    &\phantom{=} \EE_{h_1\dots,h_s} \EE_{x,y} \prod_{\omega \in \llbracket s \rrbracket} \cC^{|\omega|} f(x + \omega \cdot \vec h) \prod_{\omega \in \llbracket s \rrbracket}  \overline{\cC^{|\omega|} f(y + \omega \cdot \vec h)} \\
    &= \EE_x \EE_{h_1,\dots,h_{s+1}} \prod_{\omega \in \llbracket s+1 \rrbracket}\cC^{|\omega|} f\big(x + \omega \cdot (h_1,\dots,h_{s+1}) \big) 
  \end{align*}
  which is exactly the definition of $\|f\|_{U^{s+1}}^{2^{s+1}}$.
\end{proof}

The intuition when considering $\fsig_s f$ is as follows.  If $f(x)$ were actually a phase polynomial of degree $s$ such as $f(x) = e(\alpha x^s)$ (where $H = \ZZ/N\ZZ$ and $\alpha \in \frac1N \ZZ$) then the repeated multiplicative derivative $\cS_s f$ is given by
\[
  \prod_{\omega \in \llbracket s \rrbracket} \cC^{|\omega|} f(x + \omega \cdot \vec{h}) = e(s! \,\alpha\, h_1 \dots h_s)
\]
which is a multilinear function in $h_1,\dots,h_s$.  Specifically, for $h_1,\dots,h_{s-1}$ fixed, this is a character in the variable $h_{s}$ with frequency $s!\, \alpha\, h_1 \dots h_{s-1}$, so when we take Fourier transform in $h_{s}$ we find that
\[
  \fsig_s f(h_1,\dots,h_{s-1}, \chi) = \begin{cases} 1 &\colon \chi = \big(t \mapsto e(s!\, \alpha\, h_1 \dots h_{s-1}\, t) \big) \\ 0 &\colon \text{otherwise.} \end{cases}
\]
So, $\fsig_s f$ is the graph of a multilinear function $H^{s-1} \to \widehat{H}$.

When $f$ is not a phase polynomial but merely has large $U^{s+1}$-norm, some weak analogues of these facts are preserved.  Specifically, $\fsig_s f$ is large on the graph of some ``$1\%$ multilinear'' function $H^{s-1} \to \widehat{H}$.  This is the approach taken in~\cite{gowers-kap, gowers-milicevic}, for some suitable precise notion of ``$1\%$ multilinear''.

Note that a multilinear function $H^{s-1} \to \widehat{H}$ is a special case of a polynomial function $H^{s-1} \to \widehat{H}$ of degree $s-1$.  So, an alternative but related approach, we will follow here, is instead to show that $\fsig_s f$ looks like the graph of a $1\%$ polynomial function in the sense made precise above.

Specifically, we show the following.
\begin{proposition}%
  \label{prop:norm-inequality}
  Let $f \colon H \to \CC$ be any function with $\|f\|_{\infty} \le 1$, and let $\fsig_s f \colon H^{s-1} \times \widehat{H} \to [0,1]$ be as above.  Let $\tau \colon H^{s-1} \times \widehat{H} \to [0,1]$ be any function with $\sum_{\chi} \tau(h, \chi) \le 1$ for every $h \in H^{s-1}$, and such that the inner product $\langle \fsig_s f, \tau \rangle  = \EE_{h \in H^{s-1}} \sum_{\chi \in \widehat H} \fsig_s f(h,\chi) \tau(h,\chi)$ is at least $\delta$.

  Then $\tau$ respects many $s$-cube configurations on $H^{s-1}$, in the following sense:
  \[
    \EE_{a, t_1,\dots,t_{s} \in H^{s-1}} \sum_{\substack{(\chi_\omega)_{\omega \in \llbracket s \rrbracket} \in \widehat{H} \\ \sum_{\omega \in \llbracket s \rrbracket} \chi_\omega = 0}}\  \prod_{\omega \in \{0,1\}^s} \tau \big(a + \omega \cdot (t_1,\dots,t_s), \chi_\omega\big) \ge \delta^{M}
  \]
where\footnote{This value is certainly not best possible, although it is unclear by how much. In general we will not worry too much about optimizing constants that depend only on $s$.} $M = 2^{s + 2^{s-1}}$.
\end{proposition}

In particular this applies when $\tau = \fsig_s f$ and $\|f\|_{U^{s+1}}$ is large.  The case we are most interested in, however, is when $\tau$ is $\{0,1\}$-valued and hence the indicator of the graph of some (partial) function $H^{s-1} \to \widehat{H}$.  We record the conclusion in this case.

\begin{corollary}%
  \label{cor:found-1pc}
  Let $f \colon H \to \CC$ be any function with $\|f\|_{\infty} \le 1$, let $A \subseteq H^{s-1}$ be some set and $\kappa \colon A \to \widehat{H}$ some function such that $\fsig f(x, \kappa(x)) \ge \eps$ for all $x \in A$.  Then $\kappa$ is an approximate polynomial with parameter $(\eps \mu(A))^M$, where $M = 2^{s + 2^{s-1}}$.
\end{corollary}

We pay a modest price for working with $1\%$ polynomial functions rather than $1\%$ multilinear ones, in that Proposition~\ref{prop:norm-inequality} is harder to prove than the analogous multilinear statement.  As might be expected, the proof proceeds by multiple applications of the Cauchy--Schwarz inequality; however, both the number of applications required and the care involved in choosing them are greater.

\begin{proof}[Proof of Proposition~\ref{prop:norm-inequality}]
  We first remove all the Fourier transforms.  Let
  \begin{align*}
    T \colon H^s &\to \CC \\
      (h_1,\dots,h_s) &\mapsto \sum_{\chi \in \widehat{H}} \chi(h_s) \tau(h_1,\dots,h_{s-1},\chi)
  \end{align*}
  be the inverse Fourier transform of $\tau$ in its last argument.  By the hypothesis on $\tau$ we have $\|T\|_\infty \le 1$.  By Parseval, the correlation hypothesis $\langle \fsig_s f, \tau \rangle \ge \delta$ from the statement is equivalent to $|\langle \cS_s f, T \rangle| \ge \delta$ (where $\langle \cS_s f, T \rangle = \EE_{x \in H^s} \overline{f(x)} T(x)$).  Finally, the desired conclusion is equivalent to the lower bound
  \[
    \left| \EE_{a, t_1,\dots,t_{s} \in H^{s-1}}\ \EE_{y \in H}\ \prod_{\omega \in \llbracket s \rrbracket} T \big(a + \omega \cdot (t_1,\dots,t_s), y \big) \right| \ge \delta^{M} ,
  \]
  under taking Fourier transforms.  For each $y \in H$ we define $T_y \colon H^{s-1} \to \CC$ by $T_y(h_1,\dots,h_{s-1}) = T(h_1,\dots,h_{s-1},y)$ and $f_y \colon H \to \CC$ by $f_y(x) = f(x) \overline{f(x+y)}$.  It suffices to show the following.
  \begin{lemma}%
    \label{lem:norm-inequality}
    For any functions $k \colon H \to \CC$ and $R \colon H^{s-1} \to \CC$ with $\|k\|_\infty,\, \|R\|_\infty \le 1$, we have
    \[
      |\langle \cS_{s-1} k, R \rangle| \le \|R\|_{U^s}^{1/M'}
    \]
  where $M' = 2^{2^{s-1}}$.
  \end{lemma}
  Indeed, given this we could deduce that
  \[
    \delta \le \EE_y |\langle \cS_{s-1} f_y, T_y \rangle | \le \EE_y \left(\|T_y\|_{U^s}^{1/M'} \right)
    \le \left(\EE_y \|T_y\|_{U^s}^{2^s} \right)^{1/2^s M}
  \]
  by monotonicity of $\ell^p$ norms, and Proposition~\ref{prop:norm-inequality} follows.

  \begin{proof}[Proof of Lemma~\ref{lem:norm-inequality}]
    For any collection of functions $\big(k_\omega \colon H \to \CC \colon \omega \in \llbracket s-1\rrbracket\big)$, write
    \begin{align*}
      \cS_{s-1} (k_\omega) \colon H^{s-1} &\to \CC \\
      (h_1,\dots,h_{s-1}) &\mapsto \EE_x \prod_{\omega \in \llbracket s-1 \rrbracket} \cC^{|\omega|} k_\omega(x + \omega \cdot \vec h)
    \end{align*}
    for the multilinear generalization of $\cS_{s-1}$.  Also, we recall the notion of a \emph{generalized convolution} of degree $k$ (see~\cite{cfz-expo}) which is the function on an abelian group $G$ given by
    \begin{align*}
      \ast(F_1,\dots,F_{k+1}) \colon G &\to \CC \\
      x &\mapsto \EEE_{\substack{y_1,\dots,y_{k+1} \in G \\ y_1 + \cdots + y_k = x}} F_1(y_2,y_3,\dots,y_{k+1}) F_2(y_1,y_3,\dots,y_{k+1}) \dots F_{k+1}(y_1,y_2,\dots,y_k)
    \end{align*}
    where $F_i \colon G^k \to \CC$ are given functions.  A standard application of the Cauchy--Schwarz inequality $k+1$ times shows that for any function $r \colon G \to \CC$ and functions $F_i \colon G \to \CC$ with $\|F_i\|_\infty \le 1$, we have
    \begin{equation}
      \label{eq:csc}
      \big|\EE_{x \in G}\, \mathord{\ast}(F_1,\dots,F_{k+1})(x)\, r(x) \big| \le \|r\|_{U^{k+1}} .
    \end{equation}
    We first show the following.
    \begin{claim}%
      \label{claim:conv-is-easy}
      If each function $k_\omega \colon H \to \CC$ for $\omega \in \llbracket s-1 \rrbracket$ is a generalized convolution of degree $s-1$ of $1$-bounded functions $F_{\omega,i} \colon H^{s-1} \to \CC$, and $R \colon H^{s-1} \to \CC$ is any function, then
      \[
        \left|\langle \cS_{s-1}(k_\omega), R \rangle \right| \le \|R\|_{U^s} .
      \]
    \end{claim}
    \begin{proof}[Proof of claim]
      We may expand out the left hand side as
      \begin{align*}
        &\EEE_{x \in H} \EEE_{h \in H^{s-1}} \overline{R(h)} \prod_{\omega \in \llbracket s-1 \rrbracket} \cC^{|\omega|} \mathord{\ast}(F_{\omega,1},\dots,F_{\omega,s})(x + \omega \cdot h) \\
        = &\EEE_{x \in H} \EEE_{h \in H^{s-1}} \overline{R(h)} \EEE_{\substack{y_{\omega,i} \in H \\ \sum_i y_{\omega,i} = x + \omega \cdot h}} \prod_{i=1}^{s} \prod_{\omega \in \llbracket s-1 \rrbracket} \cC^{|\omega|} F_{i,\omega}(y_{\omega,1},\dots,y_{\omega,i-1},y_{\omega,i+1},\dots,y_{\omega,s})
      \end{align*}
      and introduce a change of variables\footnote{This change of variables is redundant insofar as each value of the old variables $x,h,y_{\omega,i}$ corresponds to many values of the new variables $x,h,t_i,y_{\omega,i}$, but since we are averaging and the fibers have constant size this is not a problem.} $t_1,\dots,t_s \in H^{s-1}$ and $z_{\omega,i} \in H$ (for $i \in [s]$ and $\omega \in \llbracket s-1 \rrbracket$) such that $\sum_{i=1}^s t_i = h$, $\sum_{i=1}^s z_{\omega,i} = x$ for each $\omega$ and $y_{\omega,i} = z_{\omega,i} + \omega \cdot t_i$, so that
      \begin{equation}
        \label{eq:gen-conv-stuff}
        \LHS = \EEE_{x \in H} \EEE_{\substack{z_{\omega,i} \\ \sum_{i=1}^s z_{\omega,i} = x}} \left( \EEE_{h \in H^{s-1}} \overline{R(h)} \EEE_{\substack{t_1,\dots,t_s \in H^{s-1} \\ \sum_{i=1}^s t_i = h}} \prod_{j=1}^s F^{(x,z_{\omega,i})}_j(t_1,\dots,t_{j-1},t_{j+1},t_s) \right)
      \end{equation}
      where $F_j^{(x,z_{\omega,i})} \colon (H^{s-1})^{s-1} \to \CC$ is a $1$-bounded function depending on $x$ and the variables $z_{\omega,i}$ given by
      \[
        F_j^{(x,z_{\omega,i})} (t_1,\dots,t_{j-1},t_{j+1},t_s) = \prod_{\omega \in \llbracket s-1 \rrbracket} \cC^{|\omega|} F\Big(
        \begin{aligned}[t]
          &z_{\omega,1}+\omega \cdot t_1,\dots, z_{\omega,{j-1}}+\omega \cdot t_{j-1}, \\
          &z_{\omega,{j+1}}+\omega \cdot t_{j+1}, \dots,z_{\omega,s} + \omega \cdot t_s \Big).
        \end{aligned}
      \]
      Hence, the term in parentheses in~\eqref{eq:gen-conv-stuff} has the form $\big\langle \mathord{\ast}\big(F_j^{(x,z_{\omega,i})}\big), R \big\rangle$ and so is bounded in magnitude by $\|R\|_{U^s}^{2^s}$ (by~\eqref{eq:csc} applied with $G=H^{s-1}$). The claim follows by the triangle inequality.
    \end{proof}

    It therefore suffices to show that the original quantity $|\langle \cS_{s-1} k, R \rangle|$ may be bounded in terms of a quantity $|\langle \cS_{s-1}(k_\omega), R \rangle|$ where each $k_\omega$ is a generalized convolution of bounded functions.  This is achieved by $2^{s-1}$ further applications of Cauchy--Schwarz.  For any functions $k_\omega$ and any $\eta \in \llbracket s-1 \rrbracket$, one application of Cauchy--Schwarz to the function $k_{\eta}$ yields
    \[
      \left| \langle \cS_{s-1}(k_\omega), R \rangle \right| \le \|k_{\eta}\|_2 \left| \langle \cS_{s-1}(k'_\omega), R \rangle \right|^{1/2}
    \]
    where $(k'_\omega)_{\omega \in \llbracket s-1 \rrbracket}$ are new functions given by $k'_\omega = k_\omega$ if $\omega \ne \eta$, and
    \[
      k'_{\eta}(y) = \EEE_{\substack{x \in H, h \in H^{s-1} \\ x + \eta \cdot h = y}} R(h) \prod_{\omega \in \llbracket s-1 \rrbracket \setminus \{\eta\}} \cC^{|\omega| + |\eta| + 1} k_\omega(x + \omega \cdot h)
    \]
    is a corresponding ``dual function''.

    By grouping the functions $\big\{k_\omega \colon \omega \in \llbracket s-1 \rrbracket \setminus \{\eta\}\big\}$ into $s-1$ classes based on the value of $\min\{ i \colon \omega_i \ne \eta_i\}$, and putting $R(h)$ in a class by itself, we see that $k'_\eta$ is a generalized convolution of degree $s-1$ of $1$-bounded functions (and so is itself $1$-bounded).  Applying this argument once for each $\eta$, we can replace all of the original functions $k$ by (different) functions which are generalized convolutions.  Combining this with Claim~\ref{claim:conv-is-easy} gives the result.
  \end{proof}
  This concludes the proof of Proposition~\ref{prop:norm-inequality}.
\end{proof}

\begin{remark}%
  \label{rem:find-any-1pc}
  It follows that if $\|f\|_{U^{s+1}} \ge \delta$, we can locate a subset $S \subseteq H^{s-1}$ of density $\mu(S) \ge \delta^{1/O_s(1)}$ and a function $\kappa \colon S \to \widehat H$, such that $\fsig_s f(x, \kappa(x)) \ge \delta^{1/O_s(1)}$ for all $x \in S$ and $\kappa$ is an approximate polynomial with parameter $\ge \delta^{1/O_s(1)}$.  Indeed, we can set $\kappa(x)$ to be any value in $\widehat H$ such that $\fsig_s f(x, \kappa(x))$ is large, if possible, and set $S$ to consist of all $x \in H^{s-1}$ for which it is possible.  The other properties then follow from Proposition~\ref{prop:norm-inequality} and Proposition~\ref{prop:sig-facts}.

  This is the basic mechanism for locating approximate polynomials to which we will apply Theorem~\ref{thm:1pc}.  However, for technical reasons to do with bounds we will need to find not just one such function $\kappa$, whose graph correlates weakly with $\fsig_s f$, but a family whose graphs collectively cover all large values of $\fsig_s f$.  This occurs in Lemma~\ref{lem:use-1pc} below.  However, until we get there the reader should keep in mind a single representative function $\kappa$ of this type.
\end{remark}

\subsection{Passing to a real-valued approximate polynomial}

A technical issue is that the partial approximate polynomials $\kappa \colon H^{s-1} \to \widehat{H}$ discussed in Remark~\ref{rem:find-any-1pc} have codomain $\widehat{H}$, whereas our structure theorem for approximate polynomial functions only applies when the codomain is $\RR$.  However, in the case that $H = \ZZ/N\ZZ$ for $N$ prime, and hence $\widehat{H} = \frac1N \ZZ/\ZZ$, it is straightforward to pass from a $\widehat{H}$-valued approximate polynomial to an $\RR$-valued one, at the cost of worsening the parameter by a small amount.  A similar statement is true replacing $\frac1N \ZZ/\ZZ$ by any product of $O(1)$ cyclic groups.

\begin{lemma}%
  \label{lem:lift-to-real}
  Let $G$ be any finite abelian group, $X \subseteq G$ a subset, $N$ a positive integer and $\kappa \colon X \to \frac1N \ZZ/\ZZ$ an approximate polynomial of degree $s$ with parameter $\delta$.

  Then there exists a function $\wt \kappa \colon X \to \frac1N \ZZ$ such that $\wt \kappa(x) \bmod{1} = \kappa(x)$ for all $x \in X$ and $\wt{\kappa}$ is an approximate polynomial with parameter $\delta' \ge \delta / 3^{2^{s+1}} - 2^{2s+2}/|G|$.
\end{lemma}
\begin{proof}
  For $y \in \RR/\ZZ$ we let $\{y\}$ denote the unique value in $[0,1)$ such that $\{y\} \bmod 1 = y$.  Then define a random function $\wt K \colon X \to \RR$ by %chktex 9
  \[
    \wt K(x) = \{\kappa(x)\} + R(x)
  \]
  where $R \colon X \to \{-1,0,1\}$ is a function chosen uniformly at random.

  Suppose $c \in C^{s+1}(G) \cap X^{\llbracket s+1 \rrbracket}$ is such that the values $c(\omega) \in G$ are all distinct, and satisfies $\partial^{s+1} \kappa(x) = 0$.  Then $\partial^{s+1} \{\kappa\}(c)$ is an integer (as reducing modulo $1$ recovers $\partial^{s+1} \kappa(c)$) and lies between $-2^{s}+1$ and $2^s-1$.  It follows that there is at least one choice of values $r_\omega \in \{-1,0,1\}$ for $\omega \in \llbracket s+1 \rrbracket$ such that
  \[
    \sum_{\omega \in \llbracket s+1 \rrbracket} (-1)^{|\omega|} \Big( \big\{\kappa(c(\omega))\big\} + r_\omega \Big) = 0 .
  \]
  Hence,
  \[
    \PP\big[ \partial^{s+1} \wt K(c) = 0 \big] \ge 1/3^{2^{s+1}} .
  \]
  Noting that the number of choices of $c \in C^{s+1}(G)$ with $c(\omega_1) = c(\omega_2)$ for some particular distinct $\omega_1, \omega_2 \in \llbracket s+1 \rrbracket$ is exactly $|G|^{s+1}$ (as the associated group homomorphism $C^{s+1}(G) \to G$ sending $c$ to $c(\omega_1) - c(\omega_2)$ is surjective), the number of $c \in C^{s+1}(G)$ where the values $c(\omega)$ fail to be distinct is at most $\binom{2^{s+1}}{2} |G|^{s+1}$.  Hence,
  \[
    \EEE\Big[\big|\big\{c \in C^{s+1}(G) \colon \partial^{s+1} \wt K(c) = 0\big\}\big|\Big] \ge \delta |G|^{s+2} / 3^{2^{s+1}} - 2^{2s+2} |G|^{s+1}
  \]
  and so there is some particular value $\wt K = \wt \kappa$ with the required properties.
\end{proof}

\subsection{Applying the structure theorem for approximate polynomials}

Throughout this subsection we take $H = \ZZ/N\ZZ$ for $N$ prime.

We are now in a position to apply Theorem~\ref{thm:1pc}.  As discussed in Remark~\ref{rem:find-any-1pc}, in order to prevent a parameter explosion in a later part of the proof it turns out to be preferable to show that almost all of the mass of $\fsig f$, rather than just a part of it, can be explained by a collection of nil-polynomials.

Specifically, the following lemma states that most of the large support of $\fsig f$ is covered by the graphs of a few nil-polynomials $g_i$.

\begin{lemma}%
  \label{lem:use-1pc}
  Let $f \colon H \to \RR$ be a function with $\|f\|_{\infty} \le 1$, and let $\eps > 0$ be a parameter.  Then there exists an integer $L \ll_s \eps^{O_s(1)}$, sets $A_1,\dots,A_L \subseteq H^{s-1}$ and nil-polynomials $g_1,\dots,g_L \colon H^{s-1} \to \RR$ of degree $s-1$, dimension $\ll_s \eps^{O_s(1)}$ and complexity $M$, such that $g_i(x) \in \frac1N \ZZ$ for all $x \in A_i$, and such that the following holds.
  
  For each $i \in [L]$, let $S_i$ denote the graph of the induced function $g_i|_{A_i} \colon H^{s-1} \to \widehat{H}$, i.e.\ the subset
  \[
    S_i = \big\{ \big(x,\ t \mapsto e(g_i(x)\, t)\big) \colon x \in A_i\big\} \subseteq H^{s-1} \times \widehat H
  \]
  and write $S = \bigcup_{i=1}^L S_i$.  Then
  \begin{equation}
    \label{eq:graph-cover}
    \EE_{x \in H^{s-1}} \sum_{\chi \in \widehat{H}} \big(1 - 1_{S}(x,\chi)\big)\, \fsig f(x, \chi)^2 \le \eps.
  \end{equation}
  Finally, we have $M \le \exp\big(O_s(\eps)^{O_s(1)}\big)$ if $s \le 3$ and $M \le \exp \exp\big(O_s(\eps)^{O_s(1)}\big)$ if $s \ge 4$.
\end{lemma}
\begin{proof}
  We write
  \[
    R = \big\{ (x, \chi) \in H^{s-1} \times \widehat{H} \colon \fsig f(x,\chi) \ge \eps/2 \big\}
  \]
  for the points where $\fsig f$ is not too small.  We note
  \begin{equation}
    \label{eq:small-points}
    \EE_{x \in H^{s-1}} \sum_{\chi \in \widehat{H}} (1 - 1_R(x,\chi))\, \fsig(x,\chi)^2 \le \EE_{x \in H^{s-1}} \sum_{\chi \in \widehat{H}} (\eps/2) \fsig(x,\chi) \le \eps/2
  \end{equation}
  by Proposition~\ref{prop:sig-facts}\ref{item:sig-facts-2}.
  
  Since the result is trivial if $N < C\eps^{-C}$ for some $C=C(s)$ (e.g., take $L = N^s$, all sets $A_i$ to have size $1$ and $g_i$ every possible constant function) we can assume a suitable lower bound on $N$ in terms of $\eps$ of this form.
  
  Let $\eta>0$ be a parameter to be specified later.  We choose sets $A_i$ and nil-polynomials $g_i$ iteratively, in such a way that the following properties hold: (i) the resulting sets $S_i$ are disjoint, (ii) each set $A_i$ has $\mu(A_i) \ge \eta$, and (iii) $S_i \subseteq R$ for each $i$.
  
  Given these assumptions, we have
  \[
    \EE_{x \in H^{s-1}} \sum_{\chi \in \widehat{H}} 1_{\bigcup_{i=1}^L S_i}(x,\chi)\, \fsig f(x, \chi)^2 \ge \sum_{i=1}^L \mu(A_i)\, (\eps/2)^2 \ge L\, \eps^2\, \eta / 4
  \]
  and as the left hand side is bounded by $\|f\|_{U^{s+1}}^{2^{s+1}} \le 1$ (Proposition~\ref{prop:sig-facts}\ref{item:sig-facts-3}), any process for picking $A_i$ and $g_i$ must terminate while $L \le 4 \eps^{-2} \eta^{-1}$.
  
  Suppose $A_i$ and $g_i$ have been chosen for $i \in [k]$.  If~\eqref{eq:graph-cover} holds already for $k$ then we stop.  If not, let
  \[
    T = R \setminus \bigcup_{i=1}^k S_i
  \]
  and let $U \subseteq H^{s-1}$ denote those points $x$ such that $(x,\chi) \in T$ for at least one value of $\chi \in \widehat H$.  We can bound
  \[
    \EE_{x \in H^{s-1}} \sum_{\chi \in \widehat{H}} \big(1 - 1_{\bigcup_{i=1}^k S_i}(x,\chi)\big)\, \fsig f(x, \chi)^2 \le \eps/2 + \EE_{x \in H^{s-1}} \sum_{\chi \in \widehat{H}} 1_T(x,\chi) \fsig f(x, \chi)^2
  \]
  by~\eqref{eq:small-points}, and hence 
  \[
    \mu(U) \ge \EE_{x \in H^{s-1}} \sum_{\chi \in \widehat{H}} 1_T(x,\chi) \fsig_s f(x,\chi)^2 \ge \eps/2
  \]
  using Proposition~\ref{prop:sig-facts}\ref{item:sig-facts-2} again.

  We define a function $\kappa \colon U \to \widehat{H}$ by setting $\kappa(x)$ to be one of the values such that $(x, \kappa(x)) \in T$, arbitrarily.  It follows by Corollary~\ref{cor:found-1pc} that $\kappa$ is a 1\% polynomial with parameter $\gg_s \eps^{O_s(1)}$.
  
  By Lemma~\ref{lem:lift-to-real} we can find $\kappa' \colon U \to \frac1N \ZZ$ such that $\kappa'(x) \bmod 1 = \kappa(x)$ for all $x \in U$ and $\kappa$ is a 1\% polynomial with parameter $\gg_s \eps^{O_s(1)}$ (using our lower bound on $N$ in terms of $\eps$).

  By Theorem~\ref{thm:1pc}, we can find a nil-polynomial $g_{k+1} \colon H^{s-1} \to \RR$ of degree $s-1$, dimension $O_s(\eps)^{-O_s(1)}$ and complexity $M$ as in the statement, and a set $A_{k+1} \subseteq U$ with $\mu(A_{k+1}) \gg_s \eps^{O_s(1)}$, such that $g_{k+1}|_{A_{k+1}} = \kappa'|_{A_{k+1}}$.  Pick $\eta \gg_s \eps^{O_s(1)}$ so that the lower bound on $\mu\big(A_{k+1}\big)$ is precisely $\eta$.

  It is clear by construction that $g_{k+1}$ and $A_{k+1}$ obey the properties (i), (ii), (iii) required, and so the recursive construction is complete.
\end{proof}

\subsection{Correlation of $\cS_s f$ with a sum of nilsequences}%
\label{subsec:hs-nil}

We now have to do something useful with the conclusion of Lemma~\ref{lem:use-1pc}.  This splits into two parts: first, we show that the function $\cS_s f \colon H^s \to \CC$ correlates strongly with a sum of degree $s$ nilsequences on $H^s$; then in the next subsection, we use this to find a sum of nilsequences $H \to \CC$ that correlates with the original function $f$.

We know that $\fsig f$ correlates strongly with the graphs of various nil-polynomials $g_i \colon H^{s-1} \to \RR$. By taking a Fourier transform, it would be sufficient to show that functions $H^{s-1} \times H \to \CC$ of the form $(x,y) \mapsto e\big(g_i(x)\, y\big)$ were essentially nilsequences.

Unfortunately this last fact is not quite true.  For example, a function $H^{s-1} \to \frac1N \ZZ$ taking $O(1)$ distinct values is always a nil-polynomial in some trivial way, but the corresponding function $(x,y) \mapsto e(g_i(x)\, y)$ is not nilsequence-like on $H^s$ if the values $g_i(x)$ are chosen arbitrarily.  More generally, the feature of nil-polynomials that we are free to apply arbitrary bounded $\Gamma$-translates pointwise to the function $r \colon H \to G$ (see Definition~\ref{def:nil-poly}) is not nilsequence behavior. 

If we only wish to find a nilsequence $H^{s} \to \CC$ that correlates weakly with $\cS_s f$, this is not actually a problem, essentially because $\fsig f$ is non-negative, and so it is good enough to choose a nilsequence that captures some small window of this arbitrary nil-polynomial behaviour.  However, this notion of ``weakly'' involves an exponential or double-exponential loss in parameters.  When we apply the inverse theorem inductively on $s$, this loss really hurts, leading to $O(s)$ iterated exponentials overall.

Instead, then, we want to understand most or almost all of the mass of $\cS_s f$ in terms of a family of nilsequences.  The technical statement is the following partition of unity argument.
\begin{lemma}%
  \label{lem:global-hs-nil}
  Let $f \colon H \to \RR$ be a function with $\|f\|_\infty \le 1$ and $\|f\|_{U^{s+1}} = \delta$.  Then there exists a function $B \colon H^s \to \RR$ with the following properties:---
  \begin{enumerate}[label=(\roman*)]
    \item $\|B\|_{\infty} \le 1$;
    \item the inner product $\langle \cS_s f, B \rangle$ is large, i.e.\ $\big|\EE_{h \in H^s} \cS_s f(h) \overline{B(h)} \big| \ge \delta^{2^{s+1}}/2$;
    \item there exists a family of nilsequences $\rho_i \colon H^s \to [0,1]$ for $i \in [T]$, where $T$ is some positive integer, with $\rho_i$ having degree $s$, dimension $O_s(\delta)^{-O_s(1)}$, complexity $M$ and parameter $K$; and complex numbers $\alpha_i$ for $i \in [T]$ with $|\alpha_i| \le 1$; such that
      $B = \sum_{i=1}^T \alpha_i \rho_i$;
  \end{enumerate}
  where as usual $T, M,K \ll \exp\big(O_s(\delta)^{O_s(1)}\big)$ if $s \le 3$ and $T, M, K \ll \exp \exp\big(O_s(\delta)^{O_s(1)}\big)$ if $s \ge 4$.
\end{lemma}

The principal ingredient is the following result which---necessarily in a weak sense, given the discussion above---allows us to relate a function of the type $(x,y) \mapsto e\big(g_i(x)\, y\big)$, where $g_i \colon H^{s-1} \to \RR$ is a nil-polynomial of degree $s-1$, to a nilsequence of degree $s$ on $H^s$.

\begin{lemma}%
  \label{lem:nilsequence-construction}
  Suppose $\phi \colon H^{s-1} \to \RR$ is a nil-polynomial of degree $s-1$, dimension $D$ and complexity $M$, and $A \subseteq H^{s-1}$ is some set such that $\phi(x) \in \frac1N \ZZ$ for all $x \in A$.  Also let $\eps > 0$ be a parameter.

  Then there exist a nilmanifold $G/\Gamma$ of degree $s$, dimension $\ll_s D^{O_s(1)}$ and complexity $\ll_s (D M)^{O_s(1)}$, a polynomial map $p \colon H^{s-1} \times H \to G/\Gamma$, and a family $F_z \colon G/\Gamma \to \RR$ of $K$-Lipschitz functions where $K = O_s(M D)^{O_s(1)}$ and $z$ is some auxiliary variable, such that the following holds: for any $x \in A$ there exists some choice of $z$ such that $F_z(p(x, y)) = e\big(\phi(x)\, y\big)$ for all $y \in H$.
\end{lemma}
This implies that $(x,y) \mapsto e\big(\phi(x)\, y\big)$ behaves like a Lipschitz function in the $y$ variable for each fixed $x$, under a metric on $H$ induced from a nilmanifold $G/\Gamma$ that does not depend on $x$.

Before starting on proofs we make one remark that will be used frequently, analogous to Remark~\ref{rem:nil-vs}.
\begin{remark}%
  \label{rem:prod-nilseq}
  If $\psi_i$ for $i \in [L]$ are all one-bounded nilsequences on a group $H'$ of degree $s$, dimension $D_i$, complexity $M_i$ and parameter $K_i$. Then there is a nilmanifold $G/\Gamma$ of degree $s$, dimension $D = \sum_{i=1}^L D_i$ and complexity $M = \max_{i=1}^L M_i$, such that any functions of the form $\psi_i$, $\sum_{i=1}^L \psi_i$ or $\prod_{i=1}^L \psi_i$ are nilsequences on $G/\Gamma$ with Lipschitz parameter $K = \max_{i=1}^L K_i$.

  Indeed, if $G^{(i)}/\Gamma^{(i)}$, $p_i \colon H' \to G^{(i)}/\Gamma^{(i)}$ and $F_i \colon G^{(i)}/\Gamma^{(i)} \to \CC$ are data associated to $\psi_i$, we can take direct products $G = \prod_{i=1}^L G^{(i)}$, $\Gamma = \prod_{i=1}^L \Gamma^{(i)}$, so $G/\Gamma \cong \prod_{i=1}^L G^{(i)}/\Gamma^{(i)}$, and then set $p(x) = (p_1(x),\dots,p_L(x))$.  Then, functions $F(x_1,\dots,x_L) = F_i(x_i)$, $F(x_1,\dots,x_L) = \sum_{i=1}^L F_i(x_i)$ or $F(x_1,\dots,x_L) = \prod_{i=1}^L F_i(x_i)$ give the required nilsequences $F \circ p$.

  It is easy to verify that the filtrations, explicitly presented filtered group structures, metrics and so on behave exactly as one would expect, giving the quantitative bounds claimed.
\end{remark}

Given Lemma~\ref{lem:nilsequence-construction}, the outline of the proof of Lemma~\ref{lem:global-hs-nil} is as follows.

By Lemma~\ref{lem:nilsequence-construction} and a Fourier transform, we can deduce that $\cS_s f(h_1,\dots,h_s)$ itself behaves (up to small errors) like a Lipschitz function in $h_s$ for $h_1,\dots,h_{s-1}$ fixed, for some notion of distance coming from a nilmanifold.  Because $\cS_s f$ is symmetric, the same holds for each other coordinate $h_i$.  Hence, we can potentially relate the values $\cS_s f(x)$ and $\cS_s f(y)$ by moving from $x$ to $y$, changing one coordinate at a time.

We use this induced notion of distance on all of $H^s$ and build an associated partition of unity on almost all of $H^s$, whose parts are the functions $\rho_i$ for $i \in [L]$ and such that $\cS_s f$ is approximately constant on the support of $\rho_i$ for each $i$.  We can then approximate $\cS_s f$ by a smoothed-out version of itself, built as a sum of the nilsequences $\rho_i$.

\begin{proof}[Proof of Lemma~\ref{lem:global-hs-nil} assuming Lemma~\ref{lem:nilsequence-construction}]
  We first take an integer $L$, sets $A_1,\dots,A_L \subseteq H^{s-1}$ and $S_i \subseteq H^{s-1} \times \widehat{H}$, and nil-polynomials $g_1,\dots,g_L \colon H^{s-1} \to \RR$ of degree $s-1$, dimension $D_0$ and complexity $M_0$, as provided by Lemma~\ref{lem:use-1pc} applied with some parameter $\eps$ to be determined.
  
  For each $i \in [L]$ we apply Lemma~\ref{lem:nilsequence-construction} with $\phi=g_i$, and obtain families of one-bounded nilsequences $F_z^{(i)} \circ p_i\colon H^{s-1} \times H \to \CC$ on nilmanifolds $G^{(i)}/\Gamma^{(i)}$ with the stated properties.  Specifically say these have degree $s$, dimension $D_1 \ll_s D_0^{O_s(1)}$, complexity $M$ and parameter $K_1$, where $M,K_1 \ll_s (D_0 M_0)^{O_s(1)}$.

  By Remark~\ref{rem:prod-nilseq} we can form one big nilmanifold $G/\Gamma$ of degree $s$, dimension $D_2 = L D_1$ and complexity $M$, and a polynomial map $p \colon H^{s-1} \times H \to G/\Gamma$, such that each nilsequence $F_z^{(i)} \circ p_i$ is equal to a nilsequence ${F'}_z^{(i)} \circ p$ on $G/\Gamma$ with parameter $K_1$.
  
  For $x,y \in H^s$, define $d(x,y) = d_G(p(x), p(y))$.  This can be thought of as a nilsequence $H^{2s} \to \RR$ on the nilmanifold $(G/\Gamma)^2$, since $(x,y) \mapsto (p(x), p(y))$ is a polynomial map and $(u,w) \mapsto d_{G/\Gamma}(u,w)$ is a Lipschitz function $(G/\Gamma)^2 \to \RR$ with parameter $1$, under the product metric on $(G/\Gamma)^2$.

  Consider the function $\Lambda \colon H^{s-1} \times H \to \CC$ obtained by restricting $\fsig f$ to $S = \bigcup_{i=1}^L S_i$ and taking an inverse Fourier transform; that is,
  \[
    \Lambda(x,y) = \sum_{\chi \in \widehat{H}} 1_S(x,\chi) \fsig f(x,\chi) \chi(y) .
  \]
  By Proposition~\ref{prop:sig-facts} we have $\|\Lambda(x,y)\|_\infty \le 1$, and by the conclusion of Lemma~\ref{lem:use-1pc} and Parseval we have
  \begin{equation}
    \label{eq:sf-lambda-l2}
    \EE_{h \in H^s} \big| \cS_s f(h) - \Lambda(h) \big|^2 \le \eps.
  \end{equation}
  We recall that $\cS_s f$ is a symmetric function in its arguments. However, $\Lambda$ will not be in general.  If $\pi$ is a bijection $[s] \to [s]$ we abuse notation to write $\pi \colon H^s \to H^s$ for the map sending $(h_1,\dots,h_s)$ to $\big(h_{\pi(1)},\dots,h_{\pi(s)}\big)$, and write $\Lambda_\pi = \Lambda \circ \pi$.  We then define $\Omega_0 \subseteq H^s$ by
  \begin{equation}
    \label{eq:sf-lambda-control}
    \Omega_0 = \big\{ h \in H^s \colon |\cS_s f(h) - \Lambda_\pi(h) \big| > \eps^{1/4} \text{ for some } \pi \colon [s]  \to [s] \big\}.
  \end{equation}
  We deduce from~\eqref{eq:sf-lambda-l2} that $\mu(\Omega_0) \le s!\, \eps^{1/2}$.  Also, for $x,y \in H^s$ we define $d_\pi(x,y) = d(\pi(x), \pi(y))$, for $d$ as above.

  By the discussion above we know that for every $(x,\chi) \in S$ there exists some $K_1$-Lipschitz function $F \colon G/\Gamma \to \CC$ such that $\chi(y) = F(p(x, y))$ for all $y \in H$.  In particular, for every $y,y' \in H$ we have $|\chi(y) - \chi(y')| \le K_1 d\big((x,y), (x,y')\big)$.
  
  For any $x \in H^{s-1}$ the function $y \mapsto \Lambda(x,y)$ is a linear combination of characters $\chi$ for which $(x,\chi) \in S$ by definition, with total weight at most $\sum_{\chi} \fsig f(x,\chi) \le 1$ (Proposition~\ref{prop:sig-facts} again), and hence $|\Lambda(x,y) - \Lambda(x,y')| \le K_1 d\big((x,y), (x,y')\big)$ for all $x \in H^{s-1}$ and $y \in H$.
  
  Replacing $\Lambda$ by $\Lambda_\pi$ where $\pi(s) = i$, similarly if $x, x' \in H^s$ and $x_j = x'_j$ for all $j \ne i$ then
  \begin{equation}
    \label{eq:lambda-lambda-control}
    |\Lambda_\pi(x) - \Lambda_\pi(x')| \le K_1 d_\pi(x,x') .
  \end{equation}

  We now construct a rather involved approximate partition of unity on $H^s$, as follows.
  \begin{claim}%
    \label{claim:partition}
    Let $\sigma, \eta > 0$ be parameters. Then there exists an integer $T \le O(M D_2 / \eta)^{-O_s(D_2)} \sigma^{O_s(1)}$, functions $\rho_j \colon H^s \to [0,1]$ for $j \in [T]$, and an exceptional set $\Omega_1 \subseteq H^s$ with $\mu(\Omega_1) \le \sigma$, such that:---
    \begin{enumerate}[label=(\roman*)]
      \item for any $x \in H^s \setminus \Omega_0 \setminus \Omega_1$ we have $\sum_{j\in[T]} \rho_j(x) = 1$, and for all $x \in H^s$ we have $\sum_{j \in [T]} \rho_j(x) \le 1$;
      \item each $\rho_j$ is a nilsequence of degree $s$, dimension $D = 2s D_2$, complexity $M$ and parameter $O(1/\eta)$; and
      \item for each $j \in [T]$ and any $x,x'$ in the support of $\rho_j$ we have $|\cS_s f(x) - \cS_s f(x')| \ll_s K_1 \eta  + \eps^{1/4}$.
    \end{enumerate}
  \end{claim}
  \begin{proof}[Proof of claim]
    For each $i \in [s]$, fix any permutation $\pi_i \colon [s] \to [s]$ with $\pi_i(s) = i$.
    We then consider the following function on pairs $x,y \in H^s$:
    \begin{equation}
      \label{eq:big-delta}
      \Delta(x,y) = \sum_{i=1}^s d_{\pi_i}\big((y_1,\dots,y_{i-1},x_i,\dots,x_s), (y_1,\dots,y_i,x_{i+1},\dots,x_s) \big)
    \end{equation}
    i.e.\ the total distance travelled as we walk from $x$ to $y$ by changing the coordinates over one at a time.  (Note $\Delta$ need not be symmetric nor obey the triangle inequality.)

    The motivation for this definition is the following: if $x,y \in H^s$ and $(y_1,\dots,y_{i},x_{i+1},\dots,x_s) \notin \Omega_0$ for each $0 \le i \le s$, by~\eqref{eq:sf-lambda-control},~\eqref{eq:lambda-lambda-control} and repeated use of the triangle inequality we deduce that
    \begin{equation}
      \label{eq:lambda-delta-control}
      |\cS_s f(x) - \cS_s f(y)| \ll_s K_1 \Delta(x,y) + \eps^{1/4} ;
    \end{equation}
    i.e., $\cS_s f$ behaves like a Lipschitz function under $\Delta$.  We call $(x,y)$ such that $(y_1,\dots,y_{i},x_{i+1},\dots,x_s) \notin \Omega_0$ for each $0 \le i \le s$ an \emph{admissible pair}.
    
    The function $\Delta \colon H^{2s} \to \RR$ is a nilsequence on $(G/\Gamma)^{2s}$ with parameter $1$.  Indeed, we know $(x,y) \mapsto d(x,y)$ is a nilsequence on $(G/\Gamma)^2$ with parameter $1$, and hence so is each term in the sum~\eqref{eq:big-delta}, as it composes $d \colon H^{2s} \to \RR$ with a group homomorphism $H^{2s} \to H^{2s}$.  Hence the sum is also a nilsequence on $(G/\Gamma)^{2s}$ by Remark~\ref{rem:prod-nilseq}. 

    It follows that for any fixed $y \in H^s$ the function $x \mapsto \Delta(x,y)$ is a nilsequence $H^s \to \RR$ on $(G/\Gamma)^{2s}$, with parameter $1$, whose associated polynomial map depends on $y$: indeed, it is a composition of $\Delta$ with an affine-linear map $H^s \to H^{2s}$.

    We next want to show that for most $x \in H^s$ there are many $y \in H^s$ such that $\Delta(x,y)$ is small.
    \begin{claim}%
      \label{claim:many-pairs}
      There is a set $\Omega_2 \subseteq H^s$ with $\mu(\Omega_2) \le \sigma/2$ such that the following holds: for all $x \in H^s \setminus \Omega_0 \setminus \Omega_2$, there exist at least $\nu |H|^s$ values $y \in H^s$ such that $(x,y)$ is an admissible pair and $\Delta(x,y) \le \eta$.  Here we can take $\nu \ge O(M D_2 / \eta)^{-O_s(D_2)} \sigma^{s}$.
    \end{claim}
    \begin{proof}[Proof of claim]
      We observe by Proposition~\ref{prop:fun-dom} and Proposition~\ref{prop:metric-compare} that $G/\Gamma$ may be covered by $Q = O(M D_2 / \eta)^{O_s(D_2)}$ balls of radius $\eta/2s$; write $U_1,\dots,U_Q \subseteq G/\Gamma$ for these sets.  Let $U'_j \subseteq U_j$ be disjoint subsets which partition $G/\Gamma$.  Clearly, if $p(x), p(y)$ lie in the same cell $U'_j$ then $d(x,y) \le \eta/s$.

      We therefore colour $H^s \setminus \Omega_0$ with colours $[Q]^s$, where $x \in H^s \setminus \Omega_0$ is given the colour $(j_1,\dots,j_s)$ such that $p(\pi_i(x)) \in U'_{j_i}$ for each $i \in [s]$.  We say an admissible pair $(x,y)$ is \emph{completely monochromatic} if $x$, $y$ and every intermediate point $(y_1,\dots,y_i,x_{i+1},\dots,x_s)$ ($1 \le i \le s-1$) all have the same color.  By~\eqref{eq:big-delta} this implies $\Delta(x,y) \le \eta$.

      We next perform another clean-up argument in the spirit of Lemma~\ref{lem:cube-system}.
      \begin{claim}
        If $X \subseteq H^s$ is any set and $c>0$ is a parameter, there exists $X' \subseteq X$ with the property that for each $i \in [s]$ and $x \in X'$,
        \begin{equation}
          \label{eq:robust-hs-sets}
          |\{ z \in X' \colon z_j = x_j \; \forall j \ne i \}| \ge c |H|
        \end{equation}
        and $\mu(X \setminus X') \le c s$.
      \end{claim}
      \begin{proof}[Proof of claim]
        We define $X'$ by starting with the set $X$ and repeatedly applying the following procedure: if $i \in [s]$ and $x \in X'$ fails to satisfy~\eqref{eq:robust-hs-sets}, remove the whole fiber of points $z \in X'$ such that $z_j = x_j$ for all $j \ne i$ (including $x$ itself).  When this process terminates, $X'$ has the required property.  Moreover, each of the $s |H|^{s-1}$ fibers can be deleted at most once, and in doing so we lose at most $c |H|$ elements from $X'$, so $|X \setminus X'| \le (c |H|) (s |H|^{s-1})$ as required.
      \end{proof}
      We apply the claim with $c = \sigma / (2 s Q^s)$ to each of the $Q^s$ colour classes of $H^s \setminus \Omega_0$, and set $\Omega_2$ to be the union of all the deleted sets $X \setminus X'$.  By applying~\eqref{eq:robust-hs-sets} once for each $i \in [s]$ in increasing order, it is immediate that for each $x \in H^s \setminus \Omega_0 \setminus \Omega_2$ we can find $(c |H|)^s$ elements $y \in H^s \setminus \Omega_0 \setminus \Omega_2$ such that $(x,y)$ is completely monochromatic, i.e.\ by choosing the coordinates $y_1,\dots,y_s$ one at a time.  This proves Claim~\ref{claim:many-pairs}.
    \end{proof}
    We continue to write $\nu$ and $\Omega_2$ for the values given by Claim~\ref{claim:many-pairs}. 
    If we now pick $y_1,\dots,y_T \in H^s$ independently and uniformly at random, for some $T$ to be determined, then for fixed $x \in H^s \setminus \Omega_0 \setminus \Omega_2$ the probability that there is at least one $j$ such that $(x,y_j)$ is admissible and $\Delta(x,y_j) \le \eta$ is at least $1 - (1-\nu)^{T}$. Taking any $T \ge \log(2/\sigma)/ \nu$ this is at least $1-\sigma/2$.
    
    So there is some choice of $y_1,\dots,y_T$ and some further exceptional set $\Omega_3 \subseteq H^s$, such that $\mu(\Omega_3) \le \sigma/2$, and for all $x \in H^s \setminus \Omega_0 \setminus \Omega_2 \setminus \Omega_3$ there is at least one $j \in [T]$ such that $(x,y_j)$ is an admissible pair and $\Delta(x,y_j) \le \eta$.
    We set $\Omega_1 = \Omega_2 \cup \Omega_3$.

    We lastly define our functions $\rho_j$ as follows. For each $j \in [T]$, define $\beta_j, \rho_j \colon H^s \to [0,1]$ by
    \[
      \beta_j(x) = \min\big(1,\; 2 \max(0, 1 - \Delta(x, y_j) / 2 \eta)\big)
    \]
    and
    \[
      \rho_j(x) = \frac{\beta_j(x)}{\min\left(1,\ \sum_{j' \in [T]} \beta_{j'}(x)\right)} .
    \]
    Note that $\beta_j$ is supported on $\{ x \colon \Delta(x,y_j) \le 2 \eta\}$.  Hence, $|\cS_s f(x) - \cS_s f(y_j)| \ll_s \eta K_1 + \eps^{1/4}$ whenever $\rho_j(x) > 0$ (by~\eqref{eq:lambda-delta-control}). Part (iii) of Claim~\ref{claim:partition} follows.
    
    Whenever $\Delta(x, y_{j_0}) \le \eta$ for at least one $j_{0} \in [T]$ we have $\beta_{j_0}(x) =1$ and hence $\sum_{j \in [T]} \rho_{j}(x) = 1$, since the minimum in the denominator of the definition of $\rho_{j}$ does not kick in.  This shows part (i).
    
    Finally, note that $\Delta(x, y_j)$ are nilsequences on $(G/\Gamma)^{2s}$ described above, and hence so are $\beta_j$ and $\rho_j$; moreover these have dimension $D = 2s D_2$, and the latter two have Lipschitz parameter $O(1/\eta)$, as required for (ii).
  \end{proof}

  To complete the proof of Lemma~\ref{lem:global-hs-nil}, we write $\|\rho_j\|_1 = \EE_{x \in H^s} \rho_j(x)$ and define complex numbers
  \[
    \alpha_j = \frac1{\|\rho_j\|_1} \EE_{y \in H^s} \rho_j(y)\, \cS_s f(y)
  \]
  which gives $|\alpha_j| \le 1$ since it is an average of values $|\cS_s f(y)|$ and $\|\cS_s f\|_\infty \le 1$. We also define the function $B \colon H^s \to \CC$ by $B = \sum_{j=1}^T \alpha_j\, \rho_j$.  This completes part (iii) of the lemma, up to finalizing parameters.

  Moreover for $x \in H^s$,
  \[
    |B(x)| \le \sum_{j=1}^T |\alpha_j| \rho_j(x) \le 1
  \]
  by Claim~\ref{claim:partition}(i), giving part (i) of the lemma.

  We note that if $x \in H^s \setminus \Omega_0 \setminus \Omega_1$ and $j$ is such that $\rho_j(x) \ne 0$, then $|\alpha_j - \cS_s f(x)| \ll_s K_1 \eta + \eps^{1/4}$, by averaging Claim~\ref{claim:partition}(iii).  By further averaging we deduce $|B(x) - \cS_s f(x)| \ll_s K_1 \eta + \eps^{1/4}$ for all such $x$.  Hence
  \[
    \EE_{x \in H^s} |B(x) - \cS_s f(x)|^2 \ll_s (K_1 \eta + \eps^{1/4})^2 + \eps^{1/2} + \sigma
  \]
  and part (ii) of the lemma follows by the triangle inequality (and Proposition~\ref{prop:sig-facts}\ref{item:sig-facts-3}) provided the right hand side is at most $\delta^{2^{s+1}} / 4$.

  This is achievable for some $\eps$, $\sigma$ and $\eta$ with $\eps,\sigma \gg_s \delta^{O_s(1)}$, and $\eta \gg_s \delta^{O_s(1)} / K_1$.  Hence the final Lipschitz parameter $K = O(1/\eta)$ is $\ll_s \delta^{-O_s(1)} K_1$.  Substituting these into the bounds for $D$, $M$, $K_1$ and $T$ already obtained gives final bounds in terms of $\delta$ of the correct form.
\end{proof}

We now prove Lemma~\ref{lem:nilsequence-construction}.
The main issue is that the given nil-polynomial $\phi = f \circ r$ makes no topological guarantees about the polynomial map $f \colon G \to \RR_{(s-1)}$: in particular a function $G \to \CC$ such as $g \mapsto e(f(g))$ may not be Lipschitz with a well-controlled constant, even on small neighbourhoods.  By contrast, a nilsequence requires a Lipschitz function.  To deal with this we need to import the polynomial map $f$ into the nilmanifold structure itself, where its failure to be Lipschitz can do no harm.

\begin{proof}[Proof of Lemma~\ref{lem:nilsequence-construction}]
  We fix a nilmanifold $G/\Gamma$, function $r \colon H^{s-1} \to G$ and polynomial map $f \colon G \to \RR_{(s-1)}$ defining the nil-polynomial $\phi = f \circ r$.  Also let $d_1,\dots,d_s$ and $(\gamma_{i,j})_{i \in [s], j \in [d_i]}$ be the data of the explicitly presented filtered group $G$ (see Definition~\ref{def:explicit-group}).

  We first construct an auxiliary nil-polynomial from $\phi$.
  \begin{claim}
    The function $H^{s-1} \times H \to \RR$ sending $(x,y) \mapsto \phi(x)\,y$ for each $x \in H^{s-1}$ and each integer $-N/2 < y \le N/2$, is a nil-polynomial $f' \circ r'$ of degree $s$, dimension $D+1$ and complexity $M$ on a nilmanifold $G'/\Gamma'$.
  \end{claim}
  \begin{proof}[Proof of claim]
    We define the direct product $G' = \RR \times G$, where $\RR$ carries the standard filtration; so, $G'_i = \RR \times G_i$ if $i \le 1$ and $G'_i = \{0\} \times G_i$ if $i \ge 2$.  Using the one-parameter subgroups $a \mapsto (0, \gamma_{i,j}^a)$ and $a \mapsto (a, \id_G)$, it is clear this is an explicitly presented filtered group of degree $s-1$, dimension $D+1$ and complexity $M$, with integral subgroup $\Gamma' = \ZZ \times \Gamma$.
  
    We then define $r' \colon H^{s-1} \times H \to G'$ by setting $r'(x,y) = (y/N, r(x))$ whenever $y$ is an integer with $-N/2 < y \le N/2$, and $f' \colon G' \to \RR_{(s)}$ by $f'(t, g) = N t f(g)$.  It is clear $f' \circ r'$ has the form claimed, and so it suffices to show $f' \colon G' \to \RR_{(s)}$ and $r' \bmod \Gamma' \colon H^s \to G'/\Gamma'$ are polynomial maps.

    To see $f'$ is polynomial, we claim for induction the more general fact that if $h \colon G \to \RR_{(t)}$ is a polynomial map of degree $t$ then $h' \colon G' \to \RR_{(t+1)}$ given by $h'(t,g) = t\, h(g)$ is a polynomial map of degree $t+1$.  Indeed, we can observe that for any $(t_1,g_1) \in G'_i$, the map
    \[
      (t,g) \mapsto \partial_{(t_1,g_1)} h'(t,g) = t\, h(g) - (t+t_1)\, h(g_1 g) = t\, \partial_{g_1} h(g) - t_1\, h(g_1 g)
    \]
    is a polynomial map $G' \to \RR_{(t+1-i)}$.  To see this, note $g \mapsto \partial_{g_1} h(g)$ is a polynomial map $G \to \RR_{(t-i)}$ by Proposition~\ref{prop:poly-equiv}, and so $(t, g) \mapsto t\, \partial_{g_1} h(g)$ is a polynomial map $G' \to \RR_{(t-i+1)}$ by inductive hypothesis.  Also, $(t,g) \mapsto t_1\, h(g_1 g)$ is a polynomial map $G' \to \RR_{(t)}$, which suffices if $i=1$, and if $i>0$ then $t_1=0$ anyway.  By Proposition~\ref{prop:poly-equiv} again this shows $h'$ has degree $t+1$ as claimed.

  The fact that $r' \bmod \Gamma' \colon H^{s-1} \times H \to G'/\Gamma'$ is a polynomial map is clear, as it is the direct product of polynomial maps $H \to \RR/\ZZ$ and $H^{s-1} \to G/\Gamma$.
  \end{proof}

  We next perform a nilmanifold construction closely related to the proof of Lemma~\ref{lem:poly-construction}.  In particular we recall the facts in Proposition~\ref{prop:poly-space}, as they apply to the vector space of polynomial maps $W = \poly\big(G', \RR_{(s)}\big)$, its filtration $W_0\supseteq\dots\supseteq W_{s}$ of degree $s$ where $W_i = \poly\big(G', \RR_{(s-i)}\big)$, the natural basis $\{e_{i,j}\}$ for $W$ where $0 \le i \le s$ and $j \in [\fd_i]$, and the lattice $\Psi = \poly\big(\Gamma', \ZZ_{(s)}\big)$ generated by the vectors $e_{i,j}$, consisting precisely of those $w \in W$ such that $w(\gamma) \in \ZZ$ for all $\gamma \in \Gamma'$.  We also recall the right action of $G'$ on $W$ given by $\tau(w,g) = x \mapsto w(x g)$, and define an $\ell^1$-norm on $W$ with respect to the basis $\{e_{i,j}\}$.
  
  We define\footnote{For avoidance of doubt we emphasize that, unlike in Lemma~\ref{lem:poly-construction}, we use the vector space $\poly\big(G', \RR_{(s)}\big)$ itself here and not its linear dual.} a new group $\wt G = G' \ltimes W$, i.e.\ the group whose elements as a set are pairs $(g, w) \in G' \times W$ and with group law
  \[
    (g, w) \ast (g', w') = \big(g g', \tau(w, g') + w' \big) .
  \]
  This is given the filtration $\wt G_i = G'_i \ltimes W_i$, which is a filtration by Proposition~\ref{prop:poly-space}(v); however it is not proper, as $W_0 \ne W_1$, and hence $\wt G$ cannot be an explicitly presented filtered group.  We write $\wt \Gamma = \Gamma' \ltimes \Psi$.
  
  If we consider the subgroup $\wt G_1$, with filtration $\wt G_1 = \wt G_1 \supseteq \wt G_2 \supseteq \dots$, and take
  one-parameter subgroups $a \mapsto ({\gamma'}_{i,j}^a, 0)$ and $a \mapsto (0,a e_{i,j})$ (for $1 \le i \le s$), then $\wt G_1$ does have the structure of an explicitly presented filtered group, of degree $s$, dimension $\ll_s D^{O_s(1)}$ and complexity $\ll_s (D M)^{O_s(1)}$.  Its integral subgroup is exactly $\wt \Gamma_1 = \wt \Gamma \cap \wt G_1$.  We write $\wt d_1,\dots,\wt d_s$ for the dimension sequence of $\wt G_1$ and $\sigma_{\wt G_1} \colon \wt G_1 \to \bigoplus_{i=1}^s \RR^{\wt d_i}$ for its coordinate map (in the sense of Definition~\ref{def:explicit-group}).

  The map $f'$ discussed above lies in $\poly\big(G', \RR_{(s)}\big)$.  We can furthermore decompose $f' = f_0 + f_1$, where $f_1 \in \Psi$ and $f_0 = \sum_{i=0}^s \sum_{j=1}^{\fd_i} \alpha_{i,j} e_{i,j}$ with $\alpha_{i,j} \in [0,1)$ for each $i,j$; i.e., the representative of $f'$ in a fundamental domain for $\Psi$. %chktex 9

  We then define $\wt r \colon H^{s-1} \times H \to \wt G$ by
  \[
    \wt r(z) = (\id_{G'}, f_1) \ast (r'(z), 0) \ast (\id_{G'}, -f_1) = \big(r'(z), \tau(f_1, r'(z)) - f_1\big).
  \]
  Note that the second expression shows that $\wt r(z) \in \wt G_1$ for all $z$, as $f_1 - \tau(f_1, r'(z)) \in W_1$ by Proposition~\ref{prop:poly-space}(v).

  \begin{claim}
    The function $\wt r \bmod \wt \Gamma$ is a polynomial map $H^{s-1} \times H \to \wt G / \wt \Gamma$.
  \end{claim}
  \begin{proof}[Proof of claim]
    Note that, since $(\id_{G'}, -f_1) \in \wt \Gamma$,
    \[
      \wt r(z)\, \wt \Gamma = (\id_{G'}, f_1) \ast (r'(z), 0) \,\wt \Gamma .
    \]
    Also, since $r' \bmod \Gamma'$ is a polynomial map, for any $c \in C^k(H^{s-1} \times H)$ we may find $\gamma \colon \llbracket k \rrbracket \to \Gamma'$ such that $\omega \mapsto r'(c(\omega)) \gamma(\omega)$ lies in $\HK^k(G')$. Therefore
    \[
      \omega \mapsto (\id_{G'}, f_1) \ast (r'(c(\omega)), 0) \ast (\gamma(\omega), 0) = (\id_{G'}, f_1) \ast (r'(c(\omega)) \gamma(\omega), 0)
    \]
    lies in $\HK^k(\wt G)$, since constant configurations are always Host--Kra cubes and $\HK^k(\wt G)$ is a group under pointwise multiplication.  This gives the claim.
  \end{proof}

  As $\wt r$ takes values in $\wt G_1$, we can also say that $\wt r \bmod \wt \Gamma_1$ is a polynomial map to $\wt G_1 / \wt \Gamma_1$, where as above $\wt \Gamma_1 = \wt \Gamma \cap \wt G_1$.

  We next define a function $F \colon \wt G_1 \to \CC$ by
  \[
    F(g, w) = e\big(f_0(g) + w(\id_{G'})\big) .
  \]
  \begin{claim}
    For $x_1,x_2 \in \wt G_1$ we have
    \[
      |F(x_1) - F(x_2)| \ll_s (M D)^{O_s(1)} d_{\wt G_1}(x_1,x_2) \big(1 + d_{\wt G_1}(\id_{\wt G},x_1) + d_{\wt G_1}(\id_{\wt G},x_2) \big)^{O_s(1)} ;
    \]
    i.e., $F$ is Lipschitz on bounded regions around $\id_{\wt G_1}$.
  \end{claim}
  \begin{proof}[Proof of claim]
    We note as in the proof of Lemma~\ref{lem:poly-construction} that by inspection of the proof of Proposition~\ref{prop:poly-space}, the function $e_{s,1} \in \poly\big(G', \RR_{(0)}\big)$ is the constant function $1$ and $e_{i,j}(\id_{G'}) = 0$ for all other $i,j$.  Hence, the linear map $W \to \RR$ given by $w \mapsto w(\id_G)$ is equivalent to extracting the coefficient of $e_{s,1}$, and in particular $|w(\id_{G'})| \le \|w\|_1$.
    
    Because the coordinates of $w$ form a subset of $\sigma_{\wt G}(g,w)$, it follows that if $x_1=(g_1,w_1)$ and $x_2 = (g_2,w_2)$ then
    \[
      |w_1(\id_{G'}) - w_2(\id_{G'})| \le \big\|\sigma_{\wt G_1}(x_1) - \sigma_{\wt G_1}(x_2)\big\|_1 .
    \]
    We have $\|f_0\|_1 \le \dim W \ll_s D^{O_s(1)}$, and so by Proposition~\ref{prop:poly-space}(viii) it follows that
    \[
      |f_0(g_1) - f_0(g_2)| \ll_s (D M)^{O_s(1)} d_{G'}(g_1, g_2) \big(1 + d_{G'}(\id_{G'}, g_1) + d_{G'}(\id_{G'}, g_2) \big)^{O_s(1)} .
    \]
    The claim follows from these bounds and Proposition~\ref{prop:metric-compare}.
  \end{proof}

  We note that for $g' \in G'$,
  \begin{equation}
    \label{eq:f-form}
    F\big(g', \tau(f_1, g') - f_1\big) = e\big(\tau(f_1, g')(\id_{G'}) - f_1(\id_{G'}) + f_0(g')\big) = e\big(f'(g')\big)
  \end{equation}
  since $f_0+f_1=f'$ and $f_1(\id_{G'}) \in \ZZ$ by definition.  In particular, whenever $x \in H^{s-1}$ and $y$ is an integer with $-N/2 < y \le N/2$, we have
  \[
    F(\wt r(x, y)) = e\big(\phi(x)\, y\big) .
  \]
  Similarly, note that $F$ is invariant under the right action of $(\id_{G'}, \Psi_1) \subseteq \wt \Gamma_1$, since multiplying $(g,w)$ on the right by $(0,\psi)$ for $\psi \in \Psi_1$ changes $w(\id_{G'})$ by an integer and leaves $f_0(g)$ unchanged.
  However, $F$ is not invariant under multiplication by $(\Gamma',0)$ and hence does not define a function on $\wt G_1/ \wt \Gamma_1$.

  Instead, we define a family of Lipschitz functions $F_z$ on $\wt G_1 / \wt \Gamma_1$ by localizing $F$ to some neighbourhood of $G'$ and then summing the action of $\Gamma'$.  Specifically let
  \begin{align*}
    T^{(\nu)} \colon \RR &\to [0,1]  \\
    t &\mapsto \max(0,\, 1 -|t|/\nu)
  \end{align*}
  denote a tent function centered at zero and supported on $(-\nu,\nu)$.  We recall that $G' = \RR \times G$.
  For any $z \in \bigoplus_{i=1}^{s-1} \RR^{d_i}$, define a function $T_z \colon G \to [0,1]$ by
  \[
    T_z(x) = \prod_{i=1}^{s-1} \prod_{j=1}^{d_i} T^{(1/10)} \big(\sigma_{G}(x)_{i,j} - z_{i,j}\big) ;
  \]
  i.e., a multi-dimensional tent function in the coordinate space of $G$, centered at $z$, supported on $B_z = \prod_{i,j} [z_{i,j}-1/2, z_{i,j}+1/2)$. %chktex 9
  Because $T_z$ is a Lipschitz function on the coordinate space $\bigoplus_{i=1}^{s-1} \RR^{d_i}$ supported on points close to $z$, by Proposition~\ref{prop:metric-compare} again we deduce that $T_z$ is Lipschitz on $G$, or that $((t,g),w) \mapsto T_z(g)$ is Lipschitz on $\wt G_1$, with parameter $O_s(M D)^{O_s(1)} (1 + \|z\|_1)^{O_s(1)}$ in each case.

  We then define $F_z$ on $\wt G_1 / \wt \Gamma_1$ by
  \begin{align*}
    F_z((t,g), w) &= \sum_{((t',g'), w') \in ((t,g),w) \ast (\Gamma', 0)} T^{(1)}(t')\, T_z(g')\, F\big((t', g'), w'\big) \\
    &= \sum_{\gamma \in \Gamma} \sum_{m \in \ZZ} T^{(1)}(m+t)\, T_z(g \gamma)\, F\big((t+m, g\gamma), \tau(w, (m,\gamma))\big) .
  \end{align*}
  Noting Proposition~\ref{prop:fun-dom}, each function $T_z$ is supported on a fundamental domain for $\Gamma$ and so at most one term of the sum over $\Gamma$ has a non-zero summand.  Similarly, at most two terms $m$ give a non-zero summand.  Since the right action of $(\id_{G'},\Psi_1)$ leaves every function in the sum unchanged, it follows that $F_z$ is invariant under $(\Gamma', 0) (\id_{G'},\Psi_1) = \wt \Gamma_1$, so $F_z$ is a well-defined function $\wt G_1 / \wt \Gamma_1 \to \CC$.

  Moreover, note $((t, g), w) \mapsto T^{(1)}(t)\, T_z(g)\, F\big((t,g),w)$ is Lipschitz with parameter $O_s(M D)^{O_s(1)} (1 + \|z\|_1)^{O_s(1)}$, by our previous bounds and noting it is supported near $z$.
  Since the metric $d_{\wt G}$ is right-invariant and and most two terms in the sum are non-zero, $F_z$ is Lipschitz with a parameter of the same form.

  Hence, each function $F_z \circ \big(\wt r \bmod \wt \Gamma_1\big)$ is a nilsequence with acceptable bounds on its parameters, provided $\|z\|_1$ is not too large.

  It suffices to show that for each $x \in A$ there is some bounded choice of $z$ such that $F_z(\wt r(x,y)) = e\big(\phi(x)\, y\big)$ for each $y \in H$.  To do this, set $z = \sigma_G(r(x))$; so $T_z(r(x)) = 1$ and $T_z(r(x) \gamma) = 0$ for all $\gamma \in \Gamma \setminus \{\id_G\}$.
  
  For for any $y \in H$, by~\eqref{eq:f-form} we have
  \[
    F_z\big(\wt r(x,y) \wt \Gamma_1\big) = \sum_{t \in y/N + \ZZ} T^{(1)}(t) e\big(N\, \phi(x)\, t\big)
  \]
  and since $\phi(x) \in \frac1N \ZZ$ by assumption and $\sum_{t \in y/N + \ZZ} T^{(1)}(t) = 1$, this is just $e\big(\phi(x)\, y\big)$, as required.
\end{proof}

\subsection{Correlation of $f$ with a sum of nilsequences}%
\label{subsec:antidiff}

With Lemma~\ref{lem:global-hs-nil}, we have succeeded in showing that the function $\cS_s f \colon H^s \to \CC$ in $s$ variables correlates with a linear combination of nilsequences $B$ on $H^s$.  We need to deduce a similar statement for the function $f \colon H \to \CC$ itself, rather than $\cS_s f$.

We briefly remark on the approaches taken in~\cite{green-tao-u3} and~\cite{gowers-milicevic} at the analogous stage.  In both cases, a ``symmetry argument'' is deployed to show that one can essentially reduce to the case that $B$ is a symmetric function of its arguments. This is reasonable given $\cS_s f$ itself is symmetric.  Then, $B$ can be explicitly anti-differentiated: i.e., a function in one variable is exhibited which has this structured piece as its $s$-fold derivative.  This is essentially sufficient.  A related set of issues are addressed in~\cite{gtz}.

In all cases the symmetry arguments are not straightforward.  Also, in the current setting the anti-differentiation step is also troublesome, both because $B$ is potentially a polynomial object rather than a truly multilinear one, and because general nilsequences are unpleasant to work with.

Instead, we attempt to recover a one-variable function by sampling $B$ along its diagonals, e.g.\ considering $g(x) = B(x,x,\dots,x)$ or more generally $B(x+a_1,\dots,x+a_s)$ for some fixed $a_1,\dots,a_s$.  In the case that $f$ was originally a true phase polynomial $e(P(x))$, we could recover $g(x) = e(\pm s!\, P(x))$, which is not correct but not far off, provided $s!$ is invertible in the group $H$.  This can be fixed by taking a product of many copies of $g$ and then applying a change of variables. %chktex 40

Specifically, we prove the following lemma.
\begin{lemma}%
  \label{lem:key-multilin-sampling}
  Suppose $|H|$ is coprime to $s!$.  If $f \colon H \to \CC$ is some function with $\|f\|_{\infty} \le 1$ and $B \colon H^s \to \CC$ is any function such that %chktex 40
  \[
    \left| \EEE_{h \in H^s} \cS_s f(h) \overline{B(h)} \right| \ge \eta
  \]
  then there exists a function $g \colon H \to \CC$ such that $\|f\, \overline{g}\|_{U^s} \ge \eta^{L}$ and
  \[
    g(x) = \prod_{i=1}^L B\big(h_i - (x/s!,x/s!,\dots,x/s!)\big)
  \]
  for some fixed $h_1,\dots,h_L \in H^s$, where $L = s!^{s-1}$.
\end{lemma}
The proof is elementary, insofar as the only ingredients are repeated applications of the Cauchy--Schwarz inequality and a small amount of Fourier analysis.

By Remark~\ref{rem:prod-nilseq} the nilsequence structure of $B$ is inherited directly by $g$, with some mild loss of parameters.  The conclusion of Lemma~\ref{lem:key-multilin-sampling} in terms of the $U^s$-norm is not quite what we want, and corresponds roughly to relating $f$ to a nilsequence up to as yet unknown lower-degree corrections (indeed, $\cS_s f$ does not carry such lower-degree information). However, we can handle this easily by applying Theorem~\ref{thm:main} for $s-1$ inductively to $f\, \overline{g}$.

\begin{proof}[Proof of Theorem~\ref{thm:main} assuming Lemma~\ref{lem:key-multilin-sampling}]
  Note the case $s=1$ of Theorem~\ref{thm:main} is an exercise in Fourier analysis (see e.g.\ \cite[(11.9)]{tao-vu}) so we assume $s \ge 2$.

  We take $B$, $\rho_i$, $\alpha_i$, $D$, $M$, $K$ and $T$ as in Lemma~\ref{lem:global-hs-nil}, and apply Lemma~\ref{lem:key-multilin-sampling} with this $B$ to get $g$, $L$ and $h_1,\dots,h_L$ as in Lemma~\ref{lem:key-multilin-sampling}.

  By applying Theorem~\ref{thm:main} inductively to $f\, \overline{g}$ with parameter $\eta^L$, we can find a nilsequence $\psi' \colon H \to \CC$ of degree $s-1$, dimension $D' \ll_s \eta^{O_s(1)}$, complexity $M'$ and parameter $K'$ such that
  \[
    \big| \EE_{x \in H} f(x)\, \overline{g}(x)\, \overline{\psi'(x)} \big| \ge \eps'
  \]
  where for $s \le 4$ we have
  \[
    {\eps'}^{-1},\, K',\, M' \le \exp\big(O(\eta)^{O(1)}\big)
  \]
  and for $s \ge 5$
  \[
    {\eps'}^{-1},\, K',\, M' \le \exp\exp\big(O_{s}(\delta)^{O_s(1)}\big) .
  \]
  By expanding $B = \sum_{j=1}^T \alpha_j \rho_j$ and applying the triangle inequality, we deduce that there are some $j_1,\dots,j_L \in [T]$ such that
  \[
    \left| \EE_{x \in H} f(x)\, \overline{\psi'(x)}\, \prod_{i=1}^L \phi_{j_i}\big(h_i - (x/s!,x/s!,\dots,x/s!)\big) \right| \ge \eps' / T^L .
  \]
  By Remark~\ref{rem:prod-nilseq}, and composing with the affine-linear maps $H \to H^s$ given by $x \mapsto h_i - (x/s!,\dots,x/s!)$, the function
  \[
    \psi'' \colon x \mapsto \prod_{i=1}^L \phi_{j_i}\big(h_i - (x/s!,\dots,x/s!)\big)
  \]
  is a nilsequence of degree $s$ dimension $L D$, complexity $M$ and parameter $K$.  By Remark~\ref{rem:prod-nilseq} again, $x \mapsto \psi''(x) \psi'(x)$ is again a nilsequence of degree $s$, dimension $L D + D'$, complexity $\max(M,M')$ and parameter $\max(K, K')$.  This proves the theorem.
\end{proof}

We will prove Lemma~\ref{lem:key-multilin-sampling} in two stages.

\begin{lemma}%
  \label{lem:multilin-step-1}
  Suppose $|H|$ is coprime to $s!$, $f \colon H \to \CC$ has $\|f\|_\infty \le 1$, and $B \colon H^s \to \CC$ satisfies %chktex 40
  \[
    \left| \EEE_{h \in H^s} \cS_s f(h) \overline{B(h)} \right| \ge \eta \ .
  \]
  Then there exists some $h_0 \in H^s$ such that, defining $g \colon H \to \CC$ by $g(x) = B(h_0 + (x,x,\dots,x))$, we have
  \[
    \left| \EEE_{h_1,\dots,h_s \in H} \cS_s f(h_1,\dots,h_s) \overline{\cS_s g(-h_1,-h_2/2,-h_3/3,\dots,-h_s/s)} \right| \ge \eta^{2^s} .
  \]
\end{lemma}
Note that this would be exactly what we want were it not for the factors $-1,-1/2,-1/3,\dots,-1/s$ in the iterated derivative of $g$.  Again, by considering the 100\% case $f(x) = e(P(x))$ where $P$ is a genuine polynomial $H \to \RR/\ZZ$ of degree $s$ and $B = \cS_s f$, we see that simply removing these factors yields a false statement, as the phases differ by a factor of $\pm s!$. %chktex 40

We fix some notation for the remainder of this subsection. For $h \in H$ and a function $\lambda \colon H \to \CC$ we write $\Delta_h \lambda$ for the multiplicative derivative, i.e.\ the function  $H \to \CC$ given by $x \mapsto \lambda(x) \overline{\lambda(x+h)}$. In particular the definition of $\cS_s \lambda$ is equivalent to
\begin{equation}
  \label{eq:cs-as-derivs}
  \cS_s \lambda (h_1,\dots,h_s) = \EE_{x \in H} \Delta_{h_1} \dots \Delta_{h_s} \lambda(x) .
\end{equation}

\begin{proof}[Proof of Lemma~\ref{lem:multilin-step-1}]
  We may expand the original expression to
  \[
    \eta \le \left| \EEE_{x \in H} \EEE_{h \in H^s} \overline{B(h)} \prod_{\omega \in \llbracket s \rrbracket} \cC^{|\omega|} f(x + \omega \cdot h) \right|
  \]
  and then make a redundant change of variables $h = r + (t,t,\dots,t)$ for $r \in H^s$ and $t \in H$, i.e.:
  \[
    \EEE_{x \in H} \EEE_{h \in H^s} \overline{B(h)} \prod_{\omega \in \llbracket s \rrbracket} \cC^{|\omega|} f(x + \omega \cdot h) = 
    \EEE_{r \in H^s} \left( \EEE_{x, t \in H} \overline{B(r + (t,t,\dots,t))} \prod_{\omega \in \llbracket s \rrbracket} \cC^{|\omega|} f\big(x + |\omega| t + \omega \cdot r\big) \right) .
  \]
  Hence, by the triangle inequality there exists some $r \in H^s$, which we now regard as fixed, such that
  \[
    \left| \EEE_{x,t \in H} \overline{B(r + (t,t,\dots,t))} \prod_{\omega \in \llbracket s \rrbracket} \cC^{|\omega|} f\big(x + |\omega| t + \omega \cdot r\big) \right| \ge \eta .
  \]
  We now group the $f$'s into $s+1$ classes according to the value of $|\omega|$.  Define functions $F_i \colon H \to \CC$ for $1 \le i \le s$ by
  \[
    F_i(y) = \cC^{i} \prod_{|\omega| = i} f(y + \omega \cdot r)
  \]
  and also let $g(x) = B(r + (x,x,\dots,x))$, so that
  \[
    \left| \EEE_{x,t \in H} \overline{g(t)} f(x) \prod_{i=1}^s F_i(x + i t) \right| \ge \eta .
  \]
  Finally, we eliminate the functions $F_i$ by $s$ applications of Cauchy--Schwarz.  This is a standard calculation---e.g.\ the proof of~\cite[Lemma 11.4]{tao-vu} just works---but is not quotable, so we give some details.

  The general Cauchy--Schwarz step is as follows.
  \begin{claim}
    If $k \ge 1$, for any functions $\alpha, \beta \colon H \to \CC$ and $F_1,\dots,F_k \colon H \to \CC$ with $\|F_i\|_\infty \le 1$ for each $i \in [k]$, we have
    \[
      \left| \EEE_{x,t \in H} \overline{\beta(t)} \alpha(x) \prod_{i=1}^k F_i(x + i t) \right| \le \left( \EEE_{u_1,\dots,u_k \in H} \EEE_{x,t \in H} \big(\Delta_{u_1} \dots \Delta_{u_k} \alpha(x)\big) \, \big(\overline{ \Delta_{u_1} \Delta_{u_2/2} \dots \Delta_{u_k/k} \beta(t)}\big) \right)^{1/2^k} .
    \]
  \end{claim}
  \begin{proof}[Proof of claim]
    By a change of variables $y=x+k t$ and applying Cauchy--Schwarz to $F_k(y)$, we have
    \[
      \LHS \le \left( \EEE_{y,x,x' \in H} \overline{\beta((y-x)/k)} \beta((y-x')/k) \alpha(x) \overline{\alpha(x')} \prod_{i=1}^{k-1} F_i(x + i(y-x)/k) \overline{F_i(x' + i(y-x')/k)} \right)^{1/2}
    \]
    (noting the quantity under the square root is non-negative).  Under a further change of variables $t=(y-x)/k$ and $u = x'-x$, the right hand side is exactly
    \[
      \left(  \EEE_{u \in H} \EEE_{x,t \in H} \Delta_u \alpha(x)\, \overline{\Delta_{-u/k} \beta(t)} \prod_{i=1}^{k-1} \Delta_{(1-i/k) u} F_i(x+it) \right)^{1/2} .
    \]
    We apply the claim inductively to the inner expectation for every $u \in H$, to bound this by
    \[
      \left( \EEE_{u \in H} \EEE_{u_1,\dots,u_{k-1} \in H} \EEE_{x,t \in H} \left(\Delta_{u_1} \dots \Delta_{u_k} \alpha(x) \, \overline{ \Delta_{u_1} \Delta_{u_2/2} \dots \Delta_{u_k/k} \beta(t)} \right)^{1/2^{k-1}} \right)^{1/2}
    \]
    and applying H\"older's inequality this gives the claim.
  \end{proof}
  Noting~\eqref{eq:cs-as-derivs}, this completes the proof of Lemma~\ref{lem:multilin-step-1}.
\end{proof}

The second stage is to eliminate these unwanted multipliers.  We will prove the following.
\begin{lemma}%
  \label{lem:multilin-step-2}
  Suppose $f, g \colon H \to \CC$ are two functions, and $\|f\|_\infty \le 1$.  Suppose further that $a_1,\dots,a_s$ and $k \ge 1$ are integers coprime to $|H|$, and $1 \le j \le s$.  Then there exists a function $g' \colon H \to \CC$ such that
  \[
    g'(x) = \prod_{i=1}^k g(x+h_i)
  \]
  for some constants $h_i \in H$, and
  \begin{align*}
    & \left| \EEE_{h_1,\dots,h_s \in H} \cS_s f(h_1,\dots,h_s) \overline{\cS_s g(h_1/a_1,h_2/a_2,\dots,h_s/a_s)} \right| \\
    & \le 
    \left| \EEE_{h_1,\dots,h_s \in H} \cS_s f(h_1,\dots,h_s) \overline{\cS_s g'(h_1/a_1,h_2/a_2,\dots,h_j/(k a_j),\dots,h_s/a_s)} \right|^{1/k} .
  \end{align*}
\end{lemma}

This, together with Lemma~\ref{lem:multilin-step-1}, allows us to prove Lemma~\ref{lem:key-multilin-sampling}.
\begin{proof}[Proof of Lemma~\ref{lem:key-multilin-sampling} assuming Lemma~\ref{lem:multilin-step-2}]
  Given $f$ and $B$ as in the statement, since
  \[
    \left| \EEE_{h \in H^s} \cS_s f(h) \overline{B(h)} \right| \ge \eta
  \]
  by Lemma~\ref{lem:multilin-step-1} we have 
  \[
    \left| \EEE_{u_1,\dots,u_s \in H} \cS_s f(u_1,\dots,u_s) \overline{\cS_s g_1(-u_1,-u_2/2,-u_3/3,\dots,-u_s/s)} \right| \ge \eta^{2^s}
  \]
  where $g_1(x) = B(h_0 + (x,x,\dots,x))$ for some $h_0 \in H^s$.  Applying Lemma~\ref{lem:multilin-step-2} for each $j$ between $1$ and $s$, with $k=s!/j$, we deduce
  \[
    \left| \EEE_{u_1,\dots,u_s \in H} \cS_s f(u_1,\dots,u_s) \overline{\cS_s g_2(-u_1/s!, -u_2/s!, -u_3/s!,\dots,-u_s/s!)} \right| \ge \eta^{2^s\, s!^{s-1}}
  \]
where $g_2(x) = \prod_{i=1}^{s!^{s-1}} g_1(x+v_i)$ for some constants $v_i \in H$.  Setting $g_3(x) = g_2(-x/s!)$, we obtain 
  \[
    \left| \EEE_{h \in H^s} \cS_s f(h) \overline{\cS_s g_3(h)} \right| \ge \eta^{2^s\, s!^{s-1}}
  \]
  and note that the left hand side is exactly
  \[
    \EEE_{t \in H} \left\| x \mapsto f(x) \overline{g_3(x+t)} \right\|_{U^s}^{2^s} .
  \]
  Hence, we may select some $t \in H$ such that $g(x) = g_3(x+t)$ has the form required by the statement and also $\left\|f\overline{g}\right\|_{U^s} \ge \eta^{s!^{s-1}}$, as required.
\end{proof}

To prove Lemma~\ref{lem:multilin-step-2}, we isolate the following fact, which is essentially the special case $s=1$ of Lemma~\ref{lem:multilin-step-2} itself.
\begin{lemma}%
  \label{lem:multilin-helper}
  Suppose $f,g \colon H \to \CC$ are two functions, $\|f\|_\infty \le 1$, and $a$ and $k \ge 1$ are integers coprime to $|H|$.  Then
  \[
    \EEE_{h\in H} \left(\EEE_{x \in H} \Delta_h f(x)\right) \left( \EEE_{y \in H} \Delta_{h/a} \overline{g(y)} \right) \le
    \left( \EEE_{r \in H^k} \EEE_{h \in H} \left(\EEE_x \Delta_h f(x)\right) \left( \EEE_y \Delta_{h/ka} \overline{g_r(y)} \right)  \right)^{1/k}
  \]
  with both sides of the inequality being non-negative real quantities, and where $g_r \colon H \to \CC$ for $r \in H^k$ is defined by
  \[
    g_r(x) = \prod_{i=1}^k g(x+r_i) .
  \]
\end{lemma}
\begin{proof}
  First note that
  \begin{equation}
    \label{eq:deriv-fourier}
    \EEE_{h\in H} \left(\EEE_{x \in H} \Delta_h f(x)\right) \overline{\left( \EEE_{y \in H} \Delta_{h/a} g(y) \right)} = \sum_{\chi \in \widehat{H}} |\widehat{f}(\chi)|^2 |\widehat{g}(a \chi)|^2
  \end{equation}
  which is in particular non-negative.  By convexity,
  \[
    \frac{\sum_{\chi \in \widehat{H}} |\widehat{f}(\chi)|^2 |\widehat{g}(a \chi)|^2}{\sum_{\chi \in \widehat{H}} |\widehat{f}(\chi)|^2} \le 
    \left( \frac{\sum_{\chi \in \widehat{H}} |\widehat{f}(\chi)|^2 |\widehat{g}(a \chi)|^{2k}}{\sum_{\chi \in \widehat{H}} |\widehat{f}(\chi)|^2}\right)^{1/k}
  \]
and as $\sum_{\chi \in \widehat{H}} |\widehat{f}(\chi)|^2 = \|f\|_2 \le 1$ we have that~\eqref{eq:deriv-fourier} is bounded by
  \[
    \left(\sum_{\chi \in \widehat{H}} |\widehat{f}(\chi)|^2 |\widehat{g}(a \chi)|^{2k}\right)^{1/k} .
  \]
  We also have the trivial inequality
  \[
    \sum_{\chi \in \widehat{H}} |\widehat{f}(\chi)|^2 |\widehat{g}(a \chi)|^{2k} \le 
    \sum_{\chi \in \widehat{H}} |\widehat{f}(\chi)|^2 \sum_{\substack{\xi_1,\dots,\xi_k \in \widehat{H} \\ \xi_1 + \cdots + \xi_k = k a \chi}} \prod_{i=1}^k |\widehat{g}(\xi_i)|^{2}
  \]
  since every term is non-negative and the terms on the left hand side form the subset of the terms on the right hand side for which $\xi_1=\cdots=\xi_k$.  Taking a Fourier transform again, the right hand side is exactly
  \[
    \EEE_{h \in H} \left(\EEE_{x \in H} \Delta_h f(x) \right) \left(\EEE_{y \in H} \Delta_{h/ka} g(y) \right)^k
  \]
  which, introducing one redundant average, is equal to the expression
  \[
    \EEE_{r \in H^k} \EEE_{h \in H} \left(\EEE_{x \in H} \Delta_h f(x)\right) \overline{\left( \EEE_{y \in H} \Delta_{h/ka} g_r(y) \right)} 
  \]
  as required.
\end{proof}

We now finish the proof of Lemma~\ref{lem:multilin-step-2}.

\begin{proof}[Proof of Lemma~\ref{lem:multilin-step-2}]
  Since the derivatives $\Delta_{h_i}$ commute, we may assume for simplicity that $j=s$.  For each choice of $h_1,\dots,h_{s-1}$, we may apply Lemma~\ref{lem:multilin-helper} to the functions
  \[
    F_{h_1,\dots,h_{s-1}}(x) = \Delta_{h_1} \dots \Delta_{h_{s-1}} f(x)
  \]
  and
  \[
    G_{h_1,\dots,h_{s-1}}(x) = \Delta_{h_1/a_1} \dots \Delta_{h_{s-1}/a_{s-1}} g(x)
  \]
  to obtain
  \begin{align}
    \nonumber
    & \EEE_{h_1,\dots,h_s \in H} \cS_s f(h_1,\dots,h_s) \overline{\cS_s g(h_1/a_1,h_2/a_2,\dots,h_s/a_s)} \\
    \nonumber
    &= \EEE_{h_1,\dots,h_{s-1} \in H} \EEE_{h \in H} \left( \EEE_{x \in H} \Delta_h F_{h_1,\dots,h_{s-1}}(x) \right) \left( \EEE_{y \in H} \Delta_{h/a_s} G_{h_1,\dots,h_{s-1}}(y) \right) \\
    \label{eq:multilin-step-2-1}
    &\le \EEE_{h_1,\dots,h_{s-1} \in H} \left(\EEE_{r \in H^k} \EEE_{h \in H} \left( \EEE_{x \in H} \Delta_h F_{h_1,\dots,h_{s-1}}(x) \right) \left( \EEE_{y \in H} \Delta_{h/ka_s} G^{(r)}_{h_1,\dots,h_{s-1}}(y) \right)\right)^{1/k}
  \end{align}
  where again the left hand side and each term $(\dots)^{1/k}$ are real and non-negative, and where
  \begin{align*}
    G_{h_1,\dots,h_{s-1}}^{(r)}(x) &= \prod_{i=1}^k G_{h_1,\dots,h_{s-1}}(x+r_i) \\
    &= \Delta_{h_1/a_1} \dots \Delta_{h_{s-1}/a_{s-1}} g_r(x)
  \end{align*}
  for $g_r(x) = \prod_{i=1}^k g(x+r_i)$, since $\Delta$ commutes with products and shifts.  By convexity,~\eqref{eq:multilin-step-2-1} is bounded by
  \begin{align*}
    &\left( \EEE_{h_1,\dots,h_{s-1} \in H} \EEE_{r \in H^k} \EEE_{h \in H} \left( \EEE_{x \in H} \Delta_h F_{h_1,\dots,h_{s-1}}(x) \right) \left( \EEE_{y \in H} \Delta_{h/ka_s} G^{(r)}_{h_1,\dots,h_{s-1}}(y) \right)\right)^{1/k} \\
    &= \left( \EEE_{r \in H^k} \EEE_{h_1,\dots,h_s \in H} \cS_s f(h_1,\dots,h_s) \overline{\cS_s g_r(h_1/a_1,\dots,h_s/ka_s)} \right)^{1/k} \ .
  \end{align*}
  Specializing the average to some particular $r \in H^k$, the stated result follows.
\end{proof}

\appendix
\counterwithout{lemma}{subsection}
\counterwithin{lemma}{section}
\section{Expansion and robust connected components}

We recall from Section~\ref{subsec:split} the definition of the normalized Cheeger constant of a graph.
The following lemma roughly states that the induced subgraph of $\cG$ on a randomly chosen set of at least $C h(\cG)^{-C}$ vertices is almost surely connected.
\begin{lemma}%
  \label{lem:subgraph-connected}
  Let $\cG=(V,E)$ be a graph with Cheeger constant $h = h(\cG) \le 1$.  Let $x_1,\dots,x_k \in V$ be vertices chosen independently and uniformly at random, and consider the graph $\cG'=(V',E')$ where $V' = [k]$ and $E' = \{ij \colon 1 \le i < j \le k,\, x_i x_j \in E\}$.\footnote{If the $x_i$ are distinct, which almost surely they are when $|V|$ is large compared to the other parameters, this is just the induced subgraph $G[\{x_1,\dots,x_k\}]$.}  Then
  \[
    \PP\big[\cG' \text{ is connected}\big] \ge 1 - (k+1) \exp\Big(-(k-1) h^4 / 64 \log\big(2/h^2\big) \Big) \ .
  \]
\end{lemma}
Although this lemma should be standard in some form, the author has not been able to find it in quotable form.  The general framework of this proof was suggested to the author by Jacob Fox.
\begin{proof}
  We run a two-stage argument.  In the first stage, we expose the first $\ell$ elements $x_1,\dots,x_\ell$ and find (almost surely) a subsequence $x_{i_1},\dots,x_{i_m}$ where $1 \le i_1 < \cdots < i_m \le \ell$, which induces a connected subgraph, and such that $\left|\bigcup_{j=1}^m \Gamma_{\cG}(x_{i_j}) \right| \ge (1 - \delta) |V|$ for some small $\delta$; i.e., almost every $v \in V$ is a neighbour of some element in this subsequence.  In the second stage, we expose $x_{\ell+1},\dots,x_k$, and claim that almost surely the following hold: each vertex from $x_{\ell+1},\dots,x_k$ has at least one neighbour in $x_{i_1},\dots,x_{i_m}$, and every vertex in $x_1,\dots,x_\ell$ has at least one neighbor in $x_{\ell+1},\dots,x_k$.  These conditions certainly imply that $\cG'$ is connected.

  Formally, set $S_1 = \Gamma(x_1)$, and for each $2 \le i \le \ell$, set
  \[
    S_i = \begin{cases} S_{i-1} \cup \Gamma(x_i) &\colon x_i \in S_{i-1} \\ S_{i-1} &\colon x_i \notin S_{i-1} \end{cases}
  \]
  and let $1 = i_1 < i_2 < \cdots < i_m \le \ell$ denote those indices for which $x_i \in S_{i-1}$ and the former case holds, together with $i=1$.  Then it is clear by construction that $S_{r} \subseteq \bigcup_{i_j \le r} \Gamma_{\cG}(x_{i_j})$ and that for each $1<r \le m$ the vertex $x_{i_r}$ has at least one neighbor in $x_{i_1},\dots,x_{i_{r-1}}$.  Hence the induced subgraph of $\cG'$ on $i_1,\dots,i_m$ is connected. Our next goal is to show that $S_\ell$ is almost all of $V$ with high probability.

  To do this, we roughly show a lower bound on the typical increase in size of $S_i$ over $S_{i-1}$, and use submartingale inequalities to bound the probability that $S_\ell$ is not basically everything. Note that
  \[
    \EE_{y \in S_{i-1}} |\Gamma(y) \setminus S_{i-1}| = \frac{|E(S_{i-1},\, V \setminus S_{i-1})|}{|S_{i-1}|}
  \]
  and so
  \[
    \EE_{y \in V} [y \in S_{i-1}] \,|\Gamma(y) \setminus S_{i-1}| = \frac{|E(S_{i-1},\, V \setminus S_{i-1})|}{|V|} \ge h \min(|S_{i-1}|, |V| - |S_{i-1}|) \ .
  \]
  Given that $|S_1| \ge h |V|$, we may bound the right hand side below by $h^2 (|V| - |S_{i-1}|)$, and so we get a conditional expectation bound
  \[
    \EE\big[ |S_i \setminus S_{i-1}|\ \big|\ S_1,\dots,S_{i-1} \big] \ge h^2 (|V| - |S_{i-1}|)
  \]
  for each $2 \le i \le \ell$.  Setting $S_0 = \emptyset$ the same bound holds for $i=1$ (as $|S_1| \ge h |V|$ always).

  If we define $T_i = 1 - |S_i| / |V|$ for $0 \le i \le \ell$, i.e.\ the proportion of vertices still missing from $S_i$ (with $T_0 = 1$), then $T_{i} \le T_{i-1}$ for all $1 \le i \le \ell$ and the bound above translates as
  \[
    \EE\big[ T_i\ \big|\ T_1,\dots,T_{i-1} \big] \le (1 - h^2) T_{i-1} .
  \]
  Changing variables again and letting $R_i = T_i / T_{i-1}$ for $1 \le i \ge \ell$, we have $0 \le R_i \le 1$ and 
  \[
    \EE\big[ R_i \ \big|\ R_1,R_2,\dots,R_{i-1} \big] \le 1 - h^2 .
  \]
  For technical reasons we will also define a variant where the ratio $R_i$ is capped below, i.e.\ setting $R'_i = \max(h^2/2, R_i)$. It follows that $R_i \le R'_i \le R_i + h^2/2$ and so
  \[
    \EE \big[R'_i \ \big| \ R_1,R_2,\dots,R_{i-1} \big] \le 1 - h^2/2 .
  \]
  Furthermore we have $h^2/2 \le R'_i \le 1$.  For our final change of variables to obtain a submartingale, let $L_0 = 0$ and $L_i = i \log(1 - h^2/2) - \sum_{j=1}^i \log R'_j $; so, $\log(1-h^2/2) \le L_i - L_{i-1} \le -\log(h^2/2)$, and
  \[
    \EE \big[L_i\ \big|\ L_1,\dots,L_{i-1} \big] = L_{i-1} - \EE \big[ \log R'_i\ \big|\ L_1,\dots,L_{i-1}] + \log(1 - h^2/2) \ge L_{i-1}
  \]
  (by convexity of $\log$) and hence $L_0,L_1,\dots,L_\ell$ do indeed form a submartingale.  Note $|\log(1-h^2/2)| \le \log 2 \le |\log(h^2/2)|$.  By the Azuma--Hoeffding inequality, we deduce that
  \[
    \PP[L_\ell \le -t] \le \exp\big(-t^2 / 2 \ell \log(2/h^2) \big) .
  \]
  We now recover information about previous variables. The value of $T_i$ is given by
  \[
    T_i = \prod_{j=1}^i R_i \le \prod_{j=1}^i R'_i = \exp\big(-L_i + i \log(1 - h^2/2) \big)
  \]
  and so we deduce that
  \[
    \PP\big[T_\ell \ge (1-h^2/2)^{\ell} e^t\big] \le \exp\big(-t^2 / 2 \ell \log(2/h^2) \big) .
  \]
  Setting $t = -(\ell/2) \log (1-h^2/2)$, it follows that
  \[
    \PP\big[T_\ell \ge (1-h^2/2)^{\ell/2} \big] \le \exp\left(-\frac{\ell \log(1-h^2/2)^2}{8 \log(2/h^2)} \right) \le \exp\left(-\ell (h^4 / 32 \log(2/h^2)) \right) \ .
  \]
  We write $\delta = (1-h^2/2)^{\ell/2}$ and $p = \exp\left(-\ell (h^4 / 32 \log(2/h^2))\right)$; so, we have shown that $|S_\ell| \ge (1 - \delta) |V|$ with probability at least $1-p$. We can further bound $\delta \le \exp\big(-\ell h^2/4\big)$.

  We now treat $x_1,\dots,x_\ell$ as fixed subject to the event $|S_\ell| \ge (1-\delta)|V|$, and enter the second stage where we expose $x_{\ell+1},\dots,x_k$.
  
  With probability at least $1 - (k-\ell) \delta$, each $x_i$ for $\ell+1 \le i \le k$ lies in $S_\ell$, i.e., has a neighbor among $x_{i_1},\dots,x_{i_m}$.
  Also, since each vertex $v \in \{x_1,\dots,x_\ell\}$ has degree at least $h |V|$, the probability (in $x_{\ell+1},\dots,x_k$) that $v$ has at least one neighbor among $x_{\ell+1},\dots,x_k$ is at least $1-(1-h)^{k-\ell}$.

  Hence, the probability that the graph is not connected is bounded by
  \[
    p + (k-\ell) \delta + \ell (1-h)^{k-\ell} .
  \]
  Setting $\ell = \lceil k/2 \rceil$, note that each of $p$, $\delta$ and $(1-h)^{k-\ell} \le \exp(-h(k-1)/2)$ is certainly bounded by $\exp \big(-(k-1) h^4 / 64 \log(2/h^2)\big)$, and so the total failure probability is at most
  \[
    (k+1) \exp\big(-(k-1) h^4 / 64 \log(2/h^2) \big)
  \]
  as required.
\end{proof}

\section{Discrete Cauchy--Schwarz and Cauchy--Schwarz complexity}%
\label{app:cs}

We supply the missing proof of the discrete Cauchy--Schwarz complexity result (Lemma~\ref{lem:gvn}) used above.

The proof is analogous to the proof that systems of Cauchy--Schwarz complexity $\le t$ are controlled by the $U^{t+1}$-norm (implicit in~\cite{gt-primes}, or~\cite[Theorem 2.3]{gowers-wolf}), only replacing every application of the Cauchy--Schwarz inequality with an appeal to the following lemma.

\begin{lemma}[Discrete Cauchy--Schwarz surrogate]%
  \label{lem:discrete-cs}
  Suppose $X$, $Y$ are finite sets, $Z$ is any set, and $\pi \colon X \to Y$ and $f \colon X \to Z$, $g \colon Y \to Z$ are functions.  If
  \[
    \big|\big\{ x \in X \colon f(x) = g(\pi(x)) \big\}\big| \ge \delta |X|
  \]
  then
  \[
    \big|\big\{ (x,x') \in X^2 \colon \pi(x) = \pi(x'),\, f(x) = f(x')\big\}\big| \ge \delta^2 |X|^2/|Y| \ .
  \]
  If furthermore every non-empty fiber $\pi^{-1}(y)$ for $y \in Y$ has the same size, we can rewrite this as
  \[
    \big|\big\{ (x,x') \in X^2 \colon \pi(x) = \pi(x'),\, f(x) = f(x')\big\}\big| \ge \delta^2 \big|\big\{ (x,x') \in X^2 \colon \pi(x) = \pi(x')\big\}\big| .
  \]
\end{lemma}
\begin{proof}
  We have
  \begin{align*}
    \delta |X| \le \big|\big\{ x \in X \colon f(x) = g(\pi(x)) \big\}\big| &= \sum_{y \in Y} \big|\big\{ x \in \pi^{-1}(y) \colon f(x) = g(y)\big\}\big|  \\
    &\le \sum_{y \in Y} \big|\big\{ (x,x') \in (\pi^{-1}(y))^2 \colon f(x) = f(x') \big\}\big|^{1/2}
  \end{align*}
  since the latter set contains all pairs $(x,x')$ with $f(x) = f(x') = g(y)$.  By Cauchy--Schwarz, we have
  \[
    \LHS \le |Y|^{1/2} \left(\sum_{y \in Y} \big|\big\{ (x,x') \in (\pi^{-1}(y))^2 \colon f(x) = f(x') \big\}\big|\right)^{1/2}
  \]
  and rearranging gives the first stated result.  For the second, note that under the assumption on the fibers,
  \[
    \big|\big\{(x,x') \in X^2 \colon \pi(x) = \pi(x')\big\}\big| = \sum_{y \in Y} |\pi^{-1}(y)|^2
    = \frac1{|Y|} \left( \sum_{y \in Y} |\pi^{-1}(y)| \right)^2
    = |X|^2/|Y|
  \]
  by the converse to Cauchy--Schwarz.
\end{proof}

From now on we write $X \times_\pi X$ to denote the set $\{ (x,x') \in X^2 \colon \pi(x) = \pi(x')\}$.
The next lemma captures what happens when we apply this bound a number of times, in the setting of abelian groups.

\begin{lemma}%
  \label{lem:multidim-cs}
  Suppose $k \ge 0$, $X$ is a finite abelian group, $K_1,\dots,K_k \subseteq X$ are subgroups, $Z$ is any abelian group, and $f \colon X \to Z$ and $g_i \colon X / K_i \to Z$ for $i \in [k]$ are any functions.  Let $\wt X$ denote the subset of $X^{\llbracket k\rrbracket}$ given by
  \[
    \wt X = \big\{ \omega \mapsto x + \omega_1 h_1 + \cdots + \omega_k h_k \colon x \in X, \, h_i \in K_i \big\} .
  \]
  If
  \[
    \left|\left\{ x \in X \colon f(x) = \sum_{i=1}^k g_i(x + K_i) \right\} \right| \ge \delta |X|
  \]
  then
  \[
    \big|\big\{ c \in \wt X \colon \partial^{k} f(c) = 0\big\}\big| \ge \delta^{2^k} | \wt X| .
  \]
\end{lemma}
\begin{proof}
  We apply induction on $k$; the case $k=0$ is trivial so we fix $k \ge 1$.  Write $\pi_i \colon X \to X/K_i$ for the quotient map, and set $F(x) = f(x) - \sum_{i=1}^{k-1} g_i(x + K_i)$.  Then
  \[
    \big|\big\{ x \in X \colon F(x) = g_k(x+K_k) \big\}\big| \ge \delta |X|  
  \]
  and so by Lemma~\ref{lem:discrete-cs} we have
  \[
    \big|\big\{ (x,x') \in X \times_{\pi_k} X \colon F(x) - F(x') = 0 \big\}\big| \ge \delta^2\, |X \times_{\pi_k} X| \ ,
  \]
  noting that the condition on fibers is automatic for group homomorphisms.  If we define
  $X' = X \times_{\pi_k} X$, or equivalently
  \[
    X' = \big\{ (x, x+h_k) \colon x \in X,\, h_k \in K_k \big\},
  \]
  as well as subgroups $K'_i = X' \cap (K_i \times K_i)$, and functions $f'(x,x') = f(x) - f(x')$ for $(x,x') \in X'$ and
  \[
    g'_i\big((x,x') + K'_i\big) = g_i(x + K_i) - g_i(x' + K_i),
  \]
  then we have that
  \[
    \left|\left\{ y \in X' \colon f'(y) = \sum_{i=1}^{k-1} g_i'(y + K'_i) \right\}\right| \ge \delta^2 |X'|
  \]
  and so by inductive hypothesis,
  \[
    \big|\big\{ c \in \wt{X'} \colon \partial^{k-1} f'(c) = 0\big\}\big| \ge \delta^{2^k} \big|\wt{X'}\big| .
  \]
  However, it is straightforward to verify that $\wt{X'} = \wt X$ and $\partial^{k-1} f' = \partial^k f$, so the result follows.
\end{proof}

\begin{proof}[Proof of Lemma~\ref{lem:gvn}]
  We use the notation of Definition~\ref{def:cauchy-schwarz-complexity} and the statement.  In particular, recall we are given functions $f_1,\dots,f_k \to \RR$, linear forms $\phi_1,\dots,\phi_k \colon \ZZ^d \to \ZZ$, sets $\Sigma_1,\dots,\Sigma_{t+1}$ that cover $[k] \setminus \{j\}$, and elements $\sigma_1,\dots,\sigma_{t+1} \in \ZZ^d$ witnessing the Cauchy--Schwarz complexity of $\phi_1,\dots,\phi_k$ at index $j \in [k]$.  Note that without loss of generality we may assume $\Sigma_i$ are disjoint.

  Set $X = H^{d}$, and for $r \in [t+1]$ consider the subgroup $K_r = \{ h \sigma_r \colon h \in H\} \subseteq H^d$ and function
  \begin{align*}
    g_r \colon H^d / K_r &\to \RR \\
    x + K_r &\mapsto \sum_{i \in \Sigma_r} -f_i(\phi_i(x))
  \end{align*}
  which is well-defined as $\phi_i(\sigma_r) = 0$ for $i \in \Sigma_r$ by assumption.  Finally write $f(x) = f_j(\phi_j(x))$.  It follows that
  \[
    \left|\left\{x \in H^d \colon f(x) = \sum_{r=1}^{t+1} g_r(x + K_r) \right\}\right| \ge \delta |X|
  \]
  and so applying Lemma~\ref{lem:multidim-cs} we deduce that for
  \[
    \wt X = \big\{ \omega \mapsto x + \omega_1 h_1 \sigma_1 + \cdots + \omega_{t+1} h_{t+1} \sigma_{t+1} \colon x \in H^d, \, h_i \in H \big\} \subseteq C^{t+1}(X)
  \]
  we have
  \[
    \big|\big\{ c \in \wt X \colon \partial^{t+1} f_j(\phi_j(c)) = 0\big\}\big| \ge \delta^{2^{t+1}} |\wt X|.
  \]
  However, we note that
  \[
    \big\{ \phi_j(c) \colon c \in \wt X\big\} = \big\{ \omega \mapsto \phi_j(x) + \omega_1 q_1 h_1 + \cdots + \omega_{t+1} q_{t+1} h_{t+1} \colon x \in H^d, \, h_1,\dots,h_{t+1} \in H \big\}
  \]
  and since $(q_r,|H|) = 1$ for each $r$, this is precisely $C^{t+1}(H)$.  It follows that
  \[
    \big|\big\{ c \in C^{t+1}(H) \colon \partial^{t+1} f_j(c) = 0\big\}\big| \ge \delta^{2^{t+1}} |C^{t+1}(H)| .
  \]
  as required.
\end{proof}

\section{Algebraic definitions and facts concerning nilmanifolds and related objects}%
\label{app:nilmanifolds}

The theory of nilmanifolds is now an established part of additive combinatorics and ergodic theory, and accounts of the essential theory exist in the literature~\cite[\S 1.6]{tao-higher}\ \cite[Appendix A]{gmv-1}\ \cite[Appendix B]{gtz}\ \cite[Appendix E]{gt-primes}.  Given this, the purpose of this appendix is twofold:---
\begin{itemize}
  \item to fix the definitions and conventions to be used in this work, not all of which are standard; and
  \item to provide statements and in some cases proofs from the algebraic theory where these are not available in the standard literature or not straightforward to extract from it.
\end{itemize}
Little of this content is actually novel, although it may be idiosyncratic.  On the other hand, this is expressly not intended as a pedagogical introduction to the theory, with motivating examples etc., as such already exist and as this is not the role of the current work.

\subsection{Filtered groups and Host--Kra cubes}

The following are by now fairly standard concepts.

\begin{definition}[Filtered group]
  If $G$ is a group, a \emph{degree $s$ filtration} on $G$ is a decreasing sequence $G = G_0 \supseteq G_1 \supseteq \dots$ of subgroups of $G$ with $G_k = \{\id_G\}$ for $k>s$, with the property that if $g \in G_i$ and $h \in G_j$ then the commutator $[g,h] = g^{-1} h^{-1} g h$ lies in $G_{i+j}$.  The filtration is called \emph{proper} if $G_0 = G_1$.
  
  If $G$ has a topology, we insist $G_i$ are closed.
\end{definition}
We sometimes write $G_\bullet$ to emphasize that a particular filtration on $G$ has been specified.  Note that a proper filtered group is necessarily nilpotent.

\begin{definition}%
  If $H$ is an abelian group, the \emph{standard filtration} on $H$ is to take $H = H_0 = H_1$ and $H_i = \{0\}$ for $i \ge 2$.  On the other hand, we write $H_{(s)}$ to denote the \emph{standard degree $s$ filtration} $H = H_0 = \cdots = H_s$ and $H_{i} = \{0\}$ for $i \ge s+1$.
\end{definition}

\begin{definition}[Host--Kra cubes]%
  \label{def:hk-group}
  Given a filtered group $G_\bullet$ and an integer $k \ge 0$, the \emph{Host--Kra cube group} $\HK^k(G_\bullet)$ is the subgroup of $G^{\llbracket k \rrbracket}$ (with pointwise group operations) generated by elements of the form
  \begin{align*}
    [g]_F \colon \llbracket k \rrbracket &\to G \\
    \omega &\mapsto \begin{cases} g &\colon g \in F \\ \id_G &\colon g \notin F \end{cases}
  \end{align*}
  where $g \in G_i$ and $F$ is a face of $\llbracket k \rrbracket$ of codimension at most $i$.
\end{definition}

\begin{remark}
  If $H$ is an abelian group considered with the standard filtration, $\HK^k(H) = C^k(H)$, where $C^k(H)$ is defined as in~\eqref{eq:cube-definition}.
\end{remark}

\subsection{Nilmanifolds and complexity}

We now discuss nilmanifolds.  For current purposes, we always consider these objects to come with a certain amount of extra data, and to have some quantitative control on their complexity.  The data we refer to is referred to elsewhere in the literature as a set of \emph{Mal'cev coordinates of the second kind}. To avoid generalities, we give a fairly unsophisticated definition, which is equivalent (but not totally obviously so) to other conventions, notably~\cite[Definition 2.1]{gt-quantitative}.

Our main definition describes a class of filtered nilpotent groups which admit a low-complexity description.
\begin{definition}[Explicitly presented filtered group]%
  \label{def:explicit-group}
  Let $s \ge 1$ and $d_1,\dots,d_s \ge 0$ be integers, and write $d = (d_1,\dots,d_s)$.
  We say $G_\bullet$ is an \emph{explicitly presented filtered group of degree $s$, dimension $d$ and complexity $M$} if it is a proper filtered Lie group of degree $s$; and the following hold.
  \begin{enumerate}[label=(\roman*)]
    \item We are given homomorphisms $\gamma_{i,j} \colon \RR \to G$ (i.e., one-parameter subgroups) for each $i \in [s]$ and $j \in [d_i]$. We denote these by $\gamma_{i,j}^a$ for $a \in \RR$ (instead of $\gamma_{i,j}(a)$), and abbreviate $\gamma_{i,j}^1 = \gamma_{i,j}$.
    \item Every $g \in G$ has a unique representation
      $g = \prod_{i=1}^s \prod_{j=1}^{d_i} \gamma_{i,j}^{a_{i,j}}$
      for $a_{i,j} \in \RR$, where the product has $i$ and $j$ in increasing order from left to right.  We write
        $\sigma \colon G \to \bigoplus_{i=1}^s \RR^{d_i}$
      for the associated bijection sending $g$ to the tuple of real numbers $a_{i,j}$.  We require further that $\sigma$ is a homeomorphism.
    \item For each $r \in [s]$, the subgroup $G_{r} \subseteq G$ in the filtration consists precisely of those $g \in G$ such that $\sigma(g)_{i,j} = 0$ for all $i<r$ and $j \in [d_i]$.
    \item For each $i,j$ and $i',j'$, the coordinates $\sigma([\gamma_{i,j}, \gamma_{i',j'}])$ of the commutator $[\gamma_{i,j},\gamma_{i',j'}] \in G_{i+i'}$ are all integers bounded by $M$.
  \end{enumerate}
  The \emph{integral subgroup} of such a $G_\bullet$ is the set
  \[
    \Gamma = \big\{ \gamma \in G \colon \sigma(\gamma)_{i,j} \in \ZZ  \text{ for all } i,j \big\}
  \]
  and we write $\Gamma_i = G_i \cap \Gamma$, yielding a filtered group $\Gamma_\bullet$.
\end{definition}

We note the following.
\begin{proposition}[$\Gamma$ and the fundamental domain]%
  \label{prop:fun-dom}
  In the setting of Definition~\ref{def:explicit-group}, $\Gamma$ is indeed a subgroup of $G$.  Moreover, for each $i$, every element $g \in G_i$ has a unique representation $g = h \gamma$ where $\gamma \in \Gamma_i$, $h \in G_i$ and $\sigma(h)_{i',j'} \in [0,1)$ for each $i',j'$. Hence $\Gamma_i$ is co-compact in $G_i$. %chktex 9

  Moreover, the same holds replacing the fundamental domain $\prod_{i',j'} [0,1)$ in $\bigoplus_{i=1}^s \RR^{d_i}$ with a shifted version $\prod_{i',j'} [a_{i',j'}, a_{i',j'}+1)$, for any real numbers $a_{i',j'}$. %chktex 9
\end{proposition}
\begin{proof}
  By (ii) it is clear $\Gamma$ is contained in the subgroup generated by the elements $\gamma_{i,j}$. Conversely, given a word in $\gamma_{i,j}$ can be reordered so that $(i,j)$ occur in increasing order from left to right, by repeated use of identity $\gamma_{i',j'} \gamma_{i,j} = \gamma_{i,j} \gamma_{i',j'} [\gamma_{i,j}, \gamma_{i',j'}]^{-1}$, and using the filtration property and an inductive argument to deal with the lower-order correction $[\gamma_{i,j}, \gamma_{i',j'}]$.

  The second statement is proven similarly: given $g \in G_i$ we can set $h = \prod_{j=1}^{d_i} \gamma_{i,j}^{\{\sigma(g)_{i,j}\}}$, $\gamma' = \prod_{j=1}^{d_i} \gamma_{i,j}^{\lfloor \sigma(g)_{i,j} \rfloor}$ where $\{\cdot\}$, $\lfloor \cdot \rfloor$ denote the fractional and integer parts of a real number, and note that $g = h g' \gamma'$ for some $g' \in G_{i+1}$.  The claimed fact then follows by induction on $i$.

  The extension to a customized fundamental domain is routine.
\end{proof}

With this done, we consider nilmanifolds.
\begin{definition}[Nilmanifold of bounded complexity]%
  \label{def:nilmanifold-complex}
  For $s \ge 1$ and $d = (d_1,\dots,d_s)$ as above, a \emph{nilmanifold of degree $s$, dimension $d$ and complexity $M$} is precisely a space of the form $G/\Gamma$, where $G_\bullet$ is an explicitly presented filtered group of degree $s$, dimension $d$ and complexity $M$, and $\Gamma$ is its integral subgroup.

  (If we wish to be less precise, we may say the dimension of $G/\Gamma$ is $D = \sum_{i=1}^s d_i$.)
\end{definition}

Again, it is true but not obvious that every filtered nilmanifold arising as in e.g.\ \cite[\S 8]{gt-primes} has finite complexity in this sense, but we do not have to be concerned with results of this nature as our constructions always supply the full data.

We also recall a notion of Host--Kra cubes on a nilmanifold.
\begin{definition}[Host--Kra cubes on nilmanifolds]%
  \label{def:hk-nm}
  Let $G_{\bullet}$ be an explicitly presented filtered group and $G/\Gamma$ the associated nilmanifold.  For $k \ge 0$ we define $\HK^k(G/\Gamma) \subseteq (G/\Gamma)^{\llbracket k \rrbracket}$ to consist of those configurations $\omega \mapsto x(\omega) \in G/\Gamma$ with the property that for some $y \in \HK^k(G_\bullet)$ (defined as in Definition~\ref{def:hk-group}) we have $x(\omega) = y(\omega) \Gamma$ for each $\omega \in \llbracket k \rrbracket$.

  In other words, $\HK^k(G/\Gamma)$ is the image of $\HK^k(G_\bullet)$ under the projection $G \to G/\Gamma$ applied at each coordinate.
\end{definition}

We also briefly mention some properties of these Host--Kra cubes.
\begin{proposition}
  If $G_{\bullet}$ is an explicitly presented filtered group, then for $k \ge 0$ we have that $\HK^k(\Gamma_\bullet) \subseteq \HK^k(G_\bullet)$ is a discrete and co-compact subgroup; that is, $\HK^k(\Gamma_\bullet)$ is discrete and the homogeneous space $\HK^k(G_\bullet) / \HK^k(\Gamma_\bullet)$ is compact.  Moreover there is a natural continuous injective map
  \[
    \HK^k(G_\bullet) / \HK^k(\Gamma_{\bullet}) \to (G/\Gamma)^{\llbracket k \rrbracket}
  \]
  whose image is $\HK^k(G/\Gamma)$; hence, $\HK^k(G/\Gamma)$ is compact and the map induces a homeomorphism $\HK^k(G_\bullet) / \HK^k(\Gamma_{\bullet}) \cong \HK^k(G/\Gamma)$.
\end{proposition}
\begin{proof}
  Given Proposition~\ref{prop:fun-dom}, co-compactness of $\HK^k(\Gamma_\bullet)$ in $\HK^k(G_\bullet)$ is discussed in~\cite[Lemma E.10]{gt-primes}.  We note that $\HK^k(\Gamma_\bullet) = \HK^k(G_\bullet) \cap \Gamma^{\llbracket k \rrbracket}$; this is not totally obvious but follows immediately from a ``face decomposition''; see~\cite[(E.1)]{gtz} or~\cite[Proposition A.5]{gmv-1}.  It follows that the map given is well-defined and injective, and it is continuous by definition.  The remaining statements are standard consequences in topology. 
\end{proof}

\subsection{Polynomial maps}

Central to the theory of nilmanifolds and related objects are \emph{polynomial maps}, which are effectively the morphisms in this category.  Possibly unconventionally, we define them as follows.

\begin{definition}[Polynomial map]%
  \label{def:poly-map}
  If $X$ and $Y$ are each either a filtered group or finite complexity nilmanifold, and $p \colon X \to Y$ is some continuous function, we say $p$ is a \emph{polynomial map} if the induced map $p \colon X^{\llbracket k \rrbracket} \to Y^{\llbracket k \rrbracket}$ has $p\big(\HK^k(X)\big) \subseteq \HK^k(Y)$, for each $k \ge 0$, where $\HK^k(X)$, $\HK^k(Y)$ are defined as in Definition~\ref{def:hk-group} or Definition~\ref{def:hk-nm}.

  If $A$ is an abelian group, we say $p \colon X \to A$ is a polynomial map \emph{of degree $s$} to mean the same as a polynomial map $X \to A_{(s)}$.

  We write $\poly(X,Y)$ for the collection of all polynomial maps $X \to Y$.  
\end{definition}

In other words, polynomial maps are maps which send Host--Kra cubes to Host--Kra cubes.  In the case of polynomial maps between groups, there is an equivalent alternative definition.

\begin{proposition}%
  \label{prop:poly-equiv}
  Let $H_\bullet$ and $G_\bullet$ be filtered groups and $p \colon H \to G$ a continuous function.  For $h \in H$ we write $\partial_h p \colon H \to G$ for the map $x \mapsto p(x) p(h x)^{-1}$.  Then $p$ is a polynomial map if and only if the following holds: for each $m \ge 0$,  $i_1,\dots,i_m \ge 0$, $h_1 \in H_{i_1},\dots,h_m \in H_{i_m}$ and $x \in H$ we have $\partial_{h_1} \dots \partial_{h_m} p(x) \in G_{i_1+\cdots+i_m}$.
\end{proposition}
\begin{proof}
  This is a special case of~\cite[Theorem B.3]{gtz}.
\end{proof}

\begin{example}%
  \label{ex:poly-examples}
  It is clear that any group homomorphism $\phi \colon G \to G'$ between filtered groups $G_\bullet$ and $G'_\bullet$ is a polynomial map if and only if $\phi(G_i) \subseteq G'_i$ for each $i \ge 0$ (although polynomial maps do not need to be group homomorphisms in general).

  Any constant function $X \to Y$ is certainly a polynomial map, as constant elements of $Y^{\llbracket k \rrbracket}$ are always Host--Kra cubes.  The identity map $X \to X$ is also trivially a polynomial map.

  If $G_\bullet$ is a filtered group and $p_1,p_2 \colon X \to G$ are polynomial maps then so is the pointwise product $x \mapsto p_1(x) p_2(x)$ (since Host--Kra cubes on $G_\bullet$ are groups under pointwise multiplication).

  Hence e.g.\ the maps $G \to G$ given by $x \mapsto g x$ or $x \mapsto x g$ for fixed $g \in G$, or $x \mapsto x^2$, are all polynomial maps, as is $x \mapsto x^{-1}$.

  Finally, it is immediate with these definitions that a composite of two polynomial maps is polynomial.
\end{example}

For the main argument, we will need a number of facts about polynomial maps between filtered groups in various settings.  The next proposition shows that (at least in cases we care about) polynomial maps on filtered groups do not significantly generalize polynomial maps on abelian filtered groups.

\begin{proposition}%
  \label{prop:filtered-triv}
  Let $G_\bullet$ be an explicitly presented filtered group.  Consider the coordinate bijection $\sigma \colon G \to \bigoplus_{i=1}^s \RR^{d_i}$, and interpret the codomain as a filtered abelian group $G'=G'_0=G'_1 \supseteq G'_2 \supseteq \dots$, where $G'_r$ is the abelian group $\bigoplus_{i=r}^s \RR^{d_i}$. 

  Then $\sigma \colon G_\bullet \to G'_\bullet$ and $\sigma^{-1} \colon G'_\bullet \to G_\bullet$ are polynomial maps.
\end{proposition}
\begin{proof}
  We claim that a configuration $c \colon \llbracket k \rrbracket \to G$ lies in $\HK^k(G_\bullet)$ if and only if there are real numbers $a_{i,j,F}$ and an associated decomposition
  \begin{equation}
    \label{eq:cdecomp}
    c = \prod_{i=1}^s \prod_{j=1}^{d_i} \prod_{\codim(F) \le i} \big[\gamma_{i,j}^{a_{i,j,F}}\big]_F
  \end{equation}
  in $\HK^k(G_\bullet)$ (again, with pointwise group operations) where again the order of products is increasing from left to right, the third product is over all faces $F$ of $\llbracket k \rrbracket$ with the specified codimension (in any order), and the notation $[g]_F$ is as in Definition~\ref{def:hk-group}.
  
  The ``if'' direction is immediate.  For the ``only if'' part, we observe the commutator identity
  \begin{equation}
    \label{eq:face-commutator}
    \big[[g]_F, [h]_{F'} \big] = \big[ [g,h] \big]_{F \cap F'}
  \end{equation}
  for any $g,h \in G$ and any faces $F,F'$ of $\llbracket k \rrbracket$.  Hence, we are free to reorder any product of elements $[g]_F$ at the expense of introducing commutators of the same form which lie further down the filtration of $G$, and by an inductive process this proves the claim.

  Similarly, the configuration $\sigma(c) \colon \llbracket k \rrbracket \to G'$ lies in $\HK^k(G'_\bullet)$ if and only if there exist real numbers $a_{i,j,F}$ such that
  \begin{equation}
    \label{eq:sigmacdecomp}
    \sigma(c) = \sum_{i=1}^s \sum_{j=1}^{d_i} \sum_{\substack{F \\ \codim(F) \le i}} \big[\sigma\big(\gamma_{i,j}^{a_{i,j,F}}\big)\big]_F
  \end{equation}
  where we use additive notation for the group operation on $G'$: this time the group is abelian so we can reorder terms arbitrarily, and hence this is immediate from the definition of $\HK^k(G'_\bullet)$.

  Finally, note that for fixed coefficients $a_{i,j,F} \in \RR$,~\eqref{eq:cdecomp} and~\eqref{eq:sigmacdecomp} are equivalent, by the definition of $\sigma$. Therefore $c \in \HK^k(G_\bullet)$ if and only if $\sigma(c) \in \HK^k(G'_\bullet)$, as required.
\end{proof}

The next result gives an explicit characterization of polynomial maps on certain classes of filtered abelian group.  It is a variant of~\cite[Lemma B.9]{gtz}.

\begin{proposition}[``Taylor expansion'']%
  \label{prop:taylor}
  Let $d_1,\dots,d_s \ge 0$ be integers, consider the abelian group $H = \bigoplus_{i=1}^s \ZZ^{d_i}$ with filtration $H_r = \bigoplus_{i=r}^s \ZZ^{d_i}$, and let $A$ an arbitrary abelian group.  If $(k_{i,j})_{i \in [s], j \in [d_i]}$ is a tuple of non-negative integers and $v \in H$, define the integer coefficients
  \begin{equation} \label{eq:gen-binom}
    \binom{v}{k} = \prod_{i=1}^s \prod_{j=1}^{d_i} \binom{v_{i,j}}{k_{i,j}}
  \end{equation}
  \uppar{where the binomial coefficient $\binom{n}{r}$ is generalized to negative $n$, or even real numbers $n$, in the obvious way} and also write
  \begin{equation}
    \label{eq:w-def}
    w(k) = \sum_{i=1}^s i \sum_{j=1}^{d_i} k_{i,j} .
  \end{equation}
  Then for $t \ge 0$, the set $\poly\big(H, A_{(t)}\big)$ consists precisely of maps $p \colon H \to A$ of the form
  \begin{equation}
    \label{eq:taylor}
    p(v) = \sum_{k \colon w(k) \le t} \alpha_k \binom{v}{k}
  \end{equation}
  for elements $\alpha_k \in A$ determined uniquely by $p$, where the sum ranges over all non-negative integer tuples $k$ of the above shape.
\end{proposition}
\begin{proof}
  We write $e^{(i,j)}$ for the tuple $e$ with $e_{i,j} = 1$ and all other entries zero.  In particular it can be thought of as an element of $H$.

  For one inclusion, it suffices to show that $v \mapsto \binom{v}{k}$ is a polynomial map $H \to \ZZ_{(t)}$ whenever $w(k) \le t$; it follows directly that scalar multiples $v \mapsto a \binom{v}{k} \in A$, and sums of these, are also polynomial maps.  The former follows from~\cite[Proposition B.8]{gtz} and the following identity: whenever $k_{i,j} > 0$,
  \[
    \partial_{e^{(i,j)}} \binom{v}{k} = -\binom{v}{k - e^{(i,j)}}.
  \]

  Now suppose $p \colon H \to A_{(t)}$ is an arbitrary polynomial map.  To determine the structure of $p$ we deploy the ``integration'' argument outlined in~\cite[Lemma B.9]{gtz}.
  For any function $f \colon H \to A$, write $T^h f$ for the shifted function $x \mapsto f(x+h)$; so, $\partial_h f = (1 - T^h) f$ in this abelian setting.  It follows that for a non-negative integer $r$ we have
  \[
    T^{r e^{(i,j)}} p = \big(1 - \partial_{e^{(i,j)}}\big)^r p = \sum_{u=0}^t \binom{r}{u} (-1)^{u} \partial_{e^{(i,j)}}^u p
  \]
  where we may stop the sum at $t$ irrespective of $r$ since $\partial_{e^{(i,j)}}^u p \equiv 0$ for $u \ge t+1$ by hypothesis.  By reasoning formally in the ring $\langle \partial_{e^{(i,j)}} \rangle \cong \ZZ[X] / \big(X^{t+1}\big)$, the argument can be extended to give the same conclusion for negative $r$.  Therefore, for $h \in H$,
  \[
    T^h p = \left( \prod_{i,j} T^{h_{i,j} e^{(i,j)}} \right) p = \sum_{k \colon w(k) \le t} (-1)^{\sum_{i,j} k_{i,j}} \binom{h}{k} \left(\prod_{i,j} \partial_{e^{(i,j)}}^{k_{i,j}} \right) p
  \]
  and evaluating at $0$ gives
  \[
    p(h) = \sum_{k \colon w(k) \le t} \binom{h}{k} \left[ (-1)^{\sum_{i,j} k_{i,j}}  \left(\prod_{i,j} \partial_{e^{(i,j)}}^{k_{i,j}}\right) p(0) \right]
  \]
  which has the required form, setting
  \begin{equation}
    \label{eq:alpha-form}
    \alpha_k = (-1)^{\sum_{i,j} k_{i,j}}  \left(\prod_{i,j} \partial_{e^{(i,j)}}^{k_{i,j}}\right) p(0) .
  \end{equation}

  Finally, for uniqueness of the coefficients $\alpha_k$, suppose that
  \[
    \sum_{k \colon w(k) \le t} \alpha_k \binom{v}{k} = 0
  \]
  for some coefficients $\alpha_k \in A$ that are not all zero, and specifically suppose $k_0$ is a minimal index with respect to the usual partial ordering on non-negative tuples such that $\alpha_{k_0} \ne 0$.  But then
  \[
    \sum_{k \colon w(k) \le t} \alpha_k \binom{k_0}{k} = \alpha_{k_0}
  \]
  which gives a contradiction.
\end{proof}

This explicit result has a number of useful consequences.

\begin{corollary}%
  \label{cor:rd-extend}
  Let $H$ be as in Proposition~\ref{prop:taylor} and let $\cH = H \otimes \RR$, i.e.\ $\cH = \bigoplus_{i=1}^s \RR^{d_i}$ is an abelian group with filtration $\cH_r = \bigoplus_{i=r}^s \RR^{d_i}$.  We identify $H$ as the subgroup of $\cH$ consisting of integer tuples.

  If $t \ge 0$ and $p \in \poly\big(H, \RR_{(t)}\big)$, then $p$ can be extended uniquely to a polynomial map $\cH \to \RR_{(t)}$ via~\eqref{eq:taylor}.
\end{corollary}
\begin{proof}
  By treating binomial coefficients $\binom{\cdot}{r}$ as functions $\RR \to \RR$, we can interpret~\eqref{eq:taylor} as a function $\cH \to \RR_{(t)}$ extending $p$, which is clearly continuous and a polynomial map (indeed, it is a polynomial in the usual sense of the term, with the necessary bounds on the degrees of each coordinate of $\cH$).  For uniqueness, note that for any positive integer $N$ the extension of $p$ to a polynomial map $\frac1N H \to \RR_{(t)}$ implied by~\eqref{eq:taylor} is unique, by the uniqueness part of Proposition~\ref{prop:taylor} applied to $\frac1N H \cong H$. Then send $N \to \infty$ and invoke continuity.
\end{proof}

We can also use these results to argue that the group operations on an explicitly presented filtered group are given by low-complexity polynomial functions.

\begin{corollary}%
  \label{cor:explicit-is-easy}
  Let $G$ be an explicitly presented filtered group of degree $s$, dimension $d$ and complexity $M$, and write $G' = \bigoplus_{i=1}^s \RR^{d_i}$ with the filtration as in Proposition~\ref{prop:filtered-triv}.  
  
  Consider the maps $\phi_1, \phi_2 \colon G' \times G' \to G'$ and $\phi_3 \colon G' \to G'$ given by
  \begin{align*}
    \phi_1(x,y) &= \sigma(\sigma^{-1}(x) \sigma^{-1}(y)) \\
    \phi_2(x,y) &= \sigma([\sigma^{-1}(x), \sigma^{-1}(y)]) \\
    \phi_3(x) &= \sigma\big(\sigma^{-1}(x)^{-1}\big) ;
  \end{align*}
  i.e., various group operations on $G$ under the coordinate map.

  Then each entry $\phi_1(x,y)_{i,j}$, $\phi_2(x,y)_{i,j}$, $\phi_3(x)_{i,j}$ is a polynomial function of $x$ and $y$ of the form~\eqref{eq:taylor}, where the coefficients $\alpha_k$ \uppar{for $w(k) \le i$} are integers bounded by $O_s\big((DM)^{O_s(1)}\big)$, where $D = \sum_{i=1}^s d_i$.
\end{corollary}
\begin{proof}
  By Example~\ref{ex:poly-examples} and Proposition~\ref{prop:filtered-triv}, each function $\phi_1(-,-)_{i,j}$, $\phi_2(-,-)_{i,j}$ or $\phi_3(-)_{i,j}$ is a polynomial map $G' \times G' \to \RR_{(i)}$ or $G' \to \RR_{(i)}$.  Hence by Proposition~\ref{prop:taylor} and Corollary~\ref{cor:rd-extend}, these functions $p$ have the form~\eqref{eq:taylor}, with coefficients $\alpha_k$ defined by~\eqref{eq:alpha-form}.  In particular, $\alpha_k$ is an $O_s(1)$-bounded integer combination of values $p(v)$ where $v$ ranges over integer tuples with each entry in $\{0,\dots,s\}$.  Hence it suffices to show that $\|p(v)\|_1 \ll_s (DM)^{O_s(1)}$ for such tuples $v$.
  
  In all cases, we have reduced to showing that words of length $O_s(D)$ in the group elements $\gamma_{i,j}$ (with integer exponents) can be rexpressed as $\prod_{i,j} \gamma_{i,j}^{a_{i,j}}$ where $a_{i,j} \in \ZZ$ obey $\sum_{i,j} |a_{i,j}| \ll_s (D M)^{O_s(1)}$.  This follows from iterative use of commutator identities and the hypothesis on the basic commutators $[\gamma_{i,j},\gamma_{i',j'}]$.
\end{proof}

We noted above that our definition of a polynomial map $H \to G/\Gamma$ when (say) $H$ is a group and $G/\Gamma$ is a nilmanifold, may be unconventional.  When $H = \ZZ^A$, an alternative common in the literature is to say that a polynomial map $\ZZ^A \to G/\Gamma$ is the same thing as the composite of a polynomial map $\ZZ^A \to G$ (in the sense of Definition~\ref{def:hk-group}) with the projection $G \to G/\Gamma$; or not to define the former notation at all.  It turns out these are equivalent.

\begin{proposition}%
  \label{prop:z-lift}
  If $A \ge 0$ is an integer and $p \colon \ZZ^A \to G/\Gamma$ is a polynomial map, then there exists a polynomial map $\wt{p} \colon \ZZ^A \to G$ such that $p(x) = \wt{p}(x) \Gamma$ for all $x \in \ZZ^A$.
\end{proposition}
\begin{proof}
  If $G = \RR_{(t)}^k$ and $\Gamma = \ZZ_{(t)}^k$ for some $k \ge 0$, we can apply Proposition~\ref{prop:taylor} with $A = (\RR/\ZZ)_{(t)}$ to find coefficients $\alpha_k \in \RR/\ZZ$ such that~\eqref{eq:taylor} holds.  Lifting $\alpha_k$ arbitrarily to elements of $\RR$ gives a polynomial map $\ZZ^A \to \RR_{(t)}$ as required.

  Suppose $G/\Gamma$ has degree $s$.  For induction, suppose that $1 \le i \le s+1$ and that $p(x) \in G_i/\Gamma_i$ for all $x \in \ZZ^A$. If $i=s+1$ then we can just set $\wt{p}(x) = \id_G$ for all $x \in \ZZ^A$, and when $i=1$ we have not assumed anything new.

  Consider the restricted coordinate map $\sigma_{=i} \colon G_i \to \RR^{d_i}$.  By the properties of the explicitly presented filtered group $G$, this induces a polynomial map $\sigma_{=i} \bmod \ZZ^{d_i} \colon G_i / \Gamma_i \to (\RR^{d_i} / \ZZ^{d_i})_{(i)}$.  Composing and applying the remark above, we can find a polynomial map $p' \colon \ZZ^A \to \RR^{d_i}$ such that $\sigma_{=i}(p(x)) = p'(x) \bmod \ZZ^{d_i}$ for all $x \in \ZZ^A$.

  By Proposition~\ref{prop:filtered-triv}, $\sigma^{-1} \circ p' \colon \ZZ^A \to G$ is again a polynomial map.  Moreover, $p''(x) = p'(x)^{-1} p(x) \Gamma$ is a polynomial map $\ZZ^A \to G_{i+1} / \Gamma_{i+1}$. Applying induction, we can lift $p''$ to a polynomial map $\wt{p''} \colon \ZZ^A \to G$, and then set $\wt{p}(x) = p'(x) \wt{p''}(x)$ to complete the proof.
\end{proof}

If $H$ is instead a finite group, polynomial maps $H \to G$ are constant functions and this ceases to be a particularly useful way to think about polynomial maps $H \to G/\Gamma$.

\subsection{Metrics}

Given an explicitly presented filtered group $G_\bullet$, there are many ways to define metrics on $G$ or $G/\Gamma$ that generate the correct topology, and broadly they will given comparable results on bounded neighbourhoods of the identity in $G$.  For concreteness we make the following definitions.

\begin{definition}
  If $G_\bullet$ is an explicitly presented filtered group, the \emph{right-invariant word metric} $d_G$ on $G$ is defined by $d_G(x,y) = d_G(\id, y x^{-1})$ and
  \[
    d_G(\id, g) = \inf \left\{ \sum_{r=1}^K |a_r| \colon g = \gamma_{i_1,j_1}^{a_1} \dots \gamma_{i_K,j_K}^{a_K} \right\}
  \]
  i.e., the smallest $\ell^1$-norm of the exponents in a representation of $g$ as a ``word'' in the one-parameter subgroups $\gamma_{i,j}^{a}$.

  The associated metric on $G/\Gamma$ is defined to be
  \[
    d_{G/\Gamma}(x,y) = \inf \big\{ d_G(g, g') \colon g \Gamma = x,\ g' \Gamma = y \big\} .
  \]
\end{definition}

Another distance function on $G$ would be to use the $\ell^1$-distance (say) under the coordinate map, $(x,y) \mapsto \|\sigma(x) - \sigma(y)\|_1$.  This does not have good invariance properties.  However, it is sometimes easier to work with, and does control $d_G$ above and below up to polynomial losses.

\begin{proposition}%
  \label{prop:metric-compare}
  Let $G_\bullet$ be an explicitly presented filtered group with total degree $D = \sum_{i=1}^s d_i$, coordinate map $\sigma \colon G \to \bigoplus_{i=1}^s \RR^{d_i}$ and complexity $M$. For all $g \in G$ we have $d_G(\id, g) \le \|\sigma(g)\|_1$ and
  \[
    \|\sigma(g)\|_1 \ll_s (M D)^{O_s(1)} d_G(1,g) \big(1 + d_G(\id_G,g)\big)^{O_s(1)} . 
  \]
  Furthermore, for $g,g' \in G$ we have
  \[
    \|\sigma(g) - \sigma(g')\|_1 \ll_s (M D)^{O_s(1)} d_G(g,g') \big(1 + d_G(\id_G,g) + d_G(\id_G, g')\big)^{O_s(1)} .
  \]
\end{proposition}
\begin{proof}
  The bound $d_G(\id, g) \le \|\sigma(g)\|_1$ is immediate from the definitions of $d_G$ and $\sigma$.
  
  By Corollary~\ref{cor:explicit-is-easy} we have the bound
  \[
    \|\sigma([x,y])\|_1 \ll_s (DM)^{O_s(1)} \|\sigma(x)\|_1 \|\sigma(y)\|_1 (1 + \|\sigma(x)\|_1 + \|\sigma(y)\|_1)^{O_s(1)}
  \]
  where we use the observation that the polynomial $\sigma(x), \sigma(y) \mapsto \sigma([x,y])$ as in Corollary~\ref{cor:explicit-is-easy} vanishes when $\sigma(x) = 0$ or $\sigma(y) =0$, and hence every term has degree at least $1$ in some coordinate of each of $x$ and $y$.  In particular we note this when $x=\gamma_{i,j}^a$, $y=\gamma_{i',j'}^{a'}$.

  We make the following claim.
  \begin{claim}
    For any $a_1,\dots,a_K \in \RR$ and tuples $(i_1,j_1),\dots,(i_K,j_K)$,
    \[
      \big\|\sigma\big(\gamma_{i_1,j_1}^{a_1} \dots \gamma_{i_K,j_K}^{a_K} \big)\big\|_1 \ll_s (M D)^{O_s(1)} \left(\sum_{r=1}^K |a_r| \right) \left(1 + \sum_{r=1}^K |a_r| \right)^{O_s(1)} .
    \]
  \end{claim}
  \begin{proof}[Proof of claim]
    We reorder the terms $\gamma_{i_r,j_r}^{a_r}$ into increasing order of $i$ and $j$, incurring commutators farther down the filtration, and keep track of the total weights in each piece of the filtration:
    \[
      B_{i} = \sum_{r \in [K], i_r = i} |a_r|
    \]
    as the word evolves.  Specifically, we repeatedly choose a term $\gamma_{i',j'}^{a'}$ with $i'$ as small as possible, and $j'$ as small as possible for that $i'$, whose correct position in the ordering lies to its left, and move it as far left as possible.

    After doing this once, the weights of the resulting word are controlled as
    \[
      B'_{i} \le \begin{cases} B_{i} &\colon i < 2 i' \\ B_{i} + |a'| O(MD)^{O_s(1)} (1 + |a'|)^{O_s(1)} \left(1 + \sum_{j=i'}^{i-i'} B_j\right)^{O_s(1)} &\colon i \ge 2 i' .\end{cases}
    \]
    Hence, after moving all the terms with a particular value of $i'$ leftward to their correct positions, the weights $B'_{i}$ afterwards crudely obey the bound
    \[
      B'_i \ll_s (M D)^{O_s(1)} \left(\sum_{j=1}^s B_j\right) \left(1 + \sum_{j=1}^s B_{j}\right)^{O_s(1)} .
    \]
    We only have to do this $s$ times.  At the end, the final sum $\sum_{i=1}^s B'_i$ is exactly $\big\|\sigma\big(\gamma_{i_1,j_1}^{a_1} \dots \gamma_{i_K,j_K}^{a_K}  \big)\big\|_1$, and the claim follows.
  \end{proof}
  The bound on $\|\sigma(g)\|_1$ follows immediately from the claim, by taking an infimum over all words with $\gamma_{i_1,j_1}^{a_1} \dots \gamma_{i_K,j_K}^{a_K} =g$.

  We now consider the third bound from the statement.  We note that
  \[
    \| \sigma(x) - \sigma(y x)\|_1 \ll_s (M D)^{O_s(1)} \|\sigma(y)\|_1 \big(1 + \|\sigma(x)\|_1 + \|\sigma(y)\|_1 \big)^{O_s(1)}
  \]
  by Corollary~\ref{cor:explicit-is-easy} again, where we again observe that the polynomial $\sigma(x), \sigma(y) \mapsto \sigma(x) - \sigma(y x)$ vanishes when $\sigma(x) = 0$, and so every term has degree at least $1$ in $\sigma(y)$.  Therefore, setting $x=g$ and $y = g' g^{-1}$,
  \[
    \| \sigma(g) - \sigma(g') \| \ll_s (M D)^{O_s(1)} \|\sigma(g' g^{-1})\|_1 \big(1 + \|\sigma(g)\|_1 + \|\sigma(g' g^{-1})\|_1 \big)^{O_s(1)}
  \]
  and by the previous bounds applied to the right hand side we get
  \[
    \| \sigma(g) - \sigma(g') \| \ll_s (M D)^{O_s(1)} d_G(\id_G, g' g^{-1}) \big(1 + d_G(\id, g) + d_G(\id, g' g^{-1}) \big)^{O_s(1)}
  \]
  and since $d_G$ is right-invariant and a metric this gives the result.
\end{proof}

\subsection{Miscellaneous results}%
\label{appsub:misc}

We recall that, by definition, for any Host--Kra cube $c \in \HK^k(G/\Gamma)$ there is some $\wt c \in \HK^k(G)$ whose image under the projection $G \to G/\Gamma$ is $c$.  In the spirit of previous results, we will need to know that we can always choose such a lift $\wt c$ whose elements are not too large.

\begin{proposition}%
  \label{prop:effective-cubes}
  If $G$ is an explicitly presented filtered group of degree $s$, dimension $D$ and complexity $M$ and $c \in \HK^k(G/\Gamma)$ is a Host--Kra cube, then there exists $\wt{c} \in \HK^k(G)$ such that $\wt{c}(\omega) \Gamma = c(\omega)$ for each $\omega \in \llbracket k \rrbracket$, and moreover $\max_{\omega \in \llbracket k \rrbracket} d_G(\id_G, \wt{c}(\omega)) \ll_{k} D$.
\end{proposition}
\begin{proof}
  By Definition~\ref{def:hk-nm} there is some $\wt{c}' \in \HK^k(G)$ with $\wt{c}'(\omega) \Gamma = c(\omega)$ for each $\omega$.  Reordering terms as in the proof of Proposition~\ref{prop:filtered-triv}, we can write
  \[
    \wt{c}' = \prod_{i=1}^s \prod_{j=1}^{d_i} \prod_{\codim(F) \le i} \big[\gamma_{i,j}^{a_{i,j,F}}\big]_F
  \]
  as in~\eqref{eq:cdecomp}.  As in the proof of Proposition~\ref{prop:fun-dom}, and using~\eqref{eq:face-commutator}, in decreasing order of $i$ we can multiply on the right by an expression
  \[
    \prod_{j=1}^{d_i} \prod_{\codim(F) \le i} \big[\gamma_{i,j}^{-\alpha_{i,j,F}}\big]_F
  \]
  where $\alpha_{i,j,F}$ are integers, to replace $a_{i,j,F}$ with real numbers in $[0,1)$, and then absorb commutators in the parameters $a_{i',j,F}$ for $i' > i$.  When this process terminates, the exponents all lie in $[0,1)$ and the distance bound follows.%chktex 9
\end{proof}

We also record the proof of Proposition~\ref{prop:poly-space}, which incorporates much of the work of this appendix.
\begin{proof}[Proof of Proposition~\ref{prop:poly-space}]%
  Part (i) is covered in Example~\ref{ex:poly-examples}, and (ii) and (v) are immediate e.g.\ from Proposition~\ref{prop:poly-equiv}.
  
  Write $d$, $d'$ for the dimension sequences of $G$ and $G'$ respectively, and write $\sigma(G) = \bigoplus_{i=1}^s \RR^{d_i}$ and $\sigma'(G') = \bigoplus_{i=1}^{s'} \RR^{d'_i}$, where $\sigma$, $\sigma'$ are coordinate maps.  These are considered as filtered abelian groups, as in Proposition~\ref{prop:filtered-triv}.
  
  By Proposition~\ref{prop:filtered-triv}, $\poly(G, G')$ as a set is in bijection with $\poly(G, \sigma'(G'))$ or $\poly(\sigma(G), \sigma'(G'))$, since $\sigma$ and $\sigma'$ are polynomial isomorphisms.  The same holds for $\poly(G, G')_i$ and $\poly(\sigma(G), \sigma'(G'))_i$, and for $\poly(\Gamma, G')$ and $\poly(\sigma(\Gamma), \sigma'(G'))$, etc..
  
  For $X= G_\bullet$ or $\Gamma_\bullet$, a map $p \colon X \to \sigma'(G)$ is polynomial if and only if each restriction $p_{=i'} \colon X \to \RR^{d'_{i'}}_{(i')}$ is polynomial.  With Corollary~\ref{cor:rd-extend} this proves (iii).
  
  By Proposition~\ref{prop:taylor}, $p_{=i'}$ is polynomial if and only if $p_{=i'}(x) = \sum_{\substack{w(k) \le i'}} \alpha_{i',k} \binom{x}{k}$ for some $\alpha_{i',k} \in \RR^{d'_{i'}}$.  Moreover, $p_{=i'}$ lies in $\poly\Big(X, \RR^{d'_{i'}}_{(i')}\Big)_{i}$ if and only if $\alpha_{i',k}$ are supported on values with $w(k) \le i'-i$ (where $w(k)$ is as in~\eqref{eq:w-def}).
  
  Hence, a natural basis for $\poly(\sigma(G), \sigma'(G'))$ consists of elements $\wt{p}_{i',j',k}$, corresponding to $\alpha_{i',k} = e_{j'}$, the $j'$-th basis vector in $\RR^{d'_{i'}}$, and all other terms zero.  These give rise to  one-parameter subgroups $p_{i',j',k}^a \colon G \to G'$ given by
  \[
    p_{i',j',k}^a(x) = {\gamma'}_{i,j}^{\binom{\sigma(x)}{k}}
  \]
  and it follows that every polynomial map $G \to G'$ has the form
  \begin{equation}
    \label{eq:poly-decomp}
    x \mapsto \prod_{i',j',k} {\gamma'}_{i',j'}^{a_{i',j',k} \binom{\sigma(x)}{k}} = \prod_{i',j',k} p_{i',j',k}^{a_{i',j',k}}
  \end{equation}
  for some $a_{i',j',k} \in \RR$, and again this lies in $\poly(G, G')_{i}$ if and only if $\alpha_{i',j',k}$ is non-zero only when $w(k) \le i'-i$, and in $\poly(\Gamma, \Gamma')$ if and only if $a_{i',j',k}$ are integers.
  
  For $p(x)$ of the form~\eqref{eq:poly-decomp}, we recall the value $a_{i',j',k}$ is given explicitly by~\eqref{eq:alpha-form}, and hence is a linear combination of values $\sigma'\big(p(\sigma^{-1}(\ell))\big)$, as $\ell$ ranges over non-negative tuples with $\ell_{i,j} \le k_{i,j}$ for each $i,j$, with coefficients bounded by $O_{s'}(1)$.  This implies (vii).
  
  It also follows from this that if we expand $x \mapsto \big[p_{i'_1,j'_1,k_1}(x), p_{i'_2,j'_2,k_2}(x)\big]$ as in~\eqref{eq:poly-decomp}, the integer coefficients $a_{i',j',k}$ are necessarily bounded by $O_{s'}(D' M')^{O_{s'}(1)}$, since its value at any such argument $\sigma^{-1}(\ell)$ has coordinates bounded by $O_{s'}(D' M')^{O_{s'}(1)}$ by Corollary~\ref{cor:explicit-is-easy}.

  If we relabel the functions $p_{i',j',k}$, grouping by the value of $i = i' - w(k)$, this gives functions $p_{i,j}$ with all the properties needed for (iv).  It is immediate from our previous remarks that $\{p_{i,j} \colon i \ge 1\}$ make $\poly(G, G')_1$ (with filtration $\poly(G,G')_1 = \poly(G,G')_1 \supseteq \poly(G,G')_2 \supseteq \dots$) into an explicitly presented nilpotent group with the bounds required, giving (iv).

  Part (vi) follows from the formula for $a_{i,j}$ in terms of $p(\sigma^{-1}(\ell))$ above; the bound on $\|\sigma(x g)\|_1$ or $\|\sigma(g x)\|_1$ in terms of $\|\sigma(x)\|_1$ and $\|\sigma(g)\|_1$ given by Corollary~\ref{cor:explicit-is-easy}; and Proposition~\ref{prop:metric-compare}.

  Finally, (viii) is immediate from a quantitative continuity statement for the multivariate binomial coefficient $x \mapsto \binom{x}{k}$, and Proposition~\ref{prop:metric-compare} again.
\end{proof}

\bibliography{master}{}
\bibliographystyle{amsalpha}
\end{document}